\documentclass[a4paper, 11pt]{article} 

\usepackage{bbm}
\usepackage{enumitem}

\usepackage[utf8]{inputenc}
\usepackage[T1]{fontenc}
\usepackage{lmodern}
\usepackage{amsmath, amsfonts, amssymb, amsthm}

\usepackage[colorlinks=true, pagebackref]{hyperref} 

\usepackage{nonumonpart}

\usepackage{graphicx, subfig}
\usepackage{color} 
\usepackage[dvipsnames]{xcolor}

\usepackage[square]{natbib}
\usepackage{mathtools}

\usepackage[english]{babel}

\usepackage{xstring} 

\usepackage{xspace} 

\usepackage{fancyhdr}	

\usepackage{xr}	

\usepackage[hmargin=3.5cm, vmargin=3cm]{geometry}

\newtheorem{theoreme}{Theorem}[section]
\newtheorem{lemme}[theoreme]{Lemma}
\newtheorem{proposition}[theoreme]{Proposition}
\newtheorem{definition}[theoreme]{Definition}
\newtheorem{hypothese}[theoreme]{Assumption}
\newtheorem{remarque}[theoreme]{Remark}

\newtheorem{corollaire}[theoreme]{Corollary}
\newtheorem{example}[theoreme]{Example}

\newcommand{\ba}{\begin{aligned}}
\newcommand{\ea}{\end{aligned}}

\newcommand{\reeq}[1]{Equation~\eqref{eq:#1}}

\newcommand{\resec}[1]{Section~\ref{sec:#1}}
\newcommand{\redef}[1]{Definition~\ref{def:#1}}
\newcommand{\rehyp}[1]{Assumption~\ref{hyp:#1}}
\newcommand{\rerem}[1]{Remark~\ref{rem:#1}}

\newcommand{\relem}[1]{Lemma~\ref{lem:#1}}
\newcommand{\recor}[1]{Corollary~\ref{cor:#1}}

\newcommand{\reprop}[1]{Proposition~\ref{prop:#1}}

\newcommand{\rethm}[1]{Theorem~\ref{thm:#1}}

\newcommand{\relemdeux}[2]{Lemmas~\ref{lem:#1} and~\ref{lem:#2}}

\newcommand{\rehypdeux}[2]{Assumptions~\ref{hyp:#1} and~\ref{hyp:#2}}

\newcommand{\resecdeux}[2]{Sections~\ref{sec:#1} and~\ref{sec:#2}}

\newcommand{\remarq}[1]{%
\begin{remarque}%
#1
\end{remarque}%
}%

\DeclareMathOperator{\id}{Id}

\newcommand{\eps}{\varepsilon}

 \newcommand{\go}[1]{O\left( #1 \right)}
 \newcommand{\po}[1]{o\left( #1 \right)}

\newcommand{\paren}[1]{\left( #1 \right)}

\newcommand{\acco}[1]{\left\lbrace #1 \right\rbrace}
\newcommand{\abs}[1]{\left\vert #1 \right\vert}

\newcommand{\prob}[1]{P \paren{#1}}
\newcommand{\espe}[1]{\sbesperance \left[ #1 \right]}

\newcommand{\psv}[2]{\langle #1, #2 \rangle}

\newcommand{\inv}[1]{\frac{1}{#1}}

\newcommand{\nrm}[1]{\left\Vert #1 \right\Vert}

\newcommand{\nrmop}[1]{\left\Vert #1 \right\Vert_{\mathrm{op}}}

\newcommand{\tens}[2]{#1 \, \otimes \, #2}

\newcommand{\dpartf}[2]{\frac{\partial #2}{\partial #1}}
\newcommand{\ddpartf}[2]{\frac{\partial^2  #2}{\partial #1^2}}
\newcommand{\dtpartf}[2]{\frac{\partial^3  #2}{\partial #1^3}}

\newcommand{\bo}{B}

\newcommand{\som}[2]{\sum_{#1}^{#2}}	

\newcommand{\enstq}[2]{\left\lbrace #1 \left\vert \, #2 \right. \right\rbrace}	

\newcommand{\cpl}[2]{\paren{#1,\,#2}}	
\newcommand{\tpl}[3]{\paren{#1,\,#2,\,#3}}	

\newcommand{\sbesperance}{\mathbb{E}} 
\newcommand{\econd}[2]{\sbesperance\,\left[ \left. #1 \right\vert #2 \right]}

\newcommand{\nbt}{NoBackTrack}

\newcommand{\tribo}{\mathcal{F}_0}
\newcommand{\trib}[1]{\mathcal{F}_{#1}}
\newcommand{\tribd}[1]{\mathcal{F}'_{#1}}

\newcommand{\tpsk}[1]{{T_{#1}}}

\newcommand{\sm}[1]{\boldsymbol{{#1}}}

\newcommand{\prmsyst}{e}
\newcommand{\prmctrl}{\theta}

\newcommand{\opevol}{\mathbf{T}}

\newcommand{\epctrl}{\Theta}
\newcommand{\sbepsyst}{\State}
\newcommand{\epsyst}{\sbepsyst}
\newcommand{\epsystins}[1]{\sbepsyst_{#1}}	

\newcommand{\pas}{\eta}

\newcommand{\pctrlopt}{\prmctrl^*}

\newcommand{\boct}{\bo_{\epctrl}}

\newcommand{\mvpmin}{\lambda_{\min}}

\newcommand{\cpliprmup}{c_{\paramupdate}} 

\newcommand{\clipalgoprm}{\kappa_{\mathrm{lip}\param}}	

\newcommand{\cliprtrlprm}{\clipalgoprm}

\newcommand{\sbeg}{\phi}	
\newcommand{\feg}[2]{\sbeg\cpl{#1}{#2}}	

\newcommand{\hsr}{h}	

\newcommand{\olr}{\bar{\eta}}	

\newcommand{\transp}[1]{#1^{\!\top}\!}

\newcommand{\state}{s}
\newcommand{\stateopt}{\state^*}
\newcommand{\State}{\mathcal{S}}
\newcommand{\sbtube}{\mathbb{T}}
\newcommand{\tube}{\sbtube}

\newcommand{\metric}{P}

\newcommand{\jope}{J}			

\newcommand{\jopeopt}{\jope^*}
\newcommand{\tubejope}{\sbtube^{\jope}}

\newcommand{\opevolt}{\opevol_t} 	
\newcommand{\sbperte}{\mathcal{L}}	
\newcommand{\avloss}{\bar\sbperte}
\newcommand{\perte}{\sbperte}			
\newcommand{\sbfsyst}{\mathbf{\state}} 	
\newcommand{\fstate}{\sbfsyst}

\newcommand{\sblin}{\mathrm{L}}  
\newcommand{\epjopeins}[1]{\sblin\paren{\epctrl,\,\epsystins{#1}}}	

\newcommand{\cst}{\kappa}		
\newcommand{\boctopt}{\boct^*}		
\newcommand{\rayopt}{r^*}		
\newcommand{\rayoptct}{{\rayopt_\epctrl}}	
\newcommand{\sbray}{r}	
\newcommand{\raysy}{\sbray_\epsyst}	

\newcommand{\bosyinsopt}[1]{\bo_{\epsystins{#1}}^*}	

\newcommand{\sbdist}{d}	
\newcommand{\dist}[2]{\sbdist\paren{#1,\,#2}}	

\newcommand{\sbopdepct}{\Phi}	
\newcommand{\opdepct}[2]{\sbopdepct\paren{#1,\,#2}}	

\newcommand{\vtanc}{v}	

\newcommand{\tinit}{{t_0}}	
\newcommand{\thoriz}[1]{T^{\rayoptct}_{#1}}	
\newcommand{\thorizo}{\thoriz{0}}	

\newcommand{\espaux}{\mathcal{E}}	

\newcommand{\majmpc}{\sbmaj_4}		

\newcommand{\sbexpas}{b}	
\newcommand{\expas}{\sbexpas}	

\newcommand{\exmp}{\gamma}	

\newcommand{\scsp}{u}	

\newcommand{\cdvp}{\overline{\eta}}	
\newcommand{\sched}{\rho}
\newcommand{\brnp}{\cdvp_{\Tangent}}

\newcommand{\cdvpop}{\cdvp_{\mathrm{op}}}	
\newcommand{\cdvpnoise}{\cdvp_\mathrm{noise}}

\newcommand{\pgs}{r}	

\newcommand{\tasg}{b}	

\newcommand{\sbpermc}[1]{\perte_{\leadsto #1}}
\newcommand{\permctt}[2]{\perte_{#1\leadsto #2}}

\newcommand{\sbfmdper}{m} 
\newcommand{\fmdper}[1]{\sbfmdper(#1)}	
\newcommand{\mlipgrad}[1]{m_\mathrm{H}(#1)}

\newcommand{\prn}{\tilde{\pas}}	

\newcommand{\itv}{I}
\newcommand{\lng}{L}	

\newcommand{\sbflng}{\lng}	
\newcommand{\flng}[1]{\sbflng(#1)}	
\newcommand{\sbflngd}{e}	

\newcommand{\lngav}{\sbflngd_{0}}

\newcommand{\smi}[1]{\sum_{#1}}	

\newcommand{\sbfech}{f}	
\newcommand{\fech}[1]{\sbfech\paren{#1}}

\newcommand{\prdope}[3]{\Pi^{#1}_{#2,\,#3}}	

\newcommand{\rtrl}{RTRL\xspace}

\newcommand{\sbcfc}{r}	
\newcommand{\scfc}{\sbcfc}	

\newcommand{\sbcfa}{b}	
\newcommand{\scfa}{\sbcfa}	

\newcommand{\cdvpmaxconv}{\cdvp_\mathrm{conv}} 

\newcommand{\expem}{a}	
\newcommand{\exli}{A}	

\newcommand{\fur}{\mathcal{U}}	
\newcommand{\furt}[3]{\fur_t\tpl{#1}{#2}{#3}}	
\newcommand{\sbfh}{\mathcal{H}}	
\newcommand{\fht}[1]{\sbfh_t\paren{#1}}	

\newcommand{\sbmodct}{\rho}	
\newcommand{\modct}[1]{\sbmodct\paren{#1}}	

\newcommand{\epmaint}{\mathcal{M}}	
\newcommand{\maint}{\mathfrak{m}}	
\newcommand{\mntopt}{\maint^*}		
\newcommand{\epvtg}{\mathcal{V}}	
\newcommand{\Tangent}{\epvtg}		

\newcommand{\bmaint}[1]{\tube_{\epmaint_{#1}}}	
\newcommand{\bvtg}[1]{\bo_{\epvtg_{#1}}}	

\newcommand{\sbruit}{D}	
\newcommand{\bruitif}[4]{\sbruit_{#1:#2}\tpl{#3}{#4}{\sm{\pas}}}	
\newcommand{\bruit}[3]{\bruitif{0}{#1}{#2}{#3}}	
\newcommand{\bruitifcdvp}[5]{\sbruit_{#1:#2}\tpl{#3}{#4}{#5}}	

\newcommand{\sepjope}[1]{\mathrm{Rk_1}\cpl{\epctrl}{\epsyst_{#1}}}

\newcommand{\sbopnbt}{\mathcal{T}}
\newcommand{\opnbt}[4]{\sbopnbt_{#4}\tpl{#1}{#2}{#3}}	
\newcommand{\opnbtit}[3]{\opnbt{#1}{#2}{#3}{t}}	

\newcommand{\vsyst}{v^{\epsyst}}	
\newcommand{\vctrl}{v^{\epctrl}}	

\newcommand{\sbopred}{\mathcal{R}}	
\newcommand{\opredt}[4]{\sbopred_t\paren{#1,\,#2,\,#3,\,#4}}	

\newcommand{\sbopmulb}{a}	
\newcommand{\opmulb}[1]{\sbopmulb\paren{#1}}	

\newcommand{\sbopaddb}{b}	
\newcommand{\lopaddb}[1]{\sbopaddb_{#1}}	

\newcommand{\vcans}{\mathfrak{\prmsyst}}	
\newcommand{\sber}{\eps}	
\newcommand{\ber}[2]{\sber_{#1}\paren{#2}}			
\newcommand{\beri}[1]{\ber{i}{#1}}	
\newcommand{\berit}{\ber{i}{t}}	
\newcommand{\veber}[1]{\sber(#1)}	

\newcommand{\sbsignes}{\eps}	
\newcommand{\sigi}[1]{\sbsignes_{#1}}	

\newcommand{\sbopegn}{\odot}	
\newcommand{\opegn}[2]{{#1 \, \sbopegn \, #2}}	
\newcommand{\vopegn}[2]{\paren{\fegn{#1}{#2} \, #1,\,\fegn{#2}{#1} \, #2}}	
\newcommand{\fegn}[2]{\sqrt{\frac{\nrm{#2}}{\nrm{#1}}}}	

\newcommand{\vesp}{v}	

\newcommand{\dopevols}[3]{\dpartf{\prmsyst}{\opevol_{#1}}\paren{#2,\,#3}}	
\newcommand{\dopevolst}[2]{\dopevols{t}{#1}{#2}}	

\newcommand{\cvnbtaxsi}{\cpl{\vsystax}{\vctrlax}} 
\newcommand{\cvnbtsi}{\cpl{\vsyst}{\vctrl}}	

\newcommand{\vsystax}{w^{\epsyst}}	
\newcommand{\vctrlax}{w^{\epctrl}}	

\newcommand{\mesl}[1]{$\trib{#1}$--measurable} 

\newcommand{\gu}[1]{``#1''}	

\newcommand{\perteiid}{\ell}
\renewcommand{\majmpc}{\mathrm{M}}
\newcommand{\E}{\mathbb{E}}
\def\d{\operatorname{d}\!{}}
\newcommand{\stat}{\Psi}
\newcommand{\Stat}{\psi}
\DeclareMathOperator{\diag}{diag}
\newcommand{\momen}{c}
\newcommand{\ff}{F} 
\newcommand{\linalgmatrix}{\Lambda} 
\newcommand{\linform}{\sblin(\Param,\mathbb{R})} 

\newcommand{\todo}[1]{{\textcolor{blue}{#1}}\xspace}
\renewcommand{\todo}[1]{}

\begin{document}
\title{Convergence of Online Adaptive and Recurrent Optimization
Algorithms}
\author{Pierre-Yves Massé \footnote{Czech Institute of Informatics, Robotics and Cybernetics, Czech Technical University in Prague} \and Yann
Ollivier \footnote{Facebook Artificial Intelligence Research, Paris}}

\newcommand{\Param}{\epctrl}
\newcommand{\param}{\prmctrl}
\newcommand{\paramopt}{\pctrlopt}
\newcommand{\Algo}{\mathcal{A}}
\newcommand{\from}{\colon}
\newcommand{\Mem}{\epmaint}
\newcommand{\mem}{\maint}
\newcommand{\memopt}{\mntopt}
\newcommand{\gradop}{\sm{V}}  
\newcommand{\deq}{\mathrel{\mathop{:}}=}
\newcommand{\paramupdate}{\sbopdepct}

\def\R{{\mathbb{R}}}
\newcommand{\eqd}{=\mathrel{\mathop{:}}}
\newcommand{\norm}[1]{\left\lVert#1\right\rVert}

\newcommand{\algonoisy}{imperfect\xspace}
\newcommand{\Algonoisy}{Imperfect\xspace}

\newcommand{\inputsyst}{u}
\newcommand{\nrmHilb}[1]{{\left\lVert#1\right\rVert}_{\mathrm{Hilbert}}}

\maketitle

\begin{abstract}
\todo{Titres possibles:
\\Learning Dynamical Systems Online by Gradient
Descent
\\Convergence of
Not-Quite-SGD Algorithms via Dynamical Systems
\\Convergence of Online Optimization Algorithms [with time-correlated data] via Dynamical Systems (merci Léonard)
\\Convergence of Online Optimization Algorithms via Dynamical Systems
[bien mais trop general: optim peut vouloir dire algos genetiques, methodes
Newton/BFGS/...]
\\Convergence of Online Gradient Descent Algorithms via Dynamical Systems
\\Convergence of Online [Machine] Learning Algorithms via Dynamical Systems
\\Convergence of Online Deep Learning Algorithms via Dynamical Systems
\\Convergence of Online Adaptive and Recurrent Optimization Algorithms
\\Convergence of Online Recurrent and Adaptive Optimization Algorithms
}
We prove local convergence of several notable gradient descent
algorithms used in
machine learning, for which standard stochastic gradient descent theory
does not apply directly. This includes, first, online algorithms for recurrent models and dynamical
systems, such as \emph{Real-time recurrent learning} (RTRL)
\citep{jaeg,pearl} and its computationally lighter approximations 
NoBackTrack \citep{oll16} and UORO \citep{uoro}; second,
several adaptive algorithms such as RMSProp, online natural gradient, and Adam with $\beta^2\to 1$.

Despite local convergence being a relatively weak requirement for a new
optimization algorithm, no local analysis was available for these algorithms, as far as
we knew. Analysis of these algorithms does not immediately follow
from standard stochastic gradient (SGD) theory.  In fact, Adam has been proved
to lack local convergence in some simple situations \citep{j.2018on}. For recurrent models, online algorithms modify the parameter
while the model is running, which further complicates the analysis with
respect to simple SGD.

Local convergence for these various algorithms results from a single,
more general set of assumptions, in the setup of learning dynamical
systems online. Thus, these results can cover other variants of
the algorithms considered.

We adopt an ``ergodic'' rather than probabilistic viewpoint, working with
empirical time averages instead of probability distributions. This is
more data-agnostic and
creates differences with respect to standard SGD theory,
especially for the range of possible learning rates. For instance, with
cycling or per-epoch reshuffling over a finite dataset instead of pure
i.i.d.\ sampling with replacement, empirical
averages of gradients converge at rate $1/T$ instead
of $1/\sqrt{T}$ (cycling acts as a variance reduction method),
theoretically allowing
for larger learning rates than in SGD.

\end{abstract}

\tableofcontents

\todo{stuff currently inconsistent in the text:}

\todo{vérifier cohérence dans étiquetage des énoncés Proposition ou Théorème vs Lemme/Corollaire}

\todo{\\Possible extensions (not urgent):}

\todo{Ajout de termes d'erreur en $O(\pas_t)$ dans l'algo de base}

\todo{Check if we can obtain, for free, that the learning trajectory stays
$O(\pas)$-close to the ODE trajectory. Could be useful later. Or if we
can obtain that $\param-\paramopt=O(e^{-\lambda t})+O(\pas_t)$ or
something like that. People have asked me whether changing the exponent
$a$ (of the ergodic assumption) changes the rate of convergence;
important for the cycling example.}

\todo{I think I can treat Adam with both $\beta^1$ and $\beta^2$ tending
to $1$, but this is quite complicated and requires an auxiliary lemma to
prove that average of losses weighted by $\pas_t$ tend to $0$ at
$\paramopt$. Feasible but not urgent, I think.}

\todo{online version of SAG/SAGA?}

\todo{HMMs}

\todo{j'ai commenté des "todos" du style "on pourrait ajouter des termes
d'erreur en $\pas_t$", comme on pensait soumettre en l'état.
NDY: je prefere ne pas commenter les todos de ce genre (les
parties commentees du .tex sont essentiellement perdues a jamais). Pour
la soumission, on redefinit juste la commande "todo" pour qu'elle
n'affiche rien...}


\section{Introduction}


We consider, from a machine learning perspective, the problem of
optimizing in real time the parameters of a dynamical system so that its
behavior optimizes some criterion over time. This problem has a
longstanding history, especially for linear systems of small to moderate
dimension \citep{ljung84}, encompassing many classical recursive
control problems like the steering of a ship, the short-term prediction
of power demand or the transmission of speech through limited capacity
transmission channels \citep{ljung84}.
Examples that have attracted more recent
attention include recurrent models in machine learning (recurrent neural
networks), used to represent time-structured or sequentially-structured
data.  Even when the data have no time structure, a dynamical system can
also represent the internal state of a machine learning algorithm, such
as momentum variables in extensions of stochastic gradient descent.

We focus on \emph{online} (or \emph{real-time}) algorithms, that are able
to update their state or predictions as each new observation arrives, at
an algorithmic and memory cost that does not grow with the amount of data
processed. Quoting \citet{pearl}, ``An online, exact, and stable, but
computationally expensive, procedure for determining the derivatives of
functions of the states of a dynamic system with respect to that system's internal parameters has been discovered and applied to recurrent
neural networks a number of times [...], \emph{real time recurrent
learning},
RTRL. Like BPTT, the technique was known and applied to other sorts of
systems since the 1950s''.

Thus, RTRL is the algorithm that adapts the parameters of a dynamical system
by gradient descent over some criterion at each time step, in real time
while the system is running.
RTRL has both practical and theoretical shortcomings: First, its
computational burden is prohibitive even for moderately-dimensioned
systems. This has led to several lightweight approximations based on
stochastic approximation, such as \emph{NoBackTrack} and its extensions
\emph{UORO} and \emph{Kronecker-factored RTRL} \citep{oll16, uoro, NIPS2018_7894}. For
relatively short data sequences (such as sentences in natural language
processing), non-online algorithms such as backpropagation through time
(BPTT) are usually preferred.  \emph{Truncated BPTT} is an approximation
of BPTT that works online by maintaining a fixed-length memory of recent
data.

Second, as far as we know, no proof of convergence, even local, has been
given for these algorithms. A key feature of online algorithms is that
the parameters of the dynamical system are updated while the system is
running. Intuitively this is only a second-order phenomenon if learning
rates are small; but this still
complicates the analysis substantially.

We provide such a proof of local convergence for RTRL, and for some of
its variants.  Moreover, the results carry over to other non-recurrent
machine learning algorithms, such as RMSProp, Adam, or online natural
gradient.  The dynamical system viewpoint is used to handle the internal
state of these algorithms.

More precisely, we prove local convergence of various algorithms for
recurrent and non-recurrent systems:
\begin{enumerate}
\item Real-time recurrent learning (RTRL)
(Theorem~\ref{thm:cvapproxrtrlalgo});
\item Truncated backpropagation through time (TBPTT), provided the
truncation length is slowly increased at a rate related to the main
learning rate (\rethm{cvtbtt});
\item Unbiased stochastic approximations to RTRL: NoBackTrack and UORO
(Corollary~\ref{cor:cvnbtuoro});
\item Stochastic gradient descent with momentum
and any parameter-dependent or adaptive
preconditioning, where a definite positive preconditioning matrix is estimated online
from the data
(Corollary~\ref{cor:adam}). This covers algorithms such as
RMSProp and Adam with the preconditioner updated at the same rate as the
main learning rate (Corollaries~\ref{cor:adap}, \ref{cor:adam}),
a natural gradient descent with the Fisher matrix estimated online
at the same rate as the main learning rate
(Corollary~\ref{cor:adap}), or the extended Kalman filter in the static
case (for estimating the state of a fixed system via nonlinear noisy
measurements). Results for RMSProp and Adam are
known (e.g., \cite{zou2019sufficient});
our result  is less precise but more
general as it covers any kind of adaptive preconditioning rather
than a specific algorithm.
\end{enumerate}

We give a more precise overview of results in
Section~\ref{sec:presentation}.

Although local convergence is a relatively weak property for an algorithm
(compared to global convergence results obtained in convex situations),
no local analysis was available for these algorithms apart from RMSProp and
Adam, as far as we know.
Our original project was to prove local convergence for NoBackTrack and
UORO based on a convergence proof for RTRL, but we could locate no such
existing proof.
Convergence of these algorithms does not immediately follow from standard
stochastic gradient (SGD) theory.  In fact, Adam has been proved to lack
local convergence if its hyperparameter $\beta^2$ is fixed
\citep{j.2018on} (convergence occurs with a time-dependent $\beta^2\to
1$ so that the preconditioner is averaged over more and more samples).

Importantly, we prove local convergence under \emph{local} assumptions:
we do not assume that the model or system is well-behaved out of some
ball of finite radius. We believe this reflects problems encountered in
practice, when large steps can be difficult to recover from if the system
parameters reach an unsafe zone. Thus, local convergence under local
assumptions can be harder to prove than global convergence under global
assumptions. 

Most data to which recurrent models are applied cannot reasonably be
assumed to be fully Markovian (natural text has arbitrary long-term
dependencies, time series may be non-time-homogeneous). So we adopt a more
data-agnostic viewpoint, reasoning on ergodic properties of an individual
data sequence rather than on expectations. A local minimum is defined as
a parameter value that achieves locally best loss on average over time
(Assumption~\ref{hyp:opt_simple}). Ergodic properties of gradients,
averaged over time, replace expectations, and the standard stochastic
case is recovered by proving that the assumptions hold with probability
one.
This
per-trajectory viewpoint with local assumptions leads to several differences with respect to
standard SGD theory, mostly relating to learning rates:


\begin{itemize}
\item
When dealing with finite datasets, the per-trajectory
viewpoint emphasizes specific properties of cycling through
the data samples or reshuffling at every epoch, as opposed to the pure
SGD method of selecting a sample at random at every step: cycling acts as
a variance reduction method (ensuring each sample is selected exactly
once within $N$ steps, where $N$ is the size of the dataset). This results in larger possible learning
rates: with cycling or random reshuffling, learning rates $\pas_t\propto
1/t^\expas$ with any $0<\expas\leq 1$ are
suitable, as opposed to $1/2<\expas\leq 1$ in classical Robbins--Monro
theory (Corollary~\ref{cor:sgdrates}). This opens the door to more
elaborate variance reduction methods in SGD.
\item On the contrary, in a non-recurrent, online i.i.d.\ setting with an
infinite dataset, our results are sometimes suboptimal: 
depending on which moments of the noise are finite, 
we may get more
constraints on the learning rate (Section~\ref{sec:pureonline}).
This is presumably
because the ergodic Assumption~\ref{hyp:opt_simple} does not capture the full randomness of an
i.i.d.\ sequence of samples.

\item
In a dynamical system setting, the stepsizes
$\pas_t$ for the gradient descent must vary smoothly in time,
to avoid spurious correlations between the stepsize and
the state of the system, which would bias the gradient descent.
This is stricter than the classical Robbins--Monro criterion
\citep{rm}.
(For instance, if a dynamical system
exhibits periodic
phenomena of period 2, and if
$\pas_t$ vanishes for even values of $t$, the gradient descent using
$\pas_t$ may be strongly biased and diverge.)
We avoid this issue
by assuming the learning rates behave like $1/t^b$ for some $b>0$. (A
more general homogeneity condition on the learning rates is given in \rehyp{spdes}.)
\end{itemize}

Finally, we treat adaptive preconditioning (RMSProp, Adam, online natural
gradient...) by viewing the preconditioner as part of the parameter to be
estimated. The corresponding update does not follow the gradient of a
loss function; indeed, unlike a Hessian, the Jacobian $\linalgmatrix$ of the expected
update is not a symmetric, definite positive matrix. But its eigenvalues
still have positive real part (Sections~\ref{sec:adap}--\ref{sec:adam}),
which is sufficient to apply the standard Lyapunov theory for stable
matrices (Appendix~\ref{sec:positivestable}), and prove local
convergence.
A reminder on positive-stable matrices is included in
Appendix~\ref{sec:positivestable}. Focusing on positive-stable matrices
instead of positive-definite Hessians is not new in machine learning: see
for instance the classical paper \citet{polyak-judi1992} on averaged
stochastic gradient descent.

\paragraph{Some related work.} Learning of recurrent models and dynamical
systems is not a new topic (see historical references in \citet{pearl,
ljung84}), and it is impossible to be exhaustive.  For dynamical systems,
an in-depth reference is \cite{ljung84}, which discusses algorithms for
learning a dynamical system online, largely focusing on the linear case.
For \emph{linear} dynamical systems, more precise results are available.
For instance,
\citet{hardt2016gradient} prove global convergence of
non-online stochastic gradient descent on linear systems, provided the
matrix defining the system is parameterized in a particular way based on
its characteristic polynomial. For nonlinear systems,
\citet{ben-met-pri1990} present results for stochastic gradient
descent in very general time-dependent systems under strong Markovian
assumptions, but it is not clear how to cast the algorithms studied here
in their framework and how to check the technical assumptions.

Our overall approach to the proofs follows the classical ODE method for
the analysis of SGD around a local optimum
\citep{ljung77,ben-met-pri1990,borkar2000ode,kyi2003,borkar2009stochastic}. The ODE approach views the
optimization process on the parameter as an approximation of a
continuous-time, noise-free ``ideal'' gradient descent, whose timescale is defined by the
step sizes of the algorithm.
Thus, our analysis is based on bounding the difference between
the true system and an idealized system, linearized close to the optimum
and with the noise averaged out. A central role is played by the Jacobian
$\linalgmatrix$  of the optimization algorithm around the local optimum: this
is the Hessian of the loss for simple SGD, but is a more complicated,
non-symmetric matrix in adaptive algorithms such as Adam
(Sections~\ref{sec:adap} and~\ref{sec:adam}). Following the
standard theory of dynamical systems, the idealized system on the
parameter will converge when all eigenvalues of this matrix have positive
real part (namely, in the simplest case, when the Hessian of the loss is
positive definite).

For simple, non-recurrent SGD on general (non-convex) loss functions, one
of the cleanest results is probably still \cite{bertsekastsitsiklis2000},
which proves convergence to a local minimum (which may
be at infinity) under mild global assumptions
(globally Lipschitz gradients, noise bounded by the gradient norm):
namely, the loss converges and the gradient of the
loss converges to $0$. This does not cover either dynamical systems or
algorithms other than simple SGD.  Moreover, contrary to this work, we
only make local assumptions.

For adaptive gradient descent algorithms such as Adam and RMSProp,
convergence results already exist. Our result (Corollary~\ref{cor:adam}) is less precise but more
general, in that it covers any kind of adaptive preconditioning rather
than specific algorithms, also covering
the online natural gradient, for example. Among others,
\citet[Corollary 10]{zou2019sufficient}
prove a convergence result for Adam and RMSProp
over a wide range of hyperparameters,
together with finite-time bounds in expectation. We refer to
\citet{defossez2020convergence}
for more up-to-date finite-time bounds for Adam, and for additional
references. These results, and ours,
use a time-dependent Adam hyperparameter $\beta^2\to 1$ so that square gradients are averaged
over more and more samples. On the
other hand, \citet{j.2018on} show divergence of Adam with fixed hyperparameters
$\beta^1$ and $\beta^2$ when cycling over a finite dataset, contradicting
an earlier convergence claim in \citet{Adam}. 

Convergence of algorithms with adaptive preconditioners (RMSProp, Adam,
online natural gradient) with $\beta^2\to 1$ could also probably be
proved using two-timescale methods (see for instance
\citet{tadic2004}). However, two-timescale methods, as the name
suggests, require different timescales for the learning rate and the
adaptive preconditioner: the main learning rate should be smaller than the
rate at which the preconditioner is updated (which itself should tend to
$0$).  Our result (Corollary~\ref{cor:adap}, Corollary~\ref{cor:adam})
lifts this restriction by letting the main learning rate be as large as
the update rate of the preconditioner.

Finally, empirical differences between cycling over a dataset or random
per-epoch reshuffling as opposed to pure i.i.d.\ sampling from the
dataset have been observed for some time \citep{bottou2009curiously}.
Some quantitative results for convex functions are available
\citep{gurbuzbalaban2015random}, showing improved convergence
for random reshuffling compared to SGD. But these results still
require learning rates smaller than $1/\sqrt{t}$, contrary to ours.

\paragraph{Structure of the text.} In Section~\ref{sec:presentation}, we
present an overview of the results,
introduce the notation for dynamical systems, and present the standard
RTRL algorithm as well as several generalizations that will encompass
more algorithms. We then state the local convergence theorem for these
extended RTRL algorithms, after discussing the technical assumptions.
Section~\ref{sec:examples} contains several examples and applications,
both recurrent and non-recurrent: simple SGD and the influence on
learning rates of cycling over a dataset versus pure i.i.d.\ sampling,
SGD with adaptive preconditioning and with momentum (including Adam), the
original RTRL algorithm, truncated backpropagation through time with
increasing truncation, and the NoBackTrack and UORO algorithms. We then
proceed to the proof: in Sections~\ref{sec:absontadsys}
and~\ref{sec:procvabstalgo} we go to a more abstract setting using an
extended dynamical system that contains all the variables maintained by an
algorithm; in this more abstract setting, we use the ODE method to
quantify the discrepancy between the ideal continuous-time, noise-free
gradient descent and the actual online gradient descent for the dynamical
system. In Sections~\ref{sec:contrtrlappxalgoart}
and~\ref{sec:prcvalappxalgo} we bridge the abstract setting and the
concrete algorithms; especially, we check that all properties needed for
Section~\ref{sec:absontadsys} are indeed satisfied for the practical
algorithms.

\paragraph{Acknowledgements.} The authors would like to thank Léon
Bottou, Joan Bruna, and Aaron Defazio for pointing us to relevant
references.
The work of the first author was partially supported by the
European Regional Development Fund under the project IMPACT (reg.\ no.\
CZ.02.1.01/0.0/0.0/15\_003/0000468).

\section{Recurrent Models and the RTRL Algorithm}
\label{sec:presentation}

\subsection{Overview of RTRL}
\label{sec:rtrloverview}

We consider a dynamical system parameterized by $\param\in\Param$, whose state
$\state_t\in\State_t$ at time $t\geq 1$ is subjected to
the evolution equation
\begin{equation}
\label{eq:basicsystem}
\state_{t}=\opevol_{t}(\state_{t-1},\param),
\end{equation}
with some transition operator $\opevol_t$.
At each time, we are given a
loss function $\perte_t(\state_t)$, and our objective is to optimize the
parameter $\param$ as to minimize the average loss function
$\frac{1}{T}\sum_{t=1}^T\perte_t(\state_t)$ over some large time interval
$T\to\infty$, in an online manner. Formal definitions are given in
Section~\ref{sec:formaldefs} below.

This formalism encompasses non-recurrent situations, by letting
$\opevol_t$ be independent of $\state_{t-1}$. For instance, consider a regression
problem $y=\ff_\theta(x)$, with training dataset
$(x_t,y_t)_{t\in[1;T]}$, and a loss function $\perteiid(y,y_t)$ such as
$\perteiid(y,y_t)=\norm{y-y_t}^2$.
This can be represented by identifying the state
$\state$ with $y$, namely, setting
\begin{equation}
\label{eq:nonrec}
\opevol_{t}(\state_{t-1},\param)\deq
\ff_\param(x_t),\qquad\perte_t(\state)\deq \perteiid(\state,y_t).
\end{equation}
The operators $\opevol_{t}$ and $\perte_t$ depend on the data.
In this non-recurrent case, the RTRL algorithm will reduce to standard stochastic gradient descent.

Another typical system we have in mind is a recurrent model with internal
state $\state_t$, where the
time-dependent transition operator
\begin{equation}
\label{eq:rec}
\opevol_t(\state_{t-1},\param)\deq \ff_\param(\state_{t-1},x_t)
\end{equation}
is defined via a time-independent function $\ff$ with some input $x_t$ 
as an argument.\footnote{This describes an online system with unbounded time.
Finite-length training sequences are covered
by separating them by 
end-of-sentence input symbol $x_t^\dag$ and defining
$\ff_\param(\state_{t-1},x_t^\dag)\deq\stateopt_0$ to reset the system to state
$\stateopt_0$ after each sequence, with notation as in \eqref{eq:rec}. This preserves all our
assumptions below.}
Once more, we define a loss function $\perte_t(\state)\deq
\perte(\state,y_t)$ where $\perte(\state,y_t)$ typically measures the loss between a
value $y_t$ to be predicted, and some part of the state $s$ that encodes
the prediction on $y_t$. 

Thus, the data $(x_t,y_t)$ is encoded in \eqref{eq:basicsystem} via the
time dependency of $\opevol_t$ and $\perte_t$.
Recurrent neural networks (RNNs) fit this framework; for instance, a
simple RNN model is
\begin{equation}
\label{eq:rnn}
\state_t=\text{sigmoid}\paren{W\state_{t-1}+W' x_t+B},
\end{equation}
where $W$, $W'$ and $B$ are matrices or vectors of suitable dimensions,
and where $\param=(W,W',B)$.

Thus, when $\opevol_t$ is defined this way, we assume the sequence of
inputs to be fixed once and for all,\footnote{This means in particular that the system is
non-adversarial: the inputs and targets do not change based on the
behavior of the algorithm.}
and make no direct assumption on its
nature. In particular, we do not make explicit stochastic assumptions on
the data, but we 
assume they satisfy ergodic-like properties, expressed as empirical
averages over time (see \rehyp{opt_simple}).

\bigskip

Jaeger's tutorial \citep{jaeg} presents several classical recurrent
training algorithms. The most widely used is backpropagation through
time. One of its important drawbacks is the need to store and pass
through the complete sequence of past observations every time a new
observation $(x_{t+1},y_{t+1})$ becomes available: it is not possible to
process online newly arrived inputs coming from a stream of data.
On the other hand, the RTRL algorithm may be used online, but has much heavier
computational and memory requirements. Let us now describe it.

%

The \rtrl algorithm conducts an approximate gradient descent on the parameter of the
dynamical system to be trained.
%
The state $\state_t$ of the system at each time depends on the parameter
used and on the initial state.
By composition, 
the loss above on $\state_t$ may thus be viewed as a loss on the parameter
and the initial state. (We will omit the initial state for now.) We write
$\perte_t(\state_t)$ for the original loss on the state of the system at time $t$, and
$\sbpermc{t}(\param)$ for the resulting loss at time $t$, seen as a function of
the parameter via running the system up to time $t$ with parameter
$\param$
(Definition~\ref{def:fcpletaprm}). In computational terms, $\sbpermc{t}$
corresponds to
the loss of the whole computational graph leading to $\perte_t$.

The derivative of $\sbpermc{t}$ with respect to the parameter can be
computed by induction, by direct differentiation of the recurrent
equation \eqref{eq:basicsystem} that defines the system. Informally, by the
chain rule,\footnote{For Jacobians, we use the standard convention from
differential geometry, namely, if $x$ and $y$ are multidimensional
variables then $\frac{\partial y}{\partial x}$ is the matrix with entries
$\frac{\partial y_i}{\partial x_j}$. With this convention the chain rules
writes $\frac{\partial z}{\partial x}=\frac{\partial z}{\partial
y}\frac{\partial y}{\partial x}$.
This makes $\frac{\partial \perte_t}{\partial \state_t}$ a \emph{row}
vector. When working with standard RTRL, we abuse notations by omitting the transpose around $\frac{\partial \perte_t}{\partial \state}\cdot\frac{\partial \state_t}{\partial \param}$ in expressions of the form $\param \leftarrow \param - \frac{\partial \perte_t}{\partial \state}\cdot\frac{\partial \state_t}{\partial \param}$.
}
\begin{equation}
\label{eq:basicrtrlloss}
\frac{\partial
\sbpermc{t}}{\partial\param}=\frac{\partial\perte_t}{\partial\state_t}\cdot
\frac{\partial \state_t}{\partial\param}
\end{equation}
where $\frac{\partial \state_t}{\partial\param}$ is the Jacobian matrix
of the state $\state_t$ as a function of $\param$.
Then by differentiating the evolution equation \eqref{eq:basicsystem},
\begin{equation}
\label{eq:basicrtrl}
\frac{\partial\state_t}{\partial\param}=\frac{\partial\opevol_{t}}{\partial
\state_{t-1}}\cdot \frac{\partial \state_{t-1}}{\partial
\param}+\frac{\partial\opevol_t}{\partial \param}.
\end{equation}
This allows for
computing $\frac{\partial\state_t}{\partial\param}$ by
induction in an online manner: store the value of the Jacobian
$\frac{\partial\state_t}{\partial\param}$ in a variable $\jope_t$, and
update $\jope_t$ via
\eqref{eq:basicrtrl} at each time step, namely,
\begin{equation*}
\label{eq:basicrtrlJ}
\jope_t=\frac{\partial\opevol_{t}}{\partial
\state_{t-1}}\cdot \jope_{t-1}
+\frac{\partial\opevol_t}{\partial \param}
\end{equation*}
after which the stored value $\jope_{t-1}$ can be discarded.
This is the core of the RTRL
algorithm. The derivative \eqref{eq:basicrtrlloss} is then used to
obtain the parameter via a gradient descent step
\begin{equation*}
\param \gets \param -\pas_t
\left(\frac{\partial\perte_t}{\partial\state_t}\cdot\jope_t\right)
\end{equation*}
with learning rate $\pas_t$. In the non-recurrent case \eqref{eq:nonrec}, $\opevol_t$ does
not depend on $\state_{t-1}$, and RTRL reduces to standard online gradient
descent on $\perte_t$.

However, updating the parameter at every step breaks the validity of the computations
\eqref{eq:basicrtrlloss}--\eqref{eq:basicrtrl}, because RTRL will use
values of $\jope_{t-1}$ stored and computed on previous values of the parameter
$\param$, thus mixing partial derivatives taken at different parameter
values. The magnitude of the error at each step is $O(\pas_t)$ (since the
parameter changes only by $O(\eta_t)$), so
intuitively this should not matter too much for small learning rates.
But
this is a core difficulty in the analysis of RTRL.

The RTRL algorithm is computationally heavy for large-dimensional
systems, since the Jacobian $\jope_t$ is an element of the space
$\epjopeins{t}$, so that even storing it requires memory $\dim(\param)\times
\dim(\state_t)$, not to mention performing the multiplication
$\dpartf{\state}{\opevolt}\jope_{t-1}$. This justifies the practical
preference for backpropagation through time in non-online setups, and
the introduction of approximations such as UORO and NoBackTrack in online
setups.

\subsection{Overview of Results}

We provide here a semi-technical account of the main results of the
text; the full definitions and statements appear in the next sections. We
start with the most general statements covering RTRL, then provide some
corollaries for local convergence of various recurrent and non-recurrent
existing algorithms: stochastic gradient descent with adaptive
preconditioning (RMSProp, Adam, online natural gradient...), truncated
backpropagation through time, and RTRL approximations such as UORO and
NoBackTrack.

These results take the general form: if the parameter is initialized
close enough to some local optimum, then the learning algorithm converges
to that optimum.  Such a local convergence property is relatively weak,
but for most algorithms considered, we could not locate a proof of local
convergence.\footnote{
Although this is not treated in this work, we believe that convergence to each
local optimum
$\paramopt$ can be extended to the whole basin of attraction of
$\paramopt$ for the ``ideal'' infinitesimal-learning-rate gradient
descent $\d
\param_t/\d t=-\partial_\param \sbpermc{t}(\param_t)$ using the same
proof technique, assuming the learning rates are small enough. Indeed, our
whole analysis is based on deviations from this
infinitesimal-learning-rate setting, using a suitable Lyapunov function
for convergence. We give a more precise argument in
Section~\ref{sec:lyapunov}.
}
 Moreover, we only rely on local assumptions. We do not
explicitly assume a random data model. For randomized algorithms,
the assumptions are satisfied with probability $1$; this
results in convergence with probability tending to $1$ as the overall
learning rate tends to $0$ (Section~\ref{sec:mainthm}).
%

\paragraph{General results: RTRL and extended RTRL algorithms.} The
general setting is a dynamical system parameterized
by $\param\in \Param=\R^{\dim(\param)}$, whose state $\state_t\in
\State_t=\R^{\dim(\state_t)}$ at time $t\geq
1$ is subjected to the evolution equation
\begin{equation*}
\state_t=\opevol_t(\state_{t-1},\param)
\end{equation*}
given some time-dependent, $C^2$ transition operator $\opevol_t$. An
important example is
$\opevol_t(\state_{t-1},\param)=\ff_{\param}(\state_{t-1},x_t)$ using a
time-independent function $\ff$ and a sequence of external inputs
$(x_t)$: this covers, for instance, recurrent neural networks or general
dynamical systems with inputs $x_t$. A further example is the
\emph{non-recurrent} case where $s_{t-1}$ is discarded, namely,
$\opevol_t(\state_{t-1},\param)=\ff_\param(x_t)$, where again $\opevol_t$
depends on $t$ via $x_t$. In this latter case one has
$\state_t=\ff_\param(x_t)$: thus, this covers standard parametric interpolation problems, such as
feedforward neural networks.

Denote $\fstate_t(\param)$ the state obtained at time $t$ by running the
system from time $0$ to $t$ with parameter $\param$. (In this overview, we assume
$s_0$ is fixed for simplicity, and omit it.) We assume that we are given
a $C^2$ loss
function $\perte_t\from \State_t \to \R$ for each time $t$.
The goal
is to minimize the average loss
\begin{equation*}
\frac1T \sum_{t=1}^T \sbpermc{t}(\theta),\qquad \sbpermc{t}(\param)\deq
\perte_t(\fstate_t(\theta))
\end{equation*}
as a function of $\theta$, when $T\to \infty$. A typical loss function
would be $\perte_t(\state_t)=\perteiid(\state_t,y_t)$ where $\ell$ is a
fixed loss function between $\state_t$ and a desired output $y_t$ at time
$t$.

In the non-recurrent case with a random i.i.d.\ sample taken at each
time, there is no difference between minimizing the expected loss and
minimizing the temporal average of the loss (thanks to the law of large
numbers). However, with a dynamical system and with no random data model,
we define an optimum based on such temporal averages instead of
expectations, thus relying on stationarity or ergodicity
properties.

Thus, we define a local optimum for this problem as a 
parameter value $\paramopt\in \Param$ such that the average derivative of
the loss with respect to $\paramopt$ vanishes, and such that the average
Hessian of the loss at $\paramopt$ is
positive definite. For a given local optimum, the rate at which these
averages converge will affect possible learning rates for convergence
towards that optimum. (This is useful for improving learning rates when
cycling over a dataset, for instance.) Therefore, more precisely, we say
(Assumption~\ref{hyp:opt_simple})
that $\paramopt$ is a local optimum with
exponent $0 < \expem < 1$ if
gradients of the loss at $\paramopt$ average to $0$ at rate
$t^\expem/t$:
\begin{equation*}
\frac{1}{T}\sum_{t=1}^T
\frac{\partial}{\partial \param}
\sbpermc{t}(\paramopt)=O(T^\expem/T),
\end{equation*}
and if
on average, Hessians of the loss at $\paramopt$ converge to a
positive definite matrix, at rate $t^\expem/t$: there is a positive definite
matrix $H$ such that
\begin{equation*}
\frac{1}{T}\sum_{t=1}^T
\frac{\partial^2}{\partial \param^2}
\sbpermc{t}(\paramopt)=H+O(T^\expem/T).
\end{equation*}
For instance, consider a non-recurrent linear regression problem with bounded or Gaussian
centered noise $\eps_t$, namely, input $x_t$, prediction model $s_t=\param\cdot x_t$,
observations $y_t=\paramopt\cdot x_t+\eps_t$, and quadratic losses
$\ell_t(\state_t)=(\state_t-y_t)^2$. Then the derivatives of the loss at
$\paramopt$ are equal to $-2\eps_t x_t$. So with bounded $x_t$, the
assumption on gradients is satisfied
for any $a>1/2$ by the law of the iterated logarithm (and likewise for
deviations from the average Hessian
$H=\lim \frac{1}{T} \sum_{t=1}^T x_t\transp{x_t}$ assuming this limit
exists). This can be improved if cycling over a finite dataset instead of
picking i.i.d.\ samples: then empirical averages converge at rate $1/T$,
so the assumption is satisfied for any $a>0$ instead of just $a>1/2$.

We assume (\rehyp{specrad}) that the dynamical system is stable at first
order around $\paramopt$. Remember that a \emph{linear} dynamical system
$\state_t=A\state_{t-1}+B\theta+C x_t$ is stable if and only if $A$ has
spectral radius less than $1$, namely, if and only if $A^k$ is
contracting for some $k\geq 1$. Here the system may be nonlinear. 
Define $A_t\deq \frac{\partial \opevol_t}{\partial
\state}(\fstate_{t-1}(\paramopt),\paramopt)$: intuitively this
represents the value of $\partial \state_t/\partial \state_{t-1}$  along
the trajectory defined by $\paramopt$. We assume that the product of a
sufficiently large number of consecutive $A_t$ is contracting
(\rehyp{specrad}). For a
linear system, this is equivalent to standard stability. Note that this
is assumed only at $\paramopt$. If this assumption is not satisfied, then
even \emph{running} the system with fixed parameter $\paramopt$ is numerically
unstable, so there is little interest in trying to learn $\paramopt$.
In the non-recurrent case,
$\opevol_t$ does not depend on $\state_{t-1}$ so that $A_t=0$ and the
assumption is automatically satisfied.

Finally, we have a series of more ``technical'' assumptions (technical in
the sense that they are always satisfied over a finite dataset for a
smooth feedforward model): the transition functions $\opevol_T$  are
uniformly $C^2$ around the trajectory defined by $\paramopt$
(\rehyp{regftransetats}), the first and second derivatives of the loss
function with respect to $s_t$ grow at most like $t^\exmp$ for some $0
\leq \exmp < 1$ along the trajectory defined by $\paramopt$
(\rehyp{regpertes}), and the Hessians of the loss with respect to
$\param$ are uniformly
continuous in time around $\paramopt$ (\rehyp{equicontH_simple}, always
satisfied if all the functions involved are $C^3$ with uniformly bounded
first, second and third derivatives). For instance, if gradients and
Hessians of the loss are bounded over time close to $\paramopt$ (e.g., if
working with a
finite dataset), then $\gamma=0$.

Our first result is local convergence of the RTRL algorithm under these
assumptions: if the parameter is initialized close enough to the local
optimum, then RTRL converges to that optimum. The possible range of
learning rates depends on the various exponents in the assumptions,
allowing for a larger range than the classical Robbins--Monro criterion
when cycling over a finite dataset, for instance.

\begin{theoreme}[informal, see Theorem~\ref{thm:cvapproxrtrlalgo}]
\label{thm:overview}
Consider a parameterized dynamical system
$\state_t=\opevol_t(\state_{t-1},\param)$ with loss function $\perte_t$ as above, satisfying all the
assumptions above. Let $\paramopt$ be a local optimum of the empirical
loss, in the sense above.

Let $(\pas_t)_{t\geq 0}$ be a non-increasing stepsize sequence satisfying
$\pas_t=\cdvp\, t^{-\expas}\,\paren{1+\po{1/t^\exmp}}$, where $\cdvp>0$ is
the \emph{overall learning rate}, and
$b$ is any exponent such that $\max(\expem,\exmp)+2\exmp<\expas\leq 1$,
where $\expem$ and $\exmp$ are the exponents from the assumptions above,
respectively about convergence of time averages and growth of losses.

Then there exists a neighborhood $\mathcal{N}_{\paramopt}$ of
$\paramopt$, 
a neighborhood $\mathcal{N}^\jope_{0}$ of $0$, and an overall learning rate
$\cdvpmaxconv>0$ such that 
for any overall learning rate $\cdvp< \cdvpmaxconv$, the following
convergence holds.

For any initial parameter $\param_0\in
\mathcal{N}_{\paramopt}$
and any initial differential $\jope_0\in
\mathcal{N}^\jope_{0}$,
the RTRL learning trajectory
\begin{equation*}
\left\lbrace
\ba
\state_{t} &= \opevolt\paren{\state_{t-1},\,\param_{t-1}}, \\
\jope_t &= \frac{\partial \opevol_t(\state_{t-1},\param_{t-1})}{\partial
\state} \,\jope_{t-1}+\frac{\partial
\opevol_t(\state_{t-1},\param_{t-1})}{\partial
\param}
,\\
\param_t &= \param_{t-1}-\pas_t\left(
\frac{\partial \perte_t(\state_t)}{\partial \state}\cdot \jope_t
\right)
\ea \right.
\end{equation*}
satisfies $\param_t\to \paramopt$ as $t\to\infty$.
\end{theoreme}

As far as we know, this is the first general convergence result for RTRL.

\paragraph{Extended RTRL algorithms.}
This theorem generalizes to more complex update rules for $\param$: these
can
cover, for instance, adaptive preconditioners such as Adam or online
natural gradient (Section~\ref{sec:examples}), by considering the
preconditioner as part of the parameter $\param$ to be estimated.

In that case, the passage from $\param_{t-1}$ to $\param_t$ is not
necessarily a gradient step for some loss, so we will consider more
general update rules.
Assume that the update of $\param_t$
Theorem~\ref{thm:overview} is replaced with
\begin{equation*}
\param_{t} = 
\paramupdate(\param_{t-1},\pas_t\,\vtanc_t),\qquad
\vtanc_t=\furt{\frac{\partial \perte_t(\state_t)}{\partial \state}\cdot
\jope_t}{\state_t}{\param_{t-1}}
\end{equation*}
for some operators $\fur_t$ and $\paramupdate$. Namely, $\fur_t$ computes an
update direction from the RTRL gradients and the current state and
parameter, then $\paramupdate$ applies the update with stepsize $\pas_t$.

A typical example for
$\fur_t$ is preconditioning:
$\furt{\vtanc}{\state}{\param}=P(\param)\,\vtanc$ for some
matrix-valued $P$. For preconditioners $P$ estimated online by
collecting some statistics, the quantities used to
estimate $P$ can be treated as part of $\param$ (see examples in
Section~\ref{sec:examples}).

We assume that
$\paramupdate(\param,\vtanc)$ coincides with $\param-\vtanc$ at first
order in $\norm{\vtanc}$ (\rehyp{updateop_first}): this covers for instance
capped gradient steps such as
$\param-\frac{\vtanc}{\max(1,\norm{\vtanc})}$, or Riemannian exponentials
$\exp_{\param}(\vtanc)$ expressed in a coordinate system.

We do not make assumptions on the general form of $\fur_t$ except for
technical assumptions (Assumption~\ref{hyp:fupdrl}): $\fur_t$ is $C^1$, at
most linear with respect to its first argument, and with bounded
derivatives close to the optimal parameter $\paramopt$. This covers
preconditioned updates 
$\furt{\vtanc}{\state}{\param}=P(\param)\,\vtanc$
(Remark~\ref{rem:affineU}).

However, changing the update rule for $\param$ changes the definition of
a local optimum: a local optimum becomes a value $\paramopt$ such that
the average update $\fur_t$ is $0$. This is a joint property of the
dynamical system and the update rule $\fur_t$. More precisely
(Assumption~\ref{hyp:critoptrtrlnbt}), we define
a ``local optimum'' for such extended update rules, as a value
$\paramopt$ such that the average update computed at $\paramopt$ tends to $0$ at rate $T^a/T$:
\begin{equation*}
\frac{1}{T}\sum_{t=1}^T
\fur_t\left(
\frac{\partial}{\partial \param}
\sbpermc{t}(\paramopt),\fstate_t(\paramopt),\paramopt
\right)
=O(T^\expem/T).
\end{equation*}
The second-order condition for a local optimum (positivity of the Hessian) involves the
``extended Hessians'', defined as the Jacobian of the update direction
with respect to $\param$. Setting
\begin{equation*}
\sbfh_t(\param)\deq \frac{\partial}{\partial \param}\left(\param\mapsto
\fur_t\left(
\frac{\partial}{\partial \param}
\sbpermc{t}(\param),\fstate_t(\param),\param
\right)\right),
\end{equation*}
the assumption states that the average extended Hessian at $\paramopt$ converges at rate
$T^a/T$,
\begin{equation*}
\frac{1}{T}\sum_{t=1}^T \sbfh_t(\paramopt)
=\linalgmatrix+O(T^\expem/T)
\end{equation*}
to some matrix $\linalgmatrix$ all of whose eigenvalues have positive real part.
The standard case is $\fur_t(\vtanc,\state,\param)=\vtanc$: then
these conditions reduce to the average gradient being $0$ and the
average Hessian being positive definite. For a preconditioning 
$\furt{\vtanc}{\state}{\param}=P(\param)\,\vtanc$ with known
(non-adaptive) matrix $P(\param)$, these conditions hold
if the average gradient is $0$, the average Hessian at $\paramopt$ is positive definite,
and $P(\paramopt)+\transp{P(\paramopt)}$ is positive definite
(Section~\ref{sec:precond}).

For adaptive algorithms, we include other quantities as part of
the parameter $\param$ to be estimated (such as the average square gradients
in Adam). Then the update of $\param$ is not a gradient udpate anymore,
and the ``extended Hessian'' is not a symmetric matrix anymore. 
Considering the analogous continuous-time 
 dynamical system $\param'=-\fur(\param)$, it is known
that stability of a fixed point $\paramopt$ does not require the Jacobian
$\partial_\param \,\fur(\paramopt)$ of the
update to be symmetric definite positive, only for its eigenvalues to
have positive real part, and this is what we will use.

Under these assumptions on $\fur_t$ and $\paramupdate$, and under the
same conditions as in Theorem~\ref{thm:overview}, the learning
trajectories of the extended RTRL algorithm
\begin{equation*}
\left\lbrace
\ba
\state_{t} &= \opevolt\paren{\state_{t-1},\,\param_{t-1}}, \\
 \jope_t &= \frac{\partial \opevol_t(\state_{t-1},\param_{t-1})}{\partial
\state} \, \jope_{t-1}+\frac{\partial
\opevol_t(\state_{t-1},\param_{t-1})}{\partial
\param}
,\\
\vtanc_t &= \furt{\frac{\partial \perte_t(\state_t)}{\partial \state}\cdot  \jope_t
}{\state_t}{\param_{t-1}},\\
\param_{t} &= 
\paramupdate(\param_{t-1},\pas_t\,\vtanc_t),
\ea \right.
\end{equation*}
satisfy $\param_t\to \paramopt$ as $t\to\infty$
(Theorem~\ref{thm:cvapproxrtrlalgo}).

\paragraph{Corollaries for non-recurrent situations: cycling over
samples, adaptive preconditioning...} Here we present some consequences of
these results for non-recurrent models, in a standard setting for machine
learning applications. Namely, we consider a finite dataset
$D=(x_n,y_n)_{n\in[1;N]}$ of
inputs and labels (with values in any sets), together with a loss
function $\perteiid(x,y,\param)$ for an input-label pair $(x,y)$,
depending on a parameter $\param$. We assume $\perteiid$ is $C^3$ with
respect to $\param$.

A 
\emph{strict local optimum} for this problem is a local optimum of the average loss
with positive definite Hessian, namely, a parameter $\paramopt$ such that
\begin{equation*}
\frac{1}{N}\sum_{n=1}^N \partial_\param \perteiid(x_n,y_n,\paramopt)=0
,\qquad
\frac{1}{N}\sum_{n=1}^N \partial^2_\param
\perteiid(x_n,y_n,\paramopt)\succ 0.
\end{equation*}
We say that an algorithm to learn $\paramopt$ \emph{converges locally} if
there is a neighborhood of $\paramopt$ and a maximal overall learning
rate $\cdvp_\mathrm{max}$ (with learning rates as in Theorem~\ref{thm:overview})
such
that, if the parameter is initialized in this neighborhood and the
overall learning rate $\cdvp$ is smaller than $\cdvp_\mathrm{max}$, then the
sequence of parameters produced by the algorithm converges to
$\paramopt$.

First, the ``ergodic'' viewpoint used to define local optima in the
recurrent case illustrates the different behavior of different data
sampling strategies for stochastic gradient descent in the non-recurrent case. In pure i.i.d.\ sampling,
at each step a sample from the dataset is selected at random with
replacement. In that case, an empirical average of some quantity over
$T$ samples converges
to the average over the dataset at rate $1/\sqrt{T}$ (variance $1/T$),
so that the ergodic assumption above is satisfied with exponent $a>1/2$.
On the other hand, if cycling over all examples in the dataset, or if
randomly reshuffling the dataset before each pass on the dataset, then
empirical averages over $T$ samples converge to the dataset average at
rate $1/T$, so the ergodic assumption is satisfied with any exponent
$a>0$. Cycling or reshuffling acts as a variance reduction method.

Since admissible learning rates in Theorem~\ref{thm:overview}
depend on $a$, this leads to the following.

\begin{corollaire}[informal; see Section~\ref{sec:sgd}]
\label{cor:sgdrates_overview}
Consider ordinary stochastic gradient descent
\begin{equation*}
\param_{t}=\param_{t-1}-\pas_t\,\partial_{\param}\perteiid(x_{i_t},y_{i_t},\param)
\end{equation*}
over a finite dataset $D$ with loss $\perteiid$, with $i_t$ the sample
selected at step $t$. Assume the
learning rates satisfy $\pas_t\propto t^{-\expas}$ with
\begin{equation*}
\begin{cases}
0<\expas\leq 1 &\text{for cycling over $D$ or random reshuffling;}
\\
1/2<\expas\leq 1 &\text{for i.i.d.\ sampling of $i_t$.}
\end{cases}
\end{equation*}
Then this algorithm is locally convergent.
\end{corollaire}

Thus, cycling or reshuffling allows for larger learning rates than the
classical Robbins--Monro criterion.  However, it is unclear if such a
variance reduction is desirable from a statistical learning perspective:
the variance introduced by i.i.d.\ resampling is a form of bootstrap and may
be helpful to represent the inherent variance from a finite dataset.

Next,
adaptive preconditioning can be treated via the ``extended'' RTRL
algorithm using $\fur_t$ above.
We give several examples in
Sections~\ref{sec:precond}--\ref{sec:adam} (RMSProp, Adam with $\beta^2\to 1$, natural
gradient, online natural gradient). In fact, Corollary~\ref{cor:adam}
proves local convergence for stochastic gradient descent with momentum
and any parameter-dependent adaptive
preconditioning matrix $P$ estimated online
from the data,
provided $P+\transp{P}$ is positive definite when computed at
$\paramopt$ and on average over the dataset. 

Let us give the example of Adam here. Removing momentum and replacing the
entrywise square with a tensor square produces a similar result for the
online natural gradient (Section~\ref{sec:adap}).  The extended Kalman filter in the ``static''
case (for estimating a fixed state via noisy nonlinear measurements) is
strictly equivalent to a particular case of online natural gradient via a
nontrivial correspondence
\citep{ollivier2018online}, and is covered as a consequence.

\begin{corollaire}[informal, see \recor{adam}]
Consider a finite dataset $D=(x_n,y_n)$ as above. Take learning rates
$\pas_t$ as in
Corollary~\ref{cor:sgdrates_overview} depending on the sample selection
scheme.

Consider a preconditioned gradient descent algorithm with momentum, that maintains a momentum variable $J$ together
with square gradient statistics $\Stat$ updated via moving averages:
\begin{align*}
J_t&=\beta^1 J_{t-1}+(1-\beta^1)\,
\partial_\param \perteiid(x_{i_t},y_{i_t},\param_{t-1})
\\
\Stat_t&=\beta^2_t\,
\Stat_{t-1}+(1-\beta^2_t)\,(\partial_\param
\perteiid(x_{i_t},y_{i_t},\param_{t-1}))^{\odot 2}
\\
\metric_t&= \diag(\Stat_t+\eps)^{-1}
\\\param_t&=\param_{t-1}-\pas_t \metric_t J_t \qquad\text{ or }\qquad
\param_t=\param_{t-1}-\pas_t \metric_{t-1} J_t
\end{align*}
where $i_t$ is the data sampled at step $t$,
where $0\leq \beta^1<1$, where $\beta^2_t=1-\momen \eta_t$ for
some $\momen>0$, where $\odot 2$ denotes entrywise squaring of a vector, and where
$\eps>0$ is some regularizing constant.

Then this algorithm is locally convergent.
\end{corollaire}

To obtain this result, the square gradient statistics $\Stat$ collected
to compute the adaptive preconditioner are treated as a part of the
parameter to be estimated: namely, the general convergence result is
applied to $\param^+\deq(\param,\psi)$. At each step, $\psi$ is updated
by incorporating a value observed on the current sample. The update of
$\param^+$ is not a gradient step of a loss function, hence the interest
of considering the generalized update operators $\fur_t$ and the
non-symmetric generalized Hessians. Momentum is incorporated by treating
it as part of the state $\state_t$ of the dynamical system; then the
momentum variable $J_t$ coincides with the RTRL Jacobian $J_t$
(Section~\ref{sec:mom}).

\paragraph{Recurrent models: backpropagation through time, RTRL
approximations.} RTRL cannot be used
directly with large-dimensional recurrent systems due to the
impossibility to store $J_t$, whose size is $(\dim \state_t)\times (\dim
\param)$. For such systems, \emph{backpropagation through time} on time
intervals $[\tpsk{k};\tpsk{k+1}]$ allows for gradients to be computed
efficiently on each such interval \citep{jaeg,pearl}. Alternatively, low-dimensional
approximations of RTRL have been introduced, such as NoBackTrack, UORO,
or Kronecker-factored RTRL \citep{oll16, uoro, NIPS2018_7894}. We now
describe results for these situations.

Truncated backpropagation through time using time intervals
$[\tpsk{k};\tpsk{k+1}]$ of fixed length $\tpsk{k+1}-\tpsk{k}$ produces a biased algorithm:
dynamical effects exceeding the length of these intervals are ignored (see, e.g.,
the simple ``influence balancing'' example of divergence in \citet{uoro}). Thus we
let the
truncation length $\lng(T)$ grow to $\infty$ at a slow rate $t^\exli$ for some
exponent
$\exli<1$. There is a sweet spot for $\exli$, related to the learning
rates.
If $\lng(T)$ is too small, gradients are biased.
If $\flng{T}$ is too large, then the gradients computed on the time
interval $[T;T+\flng{T}]$ will be large, and the gradient step on
$\param$ at the end of each interval will be too large for convergence.
This is described by the relationship between the various exponents in
the following result; remember that $\exmp=0$ if gradients and Hessians of losses are
bounded over time close to $\paramopt$, and that $\expem$ encodes the speed at which empirical averages
along the trajectory converge to their limit over time.

\begin{theoreme}[informal, see Definition~\ref{def:tbtt} and Theorem~\ref{thm:cvtbtt}]
Consider a parameterized dynamical system
$\state_t=\opevol_t(\state_{t-1},\param)$ with loss function $\perte_t$
as above, satisfying all the
assumptions above. Let $\paramopt$ be a local optimum of the empirical
loss, in the sense above.

Let $(\pas_t)_{t\geq 0}$ be a non-increasing stepsize sequence satisfying
$\pas_t=\cdvp\, t^{-\expas}\,\paren{1+\po{1/t^\exmp}}$ where $\cdvp>0$ is
the overall learning rate and $b$ is any exponent
such that $\max(\expem,\exmp)+2\exmp<\expas\leq 1$, with 
$\expem$ and $\exmp$ the exponents in the
technical assumptions above.

Consider the truncated backpropation through time algorithm
using a sequence of time intervals
$[\tpsk{k};\tpsk{k+1}]$: the system is run with a constant parameter
during each such interval, and at time $\tpsk{k+1}$ the cumulated gradient of all
losses on $[\tpsk{k};\tpsk{k+1})$ is computed via backpropagation through
time, and the parameter $\param$ is updated by a gradient step with stepsize
$\pas_{\tpsk{k+1}}$ (Definition~\ref{def:tbtt}).

Assume $\tpsk{k+1}-\tpsk{k}$ grows like $\tpsk{k}^\exli$ for some
$\max(\expem,\exmp) < \exli < \expas-2\exmp$.

Then truncated backpropagation through time on the intervals
$[\tpsk{k};\tpsk{k+1}]$ converges locally to $\paramopt$.
\end{theoreme}

As far as we know, this type of result for truncated backpropagation
through time is new.

\bigskip

Finally, let us turn to RTRL approximations such as UORO and NoBackTrack.
In such ``imperfect'' RTRL algorithms (Definition~\ref{def:approxrtrl}),
instead of maintaining the Jacobian $\jope_t$, a smaller-dimensional
approximation $\tilde \jope_t$ is used. The computation of $\tilde
\jope_t$ follows the RTRL equation, but an additional error $E_t$ is
incurred at each time step:
\begin{equation*}
\tilde \jope_t = \frac{\partial
\opevol_t(\state_{t-1},\param_{t-1})}{\partial
\state} \,\tilde \jope_{t-1}+\frac{\partial
\opevol_t(\state_{t-1},\param_{t-1})}{\partial
\param}+E_t.
\end{equation*}
In NoBackTrack and UORO,\footnote{Although not formally covered in this
text, we believe our
results also hold for the more recently introduced
Kronecker-factored RTRL \citep{NIPS2018_7894}, which is derived from UORO. Indeed it is enough to
check that the assumptions on $E_t$ hold, in a way similar to
Section~\ref{sec:nbtuoroasapproxrtrlalgo}.} the approximation $\jope_t$
is built in a random way so that the expectation of $E_t$ is $0$ at every
step. Since the equation on $\jope_t$ is affine, all subsequent gradients are
unbiased, which allows us to prove convergence.

We give a formal mathematical description of NoBackTrack and UORO in
Section~\ref{sec:defnbtuoro}. Convergence is proved by a single result
for imperfect RTRL algorithms (Theorem~\ref{thm:cvapproxrtrlalgo}) via
general properties of the error $E_t$. Namely, convergence holds as soon as the
expectation of $E_t$ knowning previous errors $(E_s)_{s\leq t}$ is $0$ at
every step (\rehyp{wcorrnoiseapprtrl}), and that the size of the error
$E_t$ is sublinear in $\tilde \jope_t$ at every step
(\rehyp{errorgauge}). In NoBackTrack and UORO, the latter property is
ensured by the ``variance-reduction'' factors originally introduced in
the algorithm \citep{oll16, uoro} (see Section~\ref{sec:defnbtuoro}
for details): they play a major role for convergence by ensuring that the
error $E_t$ scales at most like $\sqrt{\norm{\tilde \jope_t}}$ at each
step.

For this situation we obtain a convergence result similar to RTRL, but
with stricter constraints on the learning rates, and with probability
tending to $1$ as the overall learning rate tends to $0$.

\begin{theoreme}[informal, see Theorem~\ref{thm:cvapproxrtrlalgo}]
\label{thm:overview_tbptt}
Consider a parameterized dynamical system
$\state_t=\opevol_t(\state_{t-1},\param)$ with loss function $\perte_t$ as above, satisfying all the
assumptions above. Let $\paramopt$ be a local optimum of the empirical
loss, in the sense above.

Consider an imperfect RTRL algorithm with random errors $E_t$ satisfying the
unbiasedness and sublinearity assumptions above (which hold for NoBackTrack and for UORO).

Let $(\pas_t)_{t\geq 0}$ be a non-increasing stepsize sequence satisfying
$\pas_t=\cdvp\, t^{-\expas}\,\paren{1+\po{1/t^\exmp}}$ where $\cdvp>0$ is
the overall learning rate and
$b$ is any exponent such that $\max(\expem,1/2+\exmp)+2\exmp<\expas\leq 1$,
where $\expem$ and $\exmp$ are the exponents from the assumptions above,
respectively about convergence of time averages and growth of losses.

Then there exists a neighborhood $\mathcal{N}_{\paramopt}$ of
$\paramopt$ and
a neighborhood $\mathcal{N}^\jope_{0}$ of $0$
such that
for any $\eps>0$, there exists
$\cdvpmaxconv>0$ such that for any overall learning rate $\cdvp<
\cdvpmaxconv$, with probability greater than $1-\eps$,
the following
convergence holds:

For any initial parameter $\param_0\in
\mathcal{N}_{\paramopt}$
and any initial differential $\jope_0\in
\mathcal{N}^\jope_{0}$,
the imperfect RTRL learning trajectory
\begin{equation*}
\left\lbrace
\ba
\state_{t} &= \opevolt\paren{\state_{t-1},\,\param_{t-1}}, \\
\tilde \jope_t &= \frac{\partial \opevol_t(\state_{t-1},\param_{t-1})}{\partial
\state} \,\tilde \jope_{t-1}+\frac{\partial
\opevol_t(\state_{t-1},\param_{t-1})}{\partial
\param}+E_t
,\\
\param_t &= \param_{t-1}-\pas_t\left(
\frac{\partial \perte_t(\state_t)}{\partial \state}\cdot \tilde \jope_t
\right)
\ea \right.
\end{equation*}
satisfies $\param_t\to \paramopt$ as $t\to\infty$.
\end{theoreme}

This result can also be combined with the extended update
operators $\fur_t$ and $\paramupdate$, see
Theorem~\ref{thm:cvapproxrtrlalgo}.
As far as we know, this is the first theoretical analysis of NoBackTrack
and UORO.
%

\subsection{Formal Definitions: Parameterized Dynamical System, RTRL,
Extended RTRL Algorithms}

\label{sec:formaldefs}

We now turn to fully formal definitions and technical assumptions for the
convergence theorem. Alternatively, the reader may
go directly to the applications and corollaries presented in
Section~\ref{sec:examples}.

\paragraph{Linear algebra notation.} We have tried to make the text
readable under two alternative conventions for linear algebra, with
minimal notational fuss. With the programmer's notation, the parameter
and state $\param$ and $\state$ are tuples of real numbers. This
convention makes no difference between row or column tuples or
vectors, so that we write simple stochastic gradient descent as
$\param\gets \param-\pas \,\partial \perte/\partial \param$, ignoring the
fact that the tuple $\partial \perte/\partial \param$ is formally a linear form
(row vector), that is mapped to a vector using the canonical quadratic
form on $\mathbb{R}^{\dim(\param)}$. This is the convention most relevant
for the applications (Section~\ref{sec:examples}).

For the bulk of the mathematical proof, we treat the state and parameter
of the system as elements of some finite-dimensional vector spaces and
use standard differential geometry notation. Given vector spaces $E$ and
$F$, we denote $\sblin(E,F)$ the set of linear maps from $E$ to $F$.
Given a smooth map $f\colon E \to F$ and $x \in E$, the differential
$\frac{\partial f}{\partial x}(x)$ of $f$ at $x$ is an element of
$\sblin(E,F)$, which can be represented by the Jacobian matrix $\partial
f_i(x)/\partial x_j$ in a basis. 
We write indifferently $\partial_x f$
or $\frac{\partial f}{\partial x}$, depending on typography.

In particular, derivatives of the loss
are linear forms $\partial_\param \perte\in \sblin(\Param,\mathbb{R})$,
not vectors:
this is necessary for consistency of the chain rule. This double convention occasionally leads to a few inconsistencies:
notably, the Hessian $\partial^2_\param \perte(\param)\in
\sblin(\sblin(\Param,\mathbb{R}),\mathbb{R})$ is formally a
$(0,2)$-tensor (a row vector of row vectors), but we sometimes abuse notation and
treat it as a matrix. The same occurs for the Lyapunov matrix $B$ of
Sections~\ref{sec:rtrlasalgo} and \ref{sec:contaroundparamopt}.

If the vector spaces $E$ and $F$ are equipped with some
norms, we always equip $\sblin(E,F)$ with the operator norm
$\nrmop{f}\deq \sup_{x\neq 0} \norm{f(x)}/\norm{x}$. We follow this
convention for compound spaces: for example, the proofs involve spaces of
the type $\sblin(\sblin(E,F),G)$, which is equipped with the operator
norm coming from the operator norm on $\sblin(E,F)$ and
the norm of $G$.

For pairs, such as the state-parameter pair $(\state,\param)$ appearing
in some assumptions below,
we
use the supremum norm; for instance, 
$\norm{(\state,\param)}\deq
\max(\norm{\state},\norm{\param})$.

\paragraph{Parameterized dynamical systems, RTRL.} We now provide the
formal definitions for RTRL on a parameterized dynamical system.

\begin{definition}[Parameterized dynamical system]
\label{def:prmdynsys}
We consider a dynamical system parameterized by $\param\in\Param$, whose state
$\state_t\in\State_t$ at time $t\geq 1$ is subjected to
the evolution equation
\begin{equation*}
\state_{t}=\opevol_{t}(\state_{t-1},\param),
\end{equation*}
where $\Param\simeq \R^{\dim(\param)}$ and $\State_t\simeq
\R^{\dim(\state_t)}$ are some finite-dimensional
Euclidean
vector
spaces (not necessarily of constant dimension with time $t$), and, for
each $t\geq 1$,
\begin{equation*}
\opevol_{t}\from \State_{t-1}\times \Param \to \State_{t}
\end{equation*}
is a (time-dependent) $C^2$ map, the \emph{transition operator}.

Such data will be called a \emph{parameterized dynamical system}. A
sequence of states $(\state_t)_{t\geq 0}$ satisfying the evolution
equation will be called a \emph{trajectory} with parameter $\param$.

We denote by $\fstate_t\from \State_0\times \Param \to \State_t$ the
function that to $\state_0\in \State_0$
and $\param\in \Param$, associates the value $\state_t$ at time $t$ of the
trajectory starting at $\state_0$ with parameter $\param$.
\end{definition}


As usual in statistical learning, the quality of the parameter is assessed through loss functions.
\begin{definition}[Loss function]
\label{def:fcpletaprm}
A \emph{loss function} along a parameterized dynamical system, is a
family of functions
\begin{equation*}
\perte_t\from \State_t\to \R
\end{equation*}
for each integer $t\geq 1$. We assume that $\perte_t$ is $C^2$ for all
$t$. 
Given $t\geq 1$, $\param\in \Param$ and $\state_0\in\State_0$ we
denote
\begin{equation*}
\sbpermc{t}(\state_0,\param)\deq
\perte_t(\fstate_t(\state_0,\param))
\end{equation*}
the loss function at the state obtained at time $t$ from $\param\in
\Param$ and $\state_0\in\State_0$.
\end{definition}
Our smoothness assumptions on the $\opevol_t$'s and the $\perte_t$'s
imply that $\sbpermc{t}$ is $C^2$ for all $t\geq 1$.

Here, we have assumed that the state $\state_t$ of the system at time $t$
contains all the information necessary to compute the loss. In some
applications, the loss has an additional explicit dependency on $\param$;
this can be dealt with by including the current parameter as part of the
state, namely, working on the augmented state $\state_t^+\deq
(\state_t,\param)$.\footnote{More precisely, use the extended state
space $\State_t^+\deq
\State_t\times \Param$ together with the extended transition operators
$\opevol_t^+((\state_{t-1},\param_{t-1}),\param_t)\deq
(\opevol_T(\state_{t-1},\param_t),\param_t)$, thus, always storing the
latest parameter value in the state. Notably, this does not affect the spectral
radius of the operators in Definition~\ref{def:specrad}, so that the
stability assumption~\ref{hyp:specrad} is satisfied for the extended
system if and only if it is satisfied for the basic system.}

The goal of training is to find a parameter $\paramopt$ such that the
asymptotic average loss
\begin{equation*}
\lim_{T\to \infty} \frac1T \sum_{t=1}^T \sbpermc{t}(\state_0,\paramopt)
=\lim_{T\to \infty} \frac1T \sum_{t=1}^T \perte_t(\state_t,\paramopt)
\end{equation*}
is as small as possible, where $(\state_t)$ is the trajectory defined by
$\paramopt$ and $\state_0$.\footnote{A priori this may depend on $\state_0$. We
can either decide that $\state_0$ is fixed once and for all by the
algorithm, or formally let $\state_0$ be part of the parameter to be
optimized. But in the end, under our ergodicity assumptions, the state
$\state_0$ will be forgotten and the asymptotic average loss will not
depend on $\state_0$.}


Let us now define the RTRL algorithm presented informally in
Section~\ref{sec:rtrloverview}.

\begin{definition}[RTRL algorithm]
\label{def:simplertrl}
The RTRL algorithm with step sizes $(\pas_t)_{t \geq 1}$, starting at
$\state_0\in \State_0$ and $\param_0\in \Param$, maintains a
state $\state_t\in\State_t$, a parameter $\param_t\in \Param$, and a
Jacobian estimate $\jope_t\in \epjopeins{t}$, subjected to the evolution
equations
\begin{equation*}
\left\lbrace
\ba
\state_{t} &= \opevolt\paren{\state_{t-1},\,\param_{t-1}}, \\
\jope_t &= \frac{\partial \opevol_t(\state_{t-1},\param_{t-1})}{\partial
\state} \,\jope_{t-1}+\frac{\partial
\opevol_t(\state_{t-1},\param_{t-1})}{\partial
\param}%
, \qquad \jope_0=0,\\
\vtanc_t &= \frac{\partial \perte_t(\state_t)}{\partial
\state}\cdot
\jope_t
,\\
\param_{t} &= 
\param_{t-1}-\pas_t\,\vtanc_t  
\ea \right.
\end{equation*}
for $t\geq 1$.
\end{definition} 

We will also deal with more general algorithms that perform more
complicated updates on the parameter: preconditioning, adaptive
per-parameter learning rates, additional error terms... These will be obtained by applying
transformations $\fur_t$ and $\Phi_t$ to the gradient directions computed
by RTRL; we will specify assumptions on $\fur_t$ and $\paramupdate_t$
later (Section~\ref{sec:furhyp}).

\begin{definition}[Extended RTRL algorithm]
\label{def:algortrl}
An extended RTRL algorithm with step sizes $(\pas_t)_{t \geq 1}$, starting at
$\state_0\in \State_0$ and $\param_0\in \Param$, maintains a
state $\state_t\in\State_t$, a parameter $\param_t\in \Param$, and a
Jacobian estimate $J_t\in \epjopeins{t}$, subjected to the evolution
equations
\begin{equation*}
\left\lbrace
\ba
\state_{t} &= \opevolt\paren{\state_{t-1},\,\param_{t-1}}, \\
\jope_t &= \frac{\partial \opevol_t(\state_{t-1},\param_{t-1})}{\partial
\state} \,\jope_{t-1}+\frac{\partial
\opevol_t(\state_{t-1},\param_{t-1})}{\partial
\param}%
, \qquad \jope_0=0,\\
\vtanc_t &= \furt{\frac{\partial \perte_t(\state_t)}{\partial
\state}\cdot
\jope_t
}{\state_t}{\param_{t-1}
}
,\\
\param_{t} &= 
\paramupdate_t(\param_{t-1},\pas_t\,\vtanc_t)
\ea \right.
\end{equation*}
for $t\geq 1$, for some choice of update functions $\fur_t$ and
$\paramupdate_t$.
\end{definition} 

\paragraph{\Algonoisy RTRL Algorithms.}
The RTRL algorithm is unreasonably heavy in most situations, because
$\jope_t$ is an object of dimension $\dim(\state_t)\times \dim(\param)$.
Several approximation algorithms are in existence, such as \nbt{} or UORO. They
usually store a smaller-dimensional approximation $\tilde \jope_t$ of
$\jope_t$ (such as a small rank approximation). Since computing $\jope_t$
from $\jope_{t-1}$ tends to break this smaller-dimensional structure, the approximation
has to be performed after every step. Thus, these algorithms introduce an
additional error $E_t$ at each step in the computation of $\jope_t$:
\begin{equation*}
\tilde \jope_{t}
=\dpartf{\state}{\opevolt}\paren{\state_{t-1},\,\param_{t-1}}
\cdot \tilde \jope_{t-1} +
\dpartf{\prmctrl}{\opevolt}\paren{\state_{t-1},\,\param_{t-1}}+E_t.
\end{equation*}
In \nbt, UORO, and Kronecker-factored RTRL, these errors are built in a random way to be centered on
average.

For now, we just define an \algonoisy RTRL algorithm to be one
that incurs some error $E_t$ on $\jope_t$;
Assumptions~\ref{hyp:wcorrnoiseapprtrl} and~\ref{hyp:errorgauge} below
will require this noise to be not too large (sublinear in
$\jope_t$) and centered on average.

\begin{definition}[\Algonoisy RTRL algorithm]
\label{def:approxrtrl}
An \emph{\algonoisy RTRL algorithm} with step sizes $(\pas_t)_{t \geq 1}$, starting at
$\state_0\in \State_0$ and $\param_0\in \Param$, is any algorithm that maintains a
state $\state_t\in\State_t$, a parameter $\param_t\in \Param$, and a
Jacobian estimate $\tilde J_t\in \epjopeins{t}$, subjected to the evolution
equations for $t\geq 1$
\begin{equation*}
\left\lbrace
\ba
\state_{t} &= \opevolt\paren{\state_{t-1},\,\param_{t-1}}, \\
\tilde \jope_t &= \frac{\partial \opevol_t(\state_{t-1},\param_{t-1})}{\partial
\state} \,\tilde \jope_{t-1}+\frac{\partial
\opevol_t(\state_{t-1},\param_{t-1})}{\partial
\param}+E_t
,\qquad \tilde\jope_0=0,\\
\vtanc_t &= \furt{\frac{\partial \perte_t(\state_t)}{\partial
\state}\cdot
\tilde \jope_t
}{\state_t}{\param_{t-1}}
,\\
\param_{t} &= 
\paramupdate_t(\param_{t-1},\pas_t\,\vtanc_t)
\ea
\right.
\end{equation*}
for some error term $E_t\in \epjopeins{t}$.
\end{definition}

The errors $E_t$ can be seen as noise on the computation of $\jope$
performed by the RTRL algorithm. They play a somewhat different role
from usual SGD gradient noise (which is encoded by the dependency on
$t$ in $\perte_t$, usually depending on output data $y_t$ at time
$t$): first, $E_t$ is transmitted from one step to the
next in the recurrent computation of $\jope_t$; second, these
errors are introduced by the optimization algorithm while $\perte_t$ is
part of the specification of the initial problem.

\subsection{Assumptions for Local Convergence}
\label{sec:syshyp}

We will prove local convergence of RTRL and extended RTRL algorithms
under several assumptions (\rethm{cvapproxrtrlalgo}). The various other algorithms described in the
introduction are obtained as corollaries by a suitable choice of the
update operator $\fur_t$ and a suitable definition of the system state
$\state_t$ encompassing the internal state of an algorithm; this is done
in Section~\ref{sec:examples}.

We subdivide the assumptions of our local convergence theorem into
``non-technical'' assumptions (properties of a strict local optimum, stability of
the target dynamical system, centered errors $E_t$ for \algonoisy RTRL),
and ``technical'' assumptions (grouped in Section~\ref{sec:techass}). The latter are
``technical'' in the sense that they would always be satisfied on a
finite dataset if every function involved is smooth, for the standard
parameter update operators $\fur_t$ and $\paramupdate_t$.

Let us start with the non-technical assumptions.
To prove convergence of the learning algorithm towards a local optimum
$\paramopt$, we need two key assumptions: first, that $\paramopt$ is
indeed a local optimum of the loss function.
Second, that the
system with fixed parameter $\paramopt$ is \emph{stable} in the classical
sense of dynamical systems, namely: if the parameter is fixed to
$\param=\paramopt$, and if the inputs of the system are fixed (here the
inputs are implicit in the definition of the transition operators
$\opevol_t$), then the system eventually forgets its initial state.

Moreover, for extended RTRL algorithms, we assume that applying $\fur_t$ and
$\Phi_t$ behaves reasonably like a gradient step. For
\algonoisy RTRL algorithms ($E_t\neq 0$), we will assume
that the errors are centered and sublinear in $\jope_t$.

So, let $\paramopt\in \Param$ and $\stateopt_0\in\State_0$. Let
$(\stateopt_t)_{t\geq 0}$ be the trajectory starting at $\stateopt_0$
with parameter $\paramopt$. Provided the assumptions below are satisfied, 
we will refer to $\paramopt$ as the \emph{local optimum}, and to the
trajectory $\stateopt_t\deq\fstate_t(\paramopt,\stateopt_0)$ obtained
from $\paramopt$ as the \emph{target trajectory}.

The assumptions below are all required to hold locally: either at the target
trajectory itself, or only in some
neighborhood of the target trajectory. Thus, we fix some radii
$r_\Param>0$ and $\raysy>0$, and we will require these assumptions to
hold in the balls
$B_\Param(\paramopt,r_\Param)$ in $\Param$ and
$B_{\State_{t}}(\stateopt_{t},\raysy)$ in $\State_t$.

\subsubsection{Local Optima of the Loss for a Dynamical System}

Defining a local optimum notion for a dynamical system is not
straightfoward. In stochastic optimization, the global loss $\avloss$
associated to the parameter is the expectation, over some random variable
$i$, of a loss $\perte_i$ which depends on $i$. Often, $i$ is the random
choice of a training sample among a set of data, and $\perte_i$ is the
loss computed on this sample.
In this setting, a local extremum is a parameter $\paramopt$ such that, 
the derivative of the loss evaluated at this parameter vanishes on
average:
\begin{equation*}
\mathbb{E} \left[ \dpartf{\param}{\perte_i}(\paramopt)\right] = 0.
\end{equation*}
For a dynamical system, we will replace the expectation with respect to $i$ by a temporal average, and we define a local extremum as a point where the temporal averages of gradients converge to $0$:
\begin{equation*}
\inv{T} \, \sum_{t=0}^T \dpartf{\param}{\sbpermc{t}}\cpl{\stateopt_0}{\paramopt} \to 0, 
\end{equation*}
as $T$ tends to infinity. Note that the gradient is computed through the
whole dynamics, thanks to the use of $\sbpermc{t}$, which encodes the
dependency of $\state_t$ on $\param$. Here no probabilistic assumption is made
on the inputs or outputs to the system: instead we work under this
``ergodic'' assumption of time averages. The classical case
(\reeq{nonrec}) of a
non-recurrent situation corresponds to i.i.d.\ losses $\perte_t$, so that
the ergodic assumption is satisfied with probability $1$ by the
law of large numbers.

For the extremum $\paramopt$ to be a minimum, and in order to guarantee
the convergence of the gradient descent, we also need a second order
condition. We will assume that temporal averages of the Hessians,
evaluated at the local extremum, end up being positive definite: the
smallest eigenvalue of
\begin{equation*}
\inv{T} \, \sum_{t=0}^T \ddpartf{\param}{\sbpermc{t}}\cpl{\stateopt_0}{\paramopt} 
\end{equation*}
when $T\to \infty$,
should be positive.

Actually, the rate of convergence (with respect to $T$) of these limits
will affect the range of possible learning rates. For instance, in the
non-recurrent case, cycling over a finite dataset results in a
convergence $O(1/T)$ of gradients to their average, as opposed to the usual statistical rate
$O(1/\sqrt{T})$. This will allow for learning rates $\pas_t=t^{-b}$ with any
$b>0$, instead of the classical $b>1/2$ for i.i.d.\ samples. Cycling over
samples acts as a variance reduction method; the ergodic
viewpoint makes these distinctions clear.

This is why we introduce an exponent $\expem$ in the next assumption,
controlling the rate at which the gradients at $\paramopt$ tend to their
average.


For a simpler exposition, we first state a version of the
assumption corresponding to non-extended algorithms,
$\furt{\vtanc}{\state}{\param}=\vtanc$. Then the assumption corresponds
to $\paramopt$ being a strict local minimum in the traditional
sense: gradients at $\paramopt$ average to $0$, and the Hessian of the
expected loss at $\paramopt$ is definite positive.

Remember that the loss function $\sbpermc{t}$ (Def.~\ref{def:fcpletaprm})
encodes the loss at time $t$ of a parameter $\param$ when the system is
run with that parameter
from time $0$ to time $t$, and is $C^2$ for all $t$. 

\newcounter{subhyp}
\renewcommand{\thetheoreme}{\thesection.\arabic{theoreme}.\alph{subhyp}}

\refstepcounter{subhyp}
\begin{hypothese}[$\paramopt$ is a local optimum of the average loss
function]
\label{hyp:opt_simple}
We assume the existence of a parameter $\paramopt\in \Param$, an initial state
$\stateopt_0\in \State_0$, and an exponent $0 < \expem < 1$ such that: 
\begin{enumerate}
\item Gradients of the loss at $\paramopt$ average to $0$, at rate
$t^\expem/t$:
\begin{equation*}
\frac{1}{T}\sum_{t=1}^T
\frac{\partial}{\partial \param}
\sbpermc{t}(\stateopt_0,\paramopt)=O(T^\expem/T).
\end{equation*}
\item On average, Hessians of the loss at $\paramopt$ converge to a
positive definite matrix, at rate $t^\expem/t$: there is a positive definite
matrix $H$ such that
\begin{equation*}
\frac{1}{T}\sum_{t=1}^T
\frac{\partial^2}{\partial \param^2}
\sbpermc{t}(\stateopt_0,\paramopt)=H+O(T^\expem/T).
\end{equation*}
\end{enumerate}
\end{hypothese}

Next we express the
corresponding assumption for
extended algorithms $\fur$: Assumption~\ref{hyp:critoptrtrlnbt} reduces to
Assumption~\ref{hyp:opt_simple} when
$\furt{\vtanc}{\state}{\param}=\vtanc$.
With extended algorithms (non-trivial $\fur$), the optimality assumption
works out as follows. Note that it becomes a joint property of the dynamical
system and the optimization algorithm: this expresses a condition on
$\paramopt$ to be a fixed point of the algorithm\footnote{The examples
of $\fur_t$
in Section~\ref{sec:examples} will still converge to the same local optima. But
for instance,
with one-dimensional data,
by letting $\fur_t(\vtanc,\state,\param)$ interpolate between $\vtanc$
and $\mathrm{sign}(\vtanc)$, we could interpolate between computing a mean
or a median, so that $\paramopt$ depends on the algorithm.}.

\addtocounter{theoreme}{-1}
\refstepcounter{subhyp}
\begin{hypothese}[$\paramopt$ is a local optimum of the extended algorithm]
\label{hyp:critoptrtrlnbt}
We assume the existence of a parameter $\paramopt\in \Param$, an initial state
$\stateopt_0\in \State_0$, and an exponent $0 < \expem < 1$ such that: 
\begin{enumerate}
\item Updates of the open-loop algorithm at $\paramopt$ average to $0$,
at rate $t^\expem/t$:
\begin{equation*}
\frac{1}{T}\sum_{t=1}^T
\fur_t\left(
\frac{\partial}{\partial \param}
\sbpermc{t}(\stateopt_0,\paramopt),\fstate_t(\stateopt_0,\paramopt),\paramopt
\right)
=O(T^\expem/T).
\end{equation*}
\item On average, Jacobians of the update at $\paramopt$ converge to a
positive-stable matrix, at rate $t^a/t$. Namely, denoting 
\begin{equation}
\label{eq:jacobians}
\sbfh_t(\param)\deq \frac{\partial}{\partial \param}\left(\param\mapsto \fur_t\left(
\frac{\partial}{\partial \param}
\sbpermc{t}(\stateopt_0,\param),\fstate_t(\stateopt_0,\param),\param
\right)\right),
\end{equation}
we assume
there is a 
matrix $\linalgmatrix\in \sblin(\Param,\Param)$ such that 
\begin{equation*}
\frac{1}{T}\sum_{t=1}^T \sbfh_t(\paramopt)
=\linalgmatrix+O(T^\expem/T)
\end{equation*}
and all eigenvalues of $\linalgmatrix$ have positive real part.
\end{enumerate}
\end{hypothese}

\renewcommand{\thetheoreme}{\thesection.\arabic{theoreme}}

Dealing with positive-stable matrices, instead of just symmetric definite
positive matrices, is crucially needed for adaptive algorithms such as
RMSProp and Adam (see Section~\ref{sec:examples}): the associated updates
do not correspond to a gradient direction. Technically this does not pose
added difficulties. A reminder on positive-stable matrices is included in
Appendix~\ref{sec:positivestable}.

\subsubsection{Stability of the Target Trajectory}

The next non-technical assumption deals with stability: if the
target trajectory is numerically unstable as a dynamical system, then it is unlikely that
an RTRL-like algorithm could learn it online. (Besides, the interest of
learning unstable models is debatable.) Thus, we will assume that the
target trajectory defined by the local optimum $\paramopt$ is stable.

Linear systems
$\state_t=A(\param)\state_{t-1}+B(\param)+C(\param)x_t$ with inputs $x_t$ are stable if the spectral
radius of $A$ is less than $1$ \citep{willems1970stability}, namely, if there exists $k\geq
1$ such that $\nrmop{A^k}<1$. Since we are going to consider
time-inhomogeneous, nonlinear systems we need a slightly extended
definition.

\begin{definition}[Spectral radius of a sequence of linear operators]
\label{def:specrad}
A sequence of linear operators $(A_t)_{t\geq 0}$ on a normed vector space
is said to have \emph{spectral radius less than $1$} if there exists
$\alpha>0$ and an integer $\hsr\geq 1$ (called \emph{horizon}) such that for any $t$, the product
$A_{t+\hsr-1}\ldots...A_{t+1}A_t$ has operator norm less than
$1-\alpha$.
\end{definition}

For a constant sequence $A_t\equiv A$ on a finite-dimensional space, this
is equivalent to $A$ having spectral radius less than $1$.

\begin{hypothese}[The system with parameter $\paramopt$ is stable around $\stateopt$]
\label{hyp:specrad}
Let
\begin{equation*}
A_t\deq \frac{\partial \opevol_t}{\partial
\state}(\stateopt_{t-1},\paramopt).
\end{equation*}
Then the sequence $(A_t)_{t \geq 1}$ has spectral radius less than $1$.
\end{hypothese}

For a linear system $\state_t=A(\param)\state_{t-1}+B(\param)$ this
boils down to classical stability for the parameter $\param=\paramopt$,
namely, $A(\paramopt)$ has spectral radius less than $1$.

In the non-recurrent case \eqref{eq:nonrec}, this is always satisfied,
since $\frac{\partial \opevol_t}{\partial
\state}=0$.

A sufficient condition for this criterion is that every $A_t$ has
operator norm less than $1-\alpha$. So, for a simple RNN given by
\eqref{eq:rnn}, a sufficient condition would be that the matrix $W$ has
operator norm less than $4(1-\alpha)$ (because the sigmoid is
$1/4$-Lipschitz). But this sufficient condition is far from necessary.
Stability can also be checked empirically on a learned model by adding small
perturbations.

For advanced recurrent models such as LSTMs, this criterion might
be too restrictive because it imposes a time horizon $k$ for
contractivity, thus requiring the trained model to have finite effective memory, while
LSTMs are specifically designed to have arbitrarily long memory. Allowing
the learned model to have infinite memory
in our framework would require allowing the spectral
radius of the sequence to tend to $1$ over time, but this is beyond the
scope of the present work.

\subsubsection{Extended RTRL Algorithms: Assumptions on $\fur_t$ and
$\paramupdate_t$}
\label{sec:furhyp}

The next assumptions deal with with extended RTRL algorithms, namely,
with the kind of update operators $\fur_t$ and $\paramupdate_t$ that can be used instead of directly
adding the gradient, $\param\gets \param-\pas \,\partial_\param\perte$ as
in simple SGD.

The standard RTRL
algorithm corresponds to $\fur_t(\vtanc,\state,\param)=\vtanc$. The
assumption on $\fur_t$ states that 
$\fur_t$ is smooth and behaves (sub)linearly
with respect to its first argument $\vtanc$. In extended RTRL algorithms, 
$\fur_t$ is applied to $\vtanc=\partial_\state \sbperte_t\cdot J$ which is a
linear form $\vtanc\in \linform$ that
encodes the
RTRL estimated gradient $\partial_\param \sbpermc{t}$.


\begin{hypothese}[Extended update rules $\fur_t$]
\label{hyp:fupdrl}
The extended update rules used in the extended RTRL algorithm
are $C^1$ functions $\fur_t \colon \linform \times \State_t \times \Param \to
\Param$.

We assume that, in a neighborhood of the target
trajectory $(\stateopt_t,\paramopt)$, the first derivative of
$\fur_t$ with
respect to $v$ is bounded, and its first derivative with respect to
$\cpl{\state}{\param}$
is at most linear in $\vtanc$. Namely, we assume that 
there exists a constant $\kappa_\fur>0$ such that, for any $t\geq 1$, for
any $v\in\linform$, $\state\in B_{\State_t}(\stateopt_{t},\raysy)$
and $\param\in B_\Param(\paramopt,r_\Param)$, one has
\begin{equation*}
\nrmop{\dpartf{\vtanc}{\fur_t}\tpl{\vtanc}{\state}{\param}} <
\kappa_\fur
\qquad\text{and}\qquad
\nrmop{\dpartf{\cpl{\state}{\param}}{\fur_t}\tpl{\vtanc}{\state}{\param}}
\leq \kappa_\fur\,(1+\nrm{\vtanc}).
\end{equation*}
%
Finally, we assume that
\begin{equation*}
\furt{0}{\stateopt_t}{\paramopt}=\go{t^\exmp}
\end{equation*}
when $t\to \infty$,
for some exponent
$0\leq \exmp<1$ (also used in
Assumptions~\ref{hyp:regpertes}
and~\ref{hyp:scfcstepsize}).
\end{hypothese}
  
\begin{remarque}
\label{rem:affineU}
This covers notably the case of preconditioned SGD or RTRL, namely,
$\furt{\vtanc}{\state}{\param}=P(\param)\,\vtanc$ for any smooth
matrix-valued $P$. More generally this covers the case where $\fur_t$
depends on $\vtanc$ in an affine way, namely,
$\furt{\vtanc}{\state}{\param}=P_t(\param)\,\vtanc+Q_t(\param)$ where
$P_t$
and $Q_t$ are bounded and $C^1$ close to $\paramopt$, uniformly in time.
\end{remarque}

\begin{remarque}
\label{rem:errorU}
$\fur_t$ can also be used to encode small error terms in the
algorithm. For instance, 
$\fur_t(\vtanc,\state,\param)=\vtanc+\eps_t$ where $\eps_t$ is some
time-dependent error term, corresponds to a perturbed RTRL update
$\param_t=\param_{t-1}-\pas_t\,\vtanc_t -\pas_t \,\eps_t$. The size of $\eps_t$ is limited by
\rehyp{critoptrtrlnbt} which requires that $\frac1T \sum_{t\leq T}
\,\eps_t=O(T^\expem/T)$.
\end{remarque}

Finally, we assume that the update operators
$\paramupdate_t(\param,\vtanc)$ are equal to $\param-\vtanc$ up to a
second-order error in $\norm{\vtanc}$. This covers simple SGD (no
second-order term), as well as, for instance, the exponential map
$\exp_\param(-\vtanc)$ in a
Riemannian manifold, when expressed in coordinates, and
clipped updates
such as
$\param-\frac{\vtanc}{1+\norm{\vtanc}}$
(since the algorithm applies $\paramupdate$ to $\pas_t\, v$ not $v$, this amounts to
clipping the update $\pas_t\, v$, not the gradient direction $v$).

%

\begin{hypothese}[Parameter update operators]
\label{hyp:updateop_first}
We assume that the parameter update operators $\paramupdate_t\colon
\Param\times \Param\to \Param$ can be
written as
\begin{equation*}
\paramupdate_t(\param,\vtanc)=\param-\vtanc+\norm{\vtanc}^2\paramupdate_t^{(2)}(\param,\vtanc)
\end{equation*}
where
the second-order term $\paramupdate_t^{(2)}(\param,\vtanc)$ is
bounded and Lipschitz with respect to $(\param,\vtanc)$ in some ball
$B_\Param(\paramopt,r_\Param)\times
B_{\Param}(0,\tilde r_\Tangent)$, for some $\tilde r_\Tangent > 0$, uniformly in $t$. 
\end{hypothese}

\subsubsection{Assumptions on Errors for \Algonoisy RTRL Algorithms}
\label{sec:approxrtrl}

\Algonoisy RTRL algorithms such as NoBackTrack, UORO and
Kronecker-factored RTRL introduce an additional error $E_t$ in the
definition of $\jope_t$ (Def.~\ref{def:approxrtrl}). This error has been
built to be centered on average; since the evolution equation for
$\jope_t$ is affine, this property is preserved through time, a key point
in the theoretical analysis.

Thus, we will assume that the errors $E_t$ are random, and centered on
average, knowing everything that has happened up to time $t$.

\begin{hypothese}[Unbiased errors $E_t$ for \algonoisy RTRL]
\label{hyp:wcorrnoiseapprtrl}
We assume that the errors $E_t$ are random variables that satisfy, for every $t\geq 1$,
\begin{equation*}
\econd{E_t}{E_1,\ldots,E_{t-1},\tribo}=0,
\end{equation*}
where 
$\tribo$ is the $\sigma$--algebra generated by the initial parameter
$\param_0$, the initial state $\state_0$, the initial Jacobian
estimate $\tilde \jope_0$, and all the algorithm operators,
namely $\paren{\opevol_t}_{t\geq 1}$, $\paren{\perte_t}_{t\geq 1}$,
$\paren{\fur_t}_{t\geq 1}$ and $\paren{\paramupdate_t}_{t\geq 1}$.
\end{hypothese}

Note that $\tribo$ contains all future algorithm \emph{operators} (the
mathematical operations defining the transitions of the dynamical system), not
the values of the states themselves.

In RTRL approximations such as NoBackTrack or UORO, the noise $E_t$ is
not imposed by the problem, but
user-chosen to simplify computation of $\tilde J$.
The assumption states that this noise should be uncorrelated
from all other sources of randomness of the problem, past and future,
contained in $\tribo$.
Notably, the data samples are implicitly contained in $\tribo$ via the algorithm operators
$\paren{\opevol_t}_{t\geq 1}$ and $\paren{\perte_t}_{t\geq 1}$ (see
\resec{rtrloverview}), though the states themselves are not. Thus, this
assumption precludes using a recurrent noise $E_t$ that
would be correlated to future random choices of data samples (this would
obviously produce biases).  Since
$E_t$ is user-chosen in NoBackTrack or UORO, this is not a problem: just
build $E_t$ from random numbers independent from other random choices
made by the user.  This assumption also precludes ``adversarial''
recurrent settings in which the universe would send future data that are
correlated to the user-chosen noise $E_t$.

Moreover, we assume that the error $E_t$ is almost surely sublinear with
respect to $\jope$.
For this, let us first define \emph{error gauge} functions which capture this
sublinearity.

\begin{definition}[Gauge for the error]
\label{def:errorgauge}
We call \emph{error gauge} a function $\sbeg\colon \mathbb{R}_+^2
\to\mathbb{R}_+$ such that
\begin{enumerate}
\item $\sbeg$ is bounded on any compact;
\item $\feg{x}{y}$ is negligible in front of $x$, when $x$ goes to
infinity, uniformly for $y$ in compact sets: for any compact set
$\mathcal{K}\subset \mathbb{R}_+$,
when $x\to \infty$ we have
\begin{equation*}
\sup_{y\in\mathcal{K}}\,\feg{x}{y}=\po{x}.
\end{equation*}
\end{enumerate}
\end{definition}





\begin{remarque}[Example of error gauge]
For instance, any function
$\feg{x}{y}=C\,(1+x^\beta)(1+y)$
with $C>0$ and $\beta < 1$, is an error gauge. With $\beta=1/2$, this is the error gauge for the NoBackTrack and UORO algorithms (see \resec{nbtuoroasapproxrtrlalgo}).
\end{remarque}

\begin{hypothese}[Control of the error by the gauge]
\label{hyp:errorgauge}
We assume that there exists a gauge function $\sbeg$ such that the error
$E_t$ of the \algonoisy RTRL algorithm satisfies, 
for all $t \geq 1$,
\begin{equation*}
\nrmop{E_t}\leq \feg{\nrmop{\tilde J_{t-1}}}{\nrmop{\frac{\partial \opevol_t(\state_{t-1},\param_{t-1})}{\partial(\state,\param)}}}.
\end{equation*}
\end{hypothese}

Different \algonoisy RTRL algorithms may admit the same error gauge.
This way, the bounds developed below will be satisfied for all these algorithms simultaneously.

Finally, the noise $E_t$ on $\tilde J$ needs to stay
centered after computing the update direction via $\fur(\partial_\state
\perte\cdot \tilde J,\state,\param)$, so we assume that $\fur$ is linear
with respect to its first argument.

\begin{hypothese}[Linearity of the extended updates with respect to the first argument for \algonoisy RTRL algorithms]
\label{hyp:linearfur}
For \algonoisy RTRL algorithms, we assume that the functions $\fur_t$
are linear with respect to their first argument. Namely, we assume that for each $t\geq
1$, for each $\state\in \State_t$ and $\param\in \Param$, 
there exists a linear operator $P_t(\state,\param)\colon
\sblin(\Param,\mathbb{R})\to \Param$ such that for any $\vtanc\in
\sblin(\Param,\mathbb{R})$ one has
\begin{equation*}
\fur_t(\vtanc,\state,\param)=P_t(\state,\param)\cdot \vtanc
\end{equation*}
in addition to \rehyp{fupdrl}.
\end{hypothese}

This covers, notably, preconditioned SGD algorithms such as those in
Section~\ref{sec:examples}.

\subsubsection{Technical Assumptions}
\label{sec:techass}

The following three assumptions are ``technical'' in the sense that they would
be automatically satisfied for smooth functions in the non-recurrent case
with a finite dataset (because the sup over $t$ would become a max over
the dataset). However, they still encode important properties:
\begin{itemize}
\item
Uniformity of the dynamical system around the target trajectory,
\item The output noise should not grow too fast. In the i.i.d.\ case, this
corresponds to a property of moments of the output noise, see
Section~\ref{sec:pureonline}.
\item The (extended) Hessians of the loss should be uniformly continuous
over time in some neighborhood of $\paramopt$.
\end{itemize}

\begin{hypothese}[The transition functions are uniformly smooth around the
target trajectory]
\label{hyp:regftransetats}
We assume that
the
derivatives of $\opevol_t$ are uniformly bounded over time around the
target trajectory, namely:
\begin{align*}
\sup_{t \geq 1} \, \nrmop{\frac{\partial
\opevol_t}{\partial(\state,\param)}(\stateopt_{t-1},\paramopt)} < \infty,
\end{align*}
and that 
the second derivatives of $\opevol_t$ are bounded 
around $\paramopt$ and $\stateopt_t$:
\begin{equation*}
\sup_{t\geq 1} \,\sup_{
\begin{subarray}{c}
\param\in B_\Param(\paramopt,r_\Param)\\
\state\in B_{\State_{t-1}}(\stateopt_{t-1},\raysy)
\end{subarray}
}
\nrmop{\frac{\partial^2 \opevol_t}{\partial
(\state,\param)^2}(\state,\param)}
<\infty.
\end{equation*}
\end{hypothese}


The next assumption deals with the growth derivatives of the loss
$\perte_t$ along the
trajectory. In the simplest, non-recurrent case \eqref{eq:nonrec},
$\perte_t$ encodes the error between the predicted value and the actual
observation; therefore, at $\paramopt$, the difference is equal to the
output noise of the model. So in that case, the assumption on derivatives
of $\perte_t$ implicitly encodes an assumption on the law of the output noise of the
model. This point is developed in Section~\ref{sec:pureonline}.

\begin{hypothese}[Derivatives of the loss functions have controlled
growth along
the target trajectory]
\label{hyp:regpertes}
We assume that
the
derivatives of $\perte_t$ along the 
target trajectory grow at most in a controlled way over time, namely,
that there exists an exponent $0 \leq \exmp < 1$ such that 
\begin{align*}
\nrmop{\frac{\partial
\perte_t}{\partial \state}(\stateopt_{t})}
=O(t^\exmp).
\end{align*}
Moreover we assume that 
the second
derivative of $\perte_t$ is controlled
around 
$\stateopt_t$: 
\begin{equation*}
\sup_{
\state\in B_{\State_{t}}(\stateopt_{t},\raysy)
}
\nrmop{\frac{\partial^2 \perte_t}{\partial
\state^2}(\state)}
=\go{t^\exmp}.
\end{equation*}
\end{hypothese}

\renewcommand{\thetheoreme}{\thesection.\arabic{theoreme}.\alph{subhyp}}
\setcounter{subhyp}{0}

The final technical assumption requires the Hessians (or extended
Hessians) of the loss to be uniformly continuous in time in some
neighborhood of the local optimum. We first state it for the simple
algorithm without $\fur_t$ or $\paramupdate_t$.

\refstepcounter{subhyp}
\begin{hypothese}[The Hessians of the loss are uniformly continous close to
$\paramopt$.]
\label{hyp:equicontH_simple}
We assume that Hessians of the loss are continuous close to $\paramopt$, uniformly in $t$:
there exists a continuous function $\rho\from \R^+\to \R^+$ with
$\rho(0)=0$ such that for all $t$, for all $\param\in
B_\Param(\paramopt,r_\Param)$,
\begin{equation*}
\nrmop{\frac{\partial^2}{\partial \param^2}
\sbpermc{t}(\stateopt_0,\param)-\frac{\partial^2}{\partial \param^2}
\sbpermc{t}(\stateopt_0,\paramopt)}\leq \rho(\norm{\param-\paramopt}).
\end{equation*}
\end{hypothese}

For extended algorithms with non-trivial $\fur_t$, this rewrites as follows using the
Jacobians $\sbfh_t$ of the update direction, defined by
\eqref{eq:jacobians}: these play the role of the Hessian of the loss when
the update $\fur_t$ is not the gradient of a loss. By construction,
Assumption~\ref{hyp:equicontH} reduces to
Assumption~\ref{hyp:equicontH_simple} in the basic case
$\fur_t(\vtanc,\state,\param)=\vtanc$.

\addtocounter{theoreme}{-1}
\refstepcounter{subhyp}
\begin{hypothese}[Jacobians of the updates at $\paramopt$ are uniformly
continous close to $\paramopt$.]
\label{hyp:equicontH}
We assume that the Jacobians $\fht{\param}$ of the updates at
$\paramopt$, defined by \eqref{eq:jacobians}, are continuous
close to $\paramopt$, uniformly in $t$:
there exists a continuous function $\rho\from \R^+\to \R^+$ with
$\rho(0)=0$ such that for all $t$, for all $\param\in
B_\Param(\paramopt,r_\Param)$,
\begin{equation*}
\nrmop{\fht{\param}-\fht{\paramopt}}\leq \rho(\norm{\param-\paramopt}).
\end{equation*}
\end{hypothese}

This equicontinuity assumption is arguably the most technical. However,
we prove in Appendix~\ref{sec:eqehbcase} that \rehyp{equicontH_simple} is
automatically satisfied if
the transition and loss operators are $C^3$ with uniformly bounded
first, second and third derivatives. For non-trivial $\fur_t$, \rehyp{equicontH} is satisfied if
in addition, 
$\fur_t$ is $C^2$ with second
derivatives controlled in a certain way (satisfied notably when
$\fur_t(\vtanc,\state,\param)=P_t(\state,\param)\cdot \vtanc$ with $P_t$
regular enough).

\renewcommand{\thetheoreme}{\thesection.\arabic{theoreme}}

\subsection{A Convergence Theorem for Extended RTRL Algorithms}
\label{sec:mainthm}



We now introduce constraints on the stepsize sequence $(\pas_t)_{t\geq
1}$.
As explained in the introduction, in non-i.i.d.\ settings, the step sizes
of the gradient descent must satisfy time-homogeneity conditions stricter
than the classical Robbins--Monro criterion, in order to avoid
correlations between the step sizes and the internal state of the
dynamical system. Otherwise, this could bias the gradient descent: for instance,
having a step size $0$ at every odd step will produce bad results if the
underlying recurrent system exhibits period-$2$ phenomena.
In i.i.d.\
settings, this is not necessary. Besides, in some applications, we
will need the step sizes to be constant for a few steps (for instance,
truncated BPTT corresponds to a constant learning rate on each truncation
interval): for this we introduce a ``wiggle room'' factor
$1+o(1/t^\exmp)$.

\begin{hypothese}[Stepsize sequence]
\label{hyp:scfcstepsize}
We assume the stepsize sequence $\pas_t\geq 0$ is non-increasing and
satisfies
\begin{equation*}
\pas_t=\cdvp\, t^{-\expas}\,\paren{1+\po{1/t^\exmp}}
\end{equation*}
for some $b>0$, where $\cdvp\geq 0$ is the \emph{overall learning
rate}, and where $\exmp$ is the exponent in \rehypdeux{fupdrl}{regpertes}.

For simple or extended RTRL algorithms
(Defs.~\ref{def:simplertrl}--\ref{def:algortrl}), we assume that
$\max\paren{\expem,\exmp} + 2 \exmp < \expas\leq 1$, where $\expem$ is the
exponent
in \rehypdeux{opt_simple}{critoptrtrlnbt}.\todo{old: $\expem+2\exmp<\expas\leq 1$ and $\exmp<\expem$}

For \algonoisy RTRL algorithms (Def.~\ref{def:approxrtrl}),
we assume that $\max(\expem,1/2+\exmp)+2\exmp<\expas \leq 1$.\todo{old:
we furthermore assume that $\expas>1/2+3\,\exmp$.}
\end{hypothese}


In the sequel, we will
prove convergence provided the overall learning rate $\cdvp$ is small enough.
Thus, in the whole text, $\pas_t$ is implicitly a function of
$\cdvp$.

The conditions on the exponent $b$ deserve some comment. The exponent $\expem$
encodes the speed at which empirical averages of gradients at $\paramopt$
converge to $0$, and likewise for Hessians. In typical situations, this
holds for any $a>1/2$ (the standard statistical rate for empirical
averages). But on a finite
dataset, cycling over the samples in the dataset makes it possible to go
down to $a=0$ (see Section~\ref{sec:examples}). Meanwhile, on a finite
dataset, $\exmp=0$; in general, as discussed above, $\exmp$ encodes a
bound on the
growth of
the gradients at $\paramopt$ over time (see also
Section~\ref{sec:pureonline}). Thus, for simple and extended RTRL algorithms, it
is possible for $b$ to range from $0$ to $1$ in some cases. But for
\algonoisy RTRL algorithms (NoBackTrack, UORO), the conditions impose
$b>1/2$ as in the standard Robbins--Monro criterion: this is due to
inherent added stochasticity in \algonoisy RTRL algorithms.

\bigskip

We now state the local convergence result for RTRL and extended and
\algonoisy RTRL algorithms.

\begin{definition}[Local convergence]
\label{def:introlocalcv}
Given a parameterized dynamical system as above,
we say that an algorithm producing a sequence $(\param_t)$
\emph{converges locally} around $\paramopt$ if the following holds: 
There exists an overall learning rate $\cdvpmaxconv>0$ such that,
if 
the parameter $\param_0$
is initialized close enough to $\paramopt$ and the initial state
$\state_0$ is close
enough to $\stateopt_0$, then 
for any overall learning $\cdvp<\cdvpmaxconv$,
the sequence $\param_t$
computed by the algorithm converges to $\paramopt$. For \algonoisy RTRL
algorithms (NoBackTrack, UORO...), which make random choices for $E_t$,
convergence is meant with probability tending to $1$ as the overall
learning rate tends to $0$.
\end{definition}


Local convergence is a relatively weak requirement for an algorithm;
still, as far as we know, no such statement was available for any of the
algorithms considered here. Local convergence rules out the kind of bad
surprise identified for Adam in \cite{j.2018on}.

The assumptions
themselves are only local: we consider that one step in a ``wrong'' zone
of the parameter space may be impossible to recover from.
With only local assumptions, the maximal learning rate $\cdvpmaxconv$
will usually depend on the data: in the non-recurrent case, this means
that $\cdvpmaxconv$ may depend on the random choice of input-output pairs
$(x_t,y_t)$, namely, on the dataset and SGD choices. This is different
from global convergence under global convexity assumptions. This is
unavoidable with only local assumptions: if noise is unbounded, one
single random large step could take the algorithm out of the safe zone
where the assumptions hold. So with only local assumptions, the quantifiers need to be reversed: given
the dataset (or the sequence of operators $\opevol_t$ and $\perte_t$,
encoding a sequence of observations $(x_t,y_t)$), some learning rate will
work. If noise is bounded (e.g., if the dataset is finite) there is no
such problem.

For \algonoisy RTRL algorithms, convergence occurs only with
probability close to $1$ if the learning rate is small enough. Indeed,
these algorithms introduce added stochasticity in the gradient
computation. The same remark applies to stochasticity coming from the
data (which
we consider fixed in the whole text): the assumptions will be satisfied
with probability $1$, but there will be a data-dependent maximal learning
rate. This implies convergence with probability tending to $1$ as the
overall learning rate $\cdvp$ tends to $0$.

Indeed, with only local assumptions, depending on the noise, the computed
gradients may deviate from the true gradient during an arbitrarily long
time; this is only compensated if the learning rate is small enough
compared to the deviation produced by the noise, while larger learning
rates may bring the trajectory outside of the safe zone on which the
assumptions hold. Thus, in general, with stochasticity and without any
assumptions outside of a safe zone, convergence with only occur with
probability tending to $1$ as $\cdvp$ tends to $0$.  This contrasts with
global convexity assumptions.

\begin{theoreme}[Local convergence of RTRL, extended RTRL, and \algonoisy \rtrl algorithms]
\label{thm:cvapproxrtrlalgo}
%
Let $\paren{\opevol_t}$ be a parameterized dynamical sytem
(Def.~\ref{def:prmdynsys}) with loss functions $\perte_t$
(Def.~\ref{def:fcpletaprm}). Consider an extended or \algonoisy RTRL algorithm
(Defs.~\ref{def:simplertrl} or \ref{def:algortrl} or
\ref{def:approxrtrl}) on this system.

Let $\paramopt$ be a local optimum for this system (Assumption~\ref{hyp:critoptrtrlnbt},
which reduces to Assumption~\ref{hyp:opt_simple} if $\fur_t$ and $\paramupdate_t$ are
not used), with initial state $\stateopt_0$. Assume that the system
with parameter $\paramopt$ starting at $\stateopt_0$ is
stable (Assumption~\ref{hyp:specrad}).

For extended RTRL algorithms, assume the update operators $\fur_t$ and
$\paramupdate_t$ satisfy Assumptions~\ref{hyp:fupdrl} and
\ref{hyp:updateop_first}.

For \algonoisy RTRL algorithms, assume moreover that the random RTRL errors $E_t$ are
unbiased and controlled by some error gauge (Assumptions~\ref{hyp:wcorrnoiseapprtrl} and
\ref{hyp:errorgauge}), and that the update operators $\fur_t$ are linear
with respect to the estimated gradient direction
(Assumption~\ref{hyp:linearfur}).

Assume that the first and second derivatives of $\opevol_t$ and
$\perte_t$ are controlled around the target trajectory
(Assumptions~\ref{hyp:regftransetats} and \ref{hyp:regpertes}).
Assume that the extended Hessians or Jacobians of $\fur_t$ are uniformly
continuous close to $\paramopt$ (\rehyp{equicontH}).






Let $\sm{\pas}=\paren{\pas_t}$ be a stepsize sequence satisfying
\rehyp{scfcstepsize}, with overall learning rate $\cdvp$. 

Then the algorithm converges locally around $\paramopt$; for \algonoisy
RTRL algorithms, this convergence occurs with probability tending to $1$
as the overall learning rate tends to $0$.

More precisely,

\begin{itemize}
\item For RTRL or an extended RTRL algorithm, 
there exists a neighborhood $\mathcal{N}_{\paramopt}$ of
$\paramopt$, a neighbourhood $\mathcal{N}_{\stateopt_0}$ of
$\stateopt_0$, a neighborhood $\mathcal{N}^\jope_{0}$ of $0$ in
$\epjopeins{0}$, and an overall learning rate
$\cdvpmaxconv>0$ such that 
for any overall learning rate $\cdvp< \cdvpmaxconv$, the following
convergence holds:
\item For an \algonoisy RTRL algorithm, 
there exists a neighborhood $\mathcal{N}_{\paramopt}$ of
$\paramopt$, a neighbourhood $\mathcal{N}_{\stateopt_0}$ of
$\stateopt_0$, a neighborhood $\mathcal{N}^\jope_{0}$ of $0$ in
$\epjopeins{0}$ such that
for any $\eps>0$, there exists
$\cdvpmaxconv>0$ such that for any overall learning rate $\cdvp<
\cdvpmaxconv$, with probability greater than $1-\eps$,
the following
convergence holds:
\end{itemize}

For any initial parameter $\param_0\in
\mathcal{N}_{\paramopt}$, any initial state $\state_0\in
\mathcal{N}_{\stateopt_0}$ and any initial differential $\tilde\jope_0\in
\mathcal{N}^\jope_{0}$,
the trajectory given by
\begin{equation*}
\left\lbrace
\ba
\state_{t} &= \opevolt\paren{\state_{t-1},\,\param_{t-1}}, \\
\tilde \jope_t &= \frac{\partial \opevol_t(\state_{t-1},\param_{t-1})}{\partial
\state} \,\tilde \jope_{t-1}+\frac{\partial
\opevol_t(\state_{t-1},\param_{t-1})}{\partial
\param}+E_t
,\\
\vtanc_t &= \furt{\frac{\partial \perte_t(\state_t)}{\partial \state}\cdot \tilde \jope_t
}{\state_t}{\param_{t-1}},\\
\param_{t} &= 
\paramupdate(\param_{t-1},\pas_t\,\vtanc_t),
\ea \right.
\end{equation*}
(with $E_t=0$ for non-\algonoisy RTRL algorithms)
satisfies $\param_t\to \paramopt$ as $t\to\infty$.
\end{theoreme}

\subsection{Discussion: How Local is Local Convergence?}
\label{sec:lyapunov}


The convergence in Theorem~\ref{thm:cvapproxrtrlalgo} assumes that $\param_0$ is initialized close enough
to the optimal parameter $\paramopt$. We believe this is not a
fundamental limitation of the approach, and that similar results can be
extended to initializations $\param_0$ in the whole basin of attraction
$\paramopt$ under the ``idealized'' (non-noisy,
infinitesimal-learning-rate) dynamics of the underlying ODE.

Indeed, convergence is obtained by proving a contractivity property for
some well-chosen distance (Assumption~\ref{hyp:optimprm}.2, proven via
\relem{contractpctrlopt}).
In practice, a suitable distance function is obtained by expanding the dynamics at second
order around $\paramopt$, so that locally the idealized dynamics behaves
like the ODE $\param'=-\linalgmatrix\param$ for some matrix
$\linalgmatrix$ whose eigenvalues have positive real part
($\linalgmatrix$ being the Hessian of the loss in the simplest case).
This holds in the zone where the dynamics is close enough to its Taylor
expansion.

By defining the distance via a suitable Lyapunov function, a similar
contractivity property can be obtained over the whole basin of attraction
of a stable fixed point $\paramopt$ of an \emph{arbitrary} ODE
$\param'=-U(\param)$, not only in a neighborhood of $\paramopt$. Indeed,
consider an ODE
\begin{equation*}
\param'=-U(\param)
\end{equation*}
and assume that $\paramopt$ is a stable fixed
point of the ODE. (In our setting, 
$U$ is the update operator $\fur_t$ of the RTRL gradient
descent, averaged over $t$, and the ODE represents the behavior of the
system when the learning rates tend to $0$, so that noise averages out.)

\newcommand{\basin}{\mathcal{B}}
Let $\basin$ be basin of attraction of $\paramopt$: the set
of those $\param_0$ such that the ODE starting at $\param_0$ converges to
$\paramopt$. On $\basin$, define the \emph{Lyapunov distance} by
\begin{equation*}
d_L(\param,\param')\deq \sqrt{\int_{t=0}^\infty
\norm{\param_t-\param'_t}^2}
\end{equation*}
where $\param_t$ and $\param'_t$ are the value at time $t$ of the ODE
starting at $\param$ and $\param'$, respectively.
This distance is finite on $\basin$ because both trajectories converge exponentially to
$\paramopt$.
This is indeed a distance, because it is the $L^2$ distance between the
trajectories $(\param_t)_{t\geq 0}$ defined by $\param$ and $\param'$. In the linear case
$U(\param)=\linalgmatrix\param$ with positive-stable $\linalgmatrix$, this distance $d_L$ is
exactly the Euclidean metric associated to the positive-definite matrix $B$
providing the Lyapunov function in our proof
(Lemma~\ref{lem:suitposdefm} and Appendix~\ref{sec:positivestable}).

By construction,
the Lyapunov distance decreases along the flow. More precisely, if
$(\param_t)$ and $(\param'_t)$ are trajectories of the ODE starting
at $\param_0$ and $\param'_0$, respectively, then
$d_L(\param_t,\param'_t)\leq
d_L(\param_0,\param'_0)$. 
This is because the integral defining $d_L(\param_t,\param'_t)^2$ is the
same as the integral defining $d_L(\param_0,\param'_0)^2$, minus the time
segment $[0;t]$.
A more precise
short-time contractivity is given by
\begin{equation*}
\frac{\d}{\d t}_{|t=0}
d_L(\param_t,\param'_t)^2=-\norm{\param_0-\param'_0}^2
\end{equation*}
by construction of $d_L$.

In particular
\begin{equation*}
d_L(\param_t,\paramopt)\leq d_L(\param_0,\paramopt)
\end{equation*}
On compact subsets of the basin of attraction $\basin$, by bounding $d_L$
as a function of $\norm{\cdot}$ (i.e., writing
$\norm{\param-\paramopt}^2\geq \mu \, d_L(\param,\paramopt)^2$ for some constant
$\mu>0$, obtained by compactness), this can be
strengthened to strict contractivity $\frac{\d}{\d t}_{|t=0}
d_L(\param_t,\paramopt)^2\leq -\mu \, d_L(\param_0,\paramopt)^2$, resulting in
strict contractivity $d_L(\param_t,\paramopt)\leq e^{-\mu t/2}
d_L(\param_0,\paramopt)$.

Then one can use the Lyapunov distance $d_L$ as the distance that gets
contracted for Assumption~\ref{hyp:optimprm} below. This way, proving
contractivity of the idealized dynamics would not require comparing the
dynamics to its Taylor expansion in a small neighborhood of $\paramopt$.

At the same time, such an approach would require all the other
assumptions above to hold in the whole basin of attraction $\basin$, not only
at or in a neighborhood of $\paramopt$. Notably, for
Assumption~\ref{hyp:opt_simple} we would have to assume
that for \emph{any} parameter $\param$ in $\basin$, the average gradients 
$\frac{1}{T}\sum_{t=1}^T
\fur_t\left(
\frac{\partial}{\partial \param}
\sbpermc{t}(\stateopt_0,\paramopt),\fstate_t(\stateopt_0,\paramopt),\paramopt
\right)$ converge to an asymptotic dynamics $U(\param)$
at rate $O(T^\expem/T)$, uniformly in (compact subsets of) $\basin$. Likewise,
Assumption~\ref{hyp:specrad} (that the dynamical system with a fixed
parameter $\param$ is stable) would have to hold not only at $\paramopt$, but over
all the zone in which we expect to prove convergence.

We believe these would be interesting topics for future research.
In any case, the existence of the Lyapunov distance $d_L$ proves that the ODE
method is not intrinsically limited to convergence results in a small
neighborhood of $\paramopt$.

\section{Examples and Applications}
\label{sec:examples}

The theorem above establishes local convergence of the basic RTRL
algorithm. (Local convergence is defined in Def.~\ref{def:introlocalcv}.) We
now show how other algorithms can be obtained as particular cases of
extended RTRL algorithms for suitable systems.

\subsection{Non-Recurrent Situations}
\label{sec:nonrec}

We first illustrate how these results play out in the ordinary, non-recurrent
case.

In this whole section we
consider a finite dataset $D=(x_n,y_n)_{n\in[1;N]}$ of
inputs and labels (with values in any sets).
The ``streaming'' setting in which infinitely many independent samples are
available and pure online SGD is performed, leads to a different analysis
(Section~\ref{sec:pureonline}).

We consider a loss function
$\perteiid(x,y,\param)$ depending on $\param$. We assume that $\perteiid$ is
$C^3$ 
with respect to
$\param$ for each pair $(x,y)$ in the dataset (so all ``technical''
assumptions of Section~\ref{sec:techass} are automatically satisfied, notably the equicontinuity \rehyp{equicontH}). We call
\emph{strict local optimum} a local optimum of the average loss
with positive definite Hessian, namely, a parameter $\paramopt$ such that
\begin{equation*}
\frac{1}{N}\sum_{n=1}^N \partial_\param \perteiid(x_n,y_n,\paramopt)=0
\end{equation*}
and
\begin{equation*}
H\deq \frac{1}{N}\sum_{n=1}^N \partial^2_\param
\perteiid(x_n,y_n,\paramopt)\succ 0.
\end{equation*}

We consider three variants of SGD depending on how samples are selected
at each step. These will lead to different possible learning rates.

\begin{definition}[I.i.d.\ sampling, cycling, random reshuffling]
We call respectively \emph{cycling, i.i.d.\ sampling, and random
reshuffling},
a sequence of integers $i_t\in[1;N]$ 
where for each time $t$, $i_t$ is chosen by
\begin{equation*}
i_t=\begin{cases}
t\mod N & \text{(cycling over $D$);}
\\
\mathrm{Unif}([1;N]) & \text{(i.i.d.\ sampling);}
\\
\pi_k(t\mod N) & \text{(random reshuffling)}
\end{cases}
\end{equation*}
where $k=\lceil t/N\rceil$ and for each $k$, $\pi_k$ is a random
permutation of $[1;N]$, and where for convenience purposes, we define
$t\mod N$ to take values in $[1;N]$ instead of $[0;N-1]$.

We will abbreviate
\begin{equation*}
\perteiid_t(\param)\deq \perteiid(x_{i_t},y_{i_t},\param).
\end{equation*}
\end{definition}

In this setting, the ``technical'' assumptions of
Section~\ref{sec:techass} are automatically satisfied, with exponent
$\exmp=0$, because the
dataset is finite (so a supremum over $t$ becomes a maximum over $D$) and
the functions involved are smooth. (Contrast with
Section~\ref{sec:pureonline} on infinite datasets.) But the non-technical
assumptions still have to be checked and lead to interesting phenomena.

\subsubsection{Ordinary SGD on a Finite Dataset}
\label{sec:sgd}

Ordinary SGD is cast as follows in our formalism; we consider two cases
depending on whether a random sample is taken from the dataset at each
step, or whether we cycle through all samples.

\begin{example}[Non-recurrent case]
\label{ex:nonrec}
We call \emph{non-recurrent case} the following choice of evolution
operators $\opevol_t$ and loss functions $\perte_t$: the state is just the parameter $\param$ itself, namely 
\begin{equation*}
\opevol_t(\state_{t-1},\param)\deq \param, \qquad \state_t=\param
\end{equation*}
and the loss is
\begin{equation*}
\perte_t(\state)\deq \perteiid(x_{i_t},y_{i_t},\state)
\end{equation*}
where for each time $t$, the sample $i_t$ is chosen by i.i.d.\ sampling,
random reshuffling, or cycling over the dataset.

Then the RTRL algorithm for this non-recurrent case corresponds to
ordinary stochastic gradient descent.
\end{example}

The choice of i.i.d.\ sampling for $i_t$, versus cycling over $D$ or random reshuffling,
influences the speed at which $\frac{1}{T}\sum_{t=1}^T \partial_\param
\perteiid_t(\paramopt)$ converges to $0$ (Assumption~\ref{hyp:opt_simple}). Indeed, for cycling and random
reshuffling, each sample in the dataset is sampled exactly once in every
interval $(kN;(k+1)N]$. Therefore, when summing $\partial_\param
\perteiid_t(\paramopt)$ from time $1$ to $T$, full swipes over the dataset (``epochs'') exactly
cancel to $0$, and the average $\frac{1}{T}\sum_{t=1}^T \partial_\param
\perteiid_t(\paramopt)$ tends to $0$ at rate $O(1/T)$.

On the other hand, with i.i.d.\ sampling, the averages
$\frac{1}{T}\sum_{t=1}^T \partial_\param
\perteiid_t(\paramopt)$ converge to $0$ almost surely at a rate
$O(\sqrt{(\ln\ln T)/T})$ by the law of the iterated logarithm. The same
applies to the Hessians. Therefore, we find:

\begin{proposition}
\label{prop:SGDhyp}
Assume that $\paramopt$ is a strict local optimum of the average
loss over the dataset $D$.
Then Assumption~\ref{hyp:opt_simple} is
satisfied
\begin{itemize}
\item for any $\expem\geq 0$ for the case of cycling over $D$ or random
reshuffling;
\item for any $\expem>1/2$ for the i.i.d.\ case, with probability $1$
over the choice of random samples $i_t$.
\end{itemize}
\end{proposition}

As a consequence, larger learning rates can be used when cycling over the
data than
when using i.i.d.\ samples. Indeed,
remember that the exponent $\expem$ 
directly constrains the set of possible learning
rates (Assumption~\ref{hyp:scfcstepsize}) in the convergence theorem. Here we have $\exmp=0$ so
Assumption~\ref{hyp:scfcstepsize} is satisfied with
rates $\pas_t=t^{-\expas}$ for any $\expem<\expas\leq
1$. Also note that the spectral radius assumption \ref{hyp:specrad} is
trivially satisfied because $\partial_\state \opevol_t=0$. Therefore,
Theorem~\ref{thm:cvapproxrtrlalgo} yields the following:

\begin{corollaire}
\label{cor:sgdrates}
Consider ordinary stochastic gradient descent
\begin{equation*}
\param_{t}=\param_{t-1}-\pas_t\,\partial_{\param}\perteiid(x_{i_t},y_{i_t},\param)
\end{equation*}
over a finite dataset $D$ with loss $\perteiid$ as above. Assume the
learning rates satisfy $\pas_t\propto t^{-\expas}$ with
\begin{equation*}
\begin{cases}
0<\expas\leq 1 &\text{for cycling over $D$ or random reshuffling;}
\\
1/2<\expas\leq 1 &\text{for i.i.d.\ sampling of $i_t$.}
\end{cases}
\end{equation*}
Then this algorithm is locally convergent.
\end{corollaire}

Thus, cycling over $D$ or random reshuffling allow for larger
learning rates than 
the traditional range
of exponents $\expas\in (1/2;1]$ for i.i.d.\ sampling.
Cycling acts as a basic form of
variance reduction in SGD, ensuring that every data is sampled exactly
once within each cycle of $N$ steps. (See discussion in the introduction
and related work section.)

The case of a genuine online SGD with infinitely many distinct samples
is different and is treated in Section~\ref{sec:pureonline}.

Thus, for the rest of Section~\ref{sec:nonrec} we set learning rates
$\pas_t=\cdvp\,t^{-\expas}$
with $0<\expas\leq 1$ for random reshuffling or cycling over $D$, or $1/2<\expas\leq 1$ for i.i.d.\
sampling.

\subsubsection{SGD with Known Preconditioning Matrix}
\label{sec:precond}

Let us now illustrate the case of gradient descent preconditioned by a
matrix $\metric(\param)$, of the form
\begin{equation*}
\param \gets \param-\pas_t \metric(\param)\,\partial_\param \perteiid_t
\end{equation*}
thus using a non-trivial update operator $\fur_t$. This illustrates how
Assumption~\ref{hyp:critoptrtrlnbt} plays out with the extended Hessians
$\sbfh_t$.

We first assume that we can compute $\metric(\param)$ explicitly given
$\param$. (The case where $\metric$ is estimated online is treated below,
in the section on adaptive algorithms.) This covers, for instance, the
natural gradient with $\metric(\param)$ the inverse of the Fisher matrix
at $\param$.

(For the case of Riemannian metrics, $\param \gets \param-\pas_t
\metric(\param)\,\partial_\param \perteiid_t$ directly applies the
update by addition in some coordinate system. For
true ``manifold'' Riemannian gradients with an added exponential map,
the exponential map can be put in the update operator
$\paramupdate_t$.)

\begin{example}[Preconditioned SGD]
Consider again the non-recurrent setting of Example~\ref{ex:nonrec}.
Let $\param\mapsto\metric(\param)$ be a $C^1$ map from $\Param$ to the
set of square matrices of size $\dim(\Param)$. 
We call
\emph{preconditioned SGD} the RTRL algorithm resulting from the update
operator
$
\furt{\vtanc}{\state}{\param}\deq \metric(\param)\vtanc
$.
\end{example}

By Remark~\ref{rem:affineU}, this choice of $\fur_t$ is covered by our assumptions.

Thus, a corollary of our main theorem is the following.

\begin{corollaire}[Convergence of preconditioned SGD]
\label{cor:precondSGD}
Assume that $\paramopt$ is a strict local optimum of the average
loss over the dataset $D$. Assume
moreover that $\metric(\paramopt)+\transp{\metric(\paramopt)}$ is
positive definite. Take learning rates as in
Corollary~\ref{cor:sgdrates}.

Then Assumption~\ref{hyp:critoptrtrlnbt} is satisfied.
Therefore, preconditioned SGD converges locally.

Moreover, the matrix $\linalgmatrix$
in Assumption~\ref{hyp:critoptrtrlnbt} is
\begin{equation*}
\linalgmatrix=\metric(\paramopt)H
\end{equation*}
with $H= \frac{1}{N}\sum_{n=1}^N \partial^2_\param
\perteiid(x_n,y_n,\paramopt)$ the Hessian of the average loss at
$\paramopt$.
\end{corollaire}

\begin{proof}
Let us check the first point of the assumption, namely, that the average of the updates is $0$. Indeed, we
have
\begin{equation*}
\fur_t\left(
\partial_\param
\sbpermc{t},\fstate_t,\param
\right)
=\metric(\param)\partial_\param
\sbpermc{t}=\metric(\param)\partial_\param\perteiid_t
\end{equation*}
since $\sbpermc{t}=\perteiid_t$ in the non-recurrent case.
We have to take the average over time at $\param=\paramopt$. Since
$\metric(\paramopt)$ does not depend on $t$, we just have to check that
the average of $\partial_\param
\perteiid_t$ at $\param=\paramopt$ vanishes, which is the case by
assumption.

The next point of the assumption deals with the average extended Hessians
\begin{equation*}
\sbfh_t(\param)= \partial_\param\,\fur_t\left(
\partial_\param
\sbpermc{t},\fstate_t,\param
\right)
\end{equation*}
and here with $\furt{\vtanc}{\state}{\param}= \metric(\param)\vtanc$, and
using again that $\sbpermc{t}=\perteiid_t$, we find
\begin{equation*}
\sbfh_t(\param)=\partial_\param\left(\metric(\param)\partial_\param\perteiid_t\right)
=
(\partial_\param P(\param))\,\partial_\param
\perteiid_t+P(\param)\partial^2_\param \perteiid_t
\end{equation*}
which we have to average at $\param=\paramopt$. Since $\partial_\param
\perteiid_t$ averages to $0$ at $\paramopt$, the first term averages to $0$
and we find
\begin{equation*}
\frac{1}T \sum_{t=1}^T \sbfh_t(\paramopt) \to P(\paramopt) H
\end{equation*}
where $H$ is the (ordinary) Hessian at $\paramopt$ of the average loss
over the dataset. So the matrix $\linalgmatrix$ is
\begin{equation*}
\linalgmatrix=\metric(\paramopt)H
\end{equation*}
By the assumptions on
$P(\paramopt)$ and $H$, and by one of the criteria for
positive-stability (Proposition~\ref{prop:positivestablecriteria}), this is positive-stable.
\end{proof}

\subsubsection{Adding Momentum}
\label{sec:mom}

SGD with momentum appears naturally as an RTRL algorithm with a suitable
recurrent state, as follows.

\begin{corollaire}[SGD with momentum] 
\label{cor:mom}
Consider a recurrent system with
real-valued state $\state$ subject to the evolution equation
\begin{equation*}
\state_t=\opevol_t(\state_{t-1},\param)\deq
\beta \state_{t-1}+(1-\beta)\perteiid(x_{i_t},y_{i_t},\param)
\end{equation*}
for some $0\leq \beta<1$,
where each sample index $i_t$ is chosen as in Example~\ref{ex:nonrec}.
Define the loss functions $\perte_t(\state_t)\deq \state_t$.

Then RTRL on this recurrent system is equivalent to SGD with momentum:
\begin{equation}
\label{eq:sgd-momentum}
\param_t=\param_{t-1}-\pas_t J_t,\qquad J_t=\beta J_{t-1}+(1-\beta)
\partial_\param \perteiid(x_{i_t},y_{i_t},\param_{t-1})
\end{equation}

Moreover, Assumption~\ref{hyp:opt_simple} is satisfied, with the same
exponents as in Proposition~\ref{prop:SGDhyp}. Therefore, SGD with
momentum $\beta$ converges locally.
\end{corollaire}

\begin{proof}
First, the evolution operators $\opevol_t$ are obviously
$\beta$-contracting on $\state$. Since $\beta<1$, the stability
assumption \ref{hyp:specrad} on the system is satisfied.

Second, by Definition~\ref{def:simplertrl}, the variable $\jope_t$ of RTRL
for this system exactly follows the right-hand side of \eqref{eq:sgd-momentum}. This proves that RTRL
for this system is equivalent to SGD with momentum.

Finally, let us check Assumption~\ref{hyp:opt_simple}: we have to compute
the time averages of gradients and Hessians of $\sbpermc{t}$ with respect
to $\param$. Let us prove that those behave asymptotically as in the
momentum-less case. Indeed, for this system we have (dropping $\state_0$
for simplicity)
\begin{equation*}
\sbpermc{t}(\param)=(1-\beta) \sum_{j\leq t} \beta^{t-j}
\perteiid(x_{i_j},y_{i_j},\param)
\end{equation*}
by induction. Therefore,
\begin{align*}
\sum_{t=1}^T\sbpermc{t}(\param)&=\sum_{j\leq T}
\perteiid(x_{i_j},y_{i_j},\param) (1-\beta)\sum_{t=j}^T \beta^{t-j}
\\&=\sum_{j\leq T}
\perteiid(x_{i_j},y_{i_j},\param) (1-\beta^{T-j})
\\&=\sum_{j\leq T}
\perteiid(x_{i_j},y_{i_j},\param) -\sum_{j\leq
T}\beta^{T-j}\perteiid(x_{i_j},y_{i_j},\param)
\\&=\sum_{j\leq T}
\perteiid(x_{i_j},y_{i_j},\param)+O(1)
\end{align*}
as $\sum_{j\leq T} \beta^{T-j}$ is finite and
$\perteiid(x_{i_j},y_{i_j},\param)$ is bounded (since we deal with a finite
dataset). Therefore, the time averages $\frac{1}{T}
\sum_{t=1}^T\sbpermc{t}(\param)$ coincide up to $O(1/T)$ with the
momentum-less case. The same argument applies to gradients
$\partial_\param \sbpermc{t}$ and Hessians $\partial^2_\param
\sbpermc{t}$. Therefore, at $\paramopt$, we have $\frac{1}{T}
\sum_{t=1}^T \partial_\param \sbpermc{t}(\paramopt)\to 0$ and
$\frac{1}{T}
\sum_{t=1}^T \partial^2_\param \sbpermc{t}(\paramopt)\to H$ with the same
rates as in the momentum-less case.
\end{proof}

\subsubsection{Adaptive Algorithms: Collecting Statistics Online}
\label{sec:adap}

Adaptive algorithms such as RMSProp, Adam, or the online natural
gradient, estimate a preconditioning matrix $\metric(\param)$ via a
moving average along the optimization trajectory. Thus it is not possible
to apply Corollary~\ref{cor:precondSGD}, since the latter assumes
direct access to $\metric(\param)$ for each $\param$.

So let us consider an algorithm that maintains an auxiliary
variable $\Stat$, computed by aggregating past values of some statistic
$\stat(x_t,y_t,\param_{t-1})$ depending on past observations. Namely,
consider an algorithm of the type
\begin{align*}
\param_t&=\param_{t-1}-\pas_t \metric(\param_{t-1},\Stat_{t-1})\,\partial_\param
\perteiid_t
\\
\Stat_t&=\beta_t
\Stat_{t-1}+(1-\beta_t)\,\stat(x_t,y_t,\param_{t-1})
\end{align*}
where we let the ``inertia'' parameter $\beta_t$ tend to $1$ at the same
rate as the learning rates, namely,
\begin{equation*}
\beta_t=1-\momen\,\pas_t
\end{equation*}
for some constant $\momen>0$. (This choice will be discussed later.)

For example, RMSProp and Adam are based on collecting statistics about
square gradients, namely, letting $\stat_t$ be the vector
\begin{equation*}
\stat_t(\param)\deq\left(
\partial_\param \perteiid_t
\right)^{\odot 2}
\end{equation*}
and then RMSProp uses the preconditioner
\begin{equation*}
\metric(\param,\Stat)\deq\diag(\Stat+\eps)^{-1}
\end{equation*}
for some regularizing constant $\eps$. Adam uses
momentum in addition and is treated in Section~\ref{sec:adam} below.

The
online natural gradient corresponds to a full matrix-valued $\Stat$
with\footnote{This corresponds to the ``Gauss--Newton'' or ``outer product''
version of the natural gradient
\citep{martens2014new,oll15r1}. The
other version has an expectation over predicted values of $y_t$ instead
of the actual data $y_t$, corresponding to a choice of $\stat_t$ that depends
only on $x_t$ and $\param_{t-1}$, and can be treated similarly.}
\begin{equation*}
\stat_t(\param)\deq (\partial_\param\perteiid_t)^{\otimes 2},\qquad
\metric(\param,\Stat)\deq\Stat^{-1}.
\end{equation*}
The extended Kalman filter in the ``static'' case (for estimating the
state of a fixed system from noisy nonlinear measurements) has been shown to be
equivalent to a particular case of the online natural gradient
\citep{ollivier2018online}: it is an online natural gradient with
stepsize $\pas_t=1/(t+1)$ and Gaussian noise model. Therefore, the
results here apply to the static extended Kalman filter as well.

A key idea to treat such algorithms is to view $\Stat$ as part of the
parameter to be estimated. Indeed, the update of $\Stat$ can be seen as a
gradient descent for the loss $\norm{\Stat-\stat_t}^2$. However, this
idea does not work directly: incorporating $\norm{\Stat-\stat_t}^2$ into
the loss changes the gradients for $\param$, because $\stat_t$ typically
depends on $\param$, so that extraneous gradient terms on $\param$
appear.

Instead, here we will take full advantage of the generalized Hessians
$\sbfh_t$ for non-gradient updates, and of the fact that the matrix
$\linalgmatrix$ in Assumption~\ref{hyp:critoptrtrlnbt} does not need to be
positive definite, only to have eigenvalues with positive real part. This
works out as follows.

\begin{corollaire}[Local convergence of adaptive preconditioning]
\label{cor:adap}
Consider a finite dataset $D=(x_n,y_n)$ as above. Take learning rates as in
Corollary~\ref{cor:sgdrates}.

Let
$(x,y,\param)\mapsto \stat(x,y,\param)\in \R^{\dim(\stat)}$ be any $C^1$
map. 
Let $\metric$ be any $C^1$ map
sending $(\param,\Stat\!\in \!\R^{\dim(\stat)})$ to a square matrix of size $\dim(\param)$.
Consider the following algorithm: the average of $\stat$ is estimated
online via
\begin{equation}
\label{eq:adap1}
\Stat_t=\beta_t
\Stat_{t-1}+(1-\beta_t)\,\stat(x_t,y_t,\param_{t-1})
\end{equation}
with
$\beta_t=1-\momen\,\pas_t$ for some $\momen>0$,
and the parameter is updated using the preconditioning matrix $P$ computed either
from $\Stat_t$ or $\Stat_{t-1}$,
\begin{equation}
\label{eq:adap2}
\param_t=\param_{t-1}-\pas_t \metric(\param_{t-1},\Stat_{t-1})\,\partial_\param
\perteiid_t
\end{equation}
or
\begin{equation}
\label{eq:adap3}
\param_t=\param_{t-1}-\pas_t \metric(\param_{t-1},\Stat_t)\,\partial_\param
\perteiid_t.
\end{equation}

Let $\paramopt$ be a strict local optimum for the dataset. Let
\begin{equation*}
\Stat^*\deq \frac{1}{N} \sum_{n=1}^N \stat(x_n,y_n,\paramopt)
\end{equation*}
be the average value of the statistic at $\paramopt$.
Assume that
$\metric(\paramopt,\Stat^*)+\transp{\metric(\paramopt,\Stat^*)}$ is
positive definite.

Then the adaptive algorithms \eqref{eq:adap1}--\eqref{eq:adap2} and
\eqref{eq:adap1}--\eqref{eq:adap3} converge locally.
\end{corollaire}

\begin{proof}
Define the augmented parameter
$\param^+\deq(\param,\Stat)\in \Param\times \R^{\dim(\stat)}$. Define a
system via
\begin{equation*}
\state_t=\opevol_t(\state_{t-1},\param^+)\deq \param^+
\end{equation*}
with the loss $\perteiid_t$ on $\state_t=\param^+=(\param,\Stat)$ given by
\begin{equation*}
\perteiid_t(\param,\Stat)\deq \perteiid(x_{i_t},y_{i_t},\param)
\end{equation*}
as in Example~\ref{ex:nonrec}. Thus $\partial_{\param^+}
\sbpermc{t}(\param^+)=\partial_{(\param,\Stat)}\perteiid(x_{i_t},y_{i_t},\param)=\left(\partial_\param
\perteiid(x_{i_t},y_{i_t},\param)
,0
\right)$.

Define the extended RTRL algorithm on
$\param^+=(\param,\Stat)$ with
update
\begin{equation*}
\fur_t((v_\param,v_\Stat),s,(\param,\Stat))\deq \left(
\begin{array}{c}
P(\param,\Stat)v_\param
\\
\momen \Stat-\momen\stat(x_{i_t},y_{i_t},\param)
\end{array}
\right)
\end{equation*}
where the first row describes the update of $\param$ and the second row,
the update of $\Stat$.

Then by construction,
the extended RTRL update on $\param^+=(\param,\Stat)$ coincides with
\eqref{eq:adap1}--\eqref{eq:adap2} with $\beta_t=1-\momen\pas_t$. (The
case of \eqref{eq:adap3} is treated later.)

This choice of $\fur_t$ is affine in $\vtanc$, so by
Remark~\ref{rem:affineU} it is covered by \rehyp{fupdrl} on $\fur_t$.

Let us check Assumption~\ref{hyp:critoptrtrlnbt}. For this system, the
initial state $s_0$ plays no role, so for simplicity we drop it from the
notation.
The gradients computed by this algorithm are
\begin{align*}
u_t(\param,\Stat)
&\deq \fur_t\left(\partial_{\param^+}
\sbpermc{t}(\param^+),\fstate_t(\param^+),\param^+\right)
\\&=\left(
\begin{array}{c}
P(\param,\Stat)\,\partial_\param \perteiid(x_{i_t},y_{i_t},\param)
\\
\momen \Stat-\momen\stat(x_{i_t},y_{i_t},\param)
\end{array}
\right)
\end{align*}
and the extended Hessians are $\sbfh_t=\partial_{(\param,\Stat)}u_t$. We find
\begin{equation*}
\sbfh_t=\left(
\begin{array}{cc}
P\, \partial^2_\param \perteiid_{t}+(\partial_\param P)
\partial_\param\perteiid_t
&
(\partial_{\Stat} P) \partial_\param\perteiid_{t}
\\
-\momen\, \partial_\param \stat_t
& \momen\id
\end{array}
\right)
\end{equation*}
where we have abbreviated $\perteiid_t$ for $\perteiid(x_{i_t},y_{i_t},\param)$
and likewise for $\stat_t$.

We have to prove that the average of $u_t(\paramopt,\Stat^*)$ over time
is $0$, and that the average of $\sbfh_t(\paramopt,\Stat^*)$ over time is a positive-stable
matrix $\linalgmatrix$.

We have $u_t(\paramopt,\Stat^*)=\left(
P(\paramopt,\Stat^*)\,\partial_\param \perteiid_t,\,c\Stat^*-c\stat_t
\right)$. By definition of $\paramopt$, the average of $\partial_\param
\perteiid_t$ over time is $0$. Since $P(\paramopt,\Stat^*)$ does not depend
on $t$, this proves that the first component of $u_t(\paramopt,\Stat^*)$ averages to $0$.
The second component of $u_t$ averages to $0$ by definition of $\Stat^*$.

The rates of this convergence are $O(1/t)$ if cycling over the data, or
$O(\sqrt{(\log\log t)/t})$ for the i.i.d.\ case, as in
Proposition~\ref{prop:SGDhyp}.

To compute the matrix $\linalgmatrix$, let us average $\sbfh_t(\paramopt,\Stat^*)$ over
time.
The quantities $\partial_\param P(\paramopt,\Stat^*)$ and $\partial_\Stat
P(\paramopt,\Stat^*)$ do not depend
on time, and $\partial_\param \perteiid_{t}$ averages to $0$ at $\paramopt$,
so both terms involving $\partial_\param \perteiid_{t}$ average to $0$. By the
assumption at the start of Section~\ref{sec:nonrec}, the Hessians $\partial^2_\param \perteiid_{t}$ at
$\paramopt$
average to some positive definite average Hessian $H$. Moreover, the time
averages of $\partial_\param \stat_t$ also converge to their average over the
dataset, which is some matrix $C$. Therefore we find
\begin{equation*}
\linalgmatrix=\left(
\begin{array}{cc}
P(\paramopt,\Stat^*) H
&
0
\\
-c\,C
& \momen\id
\end{array}
\right)
\end{equation*}
for some matrix $C$.

Since $\linalgmatrix$ is block-triangular, its eigenvalues are those of $PH$
and of $\momen \id$. As in Section~\ref{sec:precond}, $PH$ is positive-stable, and so is
$\momen\id$, so the matrix $\linalgmatrix$ is positive-stable.

This proves that Assumption~\ref{hyp:critoptrtrlnbt} is satisfied, for
the variant \eqref{eq:adap2} where $\param$ and $\Stat$ are updated
simultaneously.

The variant \eqref{eq:adap3} where $\Stat$ is updated
before $\param$ can be treated via a simple trick: update $\Stat$ at odd
steps and update $\param$ at even steps. More precisely, for $t\geq 1$
define
\begin{equation*}
\fur_{2t-1}((v_\param,v_\Stat),s,(\param,\Stat))\deq \left(
\begin{array}{c}
0
\\
\momen \Stat-\momen\stat(x_{i_t},y_{i_t},\param)
\end{array}
\right)
\end{equation*}
at odd times, and
\begin{equation*}
\fur_{2t}((v_\param,v_\Stat),s,(\param,\Stat))\deq \left(
\begin{array}{c}
P(\param,\Stat)v_\param
\\
0
\end{array}
\right)
\end{equation*}
at even times. Redefine the step sizes $\pas_t$ accordingly, and the
losses via $\perte_{2t-1}\deq 0$ and $\perte_{2t}\deq 
\perte(x_{i_t},y_{i_t},\param)$. Then the RTRL algorithm for this choice
of $\fur_t$ coincides with \eqref{eq:adap1}--\eqref{eq:adap3}.
Crucially, the time averages
of the new $\fur_t$ (and thus of $\sbfh_t$) from time $1$ to $2t$ at $\paramopt$
are exactly half the time averages from $1$ to $t$ in the
previous case. So $u_t$ still averages to $0$, and
$\sbfh_t$ averages to
$\linalgmatrix/2$ for the same matrix $\linalgmatrix$. This deals with the case of the update
\eqref{eq:adap3}.
\end{proof}

\begin{remarque}[Splitting an update into sub-updates]
The trick above,
of updating part of the parameter at even steps and
the rest at odd steps, works more generally:
it is always
possible to split an update $\fur_t$ into as many sub-updates as needed,
in any order. 
Indeed, our assumptions only
deal with an \gu{open-loop} system using a fixed parameter ($\pas_t=0$), 
and such split updates do not change the dynamics of the
fixed-parameter system, so the
assumptions and temporal averages will be unchanged.
\end{remarque}


%
%
%
%

\subsubsection{Adam with Fixed $\beta^1$ and $\beta^2\to 1$}
\label{sec:adam}

Algorithms like Adam can be treated by a direct combination of the arguments of
Section~\ref{sec:mom} (for momentum) and Section~\ref{sec:adap} (adaptive
preconditioning)\footnote{We ignore the so-called bias correction
factors of Adam, which tend to $1$ exponentially fast and modify the learning
rates in the first few iterations, but do not affect asymptotic
convergence.}.

\begin{corollaire}[Local convergence of adaptive preconditioning with
momentum]
\label{cor:adam}
Consider a finite dataset $D=(x_n,y_n)$ as above. Take learning rates as in
Corollary~\ref{cor:sgdrates}.

Consider an algorithm that maintains a momentum variable $J$ together
with statistics $\stat$, updated via
\begin{align}
\label{eq:adamJ}
J_t&=\beta^1 J_{t-1}+(1-\beta^1)
\partial_\param \perteiid(x_{i_t},y_{i_t},\param_{t-1})
\\
\Stat_t&=\beta^2_t\,
\Stat_{t-1}+(1-\beta^2_t)\,\stat(x_t,y_t,\param_{t-1})
\end{align}
with fixed $0\leq \beta^1<1$ and with $\beta^2_t=1-\momen \eta_t$. Here
$\Stat\colon(x,y,\param)\mapsto \stat(x,y,\param)\in \R^{\dim(\stat)}$ is any $C^1$
map.

Let the algorithm update the parameter $\param$ via
\begin{equation*}
\param_t=\param_{t-1}-\pas_t \metric(\param_t,\Stat_{t-1})J_t
\qquad\text{or}\qquad
\param_t=\param_{t-1}-\pas_t \metric(\param_t,\Stat_t)J_t
\end{equation*}
where $\metric$ is any $C^1$ map
sending $(\param,\Stat\!\in \!\R^{\dim(\stat)})$ to a square matrix of
size $\dim(\param)$.

Let $\paramopt$ be a strict local optimum for the dataset. Let
\begin{equation*}
\Stat^*\deq \frac{1}{N} \sum_{n=1}^N \stat(x_n,y_n,\paramopt)
\end{equation*}
be the average value of the statistic at $\paramopt$.
Assume that
$\metric(\paramopt,\Stat^*)+\transp{\metric(\paramopt,\Stat^*)}$ is
positive definite.

Then this algorithm converges locally.
\end{corollaire}

As in the previous section, 
Adam is recovered by
letting the statistic $\stat$ and preconditioner $\metric$ be
\begin{equation*}
\stat(x,y,\param)\deq\left(
\partial_\param \perteiid(x,y,\param)
\right)^{\odot 2}
,\qquad
\metric(\param,\Stat)\deq\diag(\Stat+\eps)^{-1}
\end{equation*}
for some regularizing constant $\eps$. 

The main difference with the counterexample in \citet{j.2018on} is
that we take $\beta^2\to1$ while the counterexample uses a fixed
$\beta^2$. Fundamentally, with a fixed $\beta^2$, Adam introduces a
non-vanishing correlation between the gradient and the preconditioner
applied to this gradient;
thus the step size becomes correlated with the gradients,
which introduces a bias in the average step. With $\beta^2\to 1$, the
preconditioner is computed from an average over more and more past
gradients, so this
bias disappears asymptotically.

\begin{proof}
As in Section~\ref{sec:adap}, define the augmented parameter
$\param^+\deq(\param,\Stat)\in \Param\times \R^{\dim(\stat)}$. As in Section~\ref{sec:mom}, consider a recurrent system with real-valued
state $\state$ subject to the evolution equation
\begin{equation*}
\state_t=\opevol_t(\state_{t-1},(\param,\Stat))\deq
\beta^1 \state_{t-1}+(1-\beta^1)\perteiid(x_{i_t},y_{i_t},\param)
\end{equation*}
and to the loss function
$\perte_t(\state_t)\deq \state_t$. By the same computation as in
Section~\ref{sec:mom}, we have (dropping $\state_0$ again for simplicity)
\begin{equation*}
\sbpermc{t}(\param,\Stat)=(1-\beta) \sum_{j\leq t} \beta^{t-j}
\perteiid(x_{i_j},y_{i_j},\param)
\end{equation*}
and in particular $\partial_\Stat\sbpermc{t}(\param,\Stat)=0$.

By Definition~\ref{def:algortrl}, the variable $\jope_t$ of RTRL
for this system exactly follows \eqref{eq:adamJ}.

Define the same update operator as in Section~\ref{sec:adap},
\begin{equation*}
\fur_t((v_\param,v_\Stat),s,(\param,\Stat))\deq \left(
\begin{array}{c}
P(\param,\Stat)v_\param
\\
\momen \Stat-\momen\stat(x_{i_t},y_{i_t},\param)
\end{array}
\right)
\end{equation*}
and let us check Assumption~\ref{hyp:critoptrtrlnbt}.

The gradients computed by this algorithm are
\begin{align*}
u_t(\param,\Stat)
&\deq \fur_t\left(\partial_{\param^+}
\sbpermc{t}(\param^+),\fstate_t(\param^+),\param^+\right)
\\&=\left(
\begin{array}{c}
P(\param,\Stat)\,\partial_\param \sbpermc{t}(\param,\Stat)
\\
\momen \Stat-\momen\stat_t
\end{array}
\right)
\end{align*}
where we have abbreviated $\stat_t$ for $\stat(x_{i_t},y_{i_t},\param)$.
The extended Hessians are $\sbfh_t=\partial_{(\param,\Stat)}u_t$. We find
\begin{equation*}
\sbfh_t=\left(
\begin{array}{cc}
P\, \partial^2_\param \sbpermc{t}+(\partial_\param P)
\partial_\param\sbpermc{t}
&
(\partial_{\Stat} P) \partial_\param\sbpermc{t}
\\
-\momen\, \partial_\param \stat_t
& \momen\id
\end{array}
\right).
\end{equation*}

The difference with Section~\ref{sec:adap} is that we get the recurrent
loss $\sbpermc{t}$ instead of the instantaneous loss
$\perteiid_t$. However, as in the proof of
Corollary~\ref{cor:mom}, this does not change temporal averages: indeed
the system is the same as in Corollary~\ref{cor:mom} and by the same
computation we have
\begin{equation*}
\sum_{t=1}^{T} \partial_\param \sbpermc{t}=
\sum_{t=1}^{T} \partial_\param\perteiid(x_{i_t},y_{i_t},\param)+O(1)
\end{equation*}
and likewise for the Hessians. Consequently, time averages of the
gradients and Hessians coincide up to $O(1/T)$ with the momentum-less
case, so the time averages are the same as in Section~\ref{sec:adap}, and
we can conclude in the same way.
In
particular we still have
\begin{equation*}
\linalgmatrix=\left(
\begin{array}{cc}
P(\paramopt,\Stat^*) H
&
0
\\
-c\,C
& \momen\id
\end{array}
\right)
\end{equation*}
for some matrix $C$.
\end{proof}


The theory of two-timescale algorithms (e.g., \cite{tadic2004}) can also
be used to deal with preconditioned SGD, where one timescale is used to
estimate the preconditioner at the current value $\param_t$, and the
parameter is updated using this preconditioner.
However, with such two-timescale algorithms, it is necessary to
update the main parameter at a slower rate than the auxiliary statistic
$\Stat$, so that $\Stat_t$ has time to converge to its average value at
$\param_t$ before $\param_t$ is updated.

Our treatment here allows the parameter $\param_t$ to be updated as fast
as the statistic $\Stat_t$. This results in the off-diagonal block in the
$\linalgmatrix$ matrix of the system; but this block has no influence on
its eigenvalues hence no bearing on local convergence.


\subsubsection{Non-Recurrent Case, Online Stochastic Setting with
Infinite Dataset}
\label{sec:pureonline}

Here we assume a pure ``online SGD'' setting with an infinite dataset
$(x_t,y_t)$ obtained from some unknown probability distribution, and
where each sample is used for exactly one gradient step. One difference
with the standard treatment of SGD on convex functions is that we work
only under local assumptions: nothing is assumed outside of some ball
around $\paramopt$, so the assumptions have to ensure that the learning
trajectory never ventures there. Another difference is that we get
more constraints on possible learning rates, depending on which
moments of the noise are finite (Proposition~\ref{prop:moments}). We
believe this may be because the empirical law of large numbers
(Assumption~\ref{hyp:opt_simple}) does not capture all properties of a
random i.i.d.\ sequence.

Thus, in this section we assume that $(x_t,y_t)_{t\geq 1}$ is a sequence
of i.i.d.\ samples from some probability distribution over some set of
input-output pairs.
We consider a loss function
$\perteiid(x,y,\param)$ depending on $\param$. We assume that $\perteiid$ is
$C^3$ with respect to
$\param$ for each pair $(x,y)$ in the dataset
(as before, this guarantees that the smoothness and
equicontinuity assumptions are satisfied). We will consider ordinary
stochastic gradient descent
\begin{equation*}
\param_t=\param_{t-1}-\pas_t \, \partial_\param \perteiid_t
\end{equation*}
where as before we abbreviate $\perteiid_t(\param)\deq
\perteiid(x_t,y_t,\param)$.

We call
\emph{strict local optimum} a local optimum of the average loss
with positive definite Hessian, namely, a parameter $\paramopt$ such that
\begin{equation*}
\E_{(x,y)} \partial_\param \perteiid(x,y,\paramopt)=0
\end{equation*}
and
\begin{equation*}
H\deq \E_{(x,y)} \partial^2_\param
\perteiid(x,y,\paramopt)\succ 0.
\end{equation*}

These assumptions have no consequences on what happens far from
$\paramopt$, where the loss function could be badly
behaved. This impacts the behavior of stochastic
gradient descent: to ensure local convergence to $\paramopt$, the
consecutive steps should never venture out of some ball around
$\paramopt$. This implies that the steps $\pas_t
\partial_\param\perteiid_t$ should be bounded. This has consequences for
the possible learning rates depending on the noise on the gradients
$\partial_\param \perteiid_t$.

In this situation, the ``technical'' assumptions from
Section~\ref{sec:techass} are not automatically satisfied. Let us examine
all assumptions in turn.

We define the dynamical system as 
in Example~\ref{ex:nonrec}, by identifying the state $\state$
with $\theta$; namely, $\opevol_t(\state,\param)\deq \param$, and
$\perte_t(\state)\deq\perteiid(x_t,y_t,\state)$. In particular,
$\sbpermc{t}(\param)=\perteiid(x_t,y_t,\param)$.

We consider the simple
RTRL algorithm (no $\fur_t$ and $E_t=0$).
Assumption~\ref{hyp:specrad} is satisfied because
$\partial_\state\opevol_t=0$. We have to check
Assumptions~\ref{hyp:opt_simple}, 
\ref{hyp:regftransetats}, \ref{hyp:regpertes} and
\ref{hyp:equicontH_simple}.

Assumption~\ref{hyp:opt_simple} encodes the speed at which empirical
averages of gradients and Hessians converge to their expectation. It is
satisfied with an exponent $a$ depending on the moments of the noise, as
follows.

\begin{proposition}
\label{prop:moments}
Given a random sample $(x,y)$, assume the random variables
$\partial_\param \perteiid(x,y,\paramopt)$ and $\partial^2_\param
\perteiid(x,y,\paramopt)$ have finite moments of order $h$ for some
$2\leq h<4$. Then with probability $1$ over the choice of samples
$(x_t,y_t)$, Assumption~\ref{hyp:opt_simple} is satisfied, with
exponent $\expem=2/h$.

In particular, if the gradients and Hessians of the loss at $\paramopt$
have moments of order $4$, then Assumption~\ref{hyp:opt_simple} is
satisfied for any exponent $\expem>1/2$.
\end{proposition}

Due to the statistical fluctuations in $1/\sqrt{t}$,
Assumption~\ref{hyp:opt_simple} is never satisfied for $a<1/2$ unless
gradients and Hessians are independent of $(x,y)$.

\begin{proof}
By Theorem 3.b in \citet{baumkatz65_lln} applied with $r=2$, we know
that i.i.d.\ centered variables $X_t$ have a moment of
order $2\leq h<4$ if and only if for every $\eps>0$, the series
$\sum_{t=1}^\infty \Pr\left(
\abs{\sum_{i=1}^t X_i}>t^{2/h}\eps
\right)$  is finite. By Borel--Cantelli, this implies that $\sum_{i=1}^t
X_i$ is $o(t^{2/h})$ with probability $1$. This is what we need for
Assumption~\ref{hyp:opt_simple} with exponent $2/h$.
\end{proof}

Assumption~\ref{hyp:regftransetats} is trivially satisfied for our choice
of $\opevol_t$.

The strongest assumption is Assumption~\ref{hyp:equicontH_simple}: it
states that $\partial^2_\param \ell(x,y,\param)$ is equicontinuous in
$\param$, uniformly in $(x,y)$ for $\param$ in some neighborhood of
$\paramopt$. This happens, for instance, if there
exists a constant $k$ such that the third derivatives of
$\ell(x,y,\param)$ are bounded by $k$ for any $(x,y)$ in a neighborhood
of $\paramopt$.

Assumption~\ref{hyp:regpertes} amounts to an almost sure bound on the
growth of $\partial_\param \perte_t$ and $\partial^2_\param \perte_t$.
For $\partial^2_\param \perte_t$ this has to be uniform in a neighborhood of $\paramopt$. If
Assumption~\ref{hyp:equicontH_simple} is satisfied, it is enough to check
this for a dense (denumerable) set of values of $\param$ in that
neighborhood. By the Markov inequality and the Borell--Cantelli lemma, if a sequence of i.i.d.\ random
variables $(X_t)$ has finite moments of order $h$, then for any
$\exmp>1/h$, we have
$X_t=O(t^\exmp)$ with probability $1$. Therefore, under
Assumption~\ref{hyp:equicontH_simple}, Assumption~\ref{hyp:regpertes} is
again
an assumption on moments of gradients and Hessians.

Putting everything together, we obtain the following.

\begin{corollaire}[SGD under local assumptions]
\label{cor:pureonline}
Assume that $\partial_\param \perteiid(x,y,\paramopt)$ and $\partial^2_\param
\perteiid(x,y,\param)$ have moments of order $h\geq 2$.
Assume the third derivatives of $\perteiid(x,y,\param)$ with respect to
$\param$ are bounded in a neighborhood of $\paramopt$, uniformly over
$(x,y)$. Then the assumptions of Theorem~\ref{thm:cvapproxrtrlalgo}
are satisfied for $a>\max(1/2,2/h)$ and $\gamma>1/h$.

Consequently, for learning rates 
$\pas_t=\cdvp\, t^{-\expas}$ with
$\max(1/2,2/h)+2/h<b\leq 1$, the stochastic gradient descent
$\param_t=\param_{t-1}-\pas_t\, \partial_\param
\perteiid(x_t,y_t,\param_{t-1})$ converges locally to $\paramopt$.
\end{corollaire}

For instance, for linear regression $y_t=x_t\cdot \param+\eps_t$
with noise $\eps_t$ and bounded $x_t$, these constraints encode the
moments of $\eps_t$.

For $h\to \infty$, we recover the classical constraint $1/2<b\leq 1$.
However, when not all moments of the gradients and Hessians are finite,
the learning rates are more constrained. This is due to working only
under local assumptions, and reflects the need for the gradient descent to
stay in a finite ball a priori.

We do not know if these bounds are optimal: the constraint
$\expas>\max(1/2,2/h)+2/h$  may be
spurious and due to analyzing the non-recurrent case from a recurrent
viewpoint.

\subsection{Truncated Backpropagation Through Time}
\label{sec:tbptt}

We now consider the truncated backpropagation through time (TBPTT) algorithm,
We refer to
\citet{jaeg,pearl} for a discussion of TBPTT. 
We assume the truncation length slowly grows to $\infty$: with fixed truncation length,
the gradient computation is biased and there is no reason for the
algorithm to converge to a local minimum (see, e.g., 
the simple ``influence balancing'' example of divergence in \citet{uoro}). Thus, we let the
truncation length grow to $\infty$ at a slow rate $t^\exli$ for some exponent
$\exli<1$, related to the learning rate.

Truncated backpropagation through time comes in several variants
\citep{williams1990efficient,williams1995gradient}: an
``overlapping'' variant in which, at each step $t$, backpropagation is
run for $L$ backwards step and the parameter is updated using the
approximate gradient of $\perte_t$ (running time $O(Lt)$: every state is
visited once forward and $L$ times backward); and a
``non-overlapping'' or ``epochwise'' variant in which the input sequence is split into
segments of size $L$ and the parameter is updated every $L$ steps using
the gradients of losses $\perte_{t-L+1},\ldots,\perte_t$ computed on 
this interval (running time $O(t)$: every state is visited once forward
and once backward). As a compromise, partly overlapping settings exist
\citep{williams1990efficient,williams1995gradient}. We treat the
non-overlapping variant here.

We denote $\fstate_{t_0:t_1}(\state_{t_0},\param)$ the state at time
$t_1$ of the dynamical system starting at time $t_0$ in state
$\state_{t_0}$, with constant parameter $\param$. We denote
$\permctt{t_0}{t_1}(\state_{t_0},\param)\deq
\perte_{t_1}(\fstate_{t_0:t_1}(\state_{t_0},\param))$ the
resulting loss at time $t_1$ as a
function of $\param$.

\begin{definition}[Truncated Backpropagation Through Time algorithm]
\label{def:tbtt}
The Truncated Backpropagation Through Time algorithm (TBPTT), with step sizes
$(\pas_t)_{t \geq 1}$, truncation times $\tpsk{k}$ (an increasing integer
sequence starting at $\tpsk{0}=0$),
starting at
$\state_0\in \State_0$ and $\param_0\in \Param$, maintains a
state $\state_t\in\State_t$, and a parameter $\param_t\in \Param$,
subjected to the following evolution
equations. For every $k\geq 0$ and every $\tpsk{k}\leq t\leq \tpsk{k+1}$,
the states are computed using parameter $\param_{\tpsk{k}}$, and the
parameter is updated using the gradient of the loss on that time
interval; more precisely,
\begin{equation*}
\left\lbrace\ba
\state_{t} &= \opevol_t(\state_{t-1},\param_{\tpsk{k}}),
\qquad  \tpsk{k}< t\leq \tpsk{k+1}
\\
\param_{\tpsk{k+1}}&=\param_{\tpsk{k}}-\pas_{\tpsk{k+1}}\sum_{t=\tpsk{k}+1}^{\tpsk{k+1}}
\frac{\partial
\permctt{\tpsk{k}}{t}\left(\state_{\tpsk{k}},\param_{\tpsk{k}}\right)}{\partial
\param}.
\ea\right.
\end{equation*}
\end{definition}

Of course, we could also use update functions $\fur_t$ and
$\paramupdate_t$ as before.

It is well-known \citep{pearl} that on a fixed time interval
$(\tpsk{k};\tpsk{k+1}]$, backpropagation through time computes the same
quantity as RTRL with a fixed parameter (``open-loop'' RTRL), namely,
both compute the gradient of the total loss over the time interval. Thus,
we can see TBPTT as a form of RTRL which resets the Jacobian $\jope$ to
$0$ at the start of every time interval, and updates the parameter only
once at the end (Proposition~\ref{prop:tbpttasrtrl}).

Thus, by a slight change of our analysis of RTRL, we obtain the following
result, proved in Section~\ref{sec:proofscvrtrlapprox}.

We find that there is a ``sweet spot'' for the growth of truncation
length: with short intervals, gradient are biased due to truncation. With
too long intervals, the parameter is updated very rarely using a large
number of gradients, and the steps may be large and diverge. This is the
meaning of the constraint ${\expem} < \exli <
b-2\exmp$ relating truncation length $t^\exli$ to the stepsize sequence exponents of
Assumption~\ref{hyp:scfcstepsize}.

\begin{theoreme}[Convergence of TBPTT]
\label{thm:cvtbtt}
Let $\paren{\opevol_t}$ be a parameterized dynamical sytem
(Def.~\ref{def:prmdynsys}) with loss functions $\perte_t$
(Def.~\ref{def:fcpletaprm}).

Let $\paramopt$ be a local optimum for this system
(Assumption~\ref{hyp:opt_simple}),
with initial state $\stateopt_0$. Assume that the system
with parameter $\paramopt$ starting at $\stateopt_0$ is
stable (Assumption~\ref{hyp:specrad}).

Assume that the first and second derivatives of $\opevol_t$ and
$\perte_t$ are controlled around the target trajectory
(Assumptions~\ref{hyp:regftransetats} and \ref{hyp:regpertes}).
Assume that the Hessians of the losses are uniformly
continuous close to $\paramopt$ (\rehyp{equicontH_simple}).

Let $\sm{\pas}=\paren{\pas_t}$ be a stepsize sequence with overall
learning rate $\cdvp$ and exponents $\expas$, $\expem$, $\exmp$
satisfying \rehyp{scfcstepsize}.

We assume that the truncation intervals for TBPTT have lengths
$\tpsk{k+1}-\tpsk{k}=\tpsk{k}^\exli$, for some $\max\paren{\expem,\exmp} < \exli <
b-2\exmp$.

Then truncated backpropagation through time on the intervals
$[\tpsk{k};\tpsk{k+1}]$ converges locally around $\paramopt$.

Explicitly, there exists $\cdvpmaxconv>0$ such that for any
$\cdvp<\cdvpmaxconv$, the following holds:
There is a neighborhood $\mathcal{N}_{\paramopt}$ of $\paramopt$ and a
neighborhood $\mathcal{N}_{\stateopt_0}$ of $\stateopt_0$ such that, for
any initial parameter $\param_0\in \mathcal{N}_{\paramopt}$ and any
initial state $\state_0\in \mathcal{N}_{\stateopt_0}$
the TBPTT trajectory given by, for every $k\geq 0$
\begin{equation*}
\left\lbrace\ba
\state_{t} &= \opevol_t(\state_{t-1},\param_{\tpsk{k}}),
\qquad  \tpsk{k}< t\leq \tpsk{k+1}
\\
\param_{\tpsk{k+1}}&=\param_{\tpsk{k}}-\pas_{\tpsk{k+1}}\sum_{t=\tpsk{k}+1}^{\tpsk{k+1}}
\frac{\partial
\permctt{\tpsk{k}}{t}\left(\state_{\tpsk{k}},\param_{\tpsk{k}}\right)}{\partial
\param},
\ea\right.
\end{equation*}
produces a sequence $\param_{\tpsk{k}}\to \paramopt$ as $k\to\infty$.
\end{theoreme}

\subsection{Approximations of RTRL: NoBackTrack and UORO}

RTRL is computationally too heavy for
large-dimensional systems, hence the need for approximations such as
NoBackTrack and UORO (see above).

Here we explain how the NoBackTrack and UORO algorithms fit into the
framework for \algonoisy RTRL algorithms (Def.~\ref{def:approxrtrl}).
These algorithms maintain a rank-one approximation $\tilde \jope_t$ of
$\jope_t$ at each step. Their key feature is that this approximation is unbiased:
namely, Assumption~\ref{hyp:wcorrnoiseapprtrl} is satisfied.

As corollaries, we will obtain local convergence of NoBackTrack and UORO
under the same assumptions as RTRL (\recor{cvnbtuoro}); note however that the learning rates
are more constrained if the output noise exponent $\exmp$ is nonzero. 

We believe that an identical argument also covers more recent extensions
of UORO, such as Kronecker-factored RTRL \citep{NIPS2018_7894},
although we do not include it explicitly. Indeed, the only properties
needed are the assumptions from Section~\ref{sec:approxrtrl}, namely,
that the noise $E_t$ with respect to true RTRL is unbiased and almost
surely bounded by some sublinear function of $\jope_t$.

Our analysis also emphasizes the crucial role of the variance reduction
scaling factors used in NoBackTrack and UORO (via norm equalization).
These ensure that the error is sublinear in $\jope$ at each step, namely,
Assumption~\ref{hyp:errorgauge} is satisfied.  In the convergence proof,
this assumption ensures that errors do not accumulate over time.

\subsubsection{The NoBackTrack and UORO Algorithms}
\label{sec:defnbtuoro}

For completeness, we include a formal definition of the NoBackTrack and
UORO algorithms. We refer to the original publications \cite{oll16,uoro} for ready-to-use
formulas.

In NoBackTrack and UORO, the unbiasedness property is achieved by a
``rank-one trick'' \citep{oll16} involving random signs at every step,
which randomly reduces $\jope_t$ to a rank-one approximation $\tilde
\jope_t$ with the
correct expectation. A variance reduction step is performed thanks to a
careful rescaling (the norm equalizing operator), which guarantees that
the approximation error on $\tilde \jope_t$ scales like $\sqrt{\nrm{\tilde \jope_t}}$.

\begin{definition}[Norm equalizing operator]
Given two normed vector spaces $\espaux_1$ and $\espaux_2$, we define the
operator $\sbopegn\colon \espaux_1 \times \espaux_2 \to \espaux_1 \times
\espaux_2$ by
\begin{equation*}
\opegn{v_1}{v_2}\deq
\left\lbrace
\ba
\vopegn{v_1}{v_2} \quad &\text{if} \quad \vesp_1 \neq 0 \quad \text{and} \quad \vesp_2 \neq 0, \\
\paren{0,\,0} \quad &\text{otherwise}.
\ea \right.
\end{equation*}
\end{definition}

Note that $\opegn{v_1}{v_2}$ is invariant by multiplying $v_1$ and
dividing $v_2$ by the same factor $\lambda>0$.

\begin{definition}[Random signs]
\label{def:signesaleat} 
We consider independent, identically distributed Bernoulli random variables
\begin{equation*}
\berit, \quad t \geq 1, \quad 1 \leq i \leq \dim \epsystins{t},
\end{equation*}
which equal $1$ or $-1$ both with probability $1/2$.

For every $t \geq 1$, we write $\veber{t}$ the vector of the $\berit$'s,
for $1 \leq i \leq \dim \epsystins{t}$.
\end{definition}

The NoBackTrack and UORO reduction operators $\sbopred_t$ are defined so that, if
$\tilde\jope_t=\vsyst_t\otimes \vctrl_t$ is a rank-one approximation of
$\jope_t$ at time $t$, then $\sbopred_t$ returns a pair
$(\vsyst_{t+1},\vctrl_{t+1})$ such that, on average over $\veber{t}$,
\begin{equation*}
\E_{\veber{t}}\left[ \vsyst_{t+1}\otimes \vctrl_{t+1}\right]=\frac{\partial
\opevol_{t+1}}{\partial s}\tilde\jope_t+\frac{\partial
\opevol_{t+1}}{\partial \param}
\end{equation*}
namely, on average, the RTRL equation is satisfied.

\begin{definition}[NoBackTrack reduction operators]
\label{def:opredinstantt}
Let $t \geq 1$. We define the NoBackTrack reduction operator at time $t$, 
\begin{equation*}
\ba
\sbopred_t & : &{\epsystins{t-1}} \times \sblin(\Param,\mathbb{R}) \times \epsystins{t-1} \times \epctrl  &\to {\epsystins{t}} \times \sblin(\Param,\mathbb{R}) \\
& & \vsyst,\,\vctrl,\,\state,\,\prmctrl & \mapsto \opredt{\vsyst}{\vctrl}{\state}{\prmctrl},
\ea
\end{equation*}
by
\begin{equation*}
\opredt{\vsyst}{\vctrl}{\state}{\prmctrl}\deq
\opegn{\paren{\dopevolst{\state}{\prmctrl} \, \vsyst}}{\vctrl} +
\som{i=1}{\dim \epsystins{t}} \beri{t} \,
\paren{\opegn{\vcans_i}{\dpartf{\prmctrl}{\opevol_t^i}\paren{\state,\,\prmctrl}}},
\end{equation*}
where at each step $t$, the $\vcans_i$'s are a (deterministic) orthonormal basis of
vectors of the state space $\epsystins{t}$ (for brevity we omit the time dependency in
the notation $\vcans_i$).
\end{definition}

Note that this operator is invariant by multiplying $\vsyst$ and
dividing $\vctrl$ by the same factor $\lambda>0$.

\begin{definition}[UORO reduction operators]
\label{def:uoroopredinstantt}
Let $t \geq 1$. We define the UORO reduction operator at time $t$, 
\begin{equation*}
\ba
\sbopred_t & : &{\epsystins{t-1}} \times \sblin(\Param,\mathbb{R}) \times \epsystins{t-1} \times \epctrl  &\to {\epsystins{t}} \times \sblin(\Param,\mathbb{R}) \\
& & \vsyst,\,\vctrl,\,\state,\,\prmctrl & \mapsto \opredt{\vsyst}{\vctrl}{\state}{\prmctrl},
\ea
\end{equation*}
by
\begin{equation*}
\ba
\opredt{\vsyst}{\vctrl}{\state}{\prmctrl}&\deq
\opegn{\paren{\dopevolst{\state}{\prmctrl} \, \vsyst}}{\vctrl} \\
&+ \opegn{\paren{\som{i=1}{\dim \epsystins{t}} \beri{t} \,
{\vcans_i}}}{\paren{\som{i=1}{\dim \epsystins{t}}\,\beri{t}\,{\dpartf{\prmctrl}{\opevol_t^i}\paren{\state,\,\prmctrl}}}},
\ea
\end{equation*}
where the $\vcans_i$'s form an orthonormal basis of vectors of $\epsystins{t}$.
\end{definition}

The difference between NoBackTrack and UORO lies in the order of the sum
and norm equalization in the second part. Notably, UORO only has two norm
equalization operations, which leads to substantial algorithmic gains
\citep{uoro}.

These operators lead to the formal definition of the
NoBackTrack and UORO algorithms, as follows.

\begin{definition}[NoBackTrack and UORO operators]
\label{def:operateurnbt}
Let $t \geq 1$. Let $\sepjope{t}$ denote the space of rank-one linear
operators from $\epctrl$ to $\epsystins{t}$. The NoBackTrack
(respectively UORO) operator
\begin{equation*}
\ba
\sbopnbt_{t} & : &\epctrl \times \epsystins{t-1} \times \sepjope{t-1} &\to \sepjope{t}\\
& & \param, \, \state,\,\jope & \mapsto \opnbtit{\prmctrl}{\state}{\jope}
\ea
\end{equation*}
is defined as follows.

\begin{enumerate}
\item Write $\jope=\tens{\vsyst}{\vctrl}$, for some pair $\cvnbtsi$
belonging to $\epsystins{t-1} \times \sblin(\Param,\mathbb{R})$ (uniquely defined up to
multiplying $\vsyst$ and dividing $\vctrl$ by some factor $\lambda\neq
0$).
\item Apply the reduction step, that is, set
\begin{equation*}
\cvnbtaxsi=\opredt{\vsyst}{\vctrl}{\state}{\prmctrl},
\end{equation*}
with $\sbopred_t$ the NoBackTrack (respectively UORO) reduction operator.
\item Finally, define
$
\opnbtit{\prmctrl}{\state}{\jope}\deq\tens{\vsystax}{\vctrlax}.
$
\end{enumerate}
\end{definition}

This defines a random operator depending on the choice of the random
signs $\veber{t}$. In law, the value of
$\opnbtit{\prmctrl}{\state}{\jope}$ is independent of the choice of
decomposition $\jope=\tens{\vsyst}{\vctrl}$: indeed, the operator
$\sbopred_t$ is invariant by multiplying $\vsyst$ and dividing $\vctrl$
by some factor $\lambda>0$, so we just have to compare the decompositions $\tens{\vsyst}{\vctrl}$
and $\tens{(-\vsyst)}{(-\vctrl)}$. Thanks to the random signs, these two
choices lead to the same law for $\opnbtit{\prmctrl}{\state}{\jope}$.

\begin{definition}[NoBackTrack and UORO]
\label{def:nbtuoro}
We define 
NoBackTrack and UORO as \algonoisy RTRL algorithms
(Def.~\ref{def:approxrtrl}) by setting 
\begin{equation*}
\tilde \jope_t\deq \opnbtit{\prmctrl_{t-1}}{\state_{t-1}}{\tilde\jope_{t-1}}
\end{equation*}
and tautologically defining
$E_t$ as the difference with respect to RTRL,
\begin{equation*}
E_t\deq \tilde \jope_t-\frac{\partial
\opevol_t(\state_{t-1},\param_{t-1})}{\partial
\state} \,\tilde \jope_{t-1}-\frac{\partial
\opevol_t(\state_{t-1},\param_{t-1})}{\partial
\param}.
\end{equation*}
\end{definition}

\subsubsection{Convergence of NoBackTrack and UORO}

We are now ready to state convergence of NoBackTrack and UORO. Compared
to RTRL, the only thing to check is that the assumptions of
Section~\ref{sec:approxrtrl} for \algonoisy RTRL algorithms are
satisfied, namely, unbiasedness and boundedness of $E_t$. For this, we
have to add the assumption that the dimension of states $\State_t$ does
not increase to infinity over time.

\begin{hypothese}[Bounded state space dimension.]
\label{hyp:bndimstatespaces}
$
\sup_{t \geq 0} \, \dim \State_t < \infty.
$
\end{hypothese}
\begin{lemme}[NoBackTrack and UORO as \algonoisy RTRL algorithms]
\label{lem:nbtuoroapproxrtrl}
Under \rehyp{bndimstatespaces} of bounded state space dimension, 
the errors $E_t$ of NoBackTrack and UORO satisfy
Assumptions~\ref{hyp:wcorrnoiseapprtrl} (unbiasedness)
and~\ref{hyp:errorgauge} (sublinearity in $\jope$).
\end{lemme}

This lemma is proved in \resec{nbtuoroasapproxrtrlalgo}. Then
Theorem~\ref{thm:cvapproxrtrlalgo} provides the following conclusion.

\begin{corollaire}[Convergence of NoBackTrack and UORO]
\label{cor:cvnbtuoro}
Under \rehyp{bndimstatespaces} of bounded state space dimension,
and under the additional constraint on stepsizes for imperfect
RTRL algorithms (Assumption~\ref{hyp:scfcstepsize}),
NoBackTrack and UORO converge locally under the same general assumptions as
RTRL,
with probability tending to $1$ when the overall learning rate tends to $0$.
\end{corollaire}

We refer to Section~\ref{sec:mainthm} for a discussion of why the convergence
only occurs with probability tending to $1$ as the learning rate tends to
$0$. In short, with only local instead of global assumptions, with some
positive probability, the noise introduced by NoBackTrack and UORO may
bring the trajectory outside of the safe zone where our assumptions
apply. This gets less and less likely as the learning rate decreases,
because the noise is more averaged out.



\section{Abstract Online Training Algorithm for Dynamical Systems}
\label{sec:absontadsys}


In this section, we define an abstract model of a learning algorithm
applied to a dynamical system, and formulate assumptions about its
behaviour. Notably, for RTRL, the abstract state $\mem_t$ will encompass
both the state $\state_t$ of the dynamical system, and the internal
variable $\jope_t$ maintained by RTRL.

\subsection{Model}

\label{sec:abstractalgos}
We start by defining the spaces the quantities used by the algorithm live on.
\begin{definition}[Parameter, maintained quantities and tangent vectors spaces]
Let $\Param$ be some metric space, the \emph{parameter space}. Let
$(\Mem_t)_{t\geq 0}$ be a sequence of metric spaces, which represents the objects
maintained in memory by an algorithm at time $t$. Let
$(\Tangent_t)_{t\geq 0}$ be a sequence of normed vector spaces, containing the
gradients computed at time $t$.
\end{definition}



The transition operator of the algorithm which we introduce below symbolises all the updates performed on the objects the algorithm maintains in memory.

\begin{definition}[Transition operator]
\label{def:transop}
We call \emph{transition operator} a family $(\Algo)_{t\geq 1}$ of
functions
\begin{equation*}
\ba
\Algo_t &\from & \Param \times \Mem_{t-1} & \to \Mem_{t}\\
 & & \param, \mem & \mapsto \Algo_t(\param,\mem).
 \ea
\end{equation*}

Given such a transition operator, we call \emph{open-loop system} 
the family of operators $\Algo_{T_i:T_f}$ for $T_i< T_f$, 
\begin{equation*}
\ba
\Algo_{T_i:T_f} & : &\Param \times \Mem_{T_i} &\to \Mem_{T_f}\\
	& & \param, \mem & \mapsto \mem_{T_f},
\ea
\end{equation*}
where the sequence $(\mem_t)$ is defined inductively by $\mem_{T_i}\deq
\mem$ and, for $T_i+1\leq t \leq T_f$,
\begin{equation*}
\mem_{t} = \Algo_t\paren{\param,\,\mem_{t-1}}.
\end{equation*}
\end{definition}

Informations about how to modify the parameter are computed thanks to the following gradient computation operators.

\begin{definition}[Gradient computation operators]
\label{def:gradcomputop}
We call \emph{gradient computation operators} a sequence
$(\gradop_t)_{t\geq 1}$ where each $\gradop_t$ is a function from $\Param\times \Mem_t$ to $\Tangent_t$.

\end{definition}
We call trajectory the sequence of objects maintained in memory, together with the gradients collected along it.
\begin{definition}[Trajectories]
Let $(\param_t)_{t\geq 0}$ be a sequence of elements of $\Param$, and let
$\mem_0\in\Mem_0$. We call \emph{trajectory with parameter $(\param_t)$
starting at $\mem_0$}, the sequence $(\mem_t)$ defined inductively by
\begin{equation*}
\mem_t=\Algo_t(\param_{t-1},\mem_{t-1})
\end{equation*}
together with the sequence $(\vtanc_t)_{t\geq 1}$ of gradients
\begin{equation*}
\vtanc_t=\gradop_t(\param_{t-1},\mem_t).
\end{equation*}

If $\param$ is a single parameter value, we extend this definition to the
constant sequence $\param_t\equiv\param$.

Later, we will call \emph{optimal trajectory} the trajectory associated to the optimal
parameter $\paramopt$ starting at $\memopt_0$, once these are introduced.
\end{definition}

Finally, we call parameter update operator the update rule of the parameter.

\begin{definition}[Parameter update operators]
\label{def:paramupdate}
We call \emph{parameter update operators} a sequence
$(\paramupdate_t)_{t\geq 1}$ where each $\paramupdate_t$ is a function from
$\Param\times \Tangent_t$ to $\Param$.

Let $(\paramupdate_t)_{t\geq 1}$ be parameter update operators.
Given a sequence of gradients $\vtanc_t\in \Tangent_t$ and two integers 
$0\leq t_1 \leq t_2$, we
denote $\paramupdate_{t_1:t_2}(\param,(\vtanc_t))$ the consecutive
application of $\paramupdate_t(\cdot,\vtanc_t)$ to $\param$ from time
$t_1+1$ to time $t_2$, namely
\begin{equation*}
\paramupdate_{t_1:t_2}(\param,(\vtanc_t))\deq \param_{t_2},
\end{equation*}
where $\param_{t}$ is defined inductively via $\param_{t_1}\deq
\param$ and 
$
\param_{t} = \paramupdate_t(\param_{t-1},\vtanc_t)
$ for $t_1<t\leq t_2$.
\end{definition}
Thanks to the operators we have just defined, we model a gradient descent trajectory by the update equations, for $t\geq 1$,
\begin{equation*}
\left\lbrace \ba 
\mem_t&=\Algo_{t}\paren{\param_{t-1},\,\mem_{t-1}}\\
\vtanc_t &=\sm{V}_{t}\paren{\param_{t-1},\,\mem_t} \\ 
\param_{t} &= \paramupdate_t\paren{\param_{t-1},\,\pas_t\,\vtanc_t},
\ea \right.
\end{equation*}
for some initial parameter $\param_0$, some initial $\mem_0\in\Mem_{0}$
and some sequence of step sizes $\paren{\pas_t}$. Our formalism also
encompasses algorithms updating the parameter only after a batch of
consecutive steps, as can be seen in \rethm{cvalgoboucleouverte}. We prove convergence of these procedures in \rethm{cvalgopti} and \rethm{cvalgoboucleouverte}.

The next definition will be used instead of \gu{the average of gradients
at $\param\in \Param$}. It replaces the average with the sum of
gradients over a number of steps, with $\param$ kept fixed. 

\begin{definition}[Open-loop update from $t_1$ to $t_2$]
\label{def:olupdate}
Let $\param\in\Param$, let $0\leq t_1\leq t_2$, let $\mem_{t_1}\in
\Mem_{t_1}$, and let $(\pas_t)_{t\geq 1}$ be a sequence of step sizes.

We call \emph{open-loop update} of $\param\in\Param$ from time $t_1$ to
time $t_2$, denoted
$\paramupdate_{t_1:t_2}(\param,\mem_{t_1},(\pas_t))$, the value of the
parameter obtained by first computing the trajectory with fixed parameter
$\param$ starting at $\mem_{t_1}$, then collecting the gradients along
this trajectory and applying them to $\param$. More precisely, define by
induction for $t>t_1$
\begin{equation*}
\mem_t=\Algo_t\paren{\param,\,\mem_{t-1}}=\Algo_{t_1:t}(\param,\,\mem_{t_1}),\qquad
\vtanc_t =\sm{V}_{t}\paren{\param,\,\mem_t} 
\end{equation*}
and then set
\begin{equation*}
\paramupdate_{t_1:t_2}(\param,\mem_{t_1},(\pas_t))\deq
\paramupdate_{t_1:t_2}(\param,(\pas_t\, \vtanc_t)).
\end{equation*}

If $\pas_t\equiv \pas$ is a constant sequence, we will abbreviate this to
$\paramupdate_{t_1:t_2}(\param,\mem_{t_1},\pas)$.
\end{definition}
\begin{definition}[Stepsize sequence: overall learning rate and stepsize schedule]
\label{def:stepseqovlrstpsch}
The sequence of step sizes $\sm{\pas}=(\pas_t)_{t\geq 1}$ of the
algorithm will be parameterized as
\begin{equation*}
\pas_t\deq \cdvp\,\sched_t,
\end{equation*}
where $\cdvp\geq 0$ is the \emph{overall learning rate} and $(\sched_t)_{t\geq 0}$
is a sequence with values in $[0;1]$, the \emph{stepsize schedule}.
\end{definition}

In the sequel we will assume that a stepsize schedule is given, and
prove convergence provided the overall learning rate is small enough.
Thus, in the whole text \emph{$\pas_t$ is implicitly a function of
$\cdvp$}. The assumptions on $\pas_t$ below are actually
assumptions on $\sched_t$.

The following definition will be useful to keep track of orders of
magnitude of the quantities involved.

\begin{definition}[Scale function]
\label{def:fechelle}
We call scale function a non-negative function $\sbfech$, defined on the non-negative real axis, which satisfies the following properties.
\begin{enumerate}
\item $\fech{t}$ tends to infinity, when $t$ tends to infinity.
\item $\sbfech$ preserves asymptotic equivalence at infinity: if $x_t\to
\infty$ and $y_t\sim x_t$ then $\fech{y_t}\sim \fech{x_t}$ as
$t\to\infty$.
\item $\sbfech$ is non-decreasing, and $\fech{1} \geq 1$.
\end{enumerate}
\end{definition}

\remarq{For instance, for every $a > 0$, $t \mapsto t^a$ is a scale function.}

\subsection{Assumptions on the Model}
\subsubsection{Assumptions about the Transition Operators}

Let $\paramopt\in \Param$  be a target parameter and let
$(\memopt_t)_{t\geq 0}$ be the corresponding trajectory, initialized at
$\memopt_0\in\Mem_0$.
In order to control the trajectory of the quantities maintained by the algorithm, we assume they are contained in some stable tube, defined below. Under some conditions of contractivity for the transition operator, the tubes reduce to sequence of balls of common radius. However, when the transition operator satisfies weak contractivity assumptions, which is our working assumption in the second part, we must use the more general definition below.

\begin{definition}[Stable tube]
\label{def:stabletube}


A \emph{stable tube} around a target trajectory $(\memopt_t)$ with
parameter $\paramopt$, is a ball $\boctopt\subset \Param$ of positive
radius $\rayoptct$ centered
\footnote{We have to assume that $\boctopt$ is centered at $\paramopt$, for
the proof of \relem{maintientpsfinibocont} (we control
$\dist{\param_t}{\paramopt}$ to show that $\param_t\in \boctopt$).} at $\paramopt$, together with sets
$\bmaint{t}\subset\Mem_t$ for each $t\geq 0$, such that:
\begin{enumerate}
\item Stability: If $t\geq 1$, $\param\in \boctopt$ and $\mem_{t-1}\in \bmaint{t-1}$, then
$\Algo_{t}(\param,\mem_{t-1})\in \bmaint{t}$;
\item Upper and lower boundedness: There exist $0<r\leq R$ such that, for all $t\geq
0$,
\begin{equation*}
B_{\Mem_t}(\memopt_t,r)\subset \bmaint{t}\subset
B_{\Mem_t}(\memopt_t,R);
\end{equation*}
\item Initial value: $\memopt_0\in\bmaint{0}$.
\end{enumerate}
\end{definition}

\begin{hypothese}[Stable tube around the target trajectory]
\label{hyp:stableballs}
There exists a stable tube around the target trajectory $(\memopt_t)$
defined by $\paramopt$.
\end{hypothese}

This assumption concerns the fixed-parameter system we start with, not
the system where the parameter is learned via the algorithm.  A priori,
learning might make the trajectories diverge.

\begin{remarque}
We can always decrease the size of $\boctopt$ while preserving the stable
tube condition. This will be useful when introducing additional
assumptions below, which may hold over a smaller set of parameters. For
simplicity we will just express these assumptions for
$\param\in\boctopt$, implicitly taking a smaller $\boctopt$ if necessary.

On the other hand, the stable tube condition is generally not stable by
reducing $\bmaint{t}$, because $\bmaint{t}$ must contain
$\Algo_t(\boctopt, \bmaint{t-1})$.
\end{remarque}




Finally, we assume the following behaviour for the transition operator on the stable tube.

\begin{hypothese}[The state update is Lipschitz w.r.t.\ the parameter.]
\label{hyp:paramlip}
We assume that $\Algo_t$ is Lipschitz over $\param$ in
$\boctopt\times \bmaint{t-1}$, uniformly in $t$. Namely, there exists a constant $\clipalgoprm$
such that for all $t\geq 1$,
for all $\param,\param'\in \boctopt$ and $\mem\in \bmaint{t-1}$,
\begin{equation*}
\dist{\Algo_t(\param,\mem)}{\Algo_t(\param',\mem)}\leq
\clipalgoprm\,\dist{\param}{\param'}.
\end{equation*}
\end{hypothese}

\begin{hypothese}[Exponential forgetting of initialization for fixed
$\param$]
\label{hyp:expforgetinit}
We assume that there exist
$0 < \alpha \leq 1$ and a constant $\cst_1$ such that, 
for any parameter $\param\in\boctopt$, for any $\tinit\geq 0$, for any states
$\mem_\tinit,\mem_\tinit'\in \bmaint{\tinit}$, the trajectories $(\mem_t)$ and
$(\mem_t')$ with parameter $\param$ starting at $\mem_\tinit$ and
$\mem'_\tinit$ at time $\tinit$,
respectively, satisfy
\begin{equation*}
\dist{\mem_t}{\mem_t'} \leq \cst_1 (1-\alpha)^{t-\tinit} \,
\dist{\mem_\tinit}{\mem'_\tinit}
\end{equation*}
for all $t \geq \tinit$.
\end{hypothese}

Note that these trajectories stay in $\bmaint{t}$, so this is a ``local''
assumption.

\subsubsection{Assumptions on Gradients}


The following assumption aims at controlling the magnitude of the gradients.
\begin{hypothese}[Locally bounded instantaneous gradients]
\label{hyp:boundedgrads}
We suppose that there exists a function $\fmdper{t}$ which is either
a scale function $\fmdper{t}\ll t$ or $\fmdper{t}\equiv 1$, such that,
when $t\to\infty$,
\begin{equation*}
\sup_{\param\in \boctopt}\,\sup_{\mem\in \bmaint{t}}
\norm{\sm{V}_t(\param,\mem)}=O(\fmdper{t}).
\end{equation*}

We denote $\bvtg{t}$ the ball of $\Tangent_t$ with radius $\sup_{\param\in
\boctopt}\,\sup_{\mem\in \bmaint{t}}
\norm{\sm{V}_t(\param,\mem)}$.
\end{hypothese}

The gradients are Lipschitz with respect to the parameter, and the quantities in memory.
\begin{hypothese}[Instantaneous gradients are locally Lipschitz.]
\label{hyp:contgrad}
We assume that there exist $\cst_5$ and a scale function $\mlipgrad{t}\ll t$ 
(or $\mlipgrad{t}\equiv 1$) such that for any parameters
$\param, \param'\in \boctopt$, for any $t\geq 1$, for any $\mem,\mem'\in
\bmaint{t}$, the instantaneous gradients satisfy
\begin{equation*}
\dist{\sm{V}_{t}\paren{\param,\,\mem}}{\sm{V}_{t}\paren{\param',\,\mem'}}
\leq \cst_5 \, (\dist{\param}{\param'}+\dist{\mem}{\mem'})\,\mlipgrad{t}.
\end{equation*}
\end{hypothese}

The reason we put a separate scale function $\mlipgrad{t}$ instead of
reusing $\fmdper{t}$ is because these can be different even in very basic
examples. Indeed, $\fmdper{t}$ controls the size of gradients whereas
$\mlipgrad{t}$ controls the Lipschitz constant of gradients, i.e.,
essentially the Hessian. For instance, for a linear model with loss
$\frac12 (y_t-\param \cdot x)^2$, the gradients are $\param\cdot x-y_t$ and if $y_t$
is unbounded, then $\fmdper{t}$ might be large. But the
\emph{differences} of gradients are $\param\cdot x-\param'\cdot x'$ which do not
depend on the norm of $y_t$: in this example the Hessian is bounded so
that $\mlipgrad{t}$ is constant. 

\subsubsection{Parameter Updates}

We would like the update operators to cover parameter-dependent updates
such as
$\paramupdate(\param,\vtanc)=\param-\metric(\param)\,\vtanc+O(\norm{\vtanc}^2)$
where $\metric$ is a
$\param$-dependent linear operator. This is covered by the following
assumptions, which basically state that $\paramupdate(\param,\vtanc)-\param$
is Lipschitz w.r.t. both $\param$ and $\vtanc$, and the Lipschitz
constant w.r.t.\ $\param$ is controlled by $\norm{\vtanc}$. 

\begin{hypothese}[Parameter update operators]
\label{hyp:updateop}
We assume that the parameter update operators $\paramupdate_t$ satisfy:
\begin{enumerate}
\item for any $\param \in \Param$ and $t\geq 1$,
$\paramupdate_t(\param,0)=\param$;
\item there exists $r_\Tangent>0$ and $\cpliprmup>0$ such that for any parameters
$\param,\param'\in \boctopt$, for any $t\geq 1$, for any
$\vtanc,\vtanc'\in \Tangent_t$ with $\norm{\vtanc}\leq r_\Tangent$ and
$\norm{\vtanc'}\leq r_\Tangent$, then
\begin{equation*}
\dist{\paramupdate_t\paren{\param,\vtanc}}{\paramupdate_t\paren{\param',\vtanc'}}
\leq \dist{\param}{\param'} +\cpliprmup
\left(\dist{\vtanc}{\vtanc'}+\left(\norm{\vtanc}+\norm{\vtanc'}\right)\dist{\param}{\param'}
\right).
\end{equation*}
%
\end{enumerate}
\end{hypothese}

For instance, the assumption is satisfied with $\param-\metric(\param)\,\vtanc+\norm{\vtanc}^2
f_2(\param,v)$ with locally Lipschitz second-order term $f_2$.
It also works with
$\paramupdate_t(\param,v)=
\arg\min_{\param'} \{v\cdot(\param'-\param)+\frac12
D_t(\param|\param')\}$ where $D_t$ is a suitable second-order penalty
between $\param$ and $\param'$, such as a KL
divergence (the trivial case being $D=\norm{\param-\param'}^2$).

%
%

\subsubsection{Local Optimality of $\paramopt$}
We now formulate the crucial assumption of local optimality for a parameter: it mimicks classical second-order optimality conditions, in a way adapted to our sequential setting.

\begin{hypothese}[Local optimality of $\paramopt$]
\label{hyp:optimprm}
We assume there exists a scale function $\lng$, negligible with respect to the identity function near infinity, and $\cdvpop > 0$ such that the following properties are satisfied.
\begin{enumerate}
\item First-order stability condition. For
$\cdvp\leq \cdvpop$,
\begin{equation*}
\dist{\paramupdate_{t:t+\flng{t}}\paren{\pctrlopt,\memopt_t,\pas_t}
}{\pctrlopt} = 
\po{\pas_t\,\flng{t}}
\end{equation*}
when $t\to\infty$, uniformly in $\cdvp\leq \cdvpop$.

\item Second-order or contractivity condition.
\label{hyp:optimprm2}
There exists $\mvpmin>0$ with the following property.
For any $\cdvp\leq \cdvpop$, for any parameter $\param\in \boctopt$
then
\begin{equation*}
\dist{\paramupdate_{t:t+\flng{t}}\paren{\param,\memopt_t,\pas_t}
}{\paramupdate_{t:t+\flng{t}}\paren{\paramopt,\memopt_t,\pas_t}}
\end{equation*}
is at most
\begin{equation*}
\paren{1-\mvpmin\,\pas_t\,\flng{t}}\,\dist{\prmctrl}{\paramopt} +
\po{\pas_t\,\flng{t}}
\end{equation*}
when $t\to\infty$.
Moreover, this $\po{}$ is uniform over $0\leq
\cdvp\leq \cdvpop$ and over $\param\in\boctopt$.
\end{enumerate}
\end{hypothese}

\subsubsection{Constraints on the Stepsize Sequence}

The various assumptions above constrain the admissible step size sequences we may use for the descent.
Remember (Def.~\ref{def:stepseqovlrstpsch}) that the step sizes $\pas_t$
are defined via
$\pas_t\deq \cdvp\,\sched_t$ with overall learning rate $\cdvp$ and
stepsize schedule $(\sched_t)$.

\begin{hypothese}[Stepsize sequence: behaviour of the stepsize schedule]
\label{hyp:spdes}
The stepsize sequence
$\sm{\pas}=\paren{\pas_t}$ and the scale functions satisfy:
\begin{enumerate}
\item $\pas_t$ is positive, and $\sum_{t=1}^{\infty}\pas_t=+\infty$.
\item
$\pas_t\,\flng{t}\,\fmdper{t}\,\mlipgrad{t}\to 0$
when $t\to\infty$.
\item $\mlipgrad{t}\ll \flng{t}$ when $t\to\infty$.
\item
When $t\to\infty$,
\begin{equation*}
\frac{\sup_{t< s\leq t+\flng{t}} \pas_s
}{
\inf_{t< s\leq t+\flng{t}} \pas_s
}
=1+o(1/\fmdper{t}).
\end{equation*}
\end{enumerate}
\end{hypothese}

\begin{remarque}
\label{rem:renormspas}
\label{rem:contsuitepasdes}
These assumptions are invariant by scaling all $\pas_s$ by the same
factor. So they are a property of the stepsize schedule and not of
$\cdvp$.
\end{remarque}

The sup/inf assumption is automatically satisfied if $\pas_t$ is the
inverse of a scale function and $\fmdper{t}\equiv 1$ (because \rehyp{optimprm} asks that we have $\flng{t}=\po{t}$, as $t\to\infty$, so that $t+\flng{t}\sim t$).

\subsubsection{Timescale Adapted to the Optimality Criterion}
To study the algorithm, we consider it in the following time-scale, which is adapted to the dynamics of the optimality assumptions.

\begin{definition}[Timescale associated with a scale function.]
\label{def:echetps}
Let $\flng{}$ be the scale function appearing in \rehyp{optimprm}. We define the
integer sequence 
$\paren{\tpsk{k}}_{k \geq 0}$
by induction via $\tpsk{0}=0$, $\tpsk{1}=1$ and, for $k\geq 1$, 
\begin{equation*}
\tpsk{k+1}=\tpsk{k}+\flng{\tpsk{k}}.
\end{equation*}
We denote the integer intervals defined by this timescale by
\begin{equation*}
\itv_k\deq(\tpsk{k};\tpsk{k+1}].
\end{equation*}
\end{definition}

\begin{lemme}[Asymptotic behavior of $\paren{\tpsk{k}}$]
\label{lem:pechtps}
$\paren{\tpsk{k}}$ is strictly increasing, and tends to infinity.
Moreover, 
$\tpsk{k+1} \sim \tpsk{k}$ when $k\to\infty$.
\end{lemme}

\begin{proof}
The first statements follow from $\flng{T}\geq 1$ for $T\geq 1$.
The last statement follows from $\flng{T}=o(T)$ when $T\to\infty$.
\end{proof}

\subsection{Noisy Updates}

We now introduce definitions to deal with sequences of states $(\mem_t)$
that do not follow the algorithm $\Algo_t$, but nevertheless stay close
to it in some sense. We refer to this as the \emph{noisy} case. This will be useful to deal with stochastic
approximations of RTRL such as NoBackTrack or UORO.
Thus, the states $\mem_t$ may be
random, but the parameter updates $\param_t$ are still computed normally from
$\mem_t$ via $\vtanc_t$.

\begin{definition}[Random trajectory]
\label{def:randomtraj}
A \emph{random trajectory} is, for every
stepsize sequence $(\pas_t)$, a probability distribution over
trajectories
$(\mem_t,\vtanc_t,\param_t)_{t\geq 0}$, such that
\begin{equation*}
\left\lbrace \ba 
\vtanc_t &=\sm{V}_{t}\paren{\param_{t-1},\,\mem_t} \\ 
\param_{t} &= \paramupdate_t\paren{\param_{t-1},\,\pas_t\,\vtanc_t}
\ea \right.
\end{equation*}
holds for all $t \geq 1$, with probability $1$.

We say that a random trajectory \emph{respects the stable tube} if,
for
any stepsize sequence, for
every $t\geq 1$, $\mem_{t-1}\in \bmaint{t-1}$
and $\param_{t-1}\in\boctopt$ imply
$\mem_t\in\bmaint{t}$ with probability $1$.
\end{definition}

The next notion captures the noise produced on $\param_t$ by a
sequence of states $(\mem_t)$ that does not follow the algorithm
$\Algo_t$. This definition compares the sequence $(\mem_t)$ to a sequence
$(\bar\mem_t)$ starting at the same state but that follows $\Algo_t$.
This extends to our setting the noise on $\jope_t$ in \algonoisy RTRL
algorithms compared to exact RTRL.

\begin{definition}[Deviation of a random trajectory from an algorithm]
\label{def:algobruite}
Let $(\mem_t)$ be any sequence of states. 
Let $\tinit\geq 0$ and let $\param_{\tinit}$ be any parameter. 
Consider the following two sequences $(\param_t)$ and $(\bar\param_t)$
defined from $(\mem_t)$ and $\param_{\tinit}$ as follows.

Define $(\param_t)$ as the
parameter trajectory defined from the sequence of states $(\mem_t)$
starting at $\param_\tinit$ at time $t_0$, namely,
\begin{equation*}
\vtanc_t =\sm{V}_{t}\paren{\param_{t-1},\,\mem_t},\qquad
\param_t=\paramupdate_t(\param_{t-1},\,\pas_t\vtanc_t)
\end{equation*}
for $t>\tinit$.

Let $(\bar \mem_t)$ be the algorithm trajectory starting at $\mem_\tinit$ with
parameters $(\param_t)$ (``regularized'' trajectory), and $\paren{\bar{\param}_t}$ be the sequence of parameters obtained with it. Namely, define
$\bar\mem_\tinit=\mem_\tinit$, $\bar\param_\tinit=\param_\tinit$ and
\begin{equation*}
\bar\mem_t=\Algo_t(\param_{t-1},\bar \mem_{t-1}),\qquad
\bar\vtanc_t =\sm{V}_{t}\paren{\param_{t-1},\,\bar\mem_t},\qquad
\bar\param_t=\paramupdate_t(\bar\param_{t-1},\,\pas_t\bar\vtanc_t).
\end{equation*}
for $t>\tinit$. (Beware $\bar\mem_t$ and $\bar\vtanc_t$ use
$\param_{t-1}$ while $\bar\param_t$ uses $\bar\param_{t-1}$.)

We call \emph{deviation of $(\mem_t)$ from $(\Algo_t)$ at
time $t_1$} the quantity
\begin{equation*}
\bruitif{\tinit}{t_1}{\param_\tinit}{(\mem_t)} \deq
\dist{\param_{t_1}}{\bar\param_{t_1}}.
\end{equation*}
\end{definition}

\begin{definition}[Negligible noise]
\label{def:negnoise}
Let $K\geq 0$, and let $\paren{\delta_k}$ be a nonnegative sequence which tends towards $0$.  

We say that a trajectory $(\param_t)$, $(\mem_t)$ has \emph{negligible
noise} in the timescale $\paren{\tpsk{k}}$, from time $K$ onwards, at speed $\paren{\delta_k}$ if, for all $k\geq K$, we have
\begin{equation*}
\param_{\tpsk{k}}\in\boctopt, \dist{\param_{\tpsk{k}}}{\paramopt}\leq\frac{\rayoptct}{3},\mem_{\tpsk{k}}\in \bmaint{\tpsk{k}}
\Rightarrow
\bruitifcdvp{\tpsk{k}}{\tpsk{k+1}}{\param_{\tpsk{k}}}{(\mem_t)}{\sm{\pas}}\leq \delta_k\,
\pas_{\tpsk{k}} \,\flng{\tpsk{k}}.
\end{equation*}



\end{definition}


\subsection{Convergence Theorems}

We now formulate the convergence theorem for our abstract model, in three
cases: non-noisy, noisy (with random trajectories), and open-loop on
intervals (which models truncated backpropagation through time).
For the
noisy case, this works under the assumption that the deviation
from the non-noisy case is negligible in the sense above, with high probability.

We first gather all assumptions on the system.

\newcommand{\mainass}{Assumption~\textbf{A}}

\begin{definition}
We call \mainass{} the following setting.

Let $\Param$ be some metric space, the \emph{parameter space}. Let
$(\Mem_t)_{t\geq 0}$ be a sequence of metric spaces, which represents the objects
maintained in memory by an algorithm at time $t$. Let
$(\Tangent_t)_{t\geq 1}$ be a sequence of normed vector spaces, containing the
gradients computed at time $t$.

Let $\paren{\mathcal{A}_t}$ be a family of transition operators (\redef{transop}) admitting a stable tube (\rehyp{stableballs}). Assume the state updates are Lipschitz with respect to the parameter (\rehyp{paramlip}), and the family forgets exponentially fast the initalisation when the parameter is fixed (\rehyp{expforgetinit}).

Let $\paren{\sm{V}_t}$ be gradient computation operators (\redef{gradcomputop}), which are locally bounded (\rehyp{boundedgrads}) and Lipschitz (\rehyp{contgrad}).

Let $\paren{\paramupdate_t}$ be parameter update operators (\redef{paramupdate}), which are Lipschitz with respect to the parameter and the tangent vectors (\rehyp{updateop}).  

Let $\sm{\pas}=\paren{\pas_t}_{t \geq 1}$ be a stepsize sequence
with overall learning rate $\cdvp$ (Def.~\ref{def:stepseqovlrstpsch}),
whose asymptotic behavior satisfies \rehyp{spdes}.

Let $\pctrlopt$ be a parameter satisfying the local optimality conditions of \rehyp{optimprm}.
\end{definition}

\begin{theoreme}[Convergence of the gradient descent algorithm]
\label{thm:cvalgopti}
Consider a system satisfying \mainass.

Then there exists $\cdvpmaxconv>0$ such that, for any overall learning
rate $\cdvp<\cdvpmaxconv$, the following convergence
holds. For any parameter $\param_0$ and maintained quantity $\mem_0$ satisfying 
\begin{equation*}
\cpl{\param_0}{\mem_0} \in \enstq{\prmctrl \in \epctrl}{\dist{\prmctrl}{\pctrlopt} \leq \frac{\rayoptct}{4}} \times \bmaint{0},
\end{equation*}
consider the gradient descent trajectory given by
\begin{equation*}
\left\lbrace \ba 
\mem_t&=\Algo_{t}\paren{\param_{t-1},\,\mem_{t-1}}\\
\vtanc_t &=\sm{V}_{t}\paren{\param_{t-1},\,\mem_t} \\ 
\param_{t} &= \paramupdate_t\paren{\param_{t-1},\,\pas_t\,\vtanc_t},
\ea \right.
\end{equation*}
for $t \geq 1$. 
Then $\param_t$ tends to $\pctrlopt$ as $t\to\infty$.
\end{theoreme}

\begin{theoreme}[Convergence of the gradient descent algorithm, noisy
case]
\label{thm:cvalgopti_noisy}
Consider a system satisfying \mainass.

Assume that we are given a
random trajectory $(\mem_t,\vtanc_t,\param_t)_{t\geq 0}$ which respects the stable tube $\paren{\bmaint{t}}_{t\geq 0}$, in the sense of \redef{randomtraj}
(namely, $\mem_t$ is random but lies in $\bmaint{t}$ provided $\mem_{t-1}\in
\bmaint{t-1}\text{ and
}\param_{t-1}\in\boctopt$, while $\vtanc_t$ and $\param_t$ are updated as
in the non-noisy case).

Assume that
$\mem_0\in\bmaint{0}$ and
$
\dist{\param_0}{\pctrlopt} \leq \frac{\rayoptct}{4}$.

Assume there exists $\eps>0$, $K\geq 0$, a non-negative
sequence $\paren{\delta_k}$ which tends to $0$, and $\cdvpnoise >0$
such that, for all $\olr\leq\cdvpnoise$, with probability greater than
$1-\eps$, the random trajectory $\cpl{\mem_t}{\param_t}$ has negligible
noise starting at $K$ at speed $\paren{\delta_k}$ (\redef{negnoise}).

Then
there exists $\cdvpmaxconv>0$ such that for any $\cdvp<\cdvpmaxconv$,
with probability at least $1-\eps$,
$\param_t$ tends to $\pctrlopt$ as $t\to\infty$.
\end{theoreme}

We now define the open-loop algorithm, which models truncated
backpropagation through time (in the ``non-overlapping'' variant, see
Section~\ref{sec:tbptt}): the parameter is updated only once at the
end of each time interval $(\tpsk{k},\,\tpsk{k+1}]$, by collecting all
gradients computed during that interval. In TBPTT using time intervals
$(\tpsk{k},\,\tpsk{k+1}]$, 
 whenever the parameter
is updated, gradients are not backpropagated through the
boundary, but reset to $0$: this corresponds to resetting the RTRL state derivative
$\jope$ to $0$. Moreover, the state may or may not be reset to some
default value at the start of each new TBPTT interval.
So here, at the end of every time interval $(\tpsk{k},\,\tpsk{k+1}]$, the running quantity $\mem_{\tpsk{k}}$ is
discarded and reset to some $\mem'_{\tpsk{k}}\in\bmaint{\tpsk{k}}$
(typically,
$\mem'_{\tpsk{k}}=0$, which belongs to $\bmaint{\tpsk{k}}$ as shown in
\recor{rtrlstabletubesJ}). Note that 
$\mem_{\tpsk{k}}$ is still used to compute the gradient update for the
last step of the
previous time interval, hence our use of a substitution $\mem_{\tpsk{k}}
\leftarrow \mem'_{\tpsk{k}}$ at the beginning of each new interval.

\begin{theoreme}[Convergence of the open-loop gradient descent algorithm]
\label{thm:cvalgoboucleouverte}
%
Consider a system satisfying \mainass.

Let $\paren{\tpsk{k}}$ be the time-scale $\paren{\tpsk{k}}$ of
\redef{echetps}, namely, $\tpsk{k+1}=\tpsk{k}+\flng{\tpsk{k}}$ where
$\flng{}$ is from \rehyp{optimprm}.

There exists $\cdvpmaxconv>0$ such that, for any overall learning
rate $\cdvp<\cdvpmaxconv$, the following convergence holds. 
For any parameter $\param_0$ and maintained quantity $\mem_0$ satisfying 
\begin{equation*}
\cpl{\param_0}{\mem_0} \in \enstq{\prmctrl \in \epctrl}{\dist{\prmctrl}{\pctrlopt} \leq \frac{\rayoptct}{4}} \times \bmaint{0},
\end{equation*}
for any sequence of reset states $(\mem'_{\tpsk{k}})_{k\geq 0}$ with
$\mem'_{\tpsk{k}}\in\bmaint{\tpsk{k}}$, 
consider the open-loop gradient descent trajectory which resets the state
$\mem$ to $\mem'$ at the start of every time interval
$(\tpsk{k},\,\tpsk{k+1}]$, and updates the parameter at the end of every
time interval; namely, for each $k\geq 0$, the computation performed in
the interval $(\tpsk{k},\,\tpsk{k+1}]$ is
\begin{equation*}
\mem_{\tpsk{k}} \leftarrow \mem'_{\tpsk{k}}\in\bmaint{\tpsk{k}}
\end{equation*}
and, for $\tpsk{k} < t \leq \tpsk{k+1}$, 
\begin{equation*}
\left\lbrace \ba 
\mem_t&=\Algo_{t}\paren{\param_{\tpsk{k}},\,\mem_{t-1}} \\
\vtanc_t &=\sm{V}_{t}\paren{\param_{\tpsk{k}},\,\mem_t} \\ 
\param_{t} &= \paramupdate_t\paren{\param_{t-1},\,\pas_{\tpsk{k+1}}\,\vtanc_t}.
\ea \right.
\end{equation*}
Then $\param_t$ tends to $\pctrlopt$ as $t\to\infty$.
\end{theoreme}

The open-loop algorithm models TBPTT when
$\paramupdate_t\cpl{\param}{\vtanc}=\param-\vtanc$. For more complicated
$\paramupdate$, we would also have the option to update the parameter
once with the sum of gradients via
$\param_{\tpsk{k+1}}=\paramupdate_{\tpsk{k+1}}(\param_{\tpsk{k}},\pas_{\tpsk{k+1}}\,\sum_{\tpsk{k}+1}^{\tpsk{k+1}}\,
\vtanc_t)$ instead of applying $\paramupdate$ at every step. In
general, the difference between these two options is of second-order, as shown in the course of the proofs.

\section{Proof of Convergence for the Abstract Algorithm}
\label{sec:procvabstalgo}

We now turn to the proof of Theorems~\ref{thm:cvalgopti}, \ref{thm:cvalgopti_noisy}
and~\ref{thm:cvalgoboucleouverte}. The notation and assumptions are as in
Section~\ref{sec:absontadsys}.

The proof proceeds in three main stages. First, we derive a priori
bounds on the trajectories. Then, we quantify the amount by which
trajectories diverge from each other as time goes on. This divergence
will be negated by the contractivity properties satisfied around the
local optimum. Finally, we are able to prove convergence.

\begin{remarque}
\label{rem:timeshift}
All statements in the first two subsections  can be made to start at an arbitrary time
$\tinit\geq 0$ rather than at time $0$, by applying them to the
operators $\Algo_{t+\tinit}$, $\sm{V}_{t+\tinit}$,
$\paramupdate_{t+\tinit}$, stepsizes $\pas_{t+\tinit}$, etc., which
satisfy the same assumptions as those using $\tinit=0$.
\end{remarque}

\begin{remarque}
\label{rem:O}
In all proofs in the text, when we write $O()$, the constants
implied in the $O()$ notation depend only on the constants
explicitly appearing in the
assumptions and on the constants implied in those $O()$ appearing in the
assumptions. In particular, the $O()$ notation is always uniform over other quantities of
interest such as $\param$, $\mem$, $\jope$, $\cdvp$, etc.
\end{remarque}

\subsection{\emph{A Priori} Bounds on Trajectories}

\subsubsection{Admissible Learning Rates}

First, we define the maximum overall learning rate we will consider in
the proofs. The bound depends on the magnitude of the gradientsi in the
assumptions. (This is not yet the maximum learning rate allowed for the
final convergence theorem, for which further constraints will be needed.)

\begin{definition}[Bound on the learning rate]
\label{def:bornepas}
We define
\begin{equation*}
	\brnp\deq 
	\min\paren{1,\frac{1}{\sup_{s\geq 1}\,\sched_s\,\fmdper{s}}} \min\left(1,\frac{r_\Tangent}
{
\sup_{t\geq 1}\,\sup_{\param\in \boctopt}\,\sup_{\mem\in \bmaint{t}}
	\norm{\sm{V}_t(\param,\mem)}\fmdper{t}^{-1}
}\right)
\end{equation*}
where $r_\Tangent$ is the value from \rehyp{updateop}.
\end{definition}

\begin{remarque}
	By the second point of \rehyp{spdes}, the sequence
	$\paren{\sched_s\,\fmdper{s}}$ is bounded, so that the supremum
	of the $\sched_s\,\fmdper{s}$'s is well-defined. By
	\rehyp{boundedgrads}, the supremum of
	$\norm{\sm{V}_t(\param,\mem)}\fmdper{t}^{-1}$ is finite.
Therefore, $\brnp>0$.
\end{remarque}

\begin{corollaire}
\label{cor:normpasv}
Let $t\geq 1$, and $\vtanc_t\in \bvtg{t}$. Then $\norm{\vtanc_t}\leq \fmdper{t}
\,r_\Tangent/\brnp$.

Moreover, for
any $0\leq \cdvp\leq \brnp$, for any $t\geq 1$, for any $\vtanc \in
\bvtg{t}$ we have $\pas_t\, \vtanc\in \bvtg{t}$ and $\norm{\pas_t \,\vtanc}\leq
r_\Tangent$.
\end{corollaire}

\begin{proof}
The first assertion is true by definition of $\brnp$: indeed $\brnp\leq
r_\Tangent/(\norm{\vtanc_t}\fmdper{t}^{-1})$ for any $\vtanc_t\in
\bvtg{t}$, by definition of $\bvtg{t}$.

If $\cdvp\leq \brnp$, then $\pas_t\leq 1$. Indeed, $\pas_t=\cdvp\sched_t$
and $\brnp\leq 1/(\sched_t \fmdper{t})\leq 1/\sched_t$ because
$\fmdper{t}\geq 1$ as a scale function. This proves that $\pas_t\,
\vtanc\in \bvtg{t}$ if $\vtanc \in
\bvtg{t}$.

For the last assertion, for any $t\geq 1$, we have
\begin{equation*}
\ba
\nrm{\pas_t\,\vtanc}&\leq \brnp\,\sched_t\,\nrm{\vtanc}\\
	&\leq \paren{\sup_{s\geq 1}\,\sched_s\,\fmdper{s}}^{-1} \sched_t\,\frac{r_\Tangent\,\nrm{\vtanc}}{\sup_{s\geq 1}\,\sup_{\param\in \boctopt}\,\sup_{\mem\in \bmaint{s}}\norm{\sm{V}_s(\param,\mem)}\fmdper{s}^{-1}}\\
	&\leq \paren{\sup_{s\geq 1}\,\sched_s\,\fmdper{s}}^{-1}\,\sched_t\,\frac{r_\Tangent\,\nrm{\vtanc}}{\sup_{\param\in \boctopt}\,\sup_{\mem\in \bmaint{t}}\norm{\sm{V}_t(\param,\mem)}\fmdper{t}^{-1}}\\
	&\leq \paren{\sup_{s\geq 1}\,\sched_s\,\fmdper{s}}^{-1}\,\sched_t\,r_\Tangent\,\fmdper{t}\\
	&\leq r_\Tangent.  \ea
\end{equation*}\end{proof}



\subsubsection{Short-Time Stability}

In this subsection, we do not assume that \rehyp{paramlip},
\rehyp{expforgetinit} or \rehyp{optimprm}
hold.

\begin{corollaire}[Stability of states]
\label{cor:stablestates}
Let $(\param_t)$ be a sequence of parameters in $\boctopt$, and let
$\mem_0\in \bmaint{0}$. Let $(\mem_t)$ be the trajectory associated with
$(\param_t)$ starting at $\mem_0$. Then
for all $t \geq 0$,
$\mem_t \in \bmaint{t}$.
\end{corollaire}

\begin{proof}
By induction from Assumption~\ref{hyp:stableballs}.
\end{proof}

\begin{corollaire}[Smoothness of parameter updates]
\label{cor:regupdate2}
\label{cor:majecartsmaj2}
There exists $\majmpc > 0$ such that,
for any $t\geq 1$, for any parameters $\param,\param'\in \boctopt$, for
any $\vtanc,\vtanc'\in \bvtg{t}$, for any $0 \leq \cdvp \leq \brnp$,
\begin{equation*}
\dist{\paramupdate_t\paren{\param,\pas_t\,\vtanc}}{\paramupdate_t\paren{\param',\pas_t\,\vtanc'}}
\leq \dist{\prmctrl}{\param'} + \majmpc\,\pas_t\,\fmdper{t}.
\end{equation*}
\end{corollaire}

\begin{proof}
By \recor{normpasv}, for $\vtanc\in \bvtg{t}$, we have $\norm{\pas_t
\vtanc} \leq \pas_t\, \fmdper{t} \,r_\Tangent/\brnp$, 
and likewise for $\vtanc'$.
Then the assertion
follows from \rehyp{updateop} by setting $\majmpc\deq
\cpliprmup(2r_\Tangent+4r_\Tangent \rayoptct 
)/\brnp$.
\end{proof}

\begin{definition}[Safe time horizon for staying in $\boctopt$]
\label{def:tvalpc}
Let $\sm{\pas}=\paren{\pas_s}$ be a stepsize sequence, and let
$\tinit\geq 0$. We define
\begin{equation*}
\thoriz{\tinit}\paren{\sm{\pas}}=\inf \enstq{t \geq \tinit+1}{
\som{s=\tinit+1}{t} \pas_s \,\fmdper{s}> \frac{\rayoptct}{3\majmpc}},
\end{equation*}
or $\thoriz{\tinit}=\infty$ if this set is empty.
\end{definition}

The next lemma shows that a parameter trajectory $\param_t =
\paramupdate_t\paren{\param_{t-1},\,\pas_t\,\vtanc_t}$ stays in the
stable tube for a time at least $\thorizo(\sm{\pas})$, provided it is
computed from states $\mem_t$ and parameters $\bar \param_t$ within the stable tube.

\begin{lemme}[Parameter trajectories stay in the stable tube for a
time $\thorizo(\sm{\pas})$.]
\label{lem:ecartdeuxtraj}
\label{lem:maintientpsfinibocont}
Let $\sm{\pas}=(\pas_t)_{t\geq 1}$ be a stepsize sequence with $\cdvp\leq \brnp$.

Let $(\param_t)$, $(\bar \param_t)$ be sequences of parameters, $(\mem_t)$
a sequence of states, and $(v_t)$ a sequence of gradients, such that 
$\param_0\in\boctopt$ with $\dist{\param_0}{\paramopt}\leq \rayoptct/3$,
$\bar\param_0\in \boctopt$,
$\mem_0\in \bmaint{0}$, and for any $t\geq 1$, 
\begin{equation*}
\begin{cases}
\mem_t = \Algo_t(\bar \param_{t-1},\mem_{t-1}) &\text{ or } \mem_t\in
\bmaint{t} \\
\vtanc_t =\sm{V}_{t}\paren{\bar \param_{t-1},\,\mem_t} &\text{ or } \vtanc_t \in
\bvtg{t}\\ 
\param_t = \paramupdate_t\paren{\param_{t-1},\,\pas_t\,\vtanc_t}\\
\bar \param_t = \param_s \text{ for some $s\leq t$,}&\text{ or } \bar\param_t \in \boctopt.
\end{cases}
\end{equation*}

Then for all $1 \leq t < \thorizo(\sm{\pas})$, the trajectory lies in the
stable tube: $\param_t, \bar\param_t\in \boctopt$,
$\vtanc_t\in
\bvtg{t}$, and $\mem_t\in \bmaint{t}$.
\end{lemme}

\begin{proof}
Let us prove by induction on $t$ that the conclusion holds and that moreover, for $t \geq 1$, we have
\begin{equation*}
\dist{\param_t}{\paramopt}\leq \frac{\rayoptct}{3}+\majmpc \,
\sum_{s=1}^{t} \pas_s\,\fmdper{s}.
\end{equation*}

This holds at time $t=0$ by assumption.

By 
\rehyp{stableballs},
$\mem_t=\Algo_{t}\paren{\bar\param_{t-1},\,\mem_{t-1}}$ belongs to $\bmaint{t}$ provided the
conclusion holds at time $t-1$.
By \rehyp{boundedgrads}, $\vtanc_t$ belongs to $\bvtg{t}$ provided the
conclusion holds at time $t-1$.

Then from \recor{regupdate2} applied to $(\param_{t-1},\pas_t\vtanc_t)$ and
$(\paramopt,0)$, we find
\begin{equation*}
\ba
\dist{\param_{t}}{\pctrlopt} &= \dist{\paramupdate_t\cpl{\param_{t-1}}{\pas_t\,\vtanc_t}}{\pctrlopt} \\
&\leq \dist{\param_{t-1}}{\pctrlopt} + \majmpc \, \pas_t\,\fmdper{t}
\\
&\leq \frac{\rayoptct}{3}+\majmpc \,
\sum_{s=1}^{t} \pas_s\,\fmdper{s}
\ea
\end{equation*}
by the induction hypothesis at time $t-1$. For $t< \thorizo$, by
definition of $\thorizo$, this is at most 
$\frac{2 \, \rayoptct}{3}$. So $\param_t$ belongs to $\boctopt$ and the
induction hypothesis holds at time $t$.
\end{proof}

\begin{lemme}[Parameter trajectories stay in the stable tube for a
time $\thorizo(\sm{\pas})$ for trajectories which respect the stable tube]
\label{lem:maintientpsfinibocontrandomtraj}
Let $\sm{\pas}=(\pas_t)_{t\geq 1}$ be a stepsize sequence with $\cdvp\leq \brnp$.

Let $(\mem_t,\,\vtanc_t,\,\param_t)_{t\geq 0}$ be
a random trajectory which respects the stable tube, in the sense of
\redef{randomtraj}. Assume that
$\param_0\in\boctopt$ with $\dist{\param_0}{\paramopt}\leq \rayoptct/3$
and $\mem_0\in \bmaint{0}$.
Then for all $1 \leq t < \thorizo(\sm{\pas})$, the trajectory lies in the
stable tube: with probability $1$, $\param_t \in \boctopt$,
$\vtanc_t\in
\bvtg{t}$, and $\mem_t\in \bmaint{t}$.
\end{lemme}
\begin{proof}
The proof is identical to that of \relem{maintientpsfinibocont}.
\end{proof}

Note that by taking the overall learning rate
$\cdvp$ small enough, we can ensure that $\thorizo(\sm{\pas})$ is arbitrarily
large. We shall need this later, in case the contractivity property of
gradient descent kicks in late, so we state the following.

\begin{lemme}[Small learning rates for arbitrary control horizon]
\label{lem:contempsfiniprmprodalgortrl}
Let $T > 0$. Then there exists $\cdvp^T > 0$ with the following property:

For any $0\leq \cdvp\leq \cdvp^T$,
for any $\param_0\in \boctopt$ such that $\dist{\param_0}{\paramopt}\leq
\frac{\rayoptct}{4}$ and any $\mem_0\in \bmaint{0}$,
consider a trajectory $(\param_t)$, $(\mem_t)$, $(\vtanc_t)$ such
that, for all $t\geq 1$,
\begin{equation*}
\left\lbrace \ba 
{\mem_t}&=\Algo_{t}\paren{\param_{s},\,\mem_{t-1}} \text{ for some } s\leq t-1, \qquad \text{ or }\mem_t\in \bmaint{t},\\
\vtanc_t &=\sm{V}_{t}\paren{\param_{s},\,\mem_t} \text{ for some } s\leq t-1, \qquad \text{ or }\vtanc_t\in B_{\Tangent_t},\\ 
\param_{t} &= \paramupdate_t\paren{\param_{t-1}, \pas_t\,\vtanc_t}.
\ea \right.
\end{equation*}

Then $\param_T\in\boctopt$, $\dist{\param_T}{\paramopt}\leq
\frac{\rayoptct}{3}$, and $\mem_T\in\bmaint{T}$.
\end{lemme}

\begin{proof}
Remember that $\pas_t=\cdvp\sched_t$ where $\sched_t$ is the stepsize schedule.
Define $\cdvp^T$ such that $\cdvp^T \sum_{t=1}^T \sched_t\,\fmdper{t}\leq
\frac{\rayoptct}{12\majmpc}$. Then proceed similarly to
\relem{ecartdeuxtraj}, as follows.

Let us prove by induction on $t$ that, for $t \geq 0$, we have
that $\mem_t\in\bmaint{t}$, that $\vtanc_t\in B_{\Tangent_t} $, that
$\param_s\in\boctopt$ for every $s\leq t-1$, and that, for $t\geq 1$,
\begin{equation*}
\dist{\param_t}{\paramopt}\leq \frac{\rayoptct}{4}+\majmpc \,
\sum_{s=1}^{t} \pas_s\,\fmdper{s}.
\end{equation*}

This holds at time $t=0$ by assumption.

By 
\rehyp{stableballs},
$\mem_t=\Algo_{t}\paren{\param_{s-1},\,\mem_{t-1}}$ for some $s\leq t-1$ belongs to $\bmaint{t}$ provided the
conclusion holds at time $t-1$.
By \rehyp{boundedgrads}, $\vtanc_t=\sm{V}_{t}\paren{\param_{s},\,\mem_t}$ for some $s\leq t-1$ belongs to $\bvtg{t}$ provided the
conclusion holds at time $t-1$.

Then from \recor{regupdate2} applied to $(\param_{t-1},\pas_t\vtanc_t)$ and
$(\paramopt,0)$ we find
\begin{equation*}
\ba
\dist{\param_{t}}{\pctrlopt} &= \dist{\paramupdate_t\cpl{\param_{t-1}}{\pas_t\,\vtanc_t}}{\pctrlopt} \\
&\leq \dist{\param_{t-1}}{\pctrlopt} + \majmpc \, \pas_t\,\fmdper{t}
\\
&\leq \frac{\rayoptct}{4}+\majmpc \,
\sum_{s=1}^{t} \pas_s\,\fmdper{s}
\ea
\end{equation*}
by the induction hypothesis at time $t-1$. For $t\leq T$, by
definition of $\cdvp^T$ this is at most 
$\frac{\rayoptct}{4}+\frac{\rayoptct}{12}=\frac{\rayoptct}{3}$. So $\param_t$ belongs to $\boctopt$ and the
induction hypothesis holds at time $t$.
\end{proof}

We now state a version of this lemma for random trajectories in the sense
of Definition~\ref{def:randomtraj}.

\begin{lemme}[Small learning rates for arbitrary control horizon, for random trajectories]
\label{lem:contempsfiniprmprodalgortrlrandom}
Let $T > 0$. Then there exists $\cdvp^T > 0$ with the following property. 

For any $0\leq \cdvp\leq \cdvp^T$,
for any $\param_0\in \boctopt$ such that $\dist{\param_0}{\paramopt}\leq
\frac{\rayoptct}{4}$ and any $\mem_0\in \bmaint{0}$,
for any random trajectory $(\param_t)$, $(\mem_t)$, $(\vtanc_t)$ which
starts at $(\param_0,\mem_0)$ and respects the stable tube
(Def.~\ref{def:randomtraj}),
 with probability $1$
it holds that $\param_T\in\boctopt$, $\dist{\param_T}{\paramopt}\leq
\frac{\rayoptct}{3}$, and $\mem_T\in\bmaint{T}$.
\end{lemme}

\begin{proof}
By Definition~\ref{def:randomtraj}, a random trajectory which respects
the stable tube satisfies, for all $t\geq 1$
\begin{equation*}
\left\lbrace \ba 
&{\mem_{t-1}}\in\bmaint{t-1}\text{ and } \param_{t-1}\in\boctopt\implies \mem_t\in\bmaint{t} \qquad \text{w.p.}\,1\\
&\vtanc_t =\sm{V}_{t}\paren{\param_{t-1},\,\mem_{t}}, \\
&\param_{t} = \paramupdate_t\paren{\param_{t-1}, \pas_t\,\vtanc_t}.
\ea \right.
\end{equation*}

Remember that $\pas_t=\cdvp\sched_t$ where $\sched_t$ is the stepsize schedule.
Define $\cdvp^T$ such that $\cdvp^T \sum_{t=1}^T \sched_t\,\fmdper{t}\leq
\frac{\rayoptct}{12\majmpc}$. Then proceed similarly to
\relem{ecartdeuxtraj}, as follows.

Let us prove by induction on $t$ that, with probability $1$, for $t \geq 0$, we have
that $\mem_t\in\bmaint{t}$, that $\vtanc_t\in B_{\Tangent_t} $, that
$\param_t\in\boctopt$ and that, for $t\geq 1$,
\begin{equation*}
\dist{\param_t}{\paramopt}\leq \frac{\rayoptct}{4}+\majmpc \,
\sum_{s=1}^{t} \pas_s\,\fmdper{s}.
\end{equation*}

This holds at time $t=0$ by assumption.

Since we assume that ${\mem_{t-1}}\in\bmaint{t-1}\text{ and }
\param_{t-1}\in\boctopt\implies \mem_t\in\bmaint{t}$, if the conclusion
holds at time $t-1$, then
$\mem_t$ belongs to $\bmaint{t}$ with probability one.
Then, by \rehyp{boundedgrads}, $\vtanc_t=\sm{V}_{t}\paren{\param_{t-1},\,\mem_{t}}$ belongs to $\bvtg{t}$ provided the
conclusion holds at time $t-1$.

Then from \recor{regupdate2} applied to $(\param_{t-1},\pas_t\vtanc_t)$ and
$(\paramopt,0)$ we find
\begin{equation*}
\ba
\dist{\param_{t}}{\pctrlopt} &= \dist{\paramupdate_t\cpl{\param_{t-1}}{\pas_t\,\vtanc_t}}{\pctrlopt} \\
&\leq \dist{\param_{t-1}}{\pctrlopt} + \majmpc \, \pas_t\,\fmdper{t}
\\
&\leq \frac{\rayoptct}{4}+\majmpc \,
\sum_{s=1}^{t} \pas_s\,\fmdper{s}
\ea
\end{equation*}
with probability $1$, by the induction hypothesis at time $t-1$. For $t\leq T$, by
definition of $\cdvp^T$ this is at most 
$\frac{\rayoptct}{4}+\frac{\rayoptct}{12}=\frac{\rayoptct}{3}$. So $\param_t$ belongs to $\boctopt$ with probability $1$ and the
induction hypothesis holds at time $t$.
\end{proof}


The proof of the next result, which controls the finite-time divergence between two sequences of parameters, is analogous to the control of the fixed point iterates in the proof of the Cauchy--Lipschitz theorem.

\begin{lemme}[Parameter updates at first order in $\pas$]
\label{lem:contordredeuxitprm2}
There exists a constant $\cst_6>0$ with the following property.

Let $\param_0,\param_0'\in \boctopt$ with $\dist{\param_0}{\paramopt}$
and $\dist{\param'_0}{\paramopt}$ at most $\rayoptct/3$, and let $\paren{\vtanc_t}$,
$\paren{\vtanc_t'}$ be two gradient sequences with $\vtanc_t,\vtanc'_t\in
\bvtg{t}$ for all $t\geq 1$. Let $\paren{\pas_t}$ be a stepsize sequence
with $\cdvp\leq \brnp$. Define by induction for $t\geq 1$,
\begin{equation*}
\param_{t} = \paramupdate_t\paren{\param_{t-1},\,\pas_t\,\vtanc_t}, \qquad
\param_{t}' = \paramupdate_t\paren{\param_{t-1}',\,\pas_t\,\vtanc_t'}.
\end{equation*}
Then for any $0 \leq t < \thorizo\paren{\sm{\pas}}$, 
\begin{equation*}
\dist{\param_t}{\param_t'} \leq 2
\dist{\param_0}{\param_0'} + \cst_6
\sum_{1\leq s \leq t} \pas_s \, \dist{\vtanc_s}{\vtanc_s'}
\end{equation*}
and
\begin{equation*}
\dist{\param_t}{\param_t'} \leq 2 \dist{\param_0}{\param_0'}
+\cst_6
\sum_{1\leq s \leq t} \pas_s\,\fmdper{s}.
\end{equation*}

In particular, taking $\vtanc'_t=0$ and $\param'_0=\param_0$, we have
\begin{equation*}
\dist{\param_t}{\param_0} \leq \cst_6  \sum_{1\leq s \leq t}
\pas_s\,\fmdper{s}.
\end{equation*}
\end{lemme}

\begin{proof}
First, by \relem{maintientpsfinibocont}, the trajectories stay in the
stable tube for $t<\thorizo$, and so the various bounds and assumptions
apply.

The second and third statements follow from the first up to increasing
$\cst_6$. Indeed, $\vtanc_s$
and $\vtanc'_s$ are bounded by $\fmdper{s}\,r_\Tangent/\brnp$ by
\recor{normpasv}. So we only have to prove the first statement.

By \recor{normpasv}, we have $
\norm{\pas_s \vtanc_s}\leq \pas_s \,\fmdper{s}\,
r_\Tangent/\brnp$  and likewise for $\vtanc'_s$.
Let us denote this bound by $\tilde\pas_s$, namely
\begin{equation*}
\tilde\pas_s\deq \pas_s \,\fmdper{s}\,r_\Tangent/\brnp.
\end{equation*}

By \rehyp{updateop},
for $t\geq 1$ we have
\begin{equation*}
\ba
\dist{\param_t}{\param'_t} & \leq
\dist{\param_{t-1}}{\param'_{t-1}} +\cpliprmup\left(
\dist{\pas_t \vtanc_t}{\pas_t \vtanc'_t}+2\tilde\pas_t\,
\dist{\param_{t-1}}{\param'_{t-1}}
\right)\\
&=\left(1+2\cpliprmup\,\tilde\pas_t\right)\dist{\param_{t-1}}{\param'_{t-1}}
+
\cpliprmup \,\pas_t\, 
\dist{\vtanc_t}{\vtanc'_t}.
\ea
\end{equation*}

Set $p_{s,t}\deq \prod_{j=s+1}^{t} (1+2\cpliprmup\,\tilde\pas_j)$.
By induction we obtain
\begin{equation*}
\dist{\param_t}{\param'_t}\leq
p_{0,t}\,\dist{\param_0}{\param'_0}+
\sum_{s=1}^t p_{s,t}\,\cpliprmup \,\pas_s \,\dist{\vtanc_s}{\vtanc'_s}
\end{equation*}
and the conclusion will follow if we prove that the various factors $p_{s,t}$ are
bounded.

Since $1+2\cpliprmup\,\tilde \pas_j\leq
\exp(2\cpliprmup\,\tilde \pas_j)$ we have
$p_{s,t}\leq \exp\left(
\sum_{j=1}^t 2\cpliprmup\,\tilde \pas_j 
\right)$. But by definition of $\thorizo$, for $t< \thorizo$ we have
\begin{equation}
\label{eq:bargpst}
\ba
\sum_{j=1}^t 2\cpliprmup\,\tilde \pas_j=\frac{2\cpliprmup\,r_\Tangent}{\brnp}\sum_{j=1}^t
\pas_j\,\fmdper{j}
\leq \frac{2\cpliprmup\,r_\Tangent\,\rayoptct}{3\brnp\,\majmpc}.
\ea
\end{equation}
The value of
$\majmpc$ from \recor{regupdate2} satisfies $\majmpc \geq 4 \cpliprmup
r_\Tangent \rayoptct/\brnp$, so the right-hand side of \reeq{bargpst} is bounded by $1/6$.
(This happens precisely
because we have taken $\thorizo$ small enough to
avoid exponential divergence of trajectories in time $t<\thorizo$.)

Therefore, for $t< \thorizo$ we have $p_{s,t}\leq \exp(1/6)\leq 2$.
This ends the proof.
\end{proof}

\subsubsection{Forgetting of Initial Conditions}

Here, we investigate the consequences of \rehyp{paramlip} and \rehyp{expforgetinit}.

\begin{corollaire}[Exponential forgetting of instantaneous gradients]
\label{cor:expforgetgrad}
Let $\param$ be a parameter in $\boctopt$, and let $\mem_0,\mem_0'\in
\bmaint{0}$. Let $(\mem_t)$ and $(\mem_t')$ be the trajectories
associated with $\param$ starting at $\mem_0$ and $\mem_0'$,
respectively. Then for all $t\geq 0$,
\begin{equation*}
\dist{\sm{V}_t\paren{\param,\,\mem_t}}{\sm{V}_{t}\paren{\param,\,\mem_t'}}
\leq \cst_1\cst_5 \,\mlipgrad{t}\,\paren{1-\alpha}^t \, \dist{\mem_0}{\mem_0'}.
\end{equation*}
\end{corollaire}

\begin{proof}
This is a direct consequence of Assumption~\ref{hyp:expforgetinit} and
Assumption~\ref{hyp:contgrad}.
\end{proof}

\begin{lemme}[Lipschitz continuity of trajectories]
\label{lem:newparamcont}
For any $\bar\param\in \boctopt$ and any
sequence of parameters $(\param_t)$ included in
$\boctopt$, for any initialization $\mem_0\in \bmaint{0}$, 
the trajectories $(\bar \mem_t)$ and $(\mem_t)$, starting at $\mem_0$ with
parameters $\bar \param$ and $(\param_t)$ respectively, satisfy
\begin{equation*}
\dist{\bar \mem_t}{\mem_t} \leq \frac{\cst_1
\clipalgoprm}{\alpha} \sup_{s \leq t-1} \, \dist{\bar \param}{\param_s},
\end{equation*}
for all $t\geq 0$.
\end{lemme}

\begin{proof}
Let us define a family of trajectories that interpolate between $(\bar \mem_t)$ and
$(\mem_t)$, by using parameters $(\param_t)$ for the first $t_c$ steps,
then parameter $\bar \param$. More precisely, given $t_c\geq 0$, define
\begin{equation*}
\mem^{t_c}_{t}=\begin{cases}
\Algo_t(\param_{t-1},\mem^{t_c}_{t-1}) &\text{if }t \leq t_c,
\\
\Algo_t(\bar \param,\mem^{t_c}_{t-1}) &\text{otherwise}
\end{cases}
\end{equation*}
so that $\bar \mem_t=\mem_t^0$ and $\mem_t=\mem_t^t$. These trajectories
lie in $\bmaint{t}$.

Now
\begin{equation*}
\dist{\bar \mem_t}{\mem_t}=\dist{\mem_t^0}{\mem_t^t}\leq \sum_{s=0}^{t-1}
\dist{\mem_t^s}{\mem_t^{s+1}}.
\end{equation*}

Up to time $t=s$, both $(\mem^s_t)$ and $(\mem^{s+1}_t)$ use parameter
$\param_t$, therefore $\mem^s_s=\mem^{s+1}_s=\mem_s$. But at time
$t=s+1$ they separate:
\begin{equation*}
\mem^s_{s+1}=\Algo_{s+1}(\bar \param,\mem^s_s)=\Algo_{s+1}(\bar \param,\mem_s)
\end{equation*}
while
\begin{equation*}
\mem^{s+1}_{s+1}=\Algo_{s+1}(\param_s,\mem^{s+1}_s)=\Algo_{s+1}(\param_s,\mem_s).
\end{equation*}
Consequently, by \rehyp{paramlip},
\begin{equation*}
\dist{\mem^s_{s+1}}{\mem^{s+1}_{s+1}}\leq \clipalgoprm\,
\dist{\param_s}{\bar \param}.
\end{equation*}

Now, from time  $s+2$ onwards, both $(\mem^s_t)$ and $(\mem^{s+1}_t)$ use parameter
$\bar \param$. Therefore for $t\geq s+2$,
\begin{equation*}
\dist{\mem^s_{t}}{\mem^{s+1}_t}\leq \cst_1 (1-\alpha)^{t-(s+1)}
\dist{\mem^s_{s+1}}{\mem^{s+1}_{s+1}}
\end{equation*}
thanks to \rehyp{expforgetinit}.

Summing, we find 
\begin{equation*}
\dist{\bar \mem_t}{\mem_t}\leq \cst_1\clipalgoprm \sum_{s=0}^{t-1}
(1-\alpha)^{t-(s+1)} \dist{\param_s}{\bar\param}\leq
\frac{\cst_1\clipalgoprm}{\alpha} \,\sup_{0\leq s\leq t-1} \dist
{\param_s}{\bar \param}.
\end{equation*}
\end{proof}

\begin{corollaire}[Continuity of instantaneous gradients]
\label{cor:newcontgrad}
Let $(\param_t)$ be a sequence of parameters
included in $\boctopt$, and let $\bar\param\in \boctopt$. Let $\mem_0\in \bmaint{0}$. Let $(\mem_t)$ and
$(\bar\mem_t)$ be the trajectories starting at $\mem_0$ with parameters
$(\param_t)$ and $\bar \param$, respectively. Then for any $t\geq 1$,
\begin{equation*}
\dist{\sm{V}_t\paren{\param_{t-1},\,\mem_t}}{\sm{V}_{t}\paren{\bar\param,\,\bar
\mem_t}} = \go{\mlipgrad{t} \sup_{0\leq s <t} \, \dist{\param_s}{\bar\param}}.
\end{equation*}
\end{corollaire}

\begin{proof}
This is a consequence of 
Assumption~\ref{hyp:contgrad} and \relem{newparamcont}.
\end{proof}

We now see the finite-time divergence between two trajectories initiated at the same parameter is controlled by the step-size sequence.

\begin{lemme}[Trajectories for a fixed parameter and different
initializations]
\label{lem:contecarttrajmprmqidiff}
Let $\mem_0,\mem'_0\in \bmaint{0}$ and let $\param_0\in\boctopt$ with
$
\dist{\param_0}{\pctrlopt} \leq \frac{\rayoptct}{3}
$.
Let $\sm{\pas}=(\pas_s)_{s\geq 1}$ be a stepsize sequence with $\cdvp\leq \brnp$.

Then for $t<\thorizo$,
\begin{equation*}
\dist{\paramupdate_{0:t}(\param_0,\mem_0,(\pas_s))}{\paramupdate_{0:t}(\param_0,\mem'_0,(\pas_s))}
= \go{\mlipgrad{t}\sup_{1 \leq s \leq t} \pas_s}.
\end{equation*}
\end{lemme}

\begin{proof}
By the definition of the open-loop updates $\paramupdate_{0:t}$, the
distance above is $\dist{\param_t}{\param'_t}$ where we define by
induction
\begin{equation*}
\left\lbrace \ba 
{\mem_{t}}&=\mathcal{A}_{t}\paren{\param_0,\,\mem_{t-1}} \\
\vtanc_t &=\sm{V}_{t}\paren{\param_0,\,\mem_t} \\ 
\param_{t} &= \paramupdate_t\paren{\param_{t-1},\,\pas_t\,\vtanc_t},
\ea \right.
\end{equation*}
and likewise with initialization $\mem'_0$
\begin{equation*}
\left\lbrace \ba 
{\mem'_{t}}&=\mathcal{A}_{t}\paren{\param_0,\,\mem'_{t-1}} \\
\vtanc'_t &=\sm{V}_{t}\paren{\param_0,\,\mem'_t} \\ 
\param'_{t} &=
\paramupdate_t\paren{\param'_{t-1},\,\pas_t\,\vtanc'_t},\qquad
\param'_0=\param_0.
\ea \right.
\end{equation*}

By \rehyp{boundedgrads} and \recor{stablestates}, for any $t\geq 0$,
$\mem_t$ and $\mem'_t$ belong to $\bmaint{t}$ and $\vtanc_t$ and
$\vtanc_t'$ to $\bvtg{t}$.

Thus, from \relem{contordredeuxitprm2} with $\param_0=\param'_0$, for $t<\thorizo$ we have
\begin{equation*}
\ba
\dist{\param_t}{\param_t'} 
&= \go{\sum_{1\leq s \leq t} \, \pas_s \, \dist{\vtanc_s}{\vtanc_s'}}.
\ea
\end{equation*}
But thanks to \recor{expforgetgrad}, for any $s\leq t$,
\begin{equation*}
\dist{\vtanc_s}{\vtanc_s'} = \go{\mlipgrad{s}\paren{1-\alpha}^s \,
\dist{\mem_0}{\mem'_0}}
\end{equation*}
and therefore, for $t<\thorizo$,
\begin{equation*}
\ba
\sum_{1\leq s \leq t} \, \pas_s \, \dist{\vtanc_s}{\vtanc_s'} 
&=
\go{
\sum_{1\leq s \leq t} \pas_s \,\mlipgrad{s}
\paren{1-\alpha}^s\dist{\mem_0}{\mem'_0}}
\\&=\go{
\frac{\dist{\mem_0}{\mem'_0}}{\alpha}\, \mlipgrad{t} \sup_{1\leq s \leq t} \pas_s}.
\ea
\end{equation*}
Since $\bmaint{0}$ has a finite diameter and $\alpha$ is fixed, the
conclusion follows.
\end{proof}



\subsection{Timescales and step sizes}
\label{sec:stutmsadpoptcrt}

Here we gather some properties that follow from \rehyp{spdes} on step
sizes and the various scale functions involved.

\subsubsection{Sums over intervals $(T;T+\flng{T}]$}

\newcommand{\itvT}{I}
\begin{corollaire}[Sums of stepsizes on an interval are negligible]
\label{cor:convzerosompasintervalles}
Let $T\geq 0$ and let $\itvT$ be the integer interval $\itvT=(T;T+\flng{T}]$.
Under \rehyp{spdes}:
\begin{enumerate}
\item $\smi{\itvT}
\pas_t\sim
\pas_{T}\,\flng{T}$ when $T\to\infty$, and this tends to
$0$.
\item $\smi{\itvT}
\pas_t\,\fmdper{t}\sim
\pas_{T}\,\flng{T}\,\fmdper{T}$ when $T\to\infty$,
and this tends to $0$.
\item $\smi{\itvT}
\pas_t\,\mlipgrad{t}\sim
\pas_{T}\,\flng{T}\,\mlipgrad{T}$ when $T\to\infty$,
and this tends to $0$.
\item $
\left(\smi{\itvT} \pas_t\,\fmdper{t}\right)
\left(\smi{\itvT}\pas_t \,\mlipgrad{t}\right)
=\po{\pas_{T}\,\flng{T}}$.
\item $\left(\smi{\itvT}\pas_t \,\fmdper{t}\right)\left(\smi{\itvT}
\pas_t\,\mlipgrad{t}\right)\sim
\pas_{T}^2\,\flng{T}^2\,\fmdper{T}\,\mlipgrad{T}$ when $T\to\infty$.
\item ${\smi{\itvT} \, \pas_t^2\,\fmdper{t}^2}\sim {\pas_T^2\,\fmdper{T}^2\,\flng{T}}$, when $T\to\infty$.
\item $\left(\sup_{\itvT}\mlipgrad{t}\right)
\left(\sup_{\itvT}\pas_t\right)=\po{\pas_{T}\,\flng{T}}$.
\item When $T\to\infty$,
\begin{equation*}
\frac{\sup_{T< s\leq T+\flng{T}} \pas_s}{\inf_{T < s\leq
T+\flng{T}} \pas_s}
=1+o(1/\fmdper{T})=1+o(1/\fmdper{T+\flng{T}}).
\end{equation*}
\item $T+\flng{T}\sim T$.
\end{enumerate}
\end{corollaire}

\begin{proof}
First, by assumption $\flng{T}\ll T$, so that $T+\flng{T}\sim T$.

By the sup/inf assumption in \rehyp{spdes}, we have
$\pas_t\sim \pas_{T}$ for $t\in\itvT$, so that $\sum_{\itvT}
\pas_t\sim \pas_{T}\,\flng{T}$.

Likewise, since $T+\flng{T} \sim T$ and since $\fmdper{}$ is a
scale function, we have $\fmdper{t}\sim \fmdper{T}$ for
$t\in\itvT$, so that $\smi{\itvT}
\pas_t\,\fmdper{t}\sim
\pas_{T}\,\flng{T}\,\fmdper{T}$. The argument is the
same with $\mlipgrad{t}$, and with ${\smi{\itvT} \, \pas_t^2\,\fmdper{t}^2}$.

These quantities all tend to $0$ by \rehyp{spdes}.

We have
$\left(\smi{\itvT}\pas_t \,\fmdper{t}\right)\left(\smi{\itvT}
\pas_t\,\mlipgrad{t}\right)\sim
\pas_{T}^2\,\flng{T}^2\,\fmdper{T}\,\mlipgrad{T}$
by the above. Since $\pas_t \flng{t}\fmdper{t}\mlipgrad{t}$ tends to $0$
by \rehyp{spdes}, this is $\po{\pas_{T} \flng{T}}$.

%
%

Since
$\mlipgrad{t}$ is a scale function, we have $\sup_{\itvT}
\mlipgrad{t}\sim \mlipgrad{T}$.
By the sup/inf assumption in \rehyp{spdes}, we have
$\mlipgrad{T}\sup_{\itvT}\pas_t\sim \mlipgrad{T}\,\pas_{T}$, which is
$\po{\pas_{T}\,\flng{T}}$ by \rehyp{spdes}.

The $\sup/\inf$ property follows directly from \rehyp{spdes} and from
$\fmdper{t}\sim \fmdper{T}$ for $t\in\itvT$.
\end{proof}

Remember that the sequence $\tpsk{k}$ is defined by
$\tpsk{k+1}=\tpsk{k}+\flng{\tpsk{k}}$ (\redef{echetps}).

\begin{remarque}
\label{rem:etaLdiv}
Since the $\pas_t$'s are nonnegative, and their series diverges according to \rehyp{spdes}, the first point of \recor{convzerosompasintervalles} implies that the series $\pas_{\tpsk{k}}\,\flng{\tpsk{k}}$ diverges as well.
\end{remarque}

\begin{corollaire}[Smallest safe interval $k$]
\label{cor:contmajitvk}
There exists an integer $k_0\geq 1$ such that, for any
$\cdvp\leq \brnp$,
for any $k\geq k_0$, one has
$\smi{(\tpsk{k};\tpsk{k+1}]} \, \pas_t \,\fmdper{t} \leq
\frac{\rayoptct}{3\majmpc}$ and $ \mvpmin \, \smi{(\tpsk{k};\tpsk{k+1}]} \, \pas_t \leq 1 $.
($\majmpc$ is defined in 
Cor.~\ref{cor:majecartsmaj2}.)
\end{corollaire}

\begin{proof}
Using
\recor{convzerosompasintervalles}, take $k_0$ such that this holds
for $\cdvp=\brnp$.
Then the same will hold for smaller $\cdvp$.
\end{proof}

The next lemma justifies the construction of the timescale $\tpsk{k}$.

\begin{lemme}
\label{lem:minmthoriztpsk}
For any $\cdvp\leq \brnp$,
for any $k \geq k_0$,
the control time $\thoriz{\tpsk{k}}(\sm{\pas})$ is (strictly) larger
than $\tpsk{k+1}$.
\end{lemme}

\begin{proof}
This follows from the \redef{tvalpc} of $\thoriz{t}$, and from
\recor{contmajitvk}.
\end{proof}

We now prove a slight technical strengthening of the sup/inf property on $\pas_t$,
involving intervals $[T;T+\flng{T}]$ instead of $(T;T+\flng{T}]$.

\begin{lemme}
\label{lem:stronghomo}
When $T\to\infty$,
\begin{equation*}
\frac{\sup_{T\leq t\leq T+\flng{T}} \pas_t
}{
\inf_{T\leq t\leq T+\flng{T}} \pas_t
}
=1+o(1/\fmdper{T})=1+o(1/\fmdper{T+\flng{T}})
\end{equation*}
and moreover for $T< t\leq T+\flng{T}$ we have
\begin{equation*}
\frac{\pas_T}{\pas_t}=1+o(1/\fmdper{T}).
\end{equation*}
\end{lemme}

\begin{proof}
The last statement follows from the first by specializing to $\pas_T$ in
the supremum.

For the first statement, write
\begin{equation*}
\sup_{T\leq s \leq
T+\flng{T}}\,\pas_s
=\sup\left(
\pas_T,\,
\sup_{T< s \leq
T+\flng{T}}\,\pas_s
\right)
\end{equation*}
and likewise for the infimum. By \rehyp{spdes} applied to time $t=T-1$,
one has $\pas_T\leq \left(1+o\left(\frac{1}{\fmdper{T-1}}\right)\right)\pas_{T+1}$ so that
\begin{equation*}
\sup_{T\leq s \leq
T+\flng{T}}\,\pas_s \leq
\left(1+o\left(\frac{1}{\fmdper{T-1}}\right)\right)\sup_{T< s \leq
T+\flng{T}}\,\pas_s
\end{equation*}
and likewise for the infimum. Thus,
\begin{equation*}
\frac{ \sup_{T\leq s \leq T+\flng{T}}\,\pas_s } {\inf_{T\leq s \leq T+\flng{T}}\,\pas_s }
\leq \left(1+o\left(\frac{1}{\fmdper{T-1}}\right)\right)^2\,
\frac{ \sup_{T < s \leq T+\flng{T}}\,\pas_s } {\inf_{T< s \leq T+\flng{T}}\,\pas_s }
\end{equation*}
and we can now apply \rehyp{spdes} to the rightmost term, yielding
\begin{equation*}
\frac{ \sup_{T\leq s \leq T+\flng{T}}\,\pas_s } {\inf_{T\leq s \leq
T+\flng{T}}\,\pas_s }
\leq \left(1+o\left(\frac{1}{\fmdper{T-1}}\right)\right)^2\,
\left(1+o\left(\frac{1}{\fmdper{T}}\right)\right).
\end{equation*}
Now, since $\fmdper{}$ is a scale function (or $1$), we have $\fmdper{T-1}\sim
\fmdper{T}$ when $T\to \infty$, so the above is $(1+o(1/\fmdper{T}))^3$
which is just $1+o(1/\fmdper{T})$.

Finally, as seen above, $T+\flng{T}\sim T$ so that
$\fmdper{T+\flng{T}}\sim \fmdper{T}$ as $\fmdper{}$ is a scale function.
\end{proof}

\subsubsection{Constant Stepsizes vs a Sequence of Stepsizes}

We now bound the difference between updates
$\paramupdate_{\tpsk{k}:\tpsk{k+1}}{\paren{\param,\paren{\pas_t\,\vtanc_t}_{\tpsk{k}<t
\leq \tpsk{k+1}}}}$ using a variable learning rate $\pas_t$, and using the
constant learning rate $\pas_{\tpsk{k}}$. This is a consequence of the
homogeneity of learning rates on intervals $(T;T+\flng{T}]$.

\begin{lemme}[Variable vs constant stepsizes]
\label{lem:changerates}
Let $\sm{\pas}$ be a sequence of stepsizes with $\cdvp\leq \brnp/2$.
Let $\lng$ be a scale function such that for $T$ large enough,
$T+\flng{T}< 
\thoriz{T}\paren{\sm{\pas}}$.

Let $(\vtanc_t)$ be a sequence of gradients with $\vtanc_t\in
\bvtg{t}$ for all $t$. Let $(\param_t)$ be a sequence of parameters with
$\dist{\param_{t}}{\paramopt}\leq \frac{\rayoptct}{3}$.
Then
\begin{equation*}
\dist{
\paramupdate_{T:T+\flng{T}}{\paren{\param_{T},\paren{ \pas_t\,\vtanc_t}}}
}{
\paramupdate_{T:T+\flng{T}}{\paren{\param_{T},\pas_{T}\paren{ \vtanc_t}}}
} = \po{\sum_{T<t
\leq T+\flng{T}} \pas_t}
\end{equation*}
when $T\to\infty$.


In particular, letting $(\vtanc_t)$ be the sequence of
gradients computed along a trajectory $(\mem_T)$ with
$\mem_{T}\in\bmaint{T}$, we find
\begin{equation*}
\dist{
\paramupdate_{T:T+\flng{T}}{\paren{\param_{T},\,\mem_{T},\paren{ \pas_t}}}
}{
\paramupdate_{T:T+\flng{T}}{\paren{\param_{T},\,\mem_{T},\,\pas_{T}}}
} = \po{\sum_{T<t
\leq T+\flng{T}} \pas_t}.
\end{equation*}
\end{lemme}

\begin{proof}
Let $T<t
\leq T+\flng{T}$. Define $\vtanc'_t$ such that
\begin{equation*}
\pas_{t}\, \vtanc'_t=\pas_{T}\, \vtanc_t,
\end{equation*}
so that
\begin{equation*}
\paramupdate_{T:T+\flng{T}}{\paren{\param_{T},\paren{
\pas_{T}\,\vtanc_t}_{T<t\leq T+\flng{T}}}}
=
\paramupdate_{T:T+\flng{T}}{\paren{\param_{T},\paren{
\pas_t\,\vtanc'_t}_{T<t\leq T+\flng{T}}}}.
\end{equation*}

By \relem{stronghomo}, we have
$\pas_{T}/\pas_{t}=1+o(1)$.
Therefore, for $t$ large enough, we have
$\norm{\vtanc'_t}\leq 2\norm{\vtanc_t}$ so that if $\cdvp<\brnp/2$, then $\pas_t \,\vtanc'_t$ lies in the
control ball $\bvtg{t}$ thanks to \recor{normpasv}.

By \relem{contordredeuxitprm2}
the distance we want to bound is at most
\begin{equation*}
\cst_6 \sum_{T<t\leq T+\flng{T}} \pas_s\,
\dist{\vtanc_s}{\vtanc'_s}
\end{equation*}
but then
$\dist{\vtanc_s}{\vtanc'_s}=\dist{\vtanc_s}{\frac{\pas_{T}}{\pas_s}\vtanc_s}=\norm{\vtanc_s}\left|\frac{\pas_{T}}{\pas_s}-1
\right|=o(1)$ since $\vtanc_s=O(\fmdper{s})$ and
$\frac{\pas_{T}}{\pas_s}=1+o(1/\fmdper{T+\flng{T}})$ with
$\fmdper{T+\flng{T}}\geq \fmdper{s}$.
\end{proof}

\subsection{Finite-Time Divergence Between Trajectories}

In this section we consider increasingly easier-to-analyze trajectories. We
start with some parameters $\param_t$ computed along a ``noisy''
trajectory where the states $(\mem_t)$ are not necessarily given by applying the
algorithm $\Algo_t$. We then consider the ``regularized'' trajectory
$\bar\mem_{t}=\Algo_{t}(\param_{t-1},\bar \mem_{t-1})$ defined by $\Algo_t$, but still using the
parameters from the noisy trajectory, and the parameter updates
$\bar\param_t$ computed from $\bar\mem_t$. These differ by the
deviation $\bruit{t}{\param_0}{(\mem_t)}$ from \redef{algobruite}.

Next we consider the ``open-loop''
trajectory $\mem'_{t}=\Algo_{t}(\param_0,\mem'_{t-1})$ and the resulting
parameter updates $\param'_t$ computed from $\mem'_t$.  This open-loop
trajectory can be compared to the trajectory with optimal parameter
$\paramopt$.



\subsubsection{Divergence Between Open-Loop and Closed-Loop Trajectories}

\begin{lemme}[Noisy closed-loop vs open-loop divergence]
\label{lem:noisyopenloopdiv}
Let $\param_0\in\boctopt$ with $\dist{\param_0}{\paramopt}\leq
\rayoptct/3$. Let $(\mem_t)$ be any sequence of states such that
$\mem_t\in\bmaint{t}$.
Let $\sm{\pas}=(\pas_t)_{t\geq 1}$ be a stepsize sequence with $\cdvp\leq \brnp$.


Define the ``closed-loop'' trajectory by induction for $t\geq 1$
\begin{equation*}
\left\lbrace \ba 
\vtanc_t &=\sm{V}_{t}\paren{\param_{t-1},\,\mem_t} \\ 
\param_{t} &= \paramupdate_t\paren{\param_{t-1},\,\pas_t\,\vtanc_t},
\ea \right.
\end{equation*}
and let $\param'_t\deq \paramupdate_{0:t}(\param_0,\mem_0,(\pas_t))$ be
the corresponding open-loop value with parameter $\param_0$, namely, $\mem'_0=\mem_0$,
$\param'_0=\param_0$, and for $t\geq 1$,
\begin{equation*}
\left\lbrace \ba 
{\mem_{t}'}&=\mathcal{A}_{t}\paren{\param_0,\,\mem_{t-1}'} \\
\vtanc_t' &=\sm{V}_{t}\paren{\param_0,\,\mem_t'} \\ 
\param_{t}' &= \paramupdate_t\paren{\param_{t-1}',\,\pas_t\,\vtanc_t'}.
\ea \right.
\end{equation*}

Then for all $0 \leq t < \thorizo(\sm{\pas})$,
\begin{equation*}
\dist{\param_t}{\param_t'} = \go{
\paren{ \sum_{1\leq s \leq t} \pas_s\,\mlipgrad{s}}
\paren{ \sum_{1\leq s \leq t} \pas_s\,\fmdper{s}}
} +
\bruit{t}{\param_0}{(\mem_s)}.
\end{equation*}

In particular, if $(\mem_t)$ itself follows the trajectory
$\mem_{t}=\Algo_t(\param_{t-1},\mem_{t-1})$, we find
\begin{equation*}
\dist{\param_t}{\param_t'} = \go{
\paren{ \sum_{1\leq s \leq t} \pas_s\,\fmdper{s}}
\paren{ \sum_{1\leq s \leq t} \pas_s\,\mlipgrad{s}}
},
\end{equation*}
as $\bruit{t}{\param_0}{(\mem_s)}$ is $0$ by definition.
\end{lemme}

\begin{proof} We first consider the ``regularized'' trajectory lying between
the other two.
Define the following trajectory by
induction initialized with $\bar\mem_0=\mem_0$, $\bar\param_0=\param_0$, and
\begin{equation*}
\left\lbrace \ba 
{\bar\mem_{t}}&=\mathcal{A}_{t}\paren{\param_{t-1},\,\bar\mem_{t-1}} \\
\bar\vtanc_t &=\sm{V}_{t}\paren{\param_{t-1},\,\bar\mem_t} \\ 
\bar\param_{t} &= \paramupdate_t\paren{\bar\param_{t-1},\,\pas_t\,\bar\vtanc_t}.
\ea \right.
\end{equation*}
By \redef{algobruite}, for any $t\geq 0$,
\begin{equation*}
\dist{\param_t}{\bar\param_t}=\bruit{t}{\param_0}{(\mem_s)}.
\end{equation*}

Note that for all three trajectories, for $t< \thorizo$, by \relem{maintientpsfinibocont}, all objects
at time $t$
belong respectively to $\boctopt$, $\bvtg{t}$, and 
$\bmaint{t}$.

We now study the divergence $\dist{\bar\param_t}{\param'_t}$ between the regularized trajectory and the
open-loop trajectory.

Since $\bar\param_0=\param_0'$, from \relem{contordredeuxitprm2},
we have for $t<\thorizo$
\begin{equation*}
\dist{\bar \param_t}{\param_t'} \leq \go{\sum_{1\leq s \leq t} \, \pas_s \,
\dist{{\bar \vtanc_s}}{{\vtanc_s'}}}.
\end{equation*}
Now
 $\bar \vtanc_s$ is computed from the
trajectory with parameters $(\param_s)$ and $\vtanc'_s$ with constant
parameter $\param_0$, so by \recor{newcontgrad}, we have
\begin{equation*}
\dist{\bar\vtanc_s}{\vtanc_s'} = \go{\mlipgrad{s}\sup_{p < s} \,
\dist{\param_p}{\param_0}}.
\end{equation*}

But for $0 \leq p < \thorizo$, by \relem{contordredeuxitprm2} we
have
\begin{equation*}
\dist{\param_p}{\param_0} = \go{\sum_{p' \leq p}  \pas_{p'}\,\fmdper{p'}},
\end{equation*}
and therefore, for $s \geq 2$, we have
\begin{equation*}
\dist{\bar \vtanc_s}{\vtanc_s'} = \go{\mlipgrad{s}\sum_{p \leq s-1}
\pas_p\,\fmdper{p}}
=\go{\mlipgrad{s}\,\sum_{1 \leq p \leq t}\,\pas_p\,\fmdper{p}}.
\end{equation*}
The bound still holds for $s=1$, since $\dist{\bar \vtanc_1}{\vtanc_1'}=0$, as they are both computed from $\param_0$ and $\mem_0$.
Therefore, for $t<\thorizo$,
\begin{equation*}
\ba
\dist{\bar \param_t}{\param_t'} 
&=
\go{\sum_{1\leq s \leq t}
\pas_s \,\mlipgrad{s}\paren{\sum_{1\leq p \leq t} \pas_p\,\fmdper{p}}}
\\&=
\go{\paren{\sum_{1\leq s \leq t}
\pas_s\,\mlipgrad{s}}\paren{\sum_{1\leq s \leq t} \pas_s\,\fmdper{s}}},
\ea
\end{equation*}
from which the conclusion follows.
\end{proof}

\subsubsection{Deviation from the Optimal Parameter in Finite Time}

\begin{lemme}[Deviation from the optimal parameter in finite time]
\label{lem:propcentalgoptimtinitzero}
\label{lem:propcentalgoptim}
Let $\param_0\in\boctopt$ with $\dist{\param_0}{\paramopt}\leq
\rayoptct/3$, and let $\mem_0\in\bmaint{0}$.
Let $\sm{\pas}=(\pas_t)_{t\geq 1}$ be a stepsize sequence with $\cdvp\leq \brnp$.

Consider a trajectory such that for $t\geq 1$
\begin{equation*}
\left\lbrace \ba 
\mem_t&=\Algo_t(\param_{t-1},\mem_{t-1}) \qquad\text{or }\, \mem_t\in
\bmaint{t},\\
\vtanc_t &=\sm{V}_{t}\paren{\param_{t-1},\,\mem_t}, \\ 
\param_{t} &= \paramupdate_t\paren{\param_{t-1},\,\pas_t\,\vtanc_t}.
\ea \right.
\end{equation*}

Then for any $0 \leq t < \thorizo$,
\begin{equation*}
\ba
\dist{\param_t}{\pctrlopt} &\leq
\dist{\paramupdate_{0:t}(\param_0,\memopt_0,(\pas_s))}{
\paramupdate_{0:t}\paren{\pctrlopt,\memopt_0,(\pas_s)}
}
+ \dist{\paramupdate_{0:t}\paren{\pctrlopt,\memopt_0,(\pas_s)}}{\pctrlopt} \\
&+\bruit{t}{\param_0}{(\mem_s)} 
+ \go{\mlipgrad{t}\sup_{1\leq s \leq t} \pas_s}
\\&
+\go{
\paren{ \sum_{1\leq s \leq t} \pas_s\,\fmdper{s}}
\paren{ \sum_{1\leq s \leq t} \pas_s\,\mlipgrad{s}}
}.
\ea
\end{equation*}
\end{lemme}

\begin{proof}
%
Let us consider the open-loop trajectory initialized with $\param_0$
and $\mem_0$, namely
\begin{equation*}
\param'_t\deq \paramupdate_{0:t}(\param_0,\mem_0,(\pas_s)).
\end{equation*}

By \relem{contecarttrajmprmqidiff}, for $0\leq t<\thorizo$, we have
\begin{equation*}
\dist{\param'_t}{\paramupdate_{0:t}(\param_0,\memopt_0,(\pas_s))} =
\go{\mlipgrad{t}\sup_{1\leq s \leq t}
\pas_s}.
\end{equation*}


On the other hand, by \relem{noisyopenloopdiv}, for any $t<\thorizo$, we have
\begin{equation*}
\dist{\param_t}{\param_t'} = \go{
\paren{ \sum_{1\leq s \leq t} \pas_s\,\fmdper{s}}
\paren{ \sum_{1\leq s \leq t} \pas_s\,\mlipgrad{s}}
}+
\bruit{t}{\param_0}{(\mem_s)},
\end{equation*}
and the conclusion follows by the triangle inequality.
\end{proof}


\subsection{Convergence of Learning}

\subsubsection{Behavior Around the Local Minimum $\paramopt$}

\begin{lemme}[At first order, $\paramopt$ is not updated in $\itv_k$]
\label{lem:extloc} 
Assume that $\cdvp\leq \min(\cdvpop,\brnp/2)$. Then when $k\to\infty$,
\begin{equation*}
\dist{\paramupdate_{\tpsk{k}:\tpsk{k+1}}\paren{\pctrlopt,\,\memopt_{\tpsk{k}},\,\paren{\pas_t}}}{\pctrlopt} =
\po{\pas_{\tpsk{k}} \,\flng{\tpsk{k}}}.
\end{equation*}
\end{lemme}

\begin{proof}
By \rehyp{optimprm} and by the
\redef{echetps} of $\tpsk{k}$, this holds when using a constant learning rate
$\pas_{\tpsk{k}}$ instead of $\pas_t$ between $\tpsk{k}$ and
$\tpsk{k+1}$; namely, we have
\begin{equation*}
\dist{\paramupdate_{\tpsk{k}:\tpsk{k+1}}\paren{\pctrlopt,\,\memopt_{\tpsk{k}},\,\pas_{\tpsk{k}}}}{\pctrlopt} =
\po{\pas_{\tpsk{k}}\,\flng{\tpsk{k}}}.
\end{equation*}

\relem{changerates} can transfer this to non-constant step sizes $\pas_s$
instead of $\pas_{\tpsk{k}}$.
Let us check that all the assumptions of \relem{changerates} are
satisfied. Remember that $\tpsk{k+1}=\tpsk{k}+\flng{\tpsk{k}}$.
The condition
$\thoriz{\tpsk{k}}>\tpsk{k+1}$ is satisfied 
for $k\geq k_0$ by \relem{minmthoriztpsk}.
The condition on step sizes is satisfied by the last point of
\recor{convzerosompasintervalles}. 
Therefore,
for $k\geq k_0$ we can apply \relem{changerates} to $T=\tpsk{k}$. This provides the
conclusion, after observing that $\sum_{\itv_k}\pas_t\sim
\pas_{\tpsk{k}}\,\flng{\tpsk{k}}$
by \recor{convzerosompasintervalles}.
What happens for $k<k_0$ is absorbed in the $\po{}$ notation.
\end{proof}

\begin{lemme}[Contractivity of open-loop updates on each interval]
\label{lem:contractitvk}
Assume that $\cdvp\leq \min(\cdvpop,\brnp/2)$. Then for $k\geq k_0$,
for any $\param\in \boctopt$ with $\dist{\param}{\paramopt}\leq
\frac{\rayoptct}{3}$,
\begin{equation*}
\dist{\paramupdate_{\tpsk{k}:\tpsk{k+1}}\paren{\param,\,\memopt_{\tpsk{k}},\paren{\pas_t}}}
{\paramupdate_{\tpsk{k}:\tpsk{k+1}}\paren{\paramopt,\,\memopt_{\tpsk{k}},\paren{\pas_t}}}
\end{equation*}
is at most
\begin{equation*}
\paren{1-\mvpmin\,\pas_{\tpsk{k}}\,\flng{\tpsk{k}}}\,\dist{\prmctrl}{\pctrlopt}
+ \po{\pas_{\tpsk{k}}\,\flng{\tpsk{k}}}.
\end{equation*}
\end{lemme}

\begin{proof}
\rehyp{optimprm} applied to the intervals $\tpsk{k}<t\leq
\tpsk{k}+\flng{\tpsk{k}}=\tpsk{k+1}$ 
provides the same statement but using constant
step size $\pas_{\tpsk{k}}$ instead of variable step size
$(\pas_t)$.

As in \relem{extloc}, we can use \relem{changerates} to bound the
distance between constant and variable step sizes.
This yields
\begin{equation*}
\dist{\paramupdate_{\tpsk{k}:\tpsk{k+1}}\paren{\param,\,\memopt_{\tpsk{k}},\paren{\pas_t}}}
{\paramupdate_{\tpsk{k}:\tpsk{k+1}}\paren{\param,\,\memopt_{\tpsk{k}},\pas_{\tpsk{k}}}}
=\po{\sum_{\itv_k}\pas_t}=\po{\pas_{\tpsk{k}}\,\flng{\tpsk{k}}}
\end{equation*}
and likewise for $\paramopt$. The conclusion follows by the triangle
inequality.
\end{proof}

\subsubsection{Contraction of Errors from $\tpsk{k}$ to $\tpsk{k+1}$}

\begin{lemme}[Contraction of errors from $\tpsk{k}$ to $\tpsk{k+1}$]
\label{lem:propcentechtpstpsk}
Let $\cdvp \leq \min(\brnp/2,\cdvpop)$.
Let $k\geq k_0$ where $k_0$ is defined in \recor{contmajitvk}.

Let $\param_\tpsk{k}$ be such that
${\dist{\param_\tpsk{k}}{\pctrlopt}\leq \frac{\rayoptct}{3}}$, and let
$\mem_\tpsk{k}\in \bmaint{\tpsk{k}}$.
Consider the learning trajectory from initial parameter
$\param_\tpsk{k}$ and initial state $\mem_\tpsk{k}$ and learning rates
$(\pas_t)$, namely,
\begin{equation*}
\left\lbrace \ba 
{\mem_{t}}&=\mathcal{A}_{t}\paren{\param_{t-1},\,\mem_{t-1}} \qquad
\text{or }\,\mem_t\in\bmaint{t}\\
\vtanc_t &=\sm{V}_{t}\paren{\param_{t-1},\,\mem_t} \\ 
\param_{t} &= \paramupdate_t\paren{\param_{t-1},\,\pas_t\,\vtanc_t}.
\ea \right.
\end{equation*}

Then for all $\tpsk{k} \leq t \leq \tpsk{k+1}$, we have
$\param_t\in\boctopt$ and $\mem_t\in \bmaint{t}$, and moreover,
\begin{equation*}
\ba
\dist{\param_{\tpsk{k+1}}}{\pctrlopt} &\leq \paren{1 - \mvpmin \,
\pas_{\tpsk{k}}\,\flng{\tpsk{k}}} \, \dist{\param_{\tpsk{k}}}{\pctrlopt}
\\&
+\bruitifcdvp{\tpsk{k}}{\tpsk{k+1}}{\param_\tpsk{k}}{(\mem_t)}{\sm{\pas}}
+ \po{\pas_{\tpsk{k}}
\,\flng{\tpsk{k}}}
\ea
\end{equation*}
where the $\po{}$ is uniform over $\param_{\tpsk{k}}$, $\mem_{\tpsk{k}}$
and $\cdvp$ satisfying the constraints above.

\end{lemme}

\begin{proof}
From \relem{minmthoriztpsk} and since $k\geq k_0$, we have
$\thoriz{\tpsk{k}}>\tpsk{k+1}$. Therefore we can apply
\relem{maintientpsfinibocont} and, for $\tpsk{k} \leq t \leq \tpsk{k+1}$,
we have
\begin{equation*}
\cpl{\param_t}{\mem_t} \in \boctopt \times \bmaint{t}.
\end{equation*}

Thus we can apply \relem{propcentalgoptim} starting at time $\tpsk{k}$,
using again that $\thoriz{\tpsk{k}}>\tpsk{k+1}$.
This yields, for any $\tpsk{k} +1\leq t \leq \tpsk{k+1}$,
\begin{equation*}
\ba
\dist{\param_t}{\pctrlopt} &\leq
\dist{
\paramupdate_{\tpsk{k}:t}\paren{\param_{\tpsk{k}},\memopt_{\tpsk{k}},\paren{\pas_s}}}
{\paramupdate_{\tpsk{k}:t}\paren{\paramopt_{\tpsk{k}},\memopt_{\tpsk{k}},\paren{\pas_s}}}\\
&+
\dist{
\paramupdate_{\tpsk{k}:t}\paren{\paramopt,\memopt_{\tpsk{k}},\paren{\pas_s}}}
{\paramopt}\\
&+
\bruitifcdvp{\tpsk{k}}{t}{\param_\tpsk{k}}{(\mem_t)}{\sm{\pas}}
\\&
+ \go{\paren{\sum_{\tpsk{k}<s\leq t} \pas_s\,\fmdper{s}}\paren{\sum_{\tpsk{k}<s\leq
t} \pas_s\mlipgrad{s}}} + \go{\mlipgrad{t}\sup_{\tpsk{k} < s \leq t} \pas_s}.
\ea
\end{equation*}

Taking $t=\tpsk{k+1}$, by \relem{contractitvk}, the first term is at most
$\paren{1-\mvpmin\,\pas_{\tpsk{k}}\,\flng{\tpsk{k}}}\,\dist{\param_{\tpsk{k}}}{\paramopt}+\po{\pas_{\tpsk{k}}
\,\flng{\tpsk{k}}}$. 

By \relem{extloc}, the second term is $\po{\pas_{\tpsk{k}}
\,\flng{\tpsk{k}}}$.

By \recor{convzerosompasintervalles}, the last two terms are
$\po{\pas_{\tpsk{k}}
\,\flng{\tpsk{k}}}$.
%
%
\end{proof}

%

\subsubsection{Convergence of the Algorithm}

\begin{lemme}
\label{lem:arithgeomconv}
Let $\sm{\sbcfc}=\paren{\scfc_k}$ and $\sm{\sbcfa}=\paren{\scfa_k}$ be
two non-negative sequences such that
\begin{enumerate}
\item $\scfc_k\to 0$ and $\sum_k \scfc_k\to \infty$;
\item $\scfa_k=o(\scfc_k)$ when $k\to\infty$.
\end{enumerate}

Let $(x_k)$  be any non-negative sequence such that for $k\geq k_0$,
\begin{equation*}
x_{k+1} \leq \paren{1-\scfc_k}  x_k + \scfa_k.
\end{equation*}

Then $x_k\to 0$.

\end{lemme}

\begin{proof}

Let us prove that $x_k\to 0$. Let $\eps>0$ and let us prove that
ultimately, $x_k\leq 2\eps$. 

Set $K\deq \inf\enstq{k\geq k_0}{\forall k'\geq k,
\,b_{k'}\leq\eps\,r_{k'}}$. For $k\geq K$, the interval $[0;\eps]$ is
stable by the map $x\mapsto (1-r_k)x+b_k$. Therefore, if
there exists $k\geq K$ such that $x_k\leq \eps$, then we have $x_{k'}\leq \eps$ for all $k'\geq k$.

If there exists no $k\geq K$ such that $x_k\leq \eps$, then we have for
all $k\geq K$,
$0\leq x_{k+1}-\eps \leq \paren{1-\pgs_k} \, x_k + \tasg_k - \eps
\leq \paren{1-\pgs_k} \, x_k + \eps \, \pgs_k - \eps
=\paren{1-\pgs_k} \, \paren{x_k-\eps}$. Therefore,
\begin{equation*}
0 \leq x_k - \eps \leq \paren{\prod_{k'=K}^{k-1} \paren{1-\pgs_{k'}}} \, \paren{x_K-\eps}.
\end{equation*}

Since $\sum_k \pgs_k$ diverges, the product $\prod (1-\pgs_k)$ tends to
$0$. Therefore, $x_k-\eps$ is less than $\eps$ for large enough $k$.

Thus in both cases, $x_k$ is ultimately less that $2\eps$, for any
$\eps>0$.
\end{proof}

\begin{lemme}[End of proof of \rethm{cvalgopti}]
\label{lem:pcvrtrlbruit}
There exists $\cdvpmaxconv>0$ such that,
for any $0\leq \cdvp\leq \cdvpmaxconv$,
the following convergence holds.

For any $\param_0$ with $\dist{\param_0}{\paramopt}\leq
\frac{\rayoptct}{4}$ and any $\mem_0\in\bmaint{0}$, consider a trajectory
given by
\begin{equation*}
\left\lbrace \ba 
\mem_t&=\Algo_{t}\paren{\param_{t-1},\,\mem_{t-1}}\\
\vtanc_t &=\sm{V}_{t}\paren{\param_{t-1},\,\mem_t} \\ 
\param_{t} &= \paramupdate_t\paren{\param_{t-1},\,\pas_t\,\vtanc_t},
\ea \right.
\end{equation*}
for $t\geq 1$.
Then $\param_t$ tends to $\pctrlopt$ as $t\to\infty$.
\end{lemme}

\begin{proof}
Take $\cdvp\leq \min(\cdvpop,\,\brnp/2)$.
(This is not yet
$\cdvpmaxconv$: there will be an
additional constraint on $\cdvp$ below.)

By \relem{propcentechtpstpsk}, there exists $k_0\geq 0$, and a sequence
$b_k=\po{\pas_{\tpsk{k}}
\,\flng{\tpsk{k}}}$ such that
\begin{equation*}
\ba
\dist{\param_{\tpsk{k+1}}}{\pctrlopt} &\leq \paren{1 - \mvpmin \,
\pas_{\tpsk{k}}\,\flng{\tpsk{k}}} \, \dist{\param_{\tpsk{k}}}{\pctrlopt}
+ b_k
\ea
\end{equation*}
holds for those values of $k \geq k_0$ such that
$\dist{\param_{\tpsk{k}}}{\paramopt}\leq \frac{\rayoptct}{3}$ and
$\mem_{\tpsk{k}}\in \bmaint{\tpsk{k}}$. (Note that
$\bruitifcdvp{\tpsk{k}}{\tpsk{k+1}}{\param_\tpsk{k}}{(\mem_t)}{\sm{\pas}}$ 
vanishes by definition because,
for all $t\geq 1$,
we have $\mem_t=\Algo_{t}\paren{\param_{t-1},\,\mem_{t-1}}$.)
By \relem{propcentechtpstpsk}, the value of $b_k$ is uniform
over $\olr$ and the values of $\param$ and $\mem$ satisfying those
assumptions.

Since $b_k$ is $\po{\pas_{\tpsk{k}}
\,\flng{\tpsk{k}}}$, there exists
$k_1\geq {k_0}$ such that $b_k$ is less than
$(\rayoptct/3)(\mvpmin\,\pas_{\tpsk{k}}
\,\flng{\tpsk{k}})$
for $k\geq k_1$. (Such a $k_1$ is uniform in the values of $\param$, $\mem$ and $\olr$ satisfying the assumptions above, because $b_k$ is.) 

Define $\cdvpmaxconv \deq
\min\paren{\cdvpop,\,\brnp/2,\,\cdvp^{\tpsk{k_1}}}$, where
$\cdvp^{\tpsk{k_1}}$ is defined by \relem{contempsfiniprmprodalgortrl}
applied to $T=\tpsk{k_1}$.

The assumptions state that $\dist{\param_0}{\paramopt}\leq
\frac{\rayoptct}{4}$ and $\mem_0\in\bmaint{0}$.
Therefore, by \relem{contempsfiniprmprodalgortrl} applied to
$T=\tpsk{k_1}$, 
if
$\cdvp\leq \cdvp^{\tpsk{k_1}}$ then
$\dist{\param_\tpsk{k_1}}{\pctrlopt}\leq
\frac{\rayoptct}{3}$ and $\mem_\tpsk{k_1}\in \bmaint{\tpsk{k_1}}$.

Set $\scfc_k\deq \mvpmin \,\pas_{\tpsk{k}}\,\flng{\tpsk{k}}$. We have
$b_k=\po{\scfc_k}$.

By \relem{propcentechtpstpsk}, if $\dist{\param_\tpsk{k}}{\pctrlopt}\leq
\frac{\rayoptct}{3}$ and
$\mem_\tpsk{k}\in \bmaint{\tpsk{k}}$, then
\begin{equation*}
\dist{\param_\tpsk{k+1}}{\pctrlopt}\leq
(1-\scfc_k)\,\dist{\param_\tpsk{k}}{\pctrlopt}
+b_k
\end{equation*}
and
$\mem_\tpsk{k+1}\in \bmaint{\tpsk{k+1}}$.

By definition of $k_1$, if $k\geq k_1$ then
$(1-\scfc_k)\frac{\rayoptct}{3}+b_k\leq \frac{\rayoptct}{3}$.

Consequently, if $k\geq k_1$ and $\dist{\param_\tpsk{k}}{\pctrlopt}\leq
\frac{\rayoptct}{3}$ and $\mem_\tpsk{k}\in \bmaint{\tpsk{k}}$, then
$\dist{\param_\tpsk{k+1}}{\pctrlopt}\leq
\frac{\rayoptct}{3}$ and $\mem_\tpsk{k+1}\in \bmaint{\tpsk{k+1}}$.

Since this holds at time $\tpsk{k_1}$, by
induction
this holds for any $k\geq k_1$:
if $\cdvp\leq \cdvp^{\tpsk{k_1}}$, then
$\dist{\param_\tpsk{k}}{\pctrlopt}\leq
\frac{\rayoptct}{3}$ and $\mem_\tpsk{k}\in \bmaint{\tpsk{k}}$ for all
$k\geq k_1$. 

Therefore, for any $k\geq k_1$, we have
\begin{equation*}
\dist{\param_\tpsk{k+1}}{\pctrlopt}\leq
(1-\scfc_k)\,\dist{\param_\tpsk{k}}{\pctrlopt}
+b_k.
\end{equation*}

By \rerem{etaLdiv}, the series $\scfc_k= \mvpmin
\,\pas_{\tpsk{k}}\,\flng{\tpsk{k}}$ diverges.
Since $b_k=\po{\scfc_k}$, 
by \relem{arithgeomconv} this implies that $\param_{\tpsk{k}}$ tends to
$\paramopt$ when $k\to\infty$.

For the intermediate times 
$\tpsk{k}<t\leq \tpsk{k+1}$, by 
\relem{contordredeuxitprm2}, we have
\begin{equation*}
\dist{\param_t}{\param_{\tpsk{k}}} \leq \cst_6 \sum_{\tpsk{k}<s\leq
\tpsk{k+1}} \pas_s \,\fmdper{s}
\end{equation*}
(we can apply \relem{contordredeuxitprm2} because we stay in the
stable tube for $t \geq \tpsk{k_1}$).
By \recor{convzerosompasintervalles}, this proves that $\param_t$
tends to $\paramopt$ if $\param_{\tpsk{k}}$ does.
\end{proof}

\begin{lemme}[End of proof of \rethm{cvalgopti_noisy}]
\label{lem:pcvrtrlbruitnoisy}
Assume that, for any $\param_0$ with $\dist{\param_0}{\paramopt}\leq
\frac{\rayoptct}{4}$ and any $\mem_0\in\bmaint{0}$, we are given a random trajectory $(\mem_t,\vtanc_t,\param_t)_{t\geq 0}$ which respects the stable tube $\paren{\bmaint{t}}_{t\geq 0}$, in the sense of \redef{randomtraj}, that is, satisfies
\begin{equation*}
\left\lbrace \ba 
&\mem_{t-1}\in \bmaint{t-1}\text{ and } \param_{t-1}\in\boctopt \,\Longrightarrow\,
\mem_t\in\bmaint{t}\\
&\vtanc_t =\sm{V}_{t}\paren{\param_{t-1},\,\mem_t} \\ 
&\param_{t} = \paramupdate_t\paren{\param_{t-1},\,\pas_t\,\vtanc_t},
\ea \right.
\end{equation*}
for $t\geq 1$.
Let $\eps>0$. Assume there exists $K\geq 0$, a non-negative sequence
$\paren{\delta_k}$ which tends towards $0$, and $\cdvpnoise >0$ such
that, for all $\olr\leq\cdvpnoise$, with probability greater than
$1-\eps$, this trajectory has negligible noise starting at $K$, at speed
$\paren{\delta_k}$ (\redef{negnoise}).

Then there exists $\cdvpmaxconv>0$ such that for $\cdvp\leq
\cdvpmaxconv$, with probability greater than $1-\eps$,
the parameter $\param_t$ tends to $\pctrlopt$ as $t\to\infty$.
\end{lemme}

\begin{proof}
Take $\cdvp\leq \min(\cdvpnoise,\cdvpop,\,\brnp/2)$.
(This is not yet
$\cdvpmaxconv$: there will be an
additional constraint on $\cdvp$ below.)

Define $b^1_k\deq\delta_k\,\pas_{\tpsk{k}}\,\flng{\tpsk{k}}$, where
$(\delta_k)$ is the sequence controlling the negligible noise in the
assumptions.

Let $k\geq0$ such that $\mem_{\tpsk{k}}\in\bmaint{\tpsk{k}}$,
$\param_{\tpsk{k}}\in\boctopt$,
and
$\dist{\param_{\tpsk{k}}}{\paramopt}\leq\rayoptct/3$.
Since $\cdvp\leq \brnp$,
by \relem{minmthoriztpsk}, for any $k\geq 0$, we
have $\tpsk{k+1}<\thoriz{\tpsk{k}}(\sm{\pas})$. Therefore, by
\relem{maintientpsfinibocontrandomtraj}, we stay in the stable tube
between $\tpsk{k}$ and $\tpsk{k+1}$, namely,
we have
$\mem_t\in\bmaint{t}$ and $\param_t\in\boctopt$ for all $\tpsk{k}+1\leq t
\leq \tpsk{k+1}$.

Since the random trajectory is in  the stable tube for $\tpsk{k}\leq
t \leq \tpsk{k+1}$, we can apply
\relem{propcentechtpstpsk}. Thus, there exists $k_0\geq 0$, and a sequence
$b^2_k=\po{\pas_{\tpsk{k}}
\,\flng{\tpsk{k}}}$ such that
\begin{equation*}
\ba
\dist{\param_{\tpsk{k+1}}}{\pctrlopt} &\leq \paren{1 - \mvpmin \,
\pas_{\tpsk{k}}\,\flng{\tpsk{k}}} \, \dist{\param_{\tpsk{k}}}{\pctrlopt}
\\&
+\bruitifcdvp{\tpsk{k}}{\tpsk{k+1}}{\param_\tpsk{k}}{(\mem_t)}{\sm{\pas}}
+ b^2_k
\ea
\end{equation*}
holds for those values of $k \geq k_0$ such that
$\dist{\param_{\tpsk{k}}}{\paramopt}\leq \frac{\rayoptct}{3}$ and
$\mem_{\tpsk{k}}\in \bmaint{\tpsk{k}}$. By \relem{propcentechtpstpsk},
the value of $b^2_k$ is uniform over $\cdvp$ and the values of $\param$
and $\mem$
satisfying those assumptions.

Set $b_k\deq b^1_k+b^2_k$. Since $b_k$ is $\po{\pas_{\tpsk{k}}
\,\flng{\tpsk{k}}}$, there exists
$k_1\geq \max\cpl{K}{k_0}$ such that $b_k$ is less than
$(\rayoptct/3)(\mvpmin\,\pas_{\tpsk{k}}
\,\flng{\tpsk{k}})$
for $k\geq k_1$. Such a $k_1$ is uniform over the values of
$\param$, $\mem$ and $\cdvp$ satisfying the assumptions above, because $b_k$ is.

Define $\cdvpmaxconv \deq
\min\paren{\cdvpnoise,\,\cdvpop,\,\brnp/2,\,\cdvp^{\tpsk{k_1}}}$, where
$\cdvp^{\tpsk{k_1}}$ is the value provided by
\relem{contempsfiniprmprodalgortrlrandom} applied to $T=\tpsk{k_1}$.

For $k\geq 0$, let
$\mathcal{E}_k=\enstq{\prmctrl \in \epctrl}{\dist{\prmctrl}{\pctrlopt}
\leq \frac{\rayoptct}{3}}\times\bmaint{\tpsk{k}}$. Consider the event
that the random
trajectory has negligible noise; more precisely, define the event
\begin{equation*}
\mathfrak{S}\paren{\olr}\deq\acco{\mathbbm{1}_{\cpl{\param_{\tpsk{k}}}{\mem_{\tpsk{k}}}\in\mathcal{E}_k}\,\bruitif{\tpsk{k}}{\tpsk{k+1}}{\param_\tpsk{k}}{(\mem_t)}
\leq \delta_k\,\pas_{\tpsk{k}}\,\flng{\tpsk{k}},\quad \forall k\geq K}.
\end{equation*}
Thus, on this event, we have
$\bruitifcdvp{\tpsk{k}}{\tpsk{k+1}}{\param_\tpsk{k}}{(\mem_t)}{\sm{\pas}}\leq
b^1_k$
for any $k\geq K$ such that $\param_{\tpsk{k}}\in\boctopt$, $\dist{\param_{\tpsk{k}}}{\pctrlopt}\leq\rayoptct/3$ 
and $\mem_{\tpsk{k}}\in
\bmaint{\tpsk{k}}$. 
By assumption and by definition of negligible noise, for
$\cdvp<\cdvpnoise$ this event has probability at least $1-\eps$.
We now assume the trajectory is such that this event holds.

Set $\scfc_k\deq \mvpmin \,\pas_{\tpsk{k}}\,\flng{\tpsk{k}}$. We have
$b_k=\po{\scfc_k}$.

By \relem{propcentechtpstpsk}, if $\dist{\param_\tpsk{k}}{\pctrlopt}\leq
\frac{\rayoptct}{3}$ and
$\mem_\tpsk{k}\in \bmaint{\tpsk{k}}$, then
\begin{equation*}
\dist{\param_\tpsk{k+1}}{\pctrlopt}\leq
(1-\scfc_k)\,\dist{\param_\tpsk{k}}{\pctrlopt}
+b_k
\end{equation*}
and
$\mem_\tpsk{k+1}\in \bmaint{\tpsk{k+1}}$.

By definition of $k_1$, if $k\geq k_1$ then
$(1-\scfc_k)\frac{\rayoptct}{3}+b_k\leq \frac{\rayoptct}{3}$.

Consequently, if $k\geq k_1$ and $\dist{\param_\tpsk{k}}{\pctrlopt}\leq
\frac{\rayoptct}{3}$ and $\mem_\tpsk{k}\in \bmaint{\tpsk{k}}$, then
$\dist{\param_\tpsk{k+1}}{\pctrlopt}\leq
\frac{\rayoptct}{3}$ and $\mem_\tpsk{k+1}\in \bmaint{\tpsk{k+1}}$.

Therefore, by induction, if $\dist{\param_\tpsk{k_1}}{\pctrlopt}\leq
\frac{\rayoptct}{3}$ and $\mem_\tpsk{k_1}\in \bmaint{\tpsk{k_1}}$, then
this holds for any $k\geq k_1$.

%

By assumption, the random trajectory respects the stable tube.
Moreover,
we have assumed that $\dist{\param_0}{\paramopt}\leq
\frac{\rayoptct}{4}$ and $\mem_0\in\bmaint{0}$.
Therefore, we can apply
\relem{contempsfiniprmprodalgortrlrandom}: the value
$\cdvp^{\tpsk{k_1}}$ provided by
\relem{contempsfiniprmprodalgortrlrandom} (and used in the definition of
$\cdvpmaxconv$
above) is such that, for any $\cdvp\leq \cdvp^{\tpsk{k_1}}$,
we have
$\dist{\param_\tpsk{k_1}}{\pctrlopt}\leq
\frac{\rayoptct}{3}$ and $\mem_\tpsk{k_1}\in \bmaint{\tpsk{k_1}}$. We
have defined $\cdvpmaxconv$ above to be no greater than $\cdvp^{\tpsk{k_1}}$,
so the constraint $\cdvp\leq \cdvp^{\tpsk{k_1}}$ is satisfied for any $\cdvp\leq \cdvpmaxconv$.

Thus, if $\cdvp\leq \cdvp^{\tpsk{k_1}}$, then
$\dist{\param_\tpsk{k}}{\pctrlopt}\leq
\frac{\rayoptct}{3}$ and $\mem_\tpsk{k}\in \bmaint{\tpsk{k}}$ for all
$k\geq k_1$.

For any $\olr \leq \cdvpmaxconv$, conditionally on the event $\mathfrak{S}\paren{\olr}$,
all of the above applies. 
Therefore, for any $k\geq k_1$ we have
\begin{equation*}
\dist{\param_\tpsk{k+1}}{\pctrlopt}\leq
(1-\scfc_k)\,\dist{\param_\tpsk{k}}{\pctrlopt}
+b_k.
\end{equation*}

Since $b_k=\po{\scfc_k}$, by \relem{arithgeomconv} this implies that $\param_{\tpsk{k}}$ tends to
$\paramopt$ when $k\to\infty$.

For the intermediate times 
$\tpsk{k}<t\leq \tpsk{k+1}$, by 
\relem{contordredeuxitprm2}, we have
\begin{equation*}
\dist{\param_t}{\param_{\tpsk{k}}} \leq \cst_6 \sum_{\tpsk{k}<s\leq
\tpsk{k+1}} \pas_s \,\fmdper{s}
\end{equation*}
(we can apply \relem{contordredeuxitprm2} because we stay in the
stable tube for $t \geq \tpsk{k_1}$).
By \recor{convzerosompasintervalles}, this proves that $\param_t$
tends to $\paramopt$ if $\param_{\tpsk{k}}$ does.

Therefore, for each $\olr \leq \cdvpmaxconv$, convergence occurs for each
trajectory such that the event $\mathfrak{S}\paren{\olr}$ holds, which,
by assumption, happens with probability greater than $1-\eps$. We have thus proven our claim.
\end{proof}

\subsubsection{Convergence of the Open-Loop Algorithm}

We now prove convergence of the open-loop algorithm
as used in Theorem~\ref{thm:cvalgoboucleouverte}.
All results established so far still hold true,
except for \relem{propcentechtpstpsk} and
\relem{pcvrtrlbruit} (and \relem{pcvrtrlbruitnoisy}); we now prove the respective analogues of \relemdeux{propcentechtpstpsk}{pcvrtrlbruit} for the open-loop 
algorithm, \relem{propcentechtpstpskopenloop} and \relem{cvoploopalgo}.
The proofs are actually simpler, since the open-loop case on intervals
$(\tpsk{k};\tpsk{k+1}]$ is actually the
basis of the analysis of the previous case.

First, \relem{pascstmorceaux} deals with the piecewise constant
stepsizes of the open-loop algorithm.

\begin{lemme}[Using piecewise constant step-sizes]
\label{lem:pascstmorceaux}
Define a modified stepsize sequence $(\prn_t)$ by setting
\begin{equation*}
\prn_t\deq \pas_{\tpsk{k+1}}
\end{equation*}
for each $\tpsk{k}+1\leq t \leq\tpsk{k+1}$.
Then this new stepsize sequence still satisfies \rehyp{spdes}.

Consequently, all previous results also apply with this new stepsize
sequence.
\end{lemme}

Thus, for the rest of this section, we assume that the stepsize sequence
$(\pas_t)$ is constant on each time interval $(\tpsk{k};\tpsk{k+1}]$.


\begin{proof}
Fof the first point of \rehyp{spdes}, write
\begin{equation*}
\sum_{t\geq 0}\,\prn_t=\sum_{k\geq
0}\,\sum_{\tpsk{k}+1}^{\tpsk{k+1}}\,\pas_{\tpsk{k+1}}=\sum_{k\geq
0}\,\pas_{\tpsk{k+1}}\,\flng{\tpsk{k}}.
\end{equation*}
Now,
$\pas_{\tpsk{k+1}}\,\flng{\tpsk{k}}\sim\pas_{\tpsk{k}}\,\flng{\tpsk{k}}\sim
\sum_{t=\tpsk{k}+1}^{\tpsk{k+1}}\,\pas_t$ by
\recor{convzerosompasintervalles}. So $\sum \prn_t$ diverges if and only
if $\sum \pas_t$ does.
There is nothing to check for the third point of the assumption. Let us
check the second point. For every $t \geq 1$, write $k_t$ the unique
integer such that $\tpsk{k_t} < t \leq \tpsk{k_t+1}$. Then
\begin{equation*}
\prn_t\,\flng{t}\,\fmdper{t}\,\mlipgrad{t} = \pas_{\tpsk{k_t+1}}\,\flng{t}\,\fmdper{t}\,\mlipgrad{t}
\end{equation*}
Since $k_t\to\infty$, when $t\to\infty$, and since $\tpsk{k_t} < t \leq
\tpsk{k_t+1}$ with
$\tpsk{k}\sim\tpsk{k+1}$ when $k\to\infty$ (by \relem{pechtps}),
we have $\tpsk{k_t+1}\sim t$, when $t\to\infty$. As a result, since
$\lng$, $\sbfmdper$ and $\mlipgrad{\cdot}$ are scale functions, and
consequently preserve asymptotic equivalence at infinity, we have $\flng{t}\,\fmdper{t}\,\mlipgrad{t}\sim \flng{\tpsk{k_t+1}}\,\fmdper{\tpsk{k_t+1}}\,\mlipgrad{\tpsk{k_t+1}}$, as $t\to\infty$. Therefore,
\begin{equation*}
\prn_t\,\flng{t}\,\fmdper{t}\,\mlipgrad{t} \sim \pas_{\tpsk{k_t+1}}\,\flng{\tpsk{k_t+1}}\,\fmdper{\tpsk{k_t+1}}\,\mlipgrad{\tpsk{k_t+1}},
\end{equation*}
as $k\to\infty$. Now, since the sequence $\paren{\pas_t}$ satisfies \rehyp{spdes}, the right-hand side converges to $0$, as $t\to\infty$, so that the sequence $\paren{\prn_t}$ indeeds satisfies the second point of \rehyp{spdes}.

For the last point, 
let $t\geq 1$. We want to bound $(\sup_{t< s \leq
t+\flng{t}}\,\prn_s)/(\inf_{t< s \leq t+\flng{t}}\,\prn_s)$. We have
\begin{equation*}
1\leq \frac{\sup_{t< s \leq t+\flng{t}}\,\prn_s}{\inf_{t< s \leq t+\flng{t}}\,\prn_s}
=\frac{\sup_{t<s \leq t+\flng{t}}\,\pas_{\tpsk{k_s+1}}}{\inf_{t< s \leq
t+\flng{t}}\,\pas_{\tpsk{k_s+1}}}.
\end{equation*}
The maps $s\mapsto k_s$ is non-decreasing. Therefore, when $s$ ranges
from $t$ to $t+\flng{t}$, $k_s$ ranges at most from $k_t$ to
$k_{t+\flng{t}}$, so that $\tpsk{k_s+1}$ ranges at most from
$\tpsk{k_t+1}$ to $\tpsk{k_{t+\flng{t}}+1}$. Therefore,
\begin{equation*}
\frac{\sup_{t<s \leq t+\flng{t}}\,\pas_{\tpsk{k_s+1}}}{\inf_{t< s \leq t+\flng{t}}\,\pas_{\tpsk{k_s+1}}}
\leq
\frac{\sup_{\tpsk{k_t+1}\leq s \leq
\tpsk{k_{t+\flng{t}}+1}}\,\pas_s}{\inf_{\tpsk{k_t+1}\leq s \leq
\tpsk{k_{t+\flng{t}}+1}}\,\pas_s}.
\end{equation*}
Next, by definition of $k_{t+\flng{t}}$, we have $\tpsk{k_{t+\flng{t}}}<
t+\flng{t}$. Moreover, $\tpsk{k_t}< t\leq\tpsk{k_t+1}$, and $\lng$ is non-decreasing, so that we have 
$\tpsk{k_t+1}=\tpsk{k_t}+\flng{\tpsk{k_t}}< t+\flng{t} \leq
\tpsk{k_t+1}+\flng{\tpsk{k_t+1}}=\tpsk{k_t+2}$.
As a result, $k_{t+\flng{t}}=k_t+1$. Therefore,
\begin{equation*}
\frac{\sup_{t<s \leq t+\flng{t}}\,\pas_{\tpsk{k_s+1}}}{\inf_{t< s \leq t+\flng{t}}\,\pas_{\tpsk{k_s+1}}}
\leq
\frac{\sup_{\tpsk{k_t+1}\leq s \leq
\tpsk{k_{t}+2}}\,\pas_s}{\inf_{\tpsk{k_t+1}\leq s \leq
\tpsk{k_{t}+2}}\,\pas_s}.
\end{equation*}
By \relem{stronghomo}, this is $1+\po{1/\fmdper{\tpsk{k_t+1}}}$.

Finally, remember that
$\tpsk{k_t+1}\sim t$. So, since $\fmdper{}$ is
a scale function, we have $\fmdper{\tpsk{k_t+1}}\sim \fmdper{t}$ and
$1+o(1/\fmdper{\tpsk{k_t+1}})=1+o(1/\fmdper{t})$. Thus, we have proven that
\begin{equation*}
1\leq \frac{\sup_{t< s \leq t+\flng{t}}\,\prn_s}{\inf_{t< s \leq
t+\flng{t}}\,\prn_s}
\leq 1+o(1/\fmdper{t})
\end{equation*}
namely, $\prn_s$ satisfies the last point of \rehyp{spdes}.
\end{proof}

\begin{lemme}[Contraction of errors from $\tpsk{k}$ to $\tpsk{k+1}$ for the open-loop algorithm]
\label{lem:propcentechtpstpskopenloop}
Let $\cdvp \leq \min(\brnp/2,\cdvpop)$.
Let $k\geq k_0$ where $k_0$ is defined in \recor{contmajitvk}.

Let $\param_\tpsk{k}$ be such that
${\dist{\param_\tpsk{k}}{\pctrlopt}\leq \frac{\rayoptct}{3}}$, and let
$\mem_\tpsk{k}\in \bmaint{\tpsk{k}}$.
Consider the learning trajectory from initial parameter
$\param_\tpsk{k}$ and initial state $\mem_\tpsk{k}$ and learning rates
$(\pas_t)$, namely,
\begin{equation*}
\left\lbrace \ba 
{\mem_{t}}&=\mathcal{A}_{t}\paren{\param_{\tpsk{k}},\,\mem_{t-1}} \\
\vtanc_t &=\sm{V}_{t}\paren{\param_{\tpsk{k}},\,\mem_t} \\ 
\param_{t} &= \paramupdate_t\paren{\param_{t-1},\,\pas_t\,\vtanc_t}.
\ea \right.
\end{equation*}
for $\tpsk{k} < t \leq \tpsk{k+1}$.
Then for all $\tpsk{k} \leq t \leq \tpsk{k+1}$, we have
$\param_t\in\boctopt$ and $\mem_t\in \bmaint{t}$, and moreover,
\begin{equation*}
\ba
\dist{\param_{\tpsk{k+1}}}{\pctrlopt} &\leq \paren{1 - \mvpmin \,
\pas_{\tpsk{k}}\,\flng{\tpsk{k}}} \, \dist{\param_{\tpsk{k}}}{\pctrlopt}
+ \po{\pas_{\tpsk{k}}
\,\flng{\tpsk{k}}}
\ea
\end{equation*}
where the $\po{}$ is uniform over $\param_{\tpsk{k}}$, $\mem_{\tpsk{k}}$
and $\cdvp$ satisfying the constraints above.

\end{lemme}
\begin{proof}
The proof is similar to that of \relem{propcentechtpstpsk}.

From \relem{minmthoriztpsk} and since $k\geq k_0$, we have
$\thoriz{\tpsk{k}}>\tpsk{k+1}$. Therefore we can apply
\relem{maintientpsfinibocont} and, for $\tpsk{k} \leq t \leq \tpsk{k+1}$,
we have
\begin{equation*}
\cpl{\param_t}{\mem_t} \in \boctopt \times \bmaint{t}.
\end{equation*}

By construction of the sequence $\param_t$ and by definition of
$\paramupdate_{t_1:t_2}(\param_{t_1},\mem_{t_1},(\pas_s))$
(Def.~\ref{def:olupdate}),
we have
\begin{equation*}
\param_{t}=\paramupdate_{\tpsk{k}:t}(\param_{\tpsk{k}},\mem_{\tpsk{k}},(\pas_s)).
\end{equation*}
Then thanks to \relem{contecarttrajmprmqidiff} starting at time
$\tpsk{k}$, for any $\tpsk{k} +1\leq t \leq \tpsk{k+1}$
(using again that $\thoriz{\tpsk{k}}>\tpsk{k+1}$), we have
\begin{equation*}
\dist{\paramupdate_{\tpsk{k}:t}(\param_{\tpsk{k}},\mem_{\tpsk{k}},(\pas_s))}{\paramupdate_{\tpsk{k}:t}(\param_{\tpsk{k}},\memopt_{\tpsk{k}},(\pas_s))} =
\go{\mlipgrad{t}\sup_{\tpsk{k}+1\leq s \leq t}\,\pas_s}
\end{equation*}
and, thanks to the triangle inequality, we obtain, for any $\tpsk{k}+1 \leq t \leq \tpsk{k+1}$,
\begin{equation*}
\ba
\dist{\param_t}{\pctrlopt} &\leq
\dist{
\paramupdate_{\tpsk{k}:t}\paren{\param_{\tpsk{k}},\memopt_{\tpsk{k}},\paren{\pas_s}}}
{\paramupdate_{\tpsk{k}:t}\paren{\paramopt,\memopt_{\tpsk{k}},\paren{\pas_s}}}\\
&+
\dist{
\paramupdate_{\tpsk{k}:t}\paren{\paramopt,\memopt_{\tpsk{k}},\paren{\pas_s}}}
{\paramopt} + \go{\mlipgrad{t}\sup_{\tpsk{k} < s \leq t} \pas_s}.
\ea
\end{equation*}

Apply this to $t=\tpsk{k+1}$.
By \relem{contractitvk}, the first term is at most
$\paren{1-\mvpmin\,\pas_{\tpsk{k}}\,\flng{\tpsk{k}}}\,\dist{\param_{\tpsk{k}}}{\paramopt}+\po{\pas_{\tpsk{k}}
\,\flng{\tpsk{k}}}$.

By \relem{extloc}, the second term is $\po{\pas_{\tpsk{k}}
\,\flng{\tpsk{k}}}$.

By \recor{convzerosompasintervalles}, the last term is
$\po{\pas_{\tpsk{k}}
\,\flng{\tpsk{k}}}$.
%
%
\end{proof}


\begin{lemme}[Convergence of the open-loop algorithm: end of proof of \rethm{cvalgoboucleouverte}]
\label{lem:cvoploopalgo}
There exists $\cdvpmaxconv>0$ such that,
for any $0\leq \cdvp\leq \cdvpmaxconv$,
the following convergence holds.

Let $\param_0\in \Param$ with $\dist{\param_0}{\paramopt}\leq
\frac{\rayoptct}{4}$, and let $\mem'_{\tpsk{k}}\in\bmaint{\tpsk{k}}$ be
any sequence of ``reset states''.
Consider the trajectory $(\param_t)_{t\geq 0}$
computed for every $k\geq0$ and $\tpsk{k} < t \leq \tpsk{k+1}$ by
resetting
\begin{equation*}
\mem_{\tpsk{k}} \leftarrow \mem'_{\tpsk{k}}\in\bmaint{\tpsk{k}}
\end{equation*}
and then
\begin{equation*}
\left\lbrace \ba 
\mem_t&=\Algo_{t}\paren{\param_{\tpsk{k}},\,\mem_{t-1}} \\
\vtanc_t &=\sm{V}_{t}\paren{\param_{\tpsk{k}},\,\mem_t} \\ 
\param_{t} &= \paramupdate_t\paren{\param_{t-1},\,\pas_t\,\vtanc_t}.
\ea \right.
\end{equation*}
Then $\param_t$
tends to $\pctrlopt$ as $t\to\infty$.
\end{lemme}
\begin{proof}
The proof unfolds much like that of \relem{pcvrtrlbruit}. It is simpler, in that the relations $\mem_{\tpsk{k}}\in\bmaint{\tpsk{k}}$, after the substitutions $\mem_{\tpsk{k}}\leftarrow\mem'_{\tpsk{k}}$, hold by assumption.

Take $\cdvp\leq \min(\cdvpop,\brnp/2)$. (This is not yet $\cdvpmaxconv$: there will be an additional constraint on $\cdvp$ below.)

By \relem{propcentechtpstpskopenloop}, there exists $k_0\geq 0$, and a sequence
$b_k=\po{\pas_{\tpsk{k}}
\,\flng{\tpsk{k}}}$ such that
\begin{equation}
\label{eq:maincontropenloop}
\ba
\dist{\param_{\tpsk{k+1}}}{\pctrlopt} &\leq \paren{1 - \mvpmin \,
\pas_{\tpsk{k}}\,\flng{\tpsk{k}}} \, \dist{\param_{\tpsk{k}}}{\pctrlopt}
+ b_k
\ea
\end{equation}
holds for those values of $k\geq k_0$ such that
$\dist{\param_{\tpsk{k}}}{\paramopt}\leq \frac{\rayoptct}{3}$ and
$\mem_{\tpsk{k}}\in \bmaint{\tpsk{k}}$. By \relem{propcentechtpstpskopenloop}, the value of $b_k$ is uniform over $\olr$ and the values of $\param$ and $\mem$ satisfying those assumptions.

Since $b_k$ is $\po{\pas_{\tpsk{k}}
\,\flng{\tpsk{k}}}$, there exists
$k_1\geq k_0$ such that $b_k$ is less than
$(\rayoptct/3)(\mvpmin\,\pas_{\tpsk{k}}
\,\flng{\tpsk{k}})$
for $k\geq k_1$. (Such a $k_1$ is uniform in the values of $\param$, $\mem$ and $\olr$ satisfying the assumptions above, because $b_k$ is.)

Define $\cdvpmaxconv \deq
\min\paren{\cdvpop,\,\brnp/2,\,\cdvp^{\tpsk{k_1}}}$, where
$\cdvp^{\tpsk{k_1}}$ is defined in \relem{contempsfiniprmprodalgortrl}.

Set $\scfc_k\deq \mvpmin \,\pas_{\tpsk{k}}\,\flng{\tpsk{k}}$. We have
$b_k=\po{\scfc_k}$.

By \relem{propcentechtpstpskopenloop}, if $\dist{\param_\tpsk{k}}{\pctrlopt}\leq
\frac{\rayoptct}{3}$ and
$\mem_\tpsk{k}\in \bmaint{\tpsk{k}}$, then
\begin{equation*}
\dist{\param_\tpsk{k+1}}{\pctrlopt}\leq
(1-\scfc_k)\,\dist{\param_\tpsk{k}}{\pctrlopt}
+b_k.
\end{equation*}

By definition of $k_1$, if $k\geq k_1$ then
$(1-\scfc_k)\frac{\rayoptct}{3}+b_k\leq \frac{\rayoptct}{3}$.

Consequently, if $k\geq k_1$ and $\dist{\param_\tpsk{k}}{\pctrlopt}\leq
\frac{\rayoptct}{3}$ and $\mem_\tpsk{k}\in \bmaint{\tpsk{k}}$, then
$\dist{\param_\tpsk{k+1}}{\pctrlopt}\leq
\frac{\rayoptct}{3}$.

By assumption, for every $k$, after the substitution $\mem_{\tpsk{k}}\leftarrow\mem'_{\tpsk{k}}$, we have $\mem_\tpsk{k}\in \bmaint{\tpsk{k}}$. 
Therefore, by induction, if $\dist{\param_\tpsk{k_1}}{\pctrlopt}\leq
\frac{\rayoptct}{3}$, 
this holds for any $k\geq k_1$.

Since we assume $\dist{\param_0}{\paramopt}\leq
\frac{\rayoptct}{4}$ and $\mem_0\in\bmaint{0}$,
by \relem{contempsfiniprmprodalgortrl} applied to $T=\tpsk{k_1}$,
for $\cdvp\leq \cdvp^{\tpsk{k_1}}$, we have
$\dist{\param_\tpsk{k_1}}{\pctrlopt}\leq
\frac{\rayoptct}{3}$.

Thus, if $\cdvp\leq \cdvp^{\tpsk{k_1}}$, then
$\dist{\param_\tpsk{k}}{\pctrlopt}\leq
\frac{\rayoptct}{3}$ for all
$k\geq k_1$.

Therefore, for any $k\geq k_1$, we have
\begin{equation*}
\dist{\param_\tpsk{k+1}}{\pctrlopt}\leq
(1-\scfc_k)\,\dist{\param_\tpsk{k}}{\pctrlopt}
+b_k.
\end{equation*}

Since $b_k=\po{\scfc_k}$, by \relem{arithgeomconv} this implies that $\param_{\tpsk{k}}$ tends to
$\paramopt$ when $k\to\infty$.

For the intermediate times 
$\tpsk{k}<t\leq \tpsk{k+1}$, by 
\relem{contordredeuxitprm2}, we have
\begin{equation*}
\dist{\param_t}{\param_{\tpsk{k}}} \leq \cst_6 \sum_{\tpsk{k}<s\leq
\tpsk{k+1}} \pas_s \,\fmdper{s}
\end{equation*}
(we can apply \relem{contordredeuxitprm2} because we stay in the
stable tube for $t \geq \tpsk{k_1}$).
By \recor{convzerosompasintervalles} this proves that $\param_t$
tends to $\paramopt$ if $\param_{\tpsk{k}}$ does.
\end{proof}

\section{Controlling \rtrl and \Algonoisy \rtrl Algorithms around the Target Trajectory}
\label{sec:contrtrlappxalgoart}

We now turn back to the setting of Section~\ref{sec:presentation}. We
proceed by making the connection with the more abstract setting of
Section~\ref{sec:absontadsys}, with a suitable abstract state
$\mem_t=(\state_t,\jope_t)$ where $\state_t$ is the state of the original
dynamical system, and $\jope_t$ is the quantity maintained by RTRL. Notably, we relate the assumptions of
Section~\ref{sec:presentation} to those of Section~\ref{sec:absontadsys}.
Convergence will then result from Theorems~\ref{thm:cvalgopti},
\ref{thm:cvalgopti_noisy}, or~\ref{thm:cvalgoboucleouverte}, depending on
the case.

Thus, we now work under the assumptions of
Section~\ref{sec:presentation}.
As before,
throughout the proof, the constants implied in
the $O()$ notation only depend on the constants and $O()$ directly appearing in
the assumptions (Remark~\ref{rem:O}).

\subsection{Applying the Abstract Convergence Theorem to RTRL}
\label{sec:rtrlasalgo}

To prove convergence of the RTRL algorithm, we will apply
\rethm{cvalgopti} or \rethm{cvalgopti_noisy} to the state of the algorithm. The latter is composed
not only of the state of the system $\state_t\in\State_t$, but also of the Jacobian
$\jope_t$, which is maintained by the algorithm. The Jacobian $\jope_t$
is an element of the space of linear maps $\epjopeins{t}$. Thus, the
state of the algorithm will be the pair $\mem_t=(\state_t,\jope_t)$. The
definition of RTRL provides the transition function on $\mem_t$, together
with the way to compute gradients.

This is gathered in the following definition.  The purpose here is
to bring the RTRL algorithm as defined in \redef{algortrl} into the
framework of \resec{abstractalgos}.

\begin{definition}[RTRL as an abstract gradient descent algorithm]
\label{def:rtrlasalgo}
Given a parameterized dynamical system (Defs.~\ref{def:prmdynsys}--\ref{def:fcpletaprm}) and an extended RTRL algorithm
(Def.~\ref{def:algortrl}),
the transition operators $(\Algo_t)$ (Def.~\ref{def:transop})
associated with this RTRL algorithm are defined as follows. The parameter space is $\Theta$ and
the state space is
\begin{equation*}
\Mem_t\deq \State_t\times \epjopeins{t}
\end{equation*}
equipped with the norm
$\norm{(\state,\jope)}=\max(\norm{\state},\norm{\jope})$. The
transition operators $\Algo_t\from \Param\times \Mem_{t-1}\to
\Mem_t$ are defined by
\begin{equation*}
\Algo_t(\param,(\state,\jope))\deq \left(\opevol_t(\state,\param),
\frac{\partial \opevol_t(\state,\param)}{\partial \state}\,\jope+\frac{\partial
\opevol_t(\state,\param)}{\partial \param}\right)
\end{equation*}
and the gradient computation operators $\sm{V}_t\from
\Param\times \Mem_t\to \Param$ (Def.~\ref{def:gradcomputop}) are set to
\begin{equation*}
\sm{V}_t(\param,(\state,\jope))\deq \furt{\frac{\partial
\perte_t(\state)}{\partial \state}\,\jope}{\state}{\param}.
\end{equation*}
Finally, the update operators
$\paramupdate_t$ of Def.~\ref{def:paramupdate} are those of
the RTRL algorithm (Def.~\ref{def:algortrl}).
\end{definition}

The rest of the text is devoted to proving that this abstract algorithm
satisfies all the assumptions of Section~\ref{sec:absontadsys}.

We have to prove that these assumptions hold for $\param$ in some ball
$\boctopt$ (Section~\ref{sec:absontadsys}).
We start
by setting $\boctopt$ to the
ball $B_\Param(\paramopt,r_\Param)$ where the assumptions of
Section~\ref{sec:presentation} hold. This ball $\boctopt$ will
be reduced several times in the course of the proof so that elements
$\param\in\boctopt$ satisfy further properties.

\begin{definition}[Notation for iterates]
\label{def:notiterates}
Let $0\leq t_1\leq t_2$.
Given $\state_{t_1}\in \State_{t_1}$ and a sequence of parameters
$(\param_t)_{t\geq 0}$ in $\Param$, we denote
\begin{equation*}
\opevol_{t_1:t_2}(\state_{t_1},(\param_t))\deq \state_{t_2}
\end{equation*}
where the sequence $(\state_t)$ is defined inductively via
$
\state_t=\opevol_t(\state_{t-1},\param_{t-1})
$
for $t>t_1$. If $(\param_t)\equiv \param$ is constant we just write
$\opevol_{t_1:t_2}(\state_{t_1},\param)$.
\end{definition}

Next, we define the norm on $\Param$ that will be used in the proof.
Indeed, convergence in Section~\ref{sec:absontadsys} is based on a
contractivity property in a certain distance
(Assumption~\ref{hyp:optimprm}.2). But the dynamics of learning is not
contractive for any distance on $\Param$, only for distances built from a
suitable Lyapunov function.

For non-extended RTRL (no $\fur_t$), the suitable norm on $\Param$ is
directly given by the Hessian of the average loss at $\paramopt$. For
extended RTRL algorithms, remember the notation from
\rehyp{critoptrtrlnbt}: the Jacobian of the update direction, over time,
averages to a matrix $\linalgmatrix$ whose eigenvalues have positive real
part, which plays the role of an extended Hessian of the average loss.
This matrix controls the asymptotic dynamics of learning around
$\paramopt$, which is equivalent to $(\param-\paramopt)'=-\linalgmatrix
(\param-\paramopt)$ in the continuous-time limit when $\param$ is close
to $\paramopt$ (see Section~\ref{sec:contaroundparamopt}).

We select a norm on $\Param$ based on $\linalgmatrix$, such that this
dynamics is contractive. This is based on a classical linear algebra
result.

\begin{lemme}[Existence of a suitable Lyapunov function]
\label{lem:suitposdefm}
There exists a positive definite matrix $B$ such that
$B\,\linalgmatrix+\linalgmatrix^T\,B$ is positive definite.
\end{lemme}
\begin{proof}
This is a consequence of the fact the eigenvalues of $\linalgmatrix$ have positive
real part. See Appendix~\ref{sec:positivestable}.
\end{proof}

From now on we endow $\Param$ with the norm given by
$B$, namely, we set
\begin{equation*}
\norm{\param}^2\deq \transp{\param}B\,\param
\end{equation*}
where $B$ is such that $B\linalgmatrix+\transp{\linalgmatrix}B$ is
positive definite and $\linalgmatrix$ is
given by \rehyp{critoptrtrlnbt}. This norm will be used as an approximate
Lyapunov function for the algorithm.

Note that the assumptions in Section~\ref{sec:presentation} have been expressed with
respect to an unspecified norm on $\Param$. Since $\Param$ is
finite-dimensional, all norms are equivalent; in particular, we can find
a ball for the new norm that is included in the original ball
$B_\Param(\paramopt,r_\Param)$
on which the assumptions hold. 
Assumptions~\ref{hyp:opt_simple}, \ref{hyp:critoptrtrlnbt},
\ref{hyp:updateop_first}, \ref{hyp:errorgauge}, \ref{hyp:regftransetats},
\ref{hyp:equicontH_simple}, and \ref{hyp:equicontH}
also involve a norm on $\Param$ via norms on objects such as derivatives
with respect to $\param$; a change to an equivalent norm only introduces
a constant factor which is absorbed in the $\go{}$ in these assumptions.
Therefore, up to restriction to the smaller ball,
all the assumptions of Section~\ref{sec:presentation} hold with respect
to the norm
we just defined.

\subsection{RTRL Computes the Correct Derivatives}

We first prove that RTRL indeed computes the correct derivatives of the
loss (this is actually how RTRL is built in the first place) when the
parameter is kept fixed. This implies that over a time interval, the
open-loop (fixed-parameter) algorithm computes a parameter update equal
to the derivative of the loss, summed over this interval. (When the
parameter is actually updated at every step, this will be true only up to
some higher-order terms, controlled in Section~\ref{sec:prcvalappxalgo}.)

\begin{lemme}[RTRL computes the correct derivatives for the open-loop
trajectory]
\label{lem:rtrlcorrectderivatives}
Call \emph{open-loop RTRL} the algorithm of \redef{algortrl} with
$\pas_t=0$ for all $t$ (i.e., $\param$ is kept fixed).

Then the quantities $\jope_t$ and $\vtanc_t$ computed by open-loop RTRL
starting at $\state_0$ and $\param_0=\param$,
are respectively equal to the Jacobian of the state with respect to the
parameter,
\begin{equation*}
\jope_t=\frac{\partial \fstate_t(\state_0,\param)}{\partial
\param}
\end{equation*}
and to the derivative of the loss with respect to the parameter fed to the extended update rule,
\begin{equation*}
\vtanc_t=\furt{\frac{\partial \sbpermc{t}(\state_0,\param)}{\partial
\param}}{\fstate_{t}\cpl{\state_0}{\param}}{\param}
\end{equation*}
for all $t\geq 1$.

In other words, the RTRL algorithm $\Algo_t$ (Definition~\ref{def:rtrlasalgo})
satisfies, for any $\param\in \Param$ and $\state_0\in \State_0$,
\begin{equation*}
\Algo_{0:t}(\param,(\state_0,0))=\left(\fstate_{t}(\state_0,\param),\,
\frac{\partial \fstate_t(\state_0,\param)}{\partial
\param}
\right)
\end{equation*}
and
\begin{equation*}
\sm{V}_t\left(\param,\,\Algo_{0:t}(\param,(\state_0,0))\right)=\furt{\frac{\partial
\sbpermc{t}(\state_0,\param)}{\partial
\param}}{\fstate_{t}\cpl{\state_0}{\param}}{\param}.
\end{equation*}
\end{lemme}

\begin{proof}
RTRL is actually built to obtain this property, as explained before
\redef{algortrl}. Indeed, when $\pas_t=0$, the parameter $\param$ is
constant along the trajectory. Then
by definition the state $s_t$ in \redef{algortrl} is 
$s_t=\fstate_t(\state_0,\param)=\opevol_t(\fstate_{t-1}(\state_0,\param),\param)$.
By differentiation, we find that $\partial \fstate_t/\partial\param$
satisfies the linear
evolution equation
\begin{equation*}
\frac{\partial \fstate_t(\state_0,\param)}{\partial
\param}=\frac{\partial \opevol_t}{\partial \state} \frac{\partial
\fstate_{t-1}(\state_0,\param)}{\partial \param}+\frac{\partial \opevol_t}{\partial
\param}
\end{equation*}
where the derivatives of $\opevol_t$ are evaluated at
$(\state_{t-1},\param)$. This is the evolution equation for $\jope$ in
\redef{algortrl}, so $J$ is equal to this quantity. (The initialization
$J=0$ corresponds to $\partial \fstate_0(\state_0,\param)/\partial
\param=\partial \state_0/\partial\param=0$.)

Then the
expressions for $\vtanc_t$ and $\sm{V}_t$ follow from their definitions
in Defs.~\ref{def:algortrl} and~\ref{def:rtrlasalgo}, and the chain rule
applied to
the definition of
$\sbpermc{t}$ (Def.~\ref{def:fcpletaprm}).
\end{proof}

Recall that by \redef{paramupdate} and
\redef{olupdate},
the RTRL algorithm defines an iterated update operator
$\paramupdate_{t_1:t_2}(\param,(\vtanc_t))$ and an open-loop update
operator $\paramupdate_{t_1:t_2}(\param,\mem_{t_1},(\pas_t))$.

\begin{corollaire}
\label{cor:RTRLiter}
Set $\memopt_0\deq (\stateopt_0,0)$ (state of the RTRL algorithm
initialized at $\stateopt_0$ with $J_0=0$).

Then 
for any $\param\in\Param$, for any $1\leq t_1\leq t_2$,
for any sequence of learning rates $(\pas_{t;\,t_1,t_2})$ (not necessarily satisfying
\rehyp{scfcstepsize}), 
the open-loop
operator of RTRL updates $\param$ by the recurrent derivatives of the
loss, fed to the extended update rule:
\begin{equation*}
\paramupdate_{t_1:t_2}(\param,
\mem_{t_1}
,(\pas_{t;\,t_1,t_2})_t)=
\paramupdate_{t_1:t_2}\paren{\param,\paren{\pas_{t;\,t_1,t_2}\,\furt{\frac{\partial
\sbpermc{t}(\stateopt_0,\param)}{\partial
\param}}{\fstate_t\cpl{\stateopt_0}{\param}}{\param}}_{\!t\,}}
\end{equation*}
where $\mem_{t_1}=\Algo_{0:t_1}(\param,\memopt_0)$ is the RTRL state
obtained at time $t_1$ from parameter $\param$.
\end{corollaire}

\begin{proof}
By \redef{olupdate}, for any family of numbers $(\pas_t)$, the quantity $\paramupdate_{t_1:t_2}(\param,
\mem_{t_1}
,(\pas_t))$ is equal to $\paramupdate_{t_1:t_2}(\param,(\pas_t
\,\vtanc_t))$ where
$
\vtanc_t=\sm{V}_t(\param,\Algo_{t_1:t}(\param,\mem_{t_1}))
$.
With $\mem_{t_1}=\Algo_{0:t_1}(\param,\memopt_0)$ we have
$\Algo_{t_1:t}(\param,\mem_{t_1})=\Algo_{0:t}(\param,\memopt_0)$.
The
result follows by the expression for $\sm{V}_t$ in
Lemma~\ref{lem:rtrlcorrectderivatives}.
\end{proof}

\subsection{On the Sequence of Step Sizes}
\label{sec:stepsizes}

Next, let us deal with the assumptions on the learning rate sequence. %
We start by building the scale function $\lng$ used in
Section~\ref{sec:absontadsys} (notably Assumption~\ref{hyp:optimprm}, and
the timescale of Definition~\ref{def:echetps}) from the assumptions in
Section~\ref{sec:presentation}.

\begin{lemme}[Intervals for averaging]
\label{lem:intvavrg}
Let $\exli$ and $\expem'$ be numbers such that $\max\paren{\expem,
\exmp} < \expem'<\exli < \expas-2\,\exmp$ for RTRL and extended RTRL algorithms,
and such that $\max\paren{\expem,\exmp+1/2} < \expem' < \exli<
\expas-2\,\exmp$ for imperfect RTRL algorithms.
In both cases, the set of such pairs $(\expem',\exli)$ is non-empty under Assumption~\ref{hyp:scfcstepsize}.

Define the scale functions
\begin{equation*}
\lng(t)\deq t^\exli , \qquad
\lngav(t)\deq t^{\expem'} , \qquad
\fmdper{t}\deq t^\exmp.
\end{equation*}

Then $\lngav(t)$ and $\fmdper{t}$ are negligible in front of $\lng$, $\lng$ is negligible in front of the identity function, and
$\pas_t\,\flng{t}\,\fmdper{t}^2\to 0$
as $t$ tends to infinity.

Moreover, \rehypdeux{opt_simple}{critoptrtrlnbt} are still satisfied with $\expem'$ instead of $\expem$.
%
%
%
\end{lemme}
\begin{remarque}
In Section~\ref{sec:presentation}, we have presented the assumptions
and results using rates $t^\expem$
and $t^\exmp$.
Actually all our results are valid
as long as these expressions are scale functions
(Definition~\ref{def:fechelle}). This is why we use the more
abstract notation with scale functions $\lng$, $\lngav$, and $\sbfmdper$
in the following.
\end{remarque}

\begin{proof}
First, note that the range of values for $\exli$ is non-empty: indeed,
by Assumption~\ref{hyp:scfcstepsize} we have $\max\paren{\expem,\exmp} + 2 \exmp<\expas$.
For \algonoisy RTRL algorithms, Assumption~\ref{hyp:scfcstepsize}
further states that $\max\paren{\expem,\exmp+1/2} + 2\exmp < \expas$. 
Therefore, the requirements that $\exli<\expas-2\,\exmp$ and
$\exli>\max\paren{\expem,\exmp+1/2}$ are mutually compatible and the range for $A$ is
non-empty. 

We know that $\lngav$ and $\fmdper{\cdot}$ are negligible in front of
$\lng$ since $\exmp<\expem'<\exli$. We have
$\pas_t\,\flng{t}\,\fmdper{t}^2\to 0$ when $t$ tends to infinity, since $-\expas+\exli+2\,\exmp<0$.


Finally, since $\expem' > \expem$,
\rehypdeux{opt_simple}{critoptrtrlnbt} are a fortiori satisfied with $\expem'$ instead of $\expem$.
\end{proof}

\begin{lemme}[Comparison relations for scale functions]
\label{lem:pasconcret}
%
Assume the overall learning rate $\cdvp$ is small enough so that
$\pas_t\leq 1$ for all $t$. Then
under Assumption~\ref{hyp:scfcstepsize}, 
the sequence $1/\pas_t$ is a
scale function.
Moreover,
$\pas_t\,\lngav(t)\,\fmdper{t}^2\to 0$ and
$\fmdper{t}=\po{\lngav(t)}$ as $t\to\infty$.
\end{lemme}

\begin{proof}
By the choice of $\cdvp$, we have $1/\pas_t\geq 1$ for all $t$. Moreover,
$1/\pas_t$ is non-decreasing by assumption on $\pas_t$. Now, by
Assumption~\ref{hyp:scfcstepsize}, $1/\pas_t$ is equivalent to $t^b$
which is a scale function, and therefore, $1/\pas_t$ preserves asymptotic
equivalence at $\infty$.

The last  statements are rewritings of the conditions
$\expem'+2\exmp<\expas$ and $\exmp<\expem'$ from \relem{intvavrg}. 
\end{proof}

\begin{lemme}[Timescale for extended RTRL algorithms]
\label{lem:tmsexrtrlalgo}
With this choice of $\lng$,
the timescale $(\tpsk{k})$ of \redef{echetps} amounts to $\tpsk{0}=0$, $\tpsk{1}=1$ and, for $k\geq1$,
\begin{equation*}
\tpsk{k+1}=\tpsk{k}+\tpsk{k}^\exli.
\end{equation*}
Moreover, it satisfies $\tpsk{k}\sim c\,k^{1/\paren{1-\exli}}$ for some
$c>0$ as $k\to\infty$.
\end{lemme}

\begin{proof}
The first statement is by direct substitution of $\lng(t)=t^\exli$ in
\redef{echetps}. For
the second statement, 
let $\beta \geq 0$. We have
\begin{equation*}
\ba
\tpsk{k+1}^\beta&=\paren{\tpsk{k}+\tpsk{k}^\exli}^\beta
=\tpsk{k}^\beta\,\paren{1+\inv{\tpsk{k}^{1-\exli}}}^\beta
=\tpsk{k}^\beta\paren{1+\frac{\beta}{\tpsk{k}^{1-\exli}}+\po{\inv{\tpsk{k}^{1-\exli}}}}\\
&=\tpsk{k}^\beta + \frac{\beta}{\tpsk{k}^{1-\exli-\beta}}+\po{\inv{\tpsk{k}^{1-\exli-\beta}}},
\ea
\end{equation*}
as $k\to\infty$.
Taking $\beta=1-\exli>0$, we obtain that
$\tpsk{k+1}^{1-\exli}-\tpsk{k}^{1-\exli}\sim 1-\exli$, as $k\to\infty$,
so that $\tpsk{k}^{1-\exli}\sim (1-A)k$ as $k\to\infty$.
\end{proof}

\begin{lemme}[Homogeneity satisfied]
\label{lem:homsatis}
For $t\geq 0$, let $\itv_t$ be the segment 
$\itv_t=[t+1,\,t+\flng{t}]$. Then
\begin{equation*}
\frac{\sup_{s\in \itv_t}\,\pas_s}{\inf_{s\in \itv_t}\,\pas_s}=1+\po{\inv{\fmdper{t}}}
\end{equation*}
as $t$ tends to infinity.
\end{lemme}

\begin{proof}
For every $t\geq 1$, by the definition of $\pas_t$ in
\rehyp{scfcstepsize}, we have
\begin{equation*}
\frac{\sup_{\itv_t}\,\pas_s}{\inf_{\itv_t}\,\pas_s}
=\frac{\pas_{t+1}}{\pas_{t+\flng{t}}}
=\paren{\frac{t+\flng{t}}{t+1}}^\expas\,\frac{\paren{1+\po{\inv{\fmdper{t+1}}}}}{\paren{1+\po{\inv{\fmdper{t+\flng{t}}}}}}.
\end{equation*}
For $t\to \infty$, we have
\begin{equation*}
\ba
\paren{\frac{t+\flng{t}}{t+1}}^\expas
&=\paren{1+\frac{\flng{t}}{t}}^\expas\,\paren{1+\po{\inv{t}}}^\expas
=\paren{1+\inv{t^{1-\exli}}}^\expas\,\paren{1+\po{\inv{t}}}\\
&=\paren{1+\go{\inv{t^{1-\exli}}}}\,\paren{1+\po{\inv{t}}}=1+\go{\inv{t^{1-\exli}}}
=1+\po{\inv{t^{\exmp}}},
\ea
\end{equation*}
since, according to \relem{intvavrg}, we have $0\leq \exli < \expas-2\,\exmp \leq 1-\exmp$, so that $1-\exli\leq 1$ and $1-\exli > \exmp$. Moreover, still when $t\to \infty$, we have
\begin{equation*}
\ba
\frac{\paren{1+\po{\inv{\fmdper{t+1}}}}}{\paren{1+\po{\inv{\fmdper{t+\flng{t}}}}}}
&=\frac{1+\po{\inv{\paren{t+1}^\exmp}}}{1+\po{\paren{\inv{t+t^\exli}}^\exmp}}
=\frac{1+\po{\inv{t^\exmp}}}{1+\po{\inv{t^\exmp}}}\\
&={1+\po{\inv{t^\exmp}}}.
\ea
\end{equation*}
As a result, when $t\to\infty$, we have
\begin{equation*}
\frac{\sup_{\itv_t}\,\pas_s}{\inf_{\itv_t}\,\pas_s}={1+\po{\inv{t^\exmp}}},
\end{equation*}
which ends the proof, since for every $t\geq 1$, $\fmdper{t}=t^\exmp$.
\end{proof}

\begin{corollaire}[Suitable stepsizes]
\label{cor:rtrlsuitstep}
The stepsize sequence $\sm{\pas}=\paren{\pas_t}$, together with the scale function $\lng$, satisfy \rehyp{spdes}, taking $\mlipgrad{t}=\fmdper{t}$.
\end{corollaire}
\begin{proof}
This is a direct consequence of \rehyp{scfcstepsize} and \relem{intvavrg}, and of the fact we use $\mlipgrad{t}=\fmdper{t}$. The homogeneity assumption is a consequence of \relem{homsatis}.
\end{proof}

\subsection{Local Boundedness of Derivatives, Short-Time Control}

\begin{lemme}[Controlling the derivatives of the transition operators
around $\paramopt$]
\label{lem:localopnorms}
Let $B$ be the bound on second derivatives appearing in
\rehyp{regftransetats}. Then for $\param\in
B_\Param(\paramopt,r_\Param)$ and $\state\in
B_{\State_{t-1}}(\stateopt_{t-1},\raysy)$, one has
\begin{equation*}
\nrmop{\frac{\partial
\opevol_t}{\partial
(\state,\param)}(\state,\param)-
\frac{\partial \opevol_t}{\partial
(\state,\param)}(\stateopt,\paramopt)}\leq
B\max(\norm{\state-\stateopt},\norm{\param-\paramopt})
\end{equation*}
and therefore
\begin{equation*}
\sup_{t\geq 1} \,\sup_{
\begin{subarray}{c}
\param\in B_\Param(\paramopt,r_\Param)\\
\state\in B_{\State_{t-1}}(\stateopt_{t-1},\raysy)
\end{subarray}
}
\nrmop{\frac{\partial \opevol_t}{\partial
(\state,\param)}(\state,\param)}
<\infty.
\end{equation*}
\end{lemme}

\begin{proof}
This is a direct consequence of \rehyp{regftransetats}. Indeed, let
$\param\in B_\Param(\paramopt,r_\Param)$ and $\state\in
B_{\State_{t-1}}(\stateopt_{t-1},\raysy)$. For $0\leq u\leq 1$, set
\begin{equation*}
\state_u=(1-u)\stateopt+u\state, \qquad \param_u=(1-u)\paramopt+u\param
\end{equation*}
so that
\begin{equation*}
\frac{\partial \opevol_t}{\partial
(\state,\param)}(\state,\param)=
\frac{\partial \opevol_t}{\partial
(\state,\param)}(\stateopt,\paramopt)
+ \int_{u=0}^1 
\left(\frac{\partial^2 \opevol_t}{\partial
(\state,\param)^2}(\state_u,\param_u)
\right)\cdot \frac{\d (\state_u,\param_u)}{\d u} \d u
\end{equation*}
and now the operator norm of $\frac{\partial^2 \opevol_t}{\partial
(\state,\param)^2}$ is bounded by \rehyp{regftransetats}, and $\frac{\d
(\state_u,\param_u)}{\d u}=(\state-\stateopt,\param-\paramopt)$. This
proves the first claim.

The second claim follows since $\frac{\partial
\opevol_t}{\partial
(\state,\param)}(\stateopt,\paramopt)$ is bounded by assumption, and
$\state-\stateopt$ and $\param-\paramopt$ are bounded by definition in the balls
considered.
\end{proof}

\begin{lemme}
\label{lem:opevollip}
The operators $\opevol_t$ are uniformly Lipschitz on
$B_\Param(\paramopt,r_\Param)$ and
$B_{\State_{t-1}}(\stateopt_{t-1},\raysy)$.
Namely,
there exists a constant $\cst_7\geq 1$ such that for any $t\geq 1$,
for any $\param,\param'\in
B_\Param(\paramopt,r_\Param)$, for any $\state,\state'\in
B_{\State_{t-1}}(\stateopt_{t-1},\raysy)$, one has
\begin{equation*}
\norm{\opevol_t(\state,\param)-\opevol_t(\state',\param')}\leq \cst_7
\max(\norm{\state-\state'},\norm{\param-\param'}).
\end{equation*}
\end{lemme}

\begin{proof}
This is a consequence of \relem{localopnorms}. Indeed, 
for $0\leq u\leq 1$, set as above
\begin{equation*}
\state_u=(1-u)\state+u\state', \qquad \param_u=(1-u)\param+u\param'
\end{equation*}
so that
\begin{equation*}
\opevol_t(\state',\param')=
\opevol_t(\state,\param)
+ \int_{u=0}^1 
\left(\frac{\partial \opevol_t}{\partial
(\state,\param)}(\state_u,\param_u)
\right)\cdot \frac{\d (\state_u,\param_u)}{\d u} \d u
\end{equation*}
and now the operator norm of $\frac{\partial \opevol_t}{\partial
(\state,\param)}$ is bounded by \relem{localopnorms}, and $\frac{\d
(\state_u,\param_u)}{\d u}=(\state-\stateopt,\param-\paramopt)$. This
proves the claim.
\end{proof}

\begin{corollaire}
\label{cor:shorttermlip}
Let $0\leq t_1\leq t_2$. 
Let $(\param_t)$ and $(\param'_t)$ be two sequences of parameters with
$\sup_t \norm{\param_t-\paramopt}\leq \min(r_\Param,r_\State/\cst_7^{t_2-t_1})$
and likewise for $\param'_t$.
Let $\state, \state'\in \State_{t_1}$ with
$\norm{\state-\stateopt_{t_1}}\leq r_\State/\cst_7^{t_2-t_1}$ and
likewise for $\state'$.

Then for every $t_1\leq t\leq t_2$,
\begin{equation*}
\norm{\opevol_{t_1:t}(\state,(\param_t))-\opevol_{t_1:t}(\state',(\param'_t))}\leq
\cst_7^{t-t_1}\max(\norm{\state-\state'},\sup_{t'}
\norm{\param_{t'}-\param'_{t'}})
\end{equation*}
and both $\opevol_{t_1:t}(\state,(\param_t))$ and
$\opevol_{t_1:t}(\state',(\param'_t))$ lie in
$B_\State(\stateopt_t,r_\State)$.
\end{corollaire}

\begin{proof}
By induction from \relem{opevollip}. First consider the case
$\state'=\stateopt_{t_1}$ and $\param'_t=\paramopt$: by induction
from \relem{opevollip}, we obtain that
\begin{equation*}
\norm{\opevol_{t_1:t}(\state,(\param_t))-\opevol_{t_1:t}(\stateopt_{t_1},\paramopt)}\leq
\cst_7^{t-t_1}\max(\norm{\state-\stateopt_{t_1}},\sup_{t'}
\norm{\param_{t'}-\paramopt}) \leq r_\State
\end{equation*}
and therefore, since
$\opevol_{t_1:t}(\stateopt_{t_1},\paramopt)=\stateopt_t$ by definition,
we obtain that $\opevol_{t_1:t}(\state,(\param_t))\in
B_\state(\stateopt_t,r_\State)$. Thus \relem{opevollip} can be applied at
the next step of the induction.

Next, consider the case of general $\state'$. By the first step above, both
$\opevol_{t_1:t}(\state',(\param_t))\in
B_\state(\stateopt_t,r_\State)$ and
$\opevol_{t_1:t}(\state',(\param'_t))\in
B_\state(\stateopt_t,r_\State)$ lie in the ball
$B_\state(\stateopt_t,r_\State)$. So \relem{opevollip} can be
applied at all times $t\leq t_2$, which gives the result by induction.
\end{proof}

\subsection{Spectral Radius Close to $\paramopt$}

\begin{proposition}[Continuity of spectral radius for sequences]
\label{prop:contspecrad}
Let $(A_t)_{t\geq 0}$ be a sequence of linear operators over a normed
vector space, with bounded operator norm. Assume that $(A_t)$ has
spectral radius $\leq 1-\alpha$ at horizon $\hsr$. Then there exists $\eps>0$ such that if
$(A'_t)$ is a sequence of linear operators with
$\nrmop{A_t-A'_t}\leq \eps$ for all $t$, then the sequence $(A'_t)$ has spectral radius
$\leq 1-\alpha/2$ at horizon $\hsr$.
\end{proposition}

\begin{proof}
Writing $A'_t\eqd A_t+r_t$ and
expanding the product $A'_{t+\hsr-1}\ldots...A'_{t+1}A'_t$, one finds
$2^{\hsr}$ terms, one of which is $A_{t+\hsr-1}\ldots...A_{t+1}A_t$ and all the
others involve at least one $r_s$ factor. Therefore, if $\nrmop{r_s}\leq
\frac{\alpha}{2^{\hsr+1}} \min(1,(1/\sup \nrmop{A_t})^\hsr)$, each of those terms has
operator norm $\leq \frac{\alpha}{2^{\hsr+1}}$. So the sum of all the terms
with at least one $r_s$ factor has operator norm
$\leq \alpha/2$ and the conclusion follows.
\end{proof}

\begin{corollaire}[Balls with spectral radius bounded away from $1$]
\label{cor:specrad}
Let $\hsr$ be the horizon for the spectral radius in \rehyp{specrad}.

There exist $r'_\Param>0$, $r'_\State>0$, and $M>0$ such that, for any sequence of
parameters $(\param_t)_{t\geq 0}$ with $\param_t\in
B_\Param(\paramopt,r'_\Param)$ and any sequence of states
$(\state_t)_{t\geq 0}$ with $\state_t\in
B_{\State_t}(\stateopt_t,r'_\State)$, the sequence of operators
\begin{equation*}
\frac{\partial
\opevol_{t}}{\partial
\state}(\state_{t-1},\param_{t-1})
\end{equation*}
has spectral radius at most $1-\alpha/2$ at horizon $\hsr$.
Moreover, any product of such consecutive operators has operator norm
bounded by
\begin{equation*}
\nrmop{
\prod_{t_1<t\leq t_2}
\frac{\partial
\opevol_{t}}{\partial \state}(\state_{t-1},\param_{t-1})
}\leq M(1-\alpha/2)^{(t_2-t_1)/\hsr}.
\end{equation*}
\end{corollaire}

\begin{proof}
One has 
\begin{equation*}
\nrmop{\frac{\partial
\opevol_t}{\partial
(\state,\param)}(\state,\param)-
\frac{\partial \opevol_t}{\partial
(\state,\param)}(\stateopt,\paramopt)}
\geq
\nrmop{\frac{\partial
\opevol_t}{\partial
\state}(\state,\param)-
\frac{\partial \opevol_t}{\partial
\state}(\stateopt,\paramopt)}
\end{equation*}
since any change of $\state$ can be seen as a change of
$(\state,\param)$ with no change on $\param$.

Therefore by \relem{localopnorms}, if $(\state_t,\param_t)$ is close
enough to $(\stateopt_t,\paramopt)$, then $\frac{\partial
\opevol_{t+1}}{\partial
\state}(\state_t,\param_t)$ is arbitrarily close to $\frac{\partial
\opevol_{t+1}}{\partial
\state}(\stateopt_t,\paramopt)$ in operator norm. The spectral radius
property follows
by \rehyp{specrad} and \reprop{contspecrad}.

For the last inequality, divide the time interval $(t_1;t_2]$ into blocks
of length $\hsr$, plus a remainder of length $<\hsr$. On each consecutive block
of length $\hsr$, by definition of the spectral radius of a sequence
(Def.~\ref{def:specrad}), the operator norm of the product is at most
$(1-\alpha/2)$. 
For the remaining interval of length $<\hsr$, define
\begin{equation*}
M_0=\max\cpl{1}{
\sup_{t\geq 1} \,\sup_{
\begin{subarray}{c}
\param\in B_\Param(\paramopt,r_\Param)\\
\state\in B_{\State_{t-1}}(\stateopt_{t-1},\raysy)
\end{subarray}
}
\nrmop{\dpartf{\state}{\opevolt}\paren{\state,\,\param}}},
\end{equation*}
which is finite thanks to \relem{localopnorms}. Thus, a product of $<\hsr$
consecutive operators has operator norm at most $M_0^\hsr$.
Defining
$M=M_0^\hsr/(1-\alpha/2)$ proves the claim (the $1/(1-\alpha/2)$
compensates for $t_2-t_1$ not being an exact multiple of $\hsr$).
\end{proof}

\begin{lemme}[Balls with contractivity at the horizon]
\label{lem:statecontr}
Let $\hsr$ be the horizon for the spectral radius in \rehyp{specrad}, and
$1-\alpha$ the corresponding operator norm.

Define $r''_\Param=
\min(r_\Param,r'_\Param,\,r_\State/\cst_7^\hsr,\,r'_\State/\cst_7^\hsr)$
and
$r''_\State=\min(r_\State,r'_\State)/\cst_7^\hsr$ with $\cst_7$ as in
\recor{shorttermlip}.

Let $\param\in
B_\Param(\paramopt,r''_\Param)$ and let $\state, \state' \in
B_{\State_t}(\stateopt_t,r''_\State)$.


Then
for all $t\leq t'\leq t+\hsr$, $\opevol_{t:t'}(\state,\param)$ and
$\opevol_{t:t'}(\state',\param)$ belong to the ball
$B_{\State_{t'}}(\stateopt_{t'},r'_\State)$, and moreover
\begin{equation*}
\norm{\opevol_{t:t+\hsr}(\state,\param)-\opevol_{t:t+\hsr}(\state',\param)}
\leq (1-\alpha/2)
\norm{\state-\state'}.
\end{equation*}

{In particular, taking $\state'=\stateopt_t$, we see that $\state_{t+\hsr}$ belongs to $B_{\State_{t+\hsr}}(\stateopt_{t+\hsr},r''_\State)$.}
\end{lemme}

\begin{proof}
For $0\leq u\leq 1$ set
$\state_u=(1-u)\state'+u\state$,
which belongs to $B_{\State_t}(\stateopt_t,r''_\State)$.

Then
\begin{equation*}
\opevol_{t:t'}(\state,\param)=\opevol_{t:t'}(\state',\param)+\int_{u=0}^1
\left(\frac{\partial \opevol_{t:t'}}{\partial
\state}(\state_u,\param)
\right)\cdot \frac{\d \state_u}{\d u} \d u.
\end{equation*}

Denote $s_{u,t'}=\opevol_{t:t'}(s_u,\param)$ for $t'\geq t$ the
trajectory starting at $s_u$ with parameter $\param$.
Since $\opevol_{t:t'+1}=\opevol_{t'+1}(\opevol_{t:t'})$, by induction
the derivative of $\opevol_{t:t'}$ is the product of derivatives along
the trajectory:
\begin{equation*}
\frac{\partial \opevol_{t:t'}}{\partial
s}(\state_u,\param)=
\frac{\partial \opevol_{t'}}{\partial s}(\state_{u,t'-1},\param)
\frac{\partial \opevol_{t'-1}}{\partial s}(\state_{u,t'-2},\param)
\cdots \frac{\partial \opevol_{t+1}}{\partial
s}(\state_{u},\param).
%
\end{equation*}

Since $s_u\in B_{\State_t}(\stateopt_t,r''_\State)$, by
\recor{shorttermlip}, for any $t\leq t'\leq t+\hsr$ we have
\begin{equation}
\label{eq:modifpotentielle}
\norm{\opevol_{t:t'}(\state_u,\param))-\opevol_{t:t'}(\stateopt_t,\paramopt)}\leq
\cst_7^{\hsr}\max(\norm{\state_u-\stateopt_t},
\norm{\param-\paramopt})\leq \cst_7^\hsr \max(r''_\State,r''_\Param)\leq
r'_\State
\end{equation}
by our definition of $r''_\Param$ and $r''_\State$.

Since $\opevol_{t:t'}(\stateopt_t,\paramopt)=\stateopt_{t'}$ and
$\opevol_{t:t'}(\state_u,\param)=s_{u,t'}$ by
definition, this means that $\state_{u,t'}$ belongs to
$B_{\State_{t'}}(\stateopt_{t'},r'_\State)$.

Therefore we can apply \recor{specrad}. We obtain that the sequence
$\frac{\partial
\opevol_{t'}}{\partial
\state}(\state_{u,t'-1},\param)
$ for $t\leq t'\leq t+\hsr$, has spectral radius at most $1-\alpha/2$ at
horizon $\hsr$. 
Therefore, taking $t'=t+\hsr$ we have
\begin{equation*}
\nrmop{\frac{\partial \opevol_{t:t+\hsr}}{\partial
s}(\state_u,\param)}\leq 1-\alpha/2.
\end{equation*}

Since $\norm{\frac{\d \state_u}{\d u} }=\norm{\state-\state'}$, the
conclusion follows.
\end{proof}

\subsection{Stable Tubes for RTRL and \Algonoisy RTRL}
\subsubsection{Existence of a Stable Tube for the States $\state_t$}
\label{sec:exstabletubestates}

We are now ready to construct stable tubes for $\opevol$. We cannot construct stable balls in a straightforward way, as contractivity needs $\hsr$ iterations to operate. As a result, we construct two sets of balls in the state spaces $\State_t$ around the target trajectory $\paren{\stateopt_t}$ (of course, all the balls are included in the balls where smoothness, and Lipschitz assumptions, are satisfied). The successive radii are smaller as the number of \gu{primes} increases.
\begin{enumerate}
\item The balls of radius $r_\State$ are those where the regularity \rehyp{regftransetats} is satisfied.
\item The balls of radius $r'_\State$ are those where the several differentials of the model are bounded, and were the $\dpartf{\state}{\opevol_t}$'s have spectral radius less than $1-\alpha$.
\item The balls of radius $r''_\State$ are those the states in which cannot escape from the balls of radius $r'_\State$ in $\hsr$ iterations.
\item The balls of radius $r''''_\State$ are stable by $\hsr$ successive iterations of the $\opevol_t$'s, provided parameters in $B_{\Param}\cpl{\pctrlopt}{r'''_\Param}$ are used.
\item In between times $t$ and $t+\hsr$, the states of trajectories issuing from a ball $B_{\State_t}(\stateopt_t,r''''_\State)$, and using parameters in $B_{\Param}\cpl{\pctrlopt}{r'''_\Param}$, may get out of balls of radius $r''''_\State$, but remain in balls of radius $r'''_\State$.
\end{enumerate}
As a result, every trajectory issuing from a ball $B_{\State_t}(\stateopt_t,r''''_\State)$ at some time $t$, and using parameters in $B_{\Param}\cpl{\pctrlopt}{r'''_\Param}$, will behave as follows.
\begin{enumerate}
\item At every time $t+n\,\hsr$, where $n\geq 0$ is an integer, $\state_{t+n\,\hsr}$ is in $B_{\State_{t+n\,\hsr}}(\stateopt_{t+n\,\hsr},r''''_\State)$.
\item At times $t+r+n\,\hsr$, with $r<\hsr$, $\state_{t+r+n\,\hsr}$ is in a ball $B_{\State_{t+r+n\,\hsr}}(\stateopt_{t+r+n\,\hsr},r'''_\State)$.
\end{enumerate}
Finally, at any time $t$, any state $\state_t\in\tube_t$ is guaranteed to stay in the balls where the $\opevol_t$'s are smooth, and \gu{have spectral radius less than $1-\alpha$} in $\hsr$ iterations: for every $t \leq t' \leq t+k$, for every sequence $\paren{\param_p}$ of parameters in $B_{\Param}\cpl{\pctrlopt}{r'''_\Param}$, we have $\opevol_{t:t'}(\state_t,\paren{\param_{p}})\in B_{\State_{t'}}\paren{\stateopt_{t'},\,r_\State}\cap B_{\State_{t'}}\paren{\stateopt_{t'},\,r'_\State}$ (this is a consequence of \recor{shorttermlip}).


\begin{lemme}[Existence of a stable tube for $s$]
\label{lem:rtrlstabletubes}
There exist a ball $\boctopt\deq B_\Param(\paramopt,r'''_\Param)$ with
positive radius,
and sets $\tube_t\subset
\State_t$ with the following properties:
\begin{enumerate}
\item Stability: for any $\param\in\boctopt$ and any $\state_t\in \tube_t$,
then $\opevol_{t+1}(\state_t,\param)\in \tube_{t+1}$;
\item The sets $\tube_t$
contain a neighborhood of $\stateopt_t$ and have bounded diameter; more
precisely, there exist $r'''_\State>0$ and $r''''_\State>0$ such that
$B_{\State_t}(\stateopt_t,r''''_\State)\subset \tube_t\subset
B_{\State_t}(\stateopt_t,r'''_\State)$ for all $t\geq 0$.
\item $r'''_\Param\leq \min(r_\Param,r'_\Param,r''_\Param)$ and likewise
for $r'''_\State$, so that inside $\boctopt$ and $\tube_t$, all
assumptions of Section~\ref{sec:syshyp} as well as all results
\ref{lem:localopnorms}--\ref{lem:statecontr} apply (with $t_2\leq t_1+\hsr$
for \recor{shorttermlip}).
\end{enumerate}
\end{lemme}

\begin{proof}
Let $\hsr$ be the horizon for the spectral radius in \rehyp{specrad}.


Let $\eps_\param$ and $\eps_\State$ be small enough, to be determined later.
Let $(\theta_t)$ be a sequence of parameters with
$\norm{\param_t-\paramopt}\leq \eps_\param$ and
let $\state\in \State_{t}$ with
$\norm{\state-\stateopt_{t}}\leq \eps_\State$.

By \recor{shorttermlip}, 
 for all $t\leq t'\leq t+\hsr$ one has
\begin{equation*}
\norm{\opevol_{t:t'}(\state,\paramopt)-\opevol_{t:t'}(\stateopt_{t},\paramopt)}\leq
\cst_7^{\hsr}\,\eps_\State
\end{equation*}
provided we take $\eps_\param \leq \eps_\State$ small enough so that the
assumption of \recor{shorttermlip} is met. 


Take $\eps_\State$ smaller than $r''_\State$ from \relem{statecontr}. 
Then we can apply 
\relem{statecontr} to obtain
\begin{equation*}
\norm{\opevol_{t:t+\hsr}(\state,\paramopt)-\opevol_{t:t+\hsr}(\stateopt_t,\paramopt)}\leq
(1-\alpha/2)\norm{\state-\stateopt_t}.
\end{equation*}

Now we have
\begin{multline*}
\norm{\opevol_{t:t+\hsr}(\state,(\param_t))-\stateopt_{t+\hsr}}
=
\norm{\opevol_{t:t+\hsr}(\state,(\param_t))-\opevol_{t:t+\hsr}(\stateopt_t,\paramopt)}
\\\leq
\norm{\opevol_{t:t+\hsr}(\state,(\param_t))-\opevol_{t:t+\hsr}(\state,\paramopt)}
+\norm{\opevol_{t:t+\hsr}(\state,\paramopt)-\opevol_{t:t+\hsr}(\stateopt_t,\paramopt)}
\\\leq \cst_7^\hsr \,\eps_\param+ (1-\alpha/2) \norm{\state-\stateopt_t},
\end{multline*}
where the last inequality follows by applying \recor{shorttermlip} to
$(\state,(\param_t))$ and $(\state,\paramopt)$.

Therefore, if $\norm{\state-\stateopt_t}\leq 2\cst_7^\hsr\,\eps_\param/\alpha$,
then 
\begin{equation*}
\norm{\opevol_{t:t+\hsr}(\state,(\param_t))-\stateopt_{t+\hsr}}\leq
\cst_7^\hsr \,\eps_\param+
(1-\alpha/2)2\cst_7^\hsr\,\eps_\param/\alpha=2\cst_7^\hsr\,\eps_\param/\alpha
\end{equation*}
again. This means that the balls of radius $2\cst_7^\hsr\,\eps_\param/\alpha$ around
$\stateopt_t$ are stable by the application of $k$ consecutive steps of
the transition operator $\opevol$, using any sequence of parameters
$(\param_t)$ such
that $\norm{\param_t-\paramopt}\leq \eps_\param$.

So if we define $\eps_\State=2\cst_7^\hsr\,\eps_\param/\alpha$ (still subject to the
constraints on $\eps_\param$ and $\eps_\State$ above), by induction we obtain
that if $(\theta_t)$ is any sequence of parameters with
$\norm{\param_t-\paramopt}\leq \eps_\param$, and
$\state\in \State_{t}$ with
$\norm{\state-\stateopt_{t}}\leq \eps_\State$, then
\begin{equation*}
\norm{\opevol_{t:t+n\hsr}(\state,(\theta_t))-\stateopt_{t+n\hsr}}\leq \eps_\State
\end{equation*}
for all $n\geq 0$.

This establishes that iterates of an element of a ball of radius
$\eps_\State$ around $\stateopt_t$, stay in such a ball at times that
are
multiples of $\hsr$.

For times in between multiples of $\hsr$, write $n\hsr\leq t< n\hsr+\hsr$ and assume
that $\norm{\state_{n\hsr}-\stateopt_{n\hsr}}\leq \eps_\State$. Then
by \recor{shorttermlip}, one has
\begin{equation*}
\norm{\opevol_{n\hsr:t}(\state_{n\hsr},(\param_t))-\opevol_{n\hsr:t}(\stateopt_{n\hsr},\paramopt)}\leq
\cst_7^{\hsr}\,\max(\eps_\State,\eps_\param)
\end{equation*}
which is bounded.

We can thus set $r'''_\Param\deq \eps_\param$, 
$r'''_\State=\cst_7^{\hsr}\,\max(\eps_\State,\eps_\param)$ and $r''''_\State=\eps_\State$. We then set
$\boctopt\deq B_\Param(\paramopt,r'''_\Param)$ and define, inductively
for $t\geq 1$,
\begin{equation*}
\tube_t \deq \opevol_t(\tube_{t-1},\boctopt)\cup
B_{\State_t}(\stateopt_t,\eps_\State),\qquad \tube_0\deq
B_{\State_0}(\stateopt_0,\eps_\State)
\end{equation*}
so that the sets $\tube_t$ are stable under $\opevol_t$ and contain
a neighborhood of $\stateopt_t$.

Then every element of $\tube_t$ is
an iterate of an element of $B_{\State_{t'}}(\stateopt_{t'},\eps_\State)$
for some $t'\leq t$. Therefore, by the above, $\tube_t$ is contained
in a ball of radius $r'''_\State$ around $\stateopt_t$.
\end{proof}

\begin{corollaire}[Forgetting of states with a fixed parameter]
\label{cor:stateforget}
Let $\hsr$ be the horizon for the spectral radius in \rehyp{specrad}, and
$1-\alpha$ the corresponding operator norm.

Let $\param\in \boctopt$
and let $\state, \state' \in \tube_t$.

Then there exists a constant $\cst_8\geq 0$ such that,
for any $t'\geq t$,
\begin{equation*}
\norm{\opevol_{t:t'}(\state,\param)-\opevol_{t:t'}(\state',\param)}
\leq \cst_8 \,(1-\alpha/2)^{(t'-t)/\hsr}
\norm{\state-\state'}.
\end{equation*}
\end{corollaire}

\begin{proof}
Write $t'-t=r+n\hsr$ with $r<\hsr$. By \relem{rtrlstabletubes}, \recor{shorttermlip} can
be applied inside $\tube_t$ provided $t_2\leq t+\hsr$. With $t_2=t+r$, this yields
\begin{equation*}
\norm{\opevol_{t:t+r}(\state,\param)-\opevol_{t:t+r}(\state',\param)}\leq
\cst_7^\hsr \norm{\state-\state'}.
\end{equation*}
Then
by induction from \relem{statecontr} (whose assumptions are satisfied in
$\tube_{t+r}$, since $r'''_\State \leq r''_\State$, according to \relem{rtrlstabletubes}), we
obtain
\begin{equation*}
\norm{\opevol_{t:t+r+n\hsr}(\state,\param)-\opevol_{t:t+r+n\hsr}(\state',\param)}
\leq \cst_7^\hsr (1-\alpha/2)^n \norm{\state-\state'},
\end{equation*}
from which the conclusion follows by setting
$\cst_8=\cst_7^\hsr/(1-\alpha/2)$ where the factor  $1/(1-\alpha/2)$
accounts for the rounding in the division $(t'-t)/\hsr$.
\end{proof}

\subsubsection{Existence of a Stable Tube for the Jacobians $\jope_t$ and $\tilde
\jope_t$}

\begin{remarque}
Let $A\from \State_t\to \State_{t+1}$ be a linear operator. Equip
$\epjopeins{t}$ with the operator norm. Then the operator norm of $A$
acting on $\epjopeins{t}$ via $J\in \epjopeins{t}\mapsto AJ\in
\epjopeins{t+1}$ is the same
as the operator norm of $A$ acting on $\State_t$.
\end{remarque}

\begin{lemme}
\label{lem:iterspecrad}
Let $(A_t)_{t\geq 1}$ be a sequence of linear operators on normed vector
spaces, with spectral radius at most $1-\alpha$ at horizon $\hsr$. Assume
the $A_t$'s have operator norm at most $\rho$.

Let $(\jope_t)_{t\geq 0}$ and $(\jope'_t)_{t\geq 0}$ be two sequences of
elements of the spaces on which the $A_t$'s act, and suppose that
\begin{equation*}
\jope_t=A_t\jope_{t-1}+B_t,\qquad \jope'_t=A_t\jope'_{t-1}+B'_t
\end{equation*}
for some $B_t$ and $B'_t$. Then for any $0\leq t_1\leq t_2$,
\begin{equation*}
\norm{\jope_{t_2}-\jope'_{t_2}}\leq  \frac{\max(1,\rho^\hsr)}{1-\alpha}\,
\max\left(
(1-\alpha)^{\frac{t_2-t_1}{\hsr}}
\norm{\jope_{t_1}-\jope'_{t_1}},\sup_{t_1\leq
t\leq t_2}(1-\alpha)^{\frac{t_2-t}{\hsr}}\norm{B_t-B'_t}
\right).
\end{equation*}
\end{lemme}

\begin{proof}
By induction we have
\begin{equation*}
\jope_{t_2}=A_{t_2}A_{t_2-1}\cdots
A_{t_1+1}\jope_{t_1}+\sum_{t=t_1+1}^{t_2}
A_{t_2}A_{t_2-1}\cdots A_{t+1} B_t
\end{equation*}
and likewise for $\jope'$ and thus also for $\jope-\jope'$.

Now, for any $t_1\leq t_2$ the product of the operators
$A_{t_2}A_{t_2-1}\cdots
A_{t_1+1}$ has operator norm at most
\begin{equation*}
\nrmop{A_{t_2}A_{t_2-1}\cdots
A_{t_1+1}}\leq {\max(1,\rho^\hsr)}
(1-\alpha)^{\frac{t_2-t_1}{\hsr}-1}.
\end{equation*}
Indeed, we can decompose $t_2-t_1=r+n\hsr$ with $r<\hsr$, and the product of
the first $r$ factors has operator norm at most $\rho^r$ (and $\rho^r\leq
\max(1,\rho)^\hsr$ because $r\leq \hsr$ and the max accounts for whether $\rho$
is larger than $1$ or not). Finally, the product of the remaining $n\hsr$
factors has operator norm at most $(1-\alpha)^n$, and $n\geq
(t_2-t_1)/\hsr-1$.
\end{proof}

\begin{proposition}
\label{prop:specraditer}
Let $0\leq \alpha\leq 1$, $\rho>0$, and $\hsr\geq 0$. Then there exists
$\eps>0$ with the following property.

Let $(A_t)_{t\geq 1}$ be any sequence of linear operators on normed vector
spaces, with operator norm bounded by $\rho$, and with spectral radius at most $1-\alpha$ at horizon $\hsr$.

Let $c \geq 0$. Let
$(\jope_t)_{t\geq 0}$ be any sequence of elements of the spaces on which
the $A_t$'s act, such that
\begin{equation*}
\jope_t=A_t\jope_{t-1}+E_t
\end{equation*}
with $\norm{E_t}\leq c+\eps \norm{\jope_{t-1}}$.

Then $\norm{\jope_t}$ is bounded when $t\to\infty$. More precisely, there
exist constants $a$ and $b$ such that for any $t\geq 0$ and any $t'\geq
t$, $\norm{\jope_{t'}}\leq
a\norm{\jope_t}+b$.

Moreover, the coefficient $a$ depends on $\rho$, $\alpha$ and $\hsr$,
while $b$ depends on $\rho$, $\alpha$, $\hsr$ and $c$. 



\end{proposition}

\begin{proof}
Up to increasing $\rho$, we can assume $\rho\geq 1$.

By induction, for $t'\geq t$ one finds
\begin{equation}
\label{eq:shorttermspecrad}
\norm{\jope_{t'}}\leq (\rho+\eps)^{t'-t}\norm{\jope_t}+(t'-t)\rho^{t'-t}c.
\end{equation}

Moreover, by induction, for $t'\geq t$,
\begin{equation*}
\jope_{t'}=A_{t'}A_{t'-1}\cdots A_{t+1}\jope_t+\sum_{s=1}^{t'-t}
A_{t'}A_{t'-1}\cdots A_{t+s+1}E_{t+s}.
\end{equation*}
Taking $t'=t+\hsr$, using the spectral radius property, then substituting
\eqref{eq:shorttermspecrad}, one finds
\begin{equation*}
\ba
\norm{\jope_{t+\hsr}} &\leq
(1-\alpha)\norm{\jope_t}+\sum_{s=1}^{\hsr}\rho^{\hsr-s}\norm{E_{t+s}}
\\&\leq
(1-\alpha)\norm{\jope_t}+\sum_{s=1}^{\hsr}\rho^{\hsr-s}(c+\eps \norm{\jope_{t+s-1}})
\\&\leq
(1-\alpha)\norm{\jope_t}+\sum_{s=1}^{\hsr}\rho^{\hsr-s}\left(c+\eps(\rho+\eps)^{s-1}\norm{\jope_t}+\eps
(s-1)\rho^{s-1}c\right)
\\&=
\left(1-\alpha+\eps \sum_{s=1}^\hsr\rho^{\hsr-s}(\rho+\eps)^{s-1}\right)\norm{\jope_t}+\sum_{s=1}^{\hsr}\rho^{\hsr-s}\left(c+\eps
(s-1)\rho^{s-1}c\right).
\ea
\end{equation*}

Since $\hsr$ and $\rho$ are fixed, by taking $\eps$ small enough one can
ensure that $1-\alpha+\eps \sum_{s=1}^\hsr\rho^{\hsr-s}(\rho+\eps)^{s-1}$ is
less than $1$. Moreover, the term $\sum_{s=1}^{\hsr}\rho^{\hsr-s}\left(c+\eps
(s-1)\rho^{s-1}c\right)$ does not depend on $t$, so is bounded when
$t\to\infty$.

It results that 
if $t'=t+n\hsr$ for some $n\geq
0$, then $\norm{\jope_{t'}}$ is bounded by $a\norm{\jope_t}+b$ for some
constants $a$ and $b$.

For $t'-t$ not a multiple of $\hsr$, write $t'=t+n\hsr+r$ with $r<\hsr$. Then by
\eqref{eq:shorttermspecrad}, we
have
\begin{equation*}
\norm{\jope_{t'}}\leq (\rho+\eps)^r \norm{\jope_{t+n\hsr}}+r\rho^r c
\end{equation*}
which is bounded as well, hence the conclusion.
\end{proof}

\begin{corollaire}[$\tilde\jope$ is bounded for \algonoisy RTRL
algorithms.]
\label{cor:boundedJ}
Let $(\param_t)$ and $(\state_t)$ be sequences of parameters and states
with $\param_t\in \boctopt$ and $\state_t\in \tube_t$ for all
$t\geq0$.

Consider a sequence $(\tilde \jope_t)_{t\geq \tinit}$ computed as in an \algonoisy RTRL
algorithm (\redef{approxrtrl}) starting at time $\tinit$, namely
\begin{equation*}
\tilde \jope_t = 
\frac{\partial \opevol_t(\state_{t-1},\param_{t-1})}{\partial
\state} \,\tilde \jope_{t-1}+\frac{\partial
\opevol_t(\state_{t-1},\param_{t-1})}{\partial
\param}+E_t ,\qquad \tilde J_\tinit\in \epjopeins{\tinit}
\end{equation*}
where $E_t\in \epjopeins{t}$ satisfies
\begin{equation*}
\nrmop{E_t}\leq \feg{\nrmop{\tilde J_{t-1}}}{\nrmop{\frac{\partial \opevol_t(\state_{t-1},\param_{t-1})}{\partial(\state,\param)}}}
\end{equation*}
for some gauge error $\sbeg$.


Then the sequence $(\tilde \jope_t)$ is bounded.
More precisely, there
exist constants $a$ and $b$ such that for any such sequence $(\tilde
\jope_t)$, for any $t\geq \tinit$ and any $t'\geq
t$, $\norm{\tilde\jope_{t'}}\leq
a\norm{\tilde \jope_t}+b$.
%
\end{corollaire}

\begin{proof}
By \relem{rtrlstabletubes}, the stable tubes $\boctopt$ and $\tube_t$
are included in balls on which all the results up to
\relem{rtrlstabletubes} hold.

So by \relem{localopnorms}, the operator norm $\nrmop{\frac{\partial
\opevol_t(\state_{t-1},\param_{t-1})}{\partial(\state,\param)}}$ is
bounded by some $\rho \geq 0$. Therefore, for all $t\geq 1$, we have 
\begin{equation*}
\nrmop{E_t}\leq \sup_{\abs{y}\leq\rho}\feg{\nrmop{\tilde
J_{t-1}}}{y},
\end{equation*}
and likewise
\begin{equation*}
\nrmop{
\frac{\partial
\opevol_t(\state_{t-1},\param_{t-1})}{\partial
\param}
+
E_t}\leq 
\rho+\sup_{\abs{y}\leq\rho}\feg{\nrmop{\tilde J_{t-1}}}{y}
\end{equation*}
since $\nrmop{\frac{\partial
\opevol_t(\state_{t-1},\param_{t-1})}{\partial
\param}}\leq \nrmop{\frac{\partial
\opevol_t(\state_{t-1},\param_{t-1})}{\partial(\state,\param)}}$.

By \recor{specrad}, the sequence of operators $\left(\frac{\partial
\opevol_t(\state_{t-1},\param_{t-1})}{\partial
\state}\right)_{t\geq 0}$ has spectral radius at most $1-\alpha/2$ at horizon $\hsr$.
Consider the value $\eps>0$ provided by \reprop{specraditer} for this
sequence of operators.

Since $\sbeg$ is an error gauge, thanks to the two properties of \redef{errorgauge}, we can find a constant $c=c\cpl{\rho}{\sbeg}\geq 0$ such that
$
\rho+\sup_{\abs{y}\leq\rho}\feg{\nrmop{\tilde J_{t-1}}}{y}
\leq c+\eps \nrmop{\tilde J_{t-1}}$.
Therefore we can apply \reprop{specraditer} which yields the conclusion.
\end{proof}

\begin{corollaire}[Stable tubes for RTRL and \algonoisy RTRL algorithms]
\label{cor:rtrlstabletubesJ}
The RTRL algorithm and all \algonoisy RTRL algorithms with errors
controlled by the same error gauge $\sbeg$ (\rehyp{errorgauge}) admit a
common stable tube
%
 on $(\state_t,\tilde \jope_t)$.
 
 More precisely, with $(\tube_t)_{t\geq 0}$ the stable
tube for $(\state_t)$ alone from \relem{rtrlstabletubes}, we can find
sets $(\tubejope_t)_{t\geq 0}$ (depending on the error gauge $\sbeg$)
such that 
$(\tube_t\times \tubejope_t)_{t\geq 0}$ is a stable tube 
for the pair $(\state_t,\jope_t)$ of RTRL, as well as for the pair
$(\state_t,\tilde
\jope_t)$ of any \algonoisy RTRL algorithm with error gauge $\sbeg$.
Namely, if
$\param_t\in \boctopt$ and $(\state_t,\jope_t)\in \tube_t\times
\tubejope_t$, then $(\state_{t+1},\jope_{t+1})$ in $\tube_{t+1}\times
\tubejope_{t+1}$ and likewise for $\tilde \jope$.

In particular, trajectories of the \algonoisy RTRL algorithms
respect this stable tube in the sense of
\redef{randomtraj}.

Moreover, the sets $(\tubejope_t)_{t\geq 0}$ are bounded, and for every
$t$, $\tubejope_t$ contains a ball around $0$, whose radius does not
depend on $t$.

Therefore, \rehyp{stableballs} is satisfied, with the same stable tube, for the RTRL algorithm and
any \algonoisy RTRL algorithm which admits $\sbeg$ as an error gauge.
\end{corollaire}

\begin{proof}
First, the RTRL algorithm is also an \algonoisy RTRL algorithm
where errors $E_t=0$ vanish, and therefore it satisfies
\rehyp{errorgauge} (for any error gauge); thus, the previous results of this section for
\algonoisy RTRL algorithms also apply to RTRL.

We have already established the existence of stable tubes for the state
$\state_t$, defined by $\boctopt$ and $\tube_t$ (for which $\jope$ and
$\tilde \jope$ play no role).

Let $\tilde\Algo_t$ be the transition function of the \algonoisy RTRL algorithm, which computes
$(\state_t,\tilde \jope_t)$ from $\param$ and $(\state_{t-1},\tilde
\jope_{t-1})$. Let $\tilde \Algo_t^{\tilde \jope}$ be the part of this function that
just returns $\tilde \jope_t$.

Let $\jopeopt_t$ be the value of $\jope_t$
computed by RTRL
along the target trajectory $\stateopt_t$ defined by $\paramopt$,
namely, $\jopeopt_t=
\frac{\partial \opevol_t(\stateopt_{t-1},\paramopt)}{\partial
\state} \,\jopeopt_{t-1}+\frac{\partial
\opevol_t(\stateopt_{t-1},\paramopt)}{\partial
\param}$
initialized with $\jopeopt_0=0$.  By
\recor{boundedJ}, $\nrmop{\jopeopt_t}$ is bounded by some value
$r_\jope^*$.

Let $r_\jope>0$ be any positive value. Define the stable tubes over
$\jope$ by taking, at each step, the image of the previous values under
any \algonoisy RTRL algorithm and adjoining a ball of fixed radius around $\jopeopt_t$,
namely, define the set
\begin{equation*}
\tubejope_0\deq \{\tilde J_0 \colon \nrmop{\tilde J_0-\jopeopt_0}\leq
r_\jope\}
\end{equation*}
and then, by induction over $t\geq 1$, define $\tubejope_t$ to be the
union of a ball around $\jopeopt_t$ and of all possible values $\tilde
\jope_t$ obtained from a value $\tilde\jope_{t-1}$ in $\tubejope_{t-1}$ by a transition that
respects the assumptions for \algonoisy RTRL algorithms; namely,
\begin{multline*}
\tubejope_t\deq 
\{\tilde J_t \colon \nrmop{\tilde
J_t-\jopeopt_t}\leq
r_\jope\}
\\\cup
\left\{
\frac{\partial \opevol_t(\state,\param)}{\partial
\state} \,\tilde \jope_{t-1}+\frac{\partial
\opevol_t(\state,\param)}{\partial
\param}+E_t ,
\qquad
\state\in \tube_{t-1},\;\param\in\boctopt,\;
\tilde J_{t-1}\in \tubejope_{t-1},
\right.
\\
\left.
\nrmop{E_t}\leq \feg{\nrmop{\tilde J_{t-1}}}{\nrmop{\frac{\partial
\opevol_t(\state,\param)}{\partial(\state,\param)}}}
\right\}
.
\end{multline*}

By
construction of $\tubejope_t$, any element of $\tubejope_t$ is the iterate by
some \algonoisy RTRL algorithm, of an element of the ball $\{\tilde J_{\tinit} \colon
\nrmop{\tilde
J_{\tinit}-\jopeopt_{\tinit}}\leq
r_\jope\}$ for some $\tinit\leq t$, where the algorithm is used at a sequence
of states
$s_t$ and parameters $\param_t$ in the stable tube.

Since $\nrmop{\jopeopt_t}$ is bounded by
$r_\jope^*$,
the elements of the ball $\{\tilde J_{\tinit} \colon
\nrmop{\tilde
J_{\tinit}-\jopeopt_{\tinit}}\leq
r_\jope\}$ are bounded by $r_\jope^*+r_\jope$. 
By construction of $\tubejope_t$, any sequence $\tilde
\jope_t\in \tubejope_t$ respects the assumptions of \recor{boundedJ}.
Therefore, by
\recor{boundedJ}, any such sequence $\tilde \jope_t$ is bounded by
$a(r_\jope^*+r_\jope)+b$. So the sets $\tubejope_t$ are bounded.

Besides, $\tubejope_t$ contains a ball around $\jopeopt_t$ by definition.
Therefore,
the sets $\tube_t\times \tubejope_t$ constitute a stable tube for
the pair $(\state_t,\tilde \jope_t)$.

Finally, since $\tubejope_t$ contains a ball of radius $r_\jope$ around $\jopeopt_t$, and
since $\norm{\jopeopt_t}\leq r_\jope^*$, if we choose $r_\jope>r_\jope^*$,
we ensure that $\tubejope_t$ contains a ball around $0$.
\end{proof}

\begin{remarque}[Taylor expansions in the stable tube] When we say ''in
the stable tube'' in the following, it means that both the state $\state$
and the Jacobian $\jope$ considered belong to their respective stable
tubes, and that the parameter $\param$ considered belongs to $\boctopt$.
Therefore, all the assumptions of Section~\ref{sec:presentation} hold
``on the stable tube''; more precisely, on balls containing the stable
tube.

In the following, we shall repeatedly use the following argument: ``Since
$\state$ and $\state'$ are in the stable tube, and since $\partial_\state
f$ is bounded on the stable tube by one of the assumptions of
Section~\ref{sec:presentation}, then
$\norm{f(\state)-f(\state')}=\go{\norm{\state-\state'}}$.'' However, the stable
tube is not convex, so integrating $\partial_\state f$ on the segment from
$\state$ to $\state'$ may exit the stable tube. Still, every assumption of
Section~\ref{sec:presentation} holds on a ball containing the stable tube;
these balls are convex so the argument is valid. We will implicitly use
this argument and just say ``the derivative of $f$ is bounded on the
stable tube.''
\end{remarque}

\subsection{Lipschitz-Type Properties of the Transition Operator of RTRL}

Remember that $\tube_t$ and $\tubejope_t$ are the stable tubes for
$\state_t$ and $\jope_t$ introduced in \relem{rtrlstabletubes} and
\recor{rtrlstabletubesJ}.

\begin{lemme}[Exponential forgetting of $(\state,\jope)$ for 
RTRL with fixed $\param$]
\label{lem:expforgstjrtrlfxprm}
Let $\Algo$ be the RTRL algorithm on $(\state,\jope)$ from \redef{rtrlasalgo}.

Then there exists a constant $\cst_9$ such that the following holds. For any
$t_0\geq 0$, for any $\param\in\boctopt$, for any $\state,\state'\in
\tube_{t_0}$ and
$\jope,\jope' \in\tubejope_{t_0}$, for any $t\geq t_0$,
\begin{equation*}
\norm{\Algo_{t_0:t}(\param,(\state,\jope))-\Algo_{t_0:t}(\param,(\state',\jope'))}\leq
\cst_9 \,(1-\alpha/2)^{\frac{t-t_0}{\hsr}}\,
\max(\norm{\state-\state'},\norm{\jope-\jope'})
\end{equation*}
and in particular, \rehyp{expforgetinit} is satisfied for the RTRL
algorithm.
\end{lemme}

\begin{proof}
Define the trajectory $(\state_{t},\jope_{t})=\Algo_{t_0:t}(\param,(\state,\jope))$ and
likewise for $(\state',\jope')$. Thus, we have to bound
$\norm{(\state_t,\jope_t)-(\state_t',\jope_t')}$.

The norm on pairs $(\state,\jope)$ is the max-norm (\redef{rtrlasalgo}),
so it is enough to prove the statement separately for
$\norm{\state_t-\state'_t}$ and
$\norm{\jope_t-\jope'_t}$.

By \redef{rtrlasalgo} of the RTRL algorithm, we have
$\state_{t}=\opevol_{t_0:t}(\state,\param)$, so the conclusion for
$\state$ is
exactly \recor{stateforget}.

By \redef{rtrlasalgo}, the sequence $(\jope_{t})$ satisfies
\begin{equation*}
\jope_t = \frac{\partial \opevol_t(\state_{t-1},\param)}{\partial
\state} \,\jope_{t-1}+\frac{\partial
\opevol_t(\state_{t-1},\param)}{\partial
\param}
\end{equation*}
for $t\geq t_0+1$, and likewise for $\jope'$. Denote
\begin{equation*}
A_t\deq \frac{\partial \opevol_t(\state_{t-1},\param)}{\partial
\state},\qquad B_t\deq \frac{\partial
\opevol_t(\state_{t-1},\param)}{\partial
\param}
\end{equation*}
and likewise for $\jope'$. Therefore, we have $\jope_t=A_t\jope_{t-1}+B_t$
and
\begin{equation*}
\jope'_t=A'_t\jope'_{t-1}+B'=A_t\jope'_{t-1}+B'_t+(A'_t-A_t)\jope'_{t-1}.
\end{equation*}

Since the trajectories lie in the stable tube, by \recor{specrad}, the
sequence $(A_t)$ has spectral radius at most $1-\alpha/2$ at horizon $\hsr$.

In order to apply \relem{iterspecrad} to $\jope_t$ and
$\jope'_t$, we have to bound
\begin{equation*}
B_t-B'_t-(A'_t-A_t)\jope'_{t-1}
\end{equation*}
for each $t$.

By definition $B_t= \frac{\partial
\opevol_t(\state_{t-1},\param)}{\partial
\param}$. By \rehyp{regftransetats}, the second derivative
$\frac{\partial^2
\opevol_t(\state_{t-1},\param)}{\partial
(\state,\param)^2}$ is bounded on the stable tube.
This implies that
$\frac{\partial
\opevol_t(\state_{t-1},\param)}{\partial
\param}$ is a Lipschitz function of $\state_{t-1}$ on the stable tube,
and in particular
\begin{equation*}
B_t-B'_t=O(\norm{\state_t-\state'_t}).
\end{equation*}

The same reasoning applies to $A_t=\frac{\partial
\opevol_t(\state_{t-1},\param}{\partial \state})$, so that
\begin{equation*}
A_t-A'_t=O(\norm{\state_t-\state'_t})
\end{equation*}
and the sequence $(\jope'_t)$ belongs to the stable tube, so by
\recor{rtrlstabletubesJ}, $(\jope'_t)$ is bounded. Consequently,
\begin{equation*}
(A'_t-A_t)\jope'_t=O(\norm{\state_t-\state'_t})
\end{equation*}
as well, and therefore
\begin{equation*}
B_t-B'_t-(A'_t-A_t)\jope'_{t-1}=O(\norm{\state_t-\state'_t}).
\end{equation*}

So \relem{iterspecrad} applied to the sequences $\jope_t$ and $\jope'_t$ provides
\begin{equation*}
\norm{\jope_t-\jope'_t}\leq
\go{\max\left((1-\alpha/2)^{\frac{t-t_0}{\hsr}}\norm{\jope_{t_0}-\jope'_{t_0}},
\sup_{t_0\leq t'\leq t} (1-\alpha/2)^{\frac{t-t'}{\hsr}}\norm{\state_{t'}-\state'_{t'}}\right)
}.
\end{equation*}

But by \recor{stateforget},
\begin{equation*}
\norm{\state_t-\state'_t}\leq
\cst_8\,(1-\alpha/2)^{\frac{t-t_0}{\hsr}}\norm{\state_{t_0}-\state'_{t_0}}
\end{equation*}
so that for $t_0\leq t'\leq t$, one has
\begin{equation*}
\ba
(1-\alpha/2)^{\frac{t-t'}{\hsr}}\norm{\state_{t'}-\state'_{t'}}
&\leq \cst_8\,
(1-\alpha/2)^{\frac{t-t'}{\hsr}}\,(1-\alpha/2)^{\frac{t'-t_0}{\hsr}}
\norm{\state_{t_0}-\state'_{t_0}}
\\&=\cst_8 \,(1-\alpha/2)^{\frac{t-t_0}{\hsr}}
\norm{\state_{t_0}-\state'_{t_0}}
\ea
\end{equation*}
and therefore
\begin{equation*}
\norm{\jope_t-\jope'_t}\leq
\go{(1-\alpha/2)^{\frac{t-t_0}{\hsr}}\max\left(\norm{\jope_{t_0}-\jope'_{t_0}},
\norm{\state_{t_0}-\state'_{t_0}}
\right)
}
\end{equation*}
as needed.

\end{proof}

\begin{corollaire}[RTRL is Lipschitz wrt $\param$]
\label{cor:rtrllipswrtprm}
Let $\Algo$ be the RTRL algorithm on $(\state,\jope)$
from \redef{rtrlasalgo}. Then $\Algo$ satisfies \rehyp{paramlip}. Namely,
for any $(\state_t,\jope_t)\in\tube_t\times \tubejope_t$ and any
$\param,\param'\in \boctopt$, one has
\begin{equation*}
\norm{\Algo_{t+1}(\param,(\state_t,\jope_t))-\Algo_{t+1}(\param',(\state_t,\jope_t))}\leq
\cliprtrlprm\norm{\param-\param'}
\end{equation*}
for some $\cliprtrlprm\geq 0$.
\end{corollaire}

\begin{proof}
By definition of $\Algo$, we have
$\Algo_{t+1}(\param,(\state_t,\jope_t))=(\state_{t+1},\jope_{t+1})$ where
$\state_{t+1}=\opevol_{t+1}(\state_t,\param)$ and
\begin{equation*}
\jope_{t+1}=\frac{\partial \opevol_{t+1}(\state_t,\param)}{\partial
\state}\,\jope_t+\frac{\partial
\opevol_{t+1}(\state_t,\param)}{\partial \param}.
\end{equation*}

Now, by \relem{localopnorms}, $\frac{\partial
\opevol_{t+1}(\state_t,\param)}{\partial \param}$ is bounded on the
stable tube, so that $\opevol_{t+1}$ is Lipschitz with respect to
$\param$ on the stable tube.
This shows that $\state_{t+1}$ depends on $\param$ in
a Lipschitz way. Since the bound in \relem{localopnorms} is uniform in
$t$, the Lipschitz constant is uniform in $t$.

For $\jope_{t+1}$, by
\rehyp{regftransetats}, the second derivative
$\nrmop{\frac{\partial^2 \opevol_{t+1}}{\partial
(\state,\param)^2}(\state,\param)}$ is bounded on the stable tube. This
implies that both $\frac{\partial
\opevol_{t+1}(\state_t,\param)}{\partial \param}$ and $\frac{\partial
\opevol_{t+1}(\state_t,\param)}{\partial \state}$ are Lipschitz with
respect to $\param$ in the stable tube. Since $\jope_t$ is bounded in the
stable tube, this implies that $\jope_{t+1}=\frac{\partial
\opevol_{t+1}(\state_t,\param)}{\partial
\state}\,\jope_t+\frac{\partial
\opevol_{t+1}(\state_t,\param)}{\partial \param}
$ is Lipschitz with respect to $\param$.
Since the bound in \rehyp{regftransetats} is uniform in $t$, the
Lipschitz constant is uniform in $t$.
\end{proof}

\subsection{Boundedness of Gradients for RTRL}

\begin{lemme}[Boundedness of gradients for RTRL]
\label{lem:boundgradrtrl}
The gradient computation operators $\sm{V}_t$ from \redef{rtrlasalgo} satisfy \rehyp{boundedgrads}
and \rehyp{contgrad}; the scale functions $\fmdper{t}$ and $\mlipgrad{t}$
appearing in these assumptions are both equal to the scale function 
$\fmdper{t}$ appearing in \rehyp{regpertes}.
\end{lemme}

Consistently with \rehyp{boundedgrads}, we denote $\bvtg{t}$ the ball of $\Tangent_t$ with radius $\sup_{\param\in
\boctopt}\,\sup_{\mem\in \bmaint{t}}
\norm{\sm{V}_t(\param,\mem)}$.

\begin{proof}
From \redef{rtrlasalgo},
\begin{equation*}
\sm{V}_t(\param,(\state,\jope))= \furt{\frac{\partial
\perte_t(\state)}{\partial \state}\cdot\jope}{\state}{\param}
\end{equation*}
and we have to prove that this quantity
is bounded by $O(\fmdper{t})$ on the stable tube
(\rehyp{boundedgrads})
and 
is
$O(\fmdper{t})$-Lipschitz with
respect to $(\param,(\state,\jope))$
on the stable tube (\rehyp{contgrad}).

Let us first study the auxiliary quantity
\begin{equation*}
\sm{V}^1_t(\state,\,\jope)= \frac{\partial
\perte_t(\state)}{\partial \state}\cdot\jope.
\end{equation*}

By \rehyp{regpertes}, the second derivative of $\perte_t$
with respect to $\state$ is $\go{\fmdper{t}}$ on the stable tube. This proves
that $\frac{\partial
\perte_t(\state)}{\partial \state}$ is $\go{\fmdper{t}}$--Lipschitz with respect to
$\state$ on the stable tube. Since $\jope$ is bounded on the
stable tube, it follows that $\sm{V}^1_t(\state,\jope)$ is 
$\go{\fmdper{t}}$--Lipschitz with respect to $\state$ on the stable tube.

Since $\frac{\partial
\perte_t(\state)}{\partial \state}$ is Lipschitz with respect to
$\state$ on the stable tube, and its values at
$(\stateopt_t)$ are bounded by $O(\fmdper{t})$ by
\rehyp{regpertes}, and since the stable tube is bounded around
$(\stateopt_t)$, it follows that $\frac{\partial
\perte_t(\state)}{\partial \state}$ is bounded by $O(\fmdper{t})$ on
the stable tube.

This proves, first, that $\sm{V}^1_t(\state,\jope))$ is
$O(\fmdper{t})$-Lipschitz with respect to $\jope$ on the stable tube. Second, since $\jope$ is
bounded on the stable tube, this proves that
$\sm{V}^1_t(\state,\jope)$ is bounded by $O(\fmdper{t})$ on
the stable tube.

Thus, 
$\sm{V}^1_t(\state,\jope)$ is $O(\fmdper{t})$-Lipschitz with respect to ${\cpl{\state}{\jope}}$ on the stable tube.

Let us now consider $\cpl{\param}{\cpl{\state}{\jope}}$ and $\cpl{\param'}{\cpl{\state'}{\jope'}}$ on the stable tube. Let $t\geq 0$. Then, the difference $\sm{V}_t\cpl{\param}{\cpl{\state}{\jope}}-\sm{V}_t\cpl{\param'}{\cpl{\state'}{\jope'}}$ equals
\begin{multline*}
\furt{\frac{\partial \perte_t(\state)}{\partial \state}\cdot\jope}{\state}{\param}
-\furt{\dpartf{\state}{\perte_t}\paren{\state'}\cdot\jope'}{\state'}{\param'}\\
=\furt{\frac{\partial \perte_t(\state)}{\partial \state}\cdot\jope}{\state}{\param}
-\furt{\dpartf{\state}{\perte_t}\paren{\state}\cdot\jope}{\state'}{\param'}\\
+\furt{\dpartf{\state}{\perte_t}\paren{\state}\cdot\jope}{\state'}{\param'}
-\furt{\dpartf{\state}{\perte_t}\paren{\state'}\cdot\jope'}{\state'}{\param'}.
\end{multline*}
As a result, the difference $\sm{V}_t\cpl{\param}{\cpl{\state}{\jope}}-\sm{V}_t\cpl{\param'}{\cpl{\state'}{\jope'}}$ is bounded by (writing $B_t^\jope$ a ball containing $\tubejope_t$, which radius we can chose independant of $t$ since the stable tube is bounded)
\begin{multline*}
\sup_{\boctopt\times B_{\State_t}(\stateopt_{t},\raysy)\times B_t^\jope} \paren{\nrmop{\dpartf{\cpl{\state}{\param}}{\fur_t}\tpl{\sm{V}^1_t{\cpl{\state''}{\jope''}}}{\state''}{\param''}}}\,
\nrm{\cpl{\state}{\param}-\cpl{\state'}{\param'}}\\
+\sup_{\boctopt\times B_{\State_t}(\stateopt_{t},\raysy)\times B_t^\jope}\,\paren{\nrmop{\dpartf{\vtanc}{\fur_t}\tpl{\sm{V}^1_t{\cpl{\state''}{\jope''}}}{\state''}{\param''}}}\,
\nrm{\sm{V}^1_t{\cpl{\state}{\jope}}-\sm{V}^1_t{\cpl{\state'}{\jope'}}}.
\end{multline*}
Thanks to \rehyp{fupdrl}, we therefore know that the difference is bounded by 
\begin{multline*}
\sup_{\boctopt\times B_{\State_t}(\stateopt_{t},\raysy)\times  B_t^\jope} \go{1+\nrm{\sm{V}^1_t{\cpl{\state''}{\jope''}}}}\,
\nrm{\cpl{\state}{\param}-\cpl{\state'}{\param'}}\\
+\sup_{\boctopt\times B_{\State_t}(\stateopt_{t},\raysy)\times B_t^\jope}\,\paren{\nrmop{\dpartf{\vtanc}{\fur_t}\tpl{\sm{V}^1_t{\cpl{\state''}{\jope''}}}{\state''}{\param''}}}\,
\nrm{\sm{V}^1_t{\cpl{\state}{\jope}}-\sm{V}^1_t{\cpl{\state'}{\jope'}}}\\
=\go{\fmdper{t}}\,\nrm{\cpl{\state}{\param}-\cpl{\state'}{\param'}}
+\sup_{\boctopt\times B_{\State_t}(\stateopt_{t},\raysy)\times B_t^\jope}\,\paren{\nrmop{\dpartf{\vtanc}{\fur_t}\tpl{\sm{V}^1_t{\cpl{\state''}{\jope''}}}{\state''}{\param''}}}\times\\
\go{\fmdper{t}}\,\nrm{{\cpl{\state}{\jope}}-{\cpl{\state'}{\jope'}}}.
\end{multline*}
Thanks to \rehyp{fupdrl}, the derivative $\dpartf{\vtanc}{\fur_t}$ is
bounded on a ball which contains the stable tube. This shows that $\sm{V}_t$ is
$\go{\fmdper{t}}$-Lipschitz with respect to $\param$, $\state$ and
$\jope$ on the stable tube.

Now, for every $t\geq 0$, for every $(\param,(\state,\jope))$ in the stable tube, we have
\begin{equation*}
\ba
\sm{V}_t\cpl{\param}{\cpl{\state}{\jope}}=
\sm{V}_t\cpl{\param}{\cpl{\state}{\jope}}-\sm{V}_t\cpl{\paramopt}{\cpl{\stateopt_t}{0}}+\sm{V}_t\cpl{\paramopt}{\cpl{\stateopt_t}{0}}
\ea
\end{equation*}
and the last term is
$\sm{V}_t\cpl{\paramopt}{\cpl{\stateopt_t}{0}}=\furt{0}{\stateopt_t}{\paramopt}$
by definition.

Thanks to \rehyp{fupdrl}, this last term is $\go{\fmdper{t}}$. 
Finally, since $\norm{\param-\paramopt}$, $\norm{\state-\stateopt_t}$ and
$\norm{\jope}$ are bounded on the stable tube, and since $\sm{V}_t$ is
$\go{\fmdper{t}}$-Lipschitz with respect to $\param$, $\state$ and
$\jope$ on the stable tube, the difference
$\sm{V}_t\cpl{\param}{\cpl{\state}{\jope}}-\sm{V}_t\cpl{\paramopt}{\cpl{\stateopt_t}{0}}$
is $\go{\fmdper{t}}$ on the stable tube.
As a result, $\sm{V}_t\cpl{\param}{\cpl{\state}{\jope}}$ is $\go{\fmdper{t}}$ on the stable tube.
\end{proof}

\section{Proving Convergence of the \rtrl Algorithm and of \Algonoisy \rtrl Algorithms}
\label{sec:prcvalappxalgo}

Let us go on with applying the results of Section~\ref{sec:absontadsys}
under the assumptions of Section~\ref{sec:presentation}.
So far, we have shown that \algonoisy RTRL algorithms admit a stable
tube, thus satisfying \rehyp{stableballs}, that the exponential
forgetting and Lipschitz \rehypdeux{expforgetinit}{paramlip} are
satisfied, that the gradient computation operators are bounded on the
stable tube, and Lipschitz, so that \rehypdeux{boundedgrads}{contgrad}
are satisfied,
and that the sequence of learning rates satisfies
\rehyp{spdes}. This leaves \rehyp{updateop} and \rehyp{optimprm}.

We now check that the local optimality \rehyp{optimprm} is valid for the RTRL
algorithm. This requires to check that the update operator leaves
$\paramopt$ almost unchanged over time intervals $[t;t+\flng{t}]$, and
brings other parameters closer to $\paramopt$ over these intervals.
This follows from
\rehyp{critoptrtrlnbt} which states that $\paramopt$ is a critical point
of the average loss function, asymptotically; for this we have to
transfer this asymptotic property to finite but long enough intervals
$[t;t+\flng{t}]$. This is the object of \resecdeux{stabprmitv}{contaroundparamopt}.

We also show that the parameter update operator satisfies \rehyp{updateop} in
\resec{prmupdfopas}, and that \algonoisy RTRL algorithms have negligible noise, in the sense of \redef{negnoise}, in \resec{noisectrl}.

\subsection{Parameter Updates at First Order in $\pas$}
\label{sec:prmupdfopas}
Here we study the parameter update operator and its iterates at first
order in the step size.

\begin{proposition}[Checking conformity of the parameter update operator]
\label{prop:prmupdatesfirstorderstepsize}
\rehyp{updateop_first} implies \rehyp{updateop}, over balls of the same radius
$r_\Tangent=\tilde r_\Tangent$.

We then define $\brnp$ as in \redef{bornepas}.
\end{proposition}

\begin{proof}
Let us check the conclusions of \rehyp{updateop}. The first point is
trivial under \rehyp{updateop_first}. For the second point,
under \rehyp{updateop_first} we have
\begin{align*}
\paramupdate_t(\param,\vtanc)-\paramupdate_t(\param',\vtanc') &=
\param-\param'-\vtanc+\vtanc'+\norm{\vtanc}^2\paramupdate_t^{(2)}(\param,\vtanc)-\norm{\vtanc'}^2\paramupdate_t^{(2)}(\param',\vtanc')
\\&=\param-\param'-(\vtanc-\vtanc')+
\norm{v}^2(\paramupdate_t^{(2)}(\param,\vtanc)-\paramupdate_t^{(2)}(\param',\vtanc'))\\
&+\paren{\norm{\vtanc}^2-\norm{\vtanc'}^2}\paramupdate_t^{(2)}(\param',\vtanc')
\end{align*}
and let us control each term separately. Since $\norm{\vtanc}^2$ is bounded
for $\vtanc \in B_{\Tangent_t}(0,\tilde r_\Tangent)$, and since
$\paramupdate_t^{(2)}$ is Lipschitz, the norm of
$\norm{v}^2(\paramupdate_t^{(2)}(\param,\vtanc)-\paramupdate_t^{(2)}(\param',\vtanc'))$
is controlled by $\norm{\vtanc-\vtanc'}+\norm{\param-\param'}$. Finally,
writing
$\norm{\vtanc}^2-\norm{\vtanc'}^2=(\norm{\vtanc}-\norm{\vtanc'})(\norm{\vtanc}+\norm{\vtanc'})$,
we see that the norm of
$\paren{\norm{\vtanc}^2-\norm{\vtanc'}^2}\paramupdate_t^{(2)}(\param',\vtanc')$
is controlled by $\norm{\vtanc}-\norm{\vtanc'}$ times the supremum of
$(\norm{\vtanc}+\norm{\vtanc'})\paramupdate_t^{(2)}(\param',\vtanc')$
over the ball; since both $\norm{\vtanc}+\norm{\vtanc'}$ and
$\paramupdate_t^{(2)}$ are uniformly bounded in the ball, the conclusion
follows.
\end{proof}

\begin{lemme}[Intervals of lengths $\flng{t}$]
\label{lem:itvlgtflng}
For any $\cdvp\leq\brnp$, for $t$ large enough, we have
$\flng{t}<T_t^{r_\Theta^*}$ (Def.~\ref{def:tvalpc}).
\end{lemme}
As a consequence, we can apply the abstract results to intervals of length $\flng{t}$.
\begin{proof}
This is the same argument as in Lemma~\ref{lem:minmthoriztpsk}, which
actually
applies to any interval $(t,t+\flng{t}]$, not only to intervals
$(\tpsk{k},\,\tpsk{k}+\flng{\tpsk{k}}]$.
\end{proof}

\begin{lemme}[Parameter updates at first order]
\label{lem:citersbopexp}
Let $t_0\geq 0$.
Let $\param_{t_0}\in \boctopt$ with $\dist{\param_{t_0}}{\paramopt}$
at most $\rayoptct/3$, and let $\paren{\vtanc_t}$
be a gradient sequence with $\vtanc_t\in
\bvtg{t}$
for all $t\geq t_0$ (where $\bvtg{t}$ is defined just after \relem{boundgradrtrl}). Let $\paren{\pas_t}$ be a stepsize sequence
with $\cdvp\leq \brnp$. Define by induction for $t>t_0$,
\begin{equation*}
\param_{t} =
\paramupdate_t\paren{\param_{t-1},\,\pas_t\,\vtanc_t}=\paramupdate_{t_0:t}(\param_{t_0},(\pas_t\,\vtanc_t)_t)
\end{equation*}
	Then for any $t_0 \leq t < t_0+\thoriz{\tinit}\paren{\sm{\pas}}$, 
\begin{equation*}
\param_t=\param_{t_0}-\sum_{s=t_0+1}^t \pas_s \vtanc_s+
\go{\sum_{s=t_0+1}^t \pas_s^2 \,\fmdper{s}^2}.
\end{equation*}
\end{lemme}

\begin{proof}

By \relem{maintientpsfinibocont}, the trajectory stays in the control
ball for $t<\thoriz{\tinit}\paren{\sm{\pas}}$.

By induction 
from \rehyp{updateop_first} we find
\begin{equation*}
\param_t=\param_0-\sum_{s=\tinit+1}^t \pas_s
\,\vtanc_s+\sum_{s=\tinit+1}^t\pas_s^2
\norm{\vtanc_s}^2\paramupdate_s^{(2)}(\param_{s-1},\pas_s\vtanc_s).
\end{equation*}

Let us bound the last term. We can apply 
\recor{normpasv}, because all assumptions used for proving this corollary
have been checked.
\recor{normpasv} yields
$\norm{\pas_s \vtanc_s}\leq r_\Tangent$, and in \reprop{prmupdatesfirstorderstepsize} we ensured $r_\Tangent \leq \tilde{r}_\Tangent$. Since $\param_{s-1}$ is in the
control ball, the term
$\paramupdate_s^{(2)}(\param_{s-1},\pas_s\vtanc_s)$ is bounded by
\rehyp{updateop_first}. By \recor{normpasv}, $\norm{\vtanc_s}^2$ is
$O(\fmdper{s}^2)$. Therefore, the last term is $O(\sum \pas_s^2\,
\fmdper{s}^2)$, so that, for any $t_0 \leq t < t_0+\thorizo\paren{\sm{\pas}}$, we have
\begin{equation*}
\param_t=\param_{t_0}-\sum_{s=t_0+1}^t \pas_s \vtanc_s+
\go{\sum_{s=t_0+1}^t \pas_s^2 \,\fmdper{s}^2}.
\end{equation*}
\end{proof}

We now prove a similar property using constant step sizes.

\begin{lemme}[Parameter updates at first order with fixed step size]
\label{lem:firstorderupdate}
Let $T\geq 0$. Let $\param_{T}\in \boctopt$ with $\dist{\param_T}{\paramopt}$
at most $\rayoptct/3$, and let $\paren{\vtanc_t}$
be a gradient sequence with $\vtanc_t\in
\bvtg{t}$
for all $t\geq T$ ($\bvtg{t}$ is defined just after \relem{boundgradrtrl}). Let $\paren{\pas_t}$ be a stepsize sequence
with $\cdvp\leq \brnp/2$. Define by induction for $t>T$,
\begin{equation*}
\param_{t} =
\paramupdate_t\paren{\param_{t-1},\,\pas_T\,\vtanc_t}=\paramupdate_{T:t}(\param_{T},(\pas_T\,\vtanc_t)_t).
\end{equation*}
Then for any $T \leq t \leq T+\flng{T}$,
\begin{equation*}
\param_t=\param_{T}-\pas_T \sum_{s=T+1}^t \vtanc_s+
\go{\flng{T}\,\pas_T^2\,\fmdper{T}^2}.
\end{equation*}
\end{lemme}

\begin{proof}
For $T \leq t \leq T+\flng{T}$,
define $\vtanc'_t\deq (\pas_T/\pas_t)\vtanc_t$, so that
\begin{equation*}
\paramupdate_{T:t}(\param_{T},(\pas_T\,\vtanc_t))=\paramupdate_{T:t}(\param_{T},(\pas_t\,\vtanc'_t)).
\end{equation*}

Since we have assumed
$\cdvp\leq \brnp/2$, we can apply \recor{normpasv} with stepsize sequence
$2\pas_t$: this proves that $\norm{2\pas_t\vtanc_t}\leq r_\Tangent$.
By \relem{stronghomo}, $\pas_T/\pas_t$ is $1+o(1)$. Therefore, for $T$
large enough, we have $\pas_T/\pas_t<2$.  
Therefore,
$\norm{\pas_t\vtanc'_t}=\norm{(\pas_T/\pas_t)\,\pas_t\vtanc_t}\leq
r_\Tangent$ for $T$ large enough.
(What happens for small $T$ is absorbed in the
$\go{}$ notation.)

Since $\norm{\pas_t\vtanc'_t}\leq r_\Tangent$, we can reason exactly as
in \relem{citersbopexp} using $\vtanc'_t$ instead of $\vtanc_t$. 
This
yields 
\begin{equation*}
\param_t=\param_{T}-\sum_{s=T+1}^t \pas_s \vtanc'_s+
\go{\sum_{s=T+1}^t \pas_s^2 \,\fmdper{s}^2}
\end{equation*}
valid at least up to time $t=T+\flng{T}$ thanks to \relem{itvlgtflng}.
This yields the conclusion after substituting
$\vtanc'_s=(\pas_T/\pas_s)\vtanc_s$, and after observing that
\begin{equation*}
\sum_{s=T+1}^{T+\flng{T}}\pas_s^2\,\fmdper{s}^2=\go{\flng{T}\,\pas_T^2\,\fmdper{T}^2}
\end{equation*}
	because for $T<s\leq T+\flng{T}$ we have $\pas_s/\pas_T=1+o(1)$
	(by \relem{stronghomo}) and because
$\fmdper{s}\sim \fmdper{T}$ as $\fmdper{}$ is a scale function.
%
%
\end{proof}

\subsection{Stability of $\paramopt$ on Intervals $(T;T+\flng{T}]$}
\label{sec:stabprmitv}

\begin{lemme}[Averages over time intervals]
\label{lem:somneglint}
Let $\paren{\scsp_t}$ be a sequence with values in some normed vector
space. Assume that the average of $\scsp_t$ tends to $0$ fast enough,
namely,
\begin{equation*}
\frac{1}{T} \sum_{t=1}^T \scsp_t =O(\lngav(T)/T)
\end{equation*}
for some scale function $\lngav(T)\ll T$ when $T\to \infty$.

Let $\flng{T}$ be any scale function with $\lngav(T)\ll \flng{T}\ll T$
when $T\to\infty$. Then the averages of $\scsp$ over intervals
$[T;T+\flng{T}]$ tend to $0$, and more precisely
\begin{equation*}
\frac{1}{\flng{T}}\sum_{t=T+1}^{T+\flng{T}}\scsp_t =
O(\lngav(T)/\flng{T})
\end{equation*}
when $T\to\infty$.
\end{lemme}

\begin{proof}
For any $T\geq 1$,
\begin{equation*}
\sum_{t=T+1}^{T+\flng{T}}\scsp_t =
\sum_{t=1}^{T+\flng{T}}\scsp_t-\sum_{t=1}^{T}\scsp_t
=O(\lngav(T+\flng{T}))+O(\lngav(T))
\end{equation*}
by assumption.

When $T\to\infty$ we have $T+\flng{T}\sim T$
because $\flng{T}\ll T$. Since $\lngav$ is a scale function, we thus
have $\lngav(T+\flng{T})\sim \lngav(T)$.

Therefore
\begin{equation*}
\sum_{t=T+1}^{T+\flng{T}}\scsp_t =
O(\lngav(T))
\end{equation*}
hence the conclusion.
\end{proof}

Remember the scale function $\lngav(t)=t^{\expem'}$ was defined in \relem{intvavrg}, 
together with the scale function $\lng$. The constraint $\lngav\ll \lng$ is satisfied thanks to the same Lemma.

\begin{lemme}[Average of $ \frac{\partial
\sbpermc{t}(\stateopt_0,\paramopt)}{\partial
\param}$ on intervals]
\label{lem:averdloss}
When $T\to\infty$, we have
\begin{equation*}
\frac{1}{\flng{T}}\sum_{t=T+1}^{T+\flng{T}} \furt{\frac{\partial
\sbpermc{t}(\stateopt_0,\paramopt)}{\partial
\param}}{\fstate_t\cpl{\stateopt_0}{\pctrlopt}}{\pctrlopt}=
O(\lngav(T)/\flng{T}).
\end{equation*}
\end{lemme}

\begin{proof}
This is a direct consequence of \relem{somneglint} and
the definition of $\expem'$ in \relem{intvavrg}. Indeed, by the latter we know that \rehyp{critoptrtrlnbt} is satisfied with $\expem'$ instead of $\expem$, so that we have
\begin{equation*}
\frac{1}{T}\sum_{t=1}^T
\furt{\frac{\partial}{\partial \param}\,\sbpermc{t}(\stateopt_0,\paramopt)}{\fstate_t\cpl{\stateopt_0}{\pctrlopt}}{\pctrlopt}
=O(\lngav(T)/T).
\end{equation*}
Therefore we can apply \relem{somneglint} to the average of
$
\furt{\frac{\partial\,\sbpermc{t}(\stateopt_0,\paramopt)}{\partial
\param}}{\fstate_t\cpl{\stateopt_0}{\pctrlopt}}{\pctrlopt}$.
\end{proof}

\begin{lemme}[Deviation from the optimal parameter for updates computed along the optimal trajectory]
\label{lem:rtrlstableopt}
Set $\memopt_0\deq (\stateopt_0,0)$ (RTRL 
initialized at $\stateopt_0$ with $J_0=0$) and
$\memopt_{t}=\Algo_{0:t}(\paramopt,\memopt_0)$, the RTRL state at time $t$
using the optimal parameter.
Assume that 
$\cdvp\leq \brnp/2$. 
Then
\begin{equation*}
\paramupdate_{T:T+\flng{T}}(\paramopt,\memopt_T,\pas_T)
=\paramopt+
\po{\pas_T\,\flng{T}}.
\end{equation*}

Consequently, the first part of \rehyp{optimprm} (first-order stability of $\paramopt$)
is satisfied for extended RTRL algorithms.
\end{lemme}
\begin{proof}
%
For any $T\geq 0$,
by \recor{RTRLiter} applied to $\paramopt$ with the stepsize
sequence $\pas_{t;\,t_1,t_2}\deq \pas_{t_1}$, we have
\begin{equation*}
\paramupdate_{T:T+\flng{T}}(\paramopt,
\memopt_{T}
,\pas_T)=
\paramupdate_{T:T+\flng{T}}\paren{\paramopt,\paren{\pas_T\,\vtanc_t}_{t\,}}
\end{equation*}
where
\begin{equation*}
\vtanc_t=
\furt{\frac{\partial
\sbpermc{t}(\stateopt_0,\paramopt)}{\partial
\param}}{\fstate_t\cpl{\stateopt_0}{\paramopt}}{\paramopt}.
\end{equation*}

Since the optimal trajectory stays in the stable tube, we have
$\vtanc_t\in \bvtg{t}$ by definition of $\bvtg{t}$.
Then by
\relem{firstorderupdate}, 
\begin{multline*}
\paramupdate_{T:T+\flng{T}}\paren{\paramopt,\paren{\pas_T\,\vtanc_t}_{t\,}}=
\\
\paramopt-\,\pas_T \sum_{t=T+1}^{T+\flng{T}} 
\furt{\frac{\partial
\sbpermc{t}(\stateopt_0,\paramopt)}{\partial
\param}}{\fstate_t\cpl{\stateopt_0}{\paramopt}}{\paramopt}
+\go{ \pas_T^2\,\flng{T}\,\fmdper{T}^2}
\end{multline*}
and by \relem{averdloss}, 
\begin{equation*}
\sum_{t=T+1}^{T+\flng{T}} 
\furt{\frac{\partial
\sbpermc{t}(\stateopt_0,\paramopt)}{\partial
\param}}{\fstate_t\cpl{\stateopt_0}{\paramopt}}{\paramopt}
=O(\lngav(T)).
\end{equation*}

Finally, by Lemma
\ref{lem:intvavrg}, both
$\pas_T \,\lngav(T)$ and $\pas_T^2\,\flng{T}\,\fmdper{T}^2$
are $\po{\pas_T\,\flng{T}}$.
\end{proof}

\subsection{Contractivity Around $\paramopt$}
\label{sec:contaroundparamopt}

We now turn to the second part of \rehyp{optimprm}, contractivity around
$\paramopt$: we have to prove that 
\begin{multline*}
\dist{\paramupdate_{t:t+\flng{t}}\paren{\param,\memopt_t,\pas_t}
}{\paramupdate_{t:t+\flng{t}}\paren{\paramopt,\memopt_t,\pas_t}}
\\\leq
\paren{1-\mvpmin\,\pas_t\,\flng{t}}\,\dist{\prmctrl}{\paramopt} +
\po{\pas_t\,\flng{t}}.
\end{multline*}
We will use a suitable Lyapunov function to define a suitable Euclidean distance
$\dist{\param}{\paramopt}$ for which this holds.

Remember the notation from \rehyp{critoptrtrlnbt}, and in particular, the
matrices $\fht{\param}$ (Jacobian of the parameter update) and
$\linalgmatrix$ (time average of $\fht{\paramopt}$).
%
%
%
Notably, remember that we have endowed $\Param$ with the norm
\begin{equation*}
\norm{\param}^2\deq \transp{\param}B\,\param
\end{equation*}
where $B$, defined in \relem{suitposdefm}, is such that $B\linalgmatrix+\transp{\linalgmatrix}B$ is
positive definite and $\linalgmatrix$ is
given by \rehyp{critoptrtrlnbt}. This norm is used as an approximate
Lyapunov function for the algorithm.

\begin{lemme}[Controlling different initialisations for open-loop trajectories]
As before, let $\memopt_0\deq (\stateopt_0,0)$. Let
$\param\in\boctopt$ with $\dist{\param}{\paramopt}\leq\rayoptct/3$.
For $T\geq 0$, let $\mem_{T}=\Algo_{0:T}(\param,\memopt_0)$.
Assume $\cdvp\leq \brnp/2$.
Then 
\begin{equation*}
\paramupdate_{T:T+\flng{T}}(\param, \mem_{T} ,\pas_T)
-
\paramupdate_{T:T+\flng{T}}(\param, \mem_{T}^* ,\pas_T)
=\po{\pas_T\,\flng{T}}.
\end{equation*}
\end{lemme}

\begin{proof}
For $t\geq T$, let us write
\begin{equation*}
\vtanc_t\deq \sm{V}_t(\param,\Algo_{T:t}(\param,\mem_T)).
\end{equation*}
($\sm{V}_t$ and $\Algo$ for RTRL are given by Def.~\ref{def:rtrlasalgo}). Since
$\param$ belongs to $\boctopt$, $\mem_T$ and
$\Algo_{T:t}(\param,\mem_T)$ belong to the stable tube, so
that $\vtanc_t\in \bvtg{t}$ by definition of $\bvtg{t}$.

By definition of the open-loop updates $\paramupdate$
(Def.~\ref{def:olupdate}), we have
\begin{equation*}
\paramupdate_{T:T+\flng{T}}(\param, \mem_{T} ,\pas_T)=
\paramupdate_{T:T+\flng{T}}(\param,(\pas_T \,\vtanc_t)_t).
\end{equation*}
Likewise with $\memopt_T$ instead of $\mem_T$, set $\vtanc'_t\deq
\sm{V}_t(\param,\Algo_{T:t}(\param,\memopt_T))$ so that
$\paramupdate_{T:T+\flng{T}}(\param, \memopt_{T} ,\pas_T)=
\paramupdate_{T:T+\flng{T}}(\param,(\pas_T \,\vtanc'_t)_t)$.

By \relem{firstorderupdate}, we have
\begin{multline*}
\paramupdate_{T:T+\flng{T}}\paren{\param,\paren{\pas_T\,\vtanc_t}_t}
-
\paramupdate_{T:T+\flng{T}}\paren{\param,\paren{\pas_T\,\vtanc_t'}_t}
=\\
\param-\param-\pas_T \sum_{t=T+1}^{T+\flng{T}} 
(\vtanc_t-\vtanc_t')
+\go{\pas_T^2\,\flng{T}\,\fmdper{T}^2
}.
\end{multline*}

By \rehyp{contgrad} (which has been checked for RTRL in the previous
section) we have
\begin{equation*}
\norm{\vtanc_t-\vtanc_t'}=\go{\fmdper{t}\norm{\Algo_{T:t}(\param,\mem_T)-\Algo_{T:t}(\param,\memopt_T)}}.
\end{equation*}
By \rehyp{expforgetinit} (which has been checked for RTRL in the previous
section, with constant $(1-\alpha/2)$),
\begin{equation*}
\norm{\Algo_{T:t}(\param,\mem_T)-\Algo_{T:t}(\param,\memopt_T)}=\go{(1-\alpha/2)^{\frac{t-T}{k}}\,\norm{\mem_T-\memopt_T}}
\end{equation*}
and $\norm{\mem_T-\memopt_T}$ is bounded because both belong to the
stable tube.
%
Therefore,
\begin{equation*}
\norm{\vtanc_t-\vtanc_t'}=
\go{\fmdper{t}\paren{1-\alpha/2}^{\frac{t-T}{k}}}.
\end{equation*}
As a result,
\begin{multline*}
\paramupdate_{T:T+\flng{T}}\paren{\param,\paren{\pas_T\,\vtanc_t}_t}
-
\paramupdate_{T:T+\flng{T}}\paren{\param,\paren{\pas_T\,\vtanc_t'}_t}
=\\
\go{\pas_T\,\sum_{t=T+1}^{T+\flng{T}}\,\fmdper{t}\,\paren{1-\alpha/2}^{\frac{t-T}{k}}}
+\go{\pas_T^2\,\flng{T}\,\fmdper{T}^2}
=\\
\go{\pas_T\,\fmdper{T}}
+\go{\pas_T^2\,\flng{T}\,\fmdper{T}^2}
=\go{\pas_T\,\fmdper{T}}
=\po{\pas_T\,\flng{T}},
\end{multline*}
thanks to \recor{convzerosompasintervalles} and the fact $\fmdper{T}=\po{\flng{T}}$.
\end{proof}

\begin{lemme}[Difference between open-loop trajectories]
\label{lem:exprediffbolt}
Let $\memopt_0\deq (\stateopt_0,0)$. Let
$\param\in\boctopt$ with $\dist{\param}{\paramopt}\leq\rayoptct/3$.
For $T\geq 0$, let $\mem_{T}=\Algo_{0:T}(\param,\memopt_0)$ be the RTRL state
obtained at time $T$ from parameter $\param$, and
$\memopt_T=\Algo_{0:T}(\paramopt,\memopt_0)$.
Assume $\cdvp\leq \brnp/2$.
For $u \in [0,1]$, denote $\param^u\deq \prmctrl+u \,
\paren{\pctrlopt-\prmctrl}$.
Then, 
\begin{multline*}
\paramupdate_{T:T+\flng{T}}(\param, \mem_{T} ,\pas_T)
-
\paramupdate_{T:T+\flng{T}}(\paramopt,\memopt_T,\pas_T)
=\\
\param-\paramopt - \,\pas_T \, \int_0^1 \,
\sum_{t=T+1}^{T+\flng{T}}\,\fht{\param^u} \cdot
\paren{\param-\paramopt} \, du  
 +\go{
 \pas_T^2\,\flng{T}\,\fmdper{T}^2
 }.
\end{multline*}
\end{lemme}

\begin{proof}
By \recor{RTRLiter} with stepsize
sequence $\pas_{t;\,t_1,t_2}=\pas_{t_1}$, we have
\begin{equation*}
\paramupdate_{T:T+\flng{T}}(\param,
\mem_{T}
,\pas_T)=
\paramupdate_{T:T+\flng{T}}\paren{\param,\pas_T\paren{\furt{\frac{\partial
\sbpermc{t}(\stateopt_0,\param)}{\partial
\param}}{\fstate_t\cpl{\stateopt_0}{\param}}{\param}}_{\!t\,}}.
\end{equation*}

Let us abbreviate
\begin{equation*}
\vtanc_t(\param)\deq \furt{\frac{\partial
\sbpermc{t}(\stateopt_0,\param)}{\partial
\param}}{\fstate_t\cpl{\stateopt_0}{\param}}{\param}.
\end{equation*}

Since $\param\in\boctopt$
and $\stateopt_0\in \tube_0$, $\mem_T$ belongs to the stable tube at time
$T$. Consequently, $\vtanc_t(\param)\in \bvtg{t}$ by definition of
$\bvtg{t}$.

By
\relem{firstorderupdate},
\begin{equation*}
\paramupdate_{T:T+\flng{T}}\paren{\param,\paren{\pas_T\,\vtanc_t(\param)}_t}=
\param-\pas_T \sum_{t=T+1}^{T+\flng{T}} 
\vtanc_t(\param)
+\go{\pas_T^2 \,\flng{T}\, \fmdper{T}^2}.
\end{equation*}

Therefore, by writing the same result at $\param=\paramopt$ and taking
differences, we find
\begin{multline*}
\paramupdate_{T:T+\flng{T}}\paren{\param,\paren{\pas_T\,\vtanc_t(\param)}_t}
-
\paramupdate_{T:T+\flng{T}}(\paramopt,\memopt_T,\pas_T)
=\\
\param-\paramopt-\pas_T \sum_{t=T+1}^{T+\flng{T}} 
(\vtanc_t(\param)-\vtanc_t(\paramopt))
+\go{
\pas_T^2\,\flng{T}\,\fmdper{T}^2
}.
\end{multline*}
Now, by the definitions of $\fht{\param}$ (\rehyp{critoptrtrlnbt}) and of
$\vtanc_t(\param)$ above, we
have
\begin{equation*}
\vtanc_t(\param)=
\vtanc_t(\paramopt)+\int_0^1 \, \fht{\param^u} \cdot
\paren{\param-\paramopt} \, du
\end{equation*}
hence the conclusion.
\end{proof}

\begin{lemme}[Average of Hessians over time intervals]
\label{lemme:hessapprox}
For $0\leq u \leq 1$ and $\theta\in\boctopt$, one has
\begin{equation*}
\sum_{t=T+1}^{T+\flng{T}}\,\fht{\param^u}=\flng{T}\left(
\linalgmatrix
+
\go{\rho\paren{\nrm{\param-\paramopt}}}
+
O(\lngav(T)/\flng{T})
\right).
\end{equation*}
Moreover, the $\go{\modct{\nrm{\param-\pctrlopt}}}$ term is uniform over $0\leq \cdvp\leq \brnp$ and over $\param\in\boctopt$.
\end{lemme}

\begin{proof}
Thanks to \rehyp{equicontH}, we know that, for all $0 \leq u \leq 1$, we have
\begin{equation*}
\ba
\sum_{t=T+1}^{T+\flng{T}}\fht{\param^u}
&=\sum_{t=T+1}^{T+\flng{T}}\paren{\fht{\pctrlopt}+\go{\rho\paren{\nrm{\param^u-\paramopt}}}}\\
&=\sum_{t=T+1}^{T+\flng{T}}\paren{\fht{\pctrlopt}+\go{\rho\paren{\nrm{\param-\paramopt}}}}.
\ea
\end{equation*}
Now, thanks to the definition of $\expem'$ in \relem{intvavrg}, we know that \rehyp{critoptrtrlnbt} is satisfied with $\lngav(t) = t^{\expem'}$, so that we have
\begin{align*}
\sum_{t=T+1}^{T+\flng{T}}\fht{\pctrlopt}
&=\sum_{t=1}^{T+\flng{T}}\fht{\pctrlopt}-\sum_{t=1}^{T}\fht{\pctrlopt}
\\&=
(T+\flng{T})\linalgmatrix+O(\lngav(T))-T\linalgmatrix+O(\lngav(T))
\\&=\flng{T}\left(
\linalgmatrix+O(\lngav(T)/\flng{T})
\right).
\end{align*}
Combining these results yields the statement.
\end{proof}

We now prove the second part of \rehyp{optimprm} (contractivity around
$\paramopt$). Unfortunately, this does not necessarily hold in the ball
$\boctopt$ that we have used so far, but in a smaller ball $B'_\Param$.
This smaller ball depends on the various quantities involved in the
assumptions (such as the constants in the $\go{}$ notation appearing in
the various assumptions, or the function $\modct{}$ in
\rehyp{equicontH}). Thus, we will have proved all assumptions of
Section~\ref{sec:absontadsys}, but over this smaller ball $B'_\Param$
instead of $\boctopt$. We will thus get the convergence of
Theorems~\ref{thm:cvalgopti}, \ref{thm:cvalgopti_noisy} and~\ref{thm:cvalgoboucleouverte}  for $\param$ in this smaller ball.

\begin{lemme}[Contractivity around $\paramopt$]
\label{lem:contractpctrlopt}
There exists a ball $B'_\Param\subset \boctopt$ centered at $\paramopt$
with positive radius, and $\lambda>0$, such that the following holds.

Let $\memopt_0\deq (\stateopt_0,0)$ (RTRL state
initialized at $\stateopt_0$ with $J_0=0$) and
$\memopt_{t}=\Algo_{0:t}(\paramopt,\memopt_0)$ (RTRL state at time $t$
using the optimal parameter).
Assume $\cdvp\leq \brnp/2$.

For every $\param \in B'_\Param$,
\begin{equation*}
\nrm{
\paramupdate_{T:T+\flng{T}}(\param,\memopt_T,\pas_T)
-
\paramupdate_{T:T+\flng{T}}(\paramopt,\memopt_T,\pas_T)}
\end{equation*}
is at most
\begin{equation*}
\paren{1-\lambda\,\pas_T\,\flng{T}}\norm{\param-\paramopt} +
\po{\pas_T \,\flng{T}}
\end{equation*}
and the $\po{}$ term is uniform over $0\leq \cdvp\leq \brnp/2$ and
over $\param\in B'_\Param$.

Therefore, the second part of \rehyp{optimprm} is satisfied by an extended
RTRL algorithm for $\param$ in the ball $B'_\Param$, for the distance
$\dist{\param}{\param'}^2\deq
\transp{(\param-\param')}B(\param-\param')$.
\end{lemme}

\begin{proof}
Let $\param\in\boctopt$ with $\dist{\param}{\paramopt}\leq\rayoptct/3$.
(We will have further constraints below to define the smaller ball
$B'_\Param$.)

By combining the last three lemmas, we obtain
\begin{multline*}
\paramupdate_{T:T+\flng{T}}(\param,\memopt_T,\pas_T)
-
\paramupdate_{T:T+\flng{T}}(\paramopt,\memopt_T,\pas_T)
=\\
\paramupdate_{T:T+\flng{T}}(\param,\mem_T,\pas_T)
-
\paramupdate_{T:T+\flng{T}}(\paramopt,\memopt_T,\pas_T)
+
\paramupdate_{T:T+\flng{T}}(\param,\memopt_T,\pas_T)
-\paramupdate_{T:T+\flng{T}}(\param,\mem_T,\pas_T)
\\=
(\param-\paramopt)
- \pas_T \flng{T} 
\left(
\linalgmatrix+\go{\modct{\nrm{\param-\paramopt}}}+O(\lngav(T)/\flng{T})\right)\cdot
(\param-\paramopt)
\\+\go{\pas_T^2\,\flng{T}\, \fmdper{T}^2}
+\po{\pas_T\,\flng{T}}
\end{multline*}
which equals
\begin{multline*}
(\id_\Param
- \pas_T \flng{T} \linalgmatrix)\cdot (\param-\paramopt)
+
O\left(
\pas_T\flng{T}\,{\modct{\norm{\param-\paramopt}}}\nrm{\param-\paramopt}\right)
\\
+O\left(\pas_T\,\lngav(T)\nrm{\param-\paramopt})\right)
+\go{\pas_T^2\,\flng{T}\, \fmdper{T}^2}
+\po{\pas_T\,\flng{T}}.
\end{multline*}


Since $\norm{\param-\paramopt}$ is bounded on $\boctopt$, the
term $\pas_T\,\lngav(T)\,\norm{\param-\paramopt}$ is
$O(\pas_T\lngav(T))$.
By Lemma
\ref{lem:intvavrg}, both
$\pas_T \,\lngav(T)$ and $\pas_T^2\,\flng{T}\,\fmdper{T}^2$
are $\po{\pas_T\,\flng{T}}$. So the last two $\go{}$ terms above are absorbed in
the $\po{\pas_T\,\flng{T}}$ term.

Remember that the norm we use on $\Param$ is defined by
$\norm{\theta}^2=\transp{\theta}B\,\theta$. Therefore, we have
\begin{multline*}
\norm{(\id_\Param
- \pas_T \flng{T} \linalgmatrix)\cdot (\param-\paramopt)}^2=
\transp{(\param-\paramopt)}B\,(\param-\paramopt)
\\-\pas_T \flng{T}
\transp{(\param-\paramopt)}\paren{B\,\linalgmatrix+\linalgmatrix^T\,B}(\param-\paramopt)
+\go{\pas_T^2 \flng{T}^2\norm{\param-\paramopt}^2}.
\end{multline*}

Thanks to \rehyp{critoptrtrlnbt} and to \relem{suitposdefm}, we know 
that ${B\,\linalgmatrix+\linalgmatrix^T\,B}$ is positive definite, so that for
some $\lambda>0$ we have
\begin{equation*}
\transp{(\param-\paramopt)}\paren{B\,\linalgmatrix+\linalgmatrix^T\,B}(\param-\paramopt)\geq 4\,\lambda
\norm{\param-\paramopt}^2,
\end{equation*}
and we obtain
\begin{multline*}
\norm{(\id_\Param
- \pas_T \flng{T} \linalgmatrix)\cdot
  (\param-\paramopt)}^2\leq\norm{\param-\paramopt}^2\left(1-4\,\lambda\,\pas_T\flng{T}+\go{\pas_T^2\flng{T}^2}\right).
\end{multline*}

Therefore, the quantity we want to compute is
\begin{multline*}
\norm{
\paramupdate_{T:T+\flng{T}}(\param,\memopt_T,\pas_T)
-
\paramupdate_{T:T+\flng{T}}(\paramopt,\memopt_T,\pas_T)}
\leq\\
\norm{\param-\paramopt}\left(1-4\,\lambda\,\pas_T\flng{T}+\go{\pas_T^2\flng{T}^2}\right)^{\inv{2}}
\\+
\go{\pas_T\flng{T}\,\modct{\norm{\param-\paramopt}}\,\norm{\param-\paramopt}}
+\po{\pas_T\,\flng{T}}
\leq\\
\norm{\param-\paramopt}\left(1-2\,\lambda\,\pas_T\flng{T}+\go{\pas_T^2\flng{T}^2}\right)
\\+
\go{\pas_T\flng{T}\,\modct{\norm{\param-\paramopt}}\,\norm{\param-\paramopt}}
+\po{\pas_T\,\flng{T}},
\end{multline*}
as $\sqrt{1-x}\leq 1-x/2$.

Since $\pas_T\flng{T}\to 0$, the term $\go{\pas_T^2\flng{T}^2}$ is
ultimately smaller than $(\lambda/2) \pas_T\,\flng{T}$.

Consider the term $\go{\pas_T\flng{T}\,\modct{\norm{\param-\paramopt}}\norm{\param-\paramopt}}$. The
constant in the $\go{}$ notation is independent of $T$ or
$\param\in\boctopt$. Let $\cst$ be that constant.
Since $\rho\to 0$ at $0$, there is a ball $B'_\Param$ of some fixed
radius around $\paramopt$, in which $\modct{\norm{\param-\paramopt}}$ is
smaller than $\lambda/2\cst$. Therefore, in that ball,
$\go{\pas_T\flng{T}\modct{\norm{\param-\paramopt}}\,\norm{\param-\paramopt}}\leq
(\lambda/2)\,\pas_T\,\flng{T}\norm{\param-\paramopt}$. On this smaller
ball $B'_\Param$,
one has
\begin{multline*}
\norm{
\paramupdate_{T:T+\flng{T}}(\param,\mem_T^*,\pas_T)
-
\paramupdate_{T:T+\flng{T}}(\paramopt,\memopt_T,\pas_T)}
\leq\\
\norm{\param-\paramopt}\left(1-2\lambda\pas_T\,\flng{T}
+(\lambda/2)\pas_T\,\flng{T}+(\lambda/2)\pas_T\,\flng{T}
\right)
+\po{\pas_T\,\flng{T}},
\end{multline*}
as needed.
\end{proof}

\subsection{Noise Control for \Algonoisy RTRL Algorithms}
\label{sec:noisectrl}

In this section we bound the divergence between \algonoisy RTRL
algorithms and exact RTRL. We consider an \algonoisy RTRL algorithm
(Def.~\ref{def:approxrtrl}) whose errors $E_t$ satisfy the unbiasedness
Assumption~\ref{hyp:wcorrnoiseapprtrl} and the error control
Assumption~\ref{hyp:errorgauge}
with some error gauge $\sbeg$
(Def.~\ref{def:errorgauge}). (Actually the
results presented here will be valid simultaneously for all \algonoisy
RTRL algorithms sharing the same error gauge.) Moreover, we assume
that the extended update rules $\fur_t$ are linear with respect to their first
argument (\rehyp{linearfur}).

We compare such \algonoisy
RTRL algorithms to the corresponding RTRL algorithm with error $E_t=0$
but the same underlying system.


\begin{lemme}[\Algonoisy Jacobians expressed in terms of exact Jacobians plus noise]
\label{lem:expdiffalgotbruitefemajrtrl}
Let
$\cpl{\state_0}{\jope_0} \in \tube_0\times \tubejope_0$,
and let $\sm{\prmctrl}=\paren{\param_t}$ be a sequence of parameters included in $\boctopt$.
Let $(\state_t)$ be
the trajectory of states starting at $\state_0$ computed from
$(\param_t)$, namely, via
$
\state_{t}=\opevolt\cpl{\state_{t-1}}{\param_{t-1}}
$ for $t\geq 1$.

We now compare Jacobians computed with the exact RTRL updates, and with
\algonoisy updates. Precisely, consider
\begin{enumerate}
\item the Jacobians $\paren{\tilde \jope_t}$ starting at $\jope_0$ and
following the \algonoisy RTRL updates
\begin{equation*}
\tilde \jope_t = \frac{\partial \opevol_t(\state_{t-1},\param_{t-1})}{\partial
\state} \,\tilde \jope_{t-1}+\frac{\partial
\opevol_t(\state_{t-1},\param_{t-1})}{\partial
\param}+E_t
\end{equation*}
for $t \geq 1$, where the errors $E_t$ satisfy
Assumption~\ref{hyp:errorgauge};
\item the Jacobians $\paren{\jope_t'}$ also starting at $\jope_0$ and following the exact RTRL updates that is, for $t \geq 1$,
\begin{equation*}
\jope_{t}'={\dpartf{\state}{\opevolt}\paren{\state_{t-1},\,\param_{t-1}} \cdot \jope_{t-1}' + \dpartf{\prmctrl}{\opevolt}\paren{\state_{t-1},\,\param_{t-1}}}.
\end{equation*}
\end{enumerate}
Then, for every $t \geq 1$, 
\begin{equation*}
\tilde \jope_t=\jope_t'+\sum_{s \leq t} \, \paren{\prod_{p=s+1}^{t} \,
\dpartf{\state}{\opevol_p}\cpl{\opevol_{0:p-1}\cpl{\state_0}{\param_0}}{\param_0}}\, E_s + \go{\sup_{s \leq t-1} \, \dist{\param_s}{\param_0}},
\end{equation*}
where the constant in the $\go{}$ term only depends on the constants
appearing in the assumptions and on the error gauge. (By convention, for
$s=t$
the empty product $\prod_{p=s+1}^{t}$ is equal to $\id$.)
\end{lemme}

\begin{proof}
For every $t \geq 1$,
\begin{equation*}
\tilde\jope_{t}-\jope_{t}'=\dpartf{\state}{\opevolt}\cpl{\state_{t-1}}{\param_{t-1}}
\, \paren{\tilde\jope_{t-1}-\jope_{t-1}'} + E_t.
\end{equation*}

Now, we have seen above that the non-\algonoisy RTRL operator on
$(s,J)$ (\redef{rtrlasalgo}) satisfies
\rehyp{paramlip} and \rehyp{expforgetinit}. Setting
$(\state''_t,\jope''_t)\deq \Algo_{0:t}(\param_0,(\state_0,\jope_0))$,
we may therefore apply \relem{newparamcont} to compare the
(non-\algonoisy) RTRL
trajectories with fixed parameter $\param_0$ and with variable parameter
$(\param_t)$: this yields, 
for all $t \geq 1$, 
\begin{equation*}
\dist{(\state_t,\jope'_t)}{(\state''_t,\jope''_t)}=\go{\sup_{s \leq t-1} \, \dist{\param_s}{\param_0}}.
\end{equation*}
Setting $S_t\deq \sup_{s \leq t-1} \, \dist{\param_s}{\param_0}$, we have a
fortiori
\begin{equation*}
\dist{\state_t}{\state''_t}=\go{S_t}.
\end{equation*}
Also note that $\state''_t=\opevol_{0:t}(\state_0,\param_0)$ by definition of
the RTRL operator $\Algo_{0:t}$.

Moreover, for all $t\geq 0$, $\param_t$ and $\state_t$ belong to the stable tube for RTRL. Thanks to 
\rehyp{regftransetats} the second derivatives of the transition
operator on the states are bounded. As a result, for all $t \geq 1$,
\begin{equation*}
\ba
\dpartf{\state}{\opevolt}\paren{\state_{t-1},\,\param_{t-1}}
&=\dpartf{\state}{\opevolt}\paren{\state''_{t-1},\,\param_0}+\go{\dist{\state_{t-1}}{\state''_{t-1}}+\dist{\param_{t-1}}{\param_0}}
\\&=\dpartf{\state}{\opevolt}\paren{\state''_{t-1},\,\param_0}+\go{S_t}.
\ea
\end{equation*} 
As a consequence, for all $t \geq 1$,
\begin{equation*}
\ba
\tilde \jope_{t}-\jope_{t}'&=\dpartf{\state}{\opevolt}\paren{\state''_{t-1},\,\param_0} \, \paren{\tilde \jope_{t-1}-\jope_{t-1}'} \\
&+ \go{S_t\, \norm{\tilde \jope_{t-1}-\jope_{t-1}'}} + E_t.
\ea
\end{equation*}
Now, since the sequences $\paren{\tilde \jope_t}$ and $\paren{\jope_t'}$ 
are computed by an \algonoisy and exact RTRL algorithm from a sequence
of parameters $(\param_t)$ in the control ball and an initialization
$(\state_0,\jope_0)$ in the stable tube, they both belong to the stable
tube, and are therefore bounded by \recor{rtrlstabletubesJ}. Thus
$\jope_t-\jope'_t$ is bounded and
\begin{equation*}
\go{
{S_t}\, \norm{\tilde
\jope_{t-1}-\jope_{t-1}'}}=\go{ S_t}
\end{equation*}
so that
\begin{equation*}
\ba
\tilde
\jope_{t}-\jope_{t}'&=\dpartf{\state}{\opevolt}\paren{\state''_{t-1},\,\param_0} \, \paren{\tilde \jope_{t-1}-\jope_{t-1}'} 
+ E_t+
r_t
\ea
\end{equation*}
for some remainder $r_t=\go{S_t}$.
From this we obtain, by induction,
\begin{equation*}
\ba
\tilde\jope_t-\jope_t'&=\sum_{s \leq t} \paren{\prod_{p=s+1}^{t} \,
\dpartf{\state}{\opevol_p}\paren{\state''_{p-1},\,\param_0}}
E_s
\\&+
\sum_{s \leq t} \paren{\prod_{p=s+1}^{t} \,
\dpartf{\state}{\opevol_p}\paren{\state''_{p-1},\,\param_0}}
r_s.
\ea
\end{equation*}

Since $\state''_t=\opevol_{0:t}(\state_0,\param_0)$, the claim will be proved
if we prove that the remainder term is $\go{S_t}$.
The norm of the remainder term is
\begin{equation*}
\norm{\sum_{s \leq t} \paren{\prod_{p=s+1}^{t} \,
\dpartf{\state}{\opevol_p}\paren{\state''_{p-1},\,\param_0}}
r_s
}\leq \sum_{s\leq t} \nrmop{\prod_{p=s+1}^{t} \,
\dpartf{\state}{\opevol_p}\paren{\state''_{p-1},\,\param_0}}\norm{r_s}.
\end{equation*}

Since $\state''_{t-1}$ belongs to the stable tube, we can apply
\recor{specrad}: 
the product $\prod_{p=s+1}^{t} \,
\dpartf{\state}{\opevol_p}\paren{\state''_{p-1},\,\param_0}$ has operator
norm at most $M(1-\alpha/2)^{{(t-s)/\hsr}}$ for some constant
$M>0$.

As a result,
\begin{equation*}
\ba
\sum_{s\leq t} \nrmop{\prod_{p=s+1}^{t} \,
\dpartf{\state}{\opevol_p}\paren{\state''_{p-1},\,\param_0}}\norm{r_s}
&\leq \left(\sup_{s\leq t}\norm{r_s}\right)\left(\sum_{s\leq t}
M(1-\alpha/2)^{{(t-s)/\hsr}}\right)
\\&\leq{(2\hsr M/\alpha)}\sup_{s\leq t}\norm{r_s}=\go{S_t}
\ea
\end{equation*}
since $r_s=\go{S_t}$.
\end{proof}

Remember that, for \algonoisy RTRL algorithms, we assume that $\fur_t$
has the form
$\fur_t(\vtanc,\state,\param)=P_t(\state,\param)\cdot \vtanc$ for some linear
operator
$P_t$ (\rehyp{linearfur}, in addition to \rehyp{fupdrl}).

\begin{lemme}[Bounds on $\fur_t$ in the linear case]
\label{lem:boundingextupdateslinear}
Under \rehypdeux{fupdrl}{linearfur},
the operator $P_t$ and its derivative with respect to $(\state,\param)$
are bounded on a ball containing the stable tube. Namely,
\begin{equation*}
\sup_{t\geq 1}
\sup_{B_{\State_t}(\stateopt_{t},\raysy)\times
B_\Param(\paramopt,r_\Param)} \,
\nrmop{P_t(\state,\param)} < \infty
\end{equation*}
and
\begin{equation*}
\sup_{t\geq 1}
\sup_{B_{\State_t}(\stateopt_{t},\raysy)\times
B_\Param(\paramopt,r_\Param)} \,
\nrmop{\partial_{(\state,\param)}P_t(\state,\param)}< \infty,
\end{equation*}
where the balls are those appearing in \rehyp{fupdrl}.
\end{lemme}

\begin{proof}
This is a direct rewriting of \rehyp{fupdrl} for the particular case
$\fur_t(\vtanc,\state,\param)=P_t(\state,\param)\cdot \vtanc$. Indeed,
$\partial_v \,\fur_t=P_t(\state,\param)$ so that the first point of
\rehyp{fupdrl} gives the first statement of the lemma.

For the second statement, consider the second point of \rehyp{fupdrl}.
Here $\partial_{(\state,\param)}\,\fur_t\tpl{\vtanc}{\state}{\param}$ is
$\partial_{(\state,\param)}(P_t(\state,\param)\cdot \vtanc)$. Its
operator norm is
$O(1+\norm{\vtanc})$ by \rehyp{fupdrl}.
Remember that $P_t(\state,\param)\in \sblin(\linform,\Param)$.
Let us now compute the operator norm of $\partial_{(\state,\param)}
P_t(\state,\param)$.  Let $(u_\state,u_\param)\in \State\times \Param$
be a unit vector that realizes this operator norm, namely,
$\nrmop{\partial_{(\state,\param)}
P_t(\state,\param)}=\nrmop{\partial_{(\state,\param)}
P_t(\state,\param)\cdot (u_\state,u_\param)}$; here the second operator
norm is as an operator on $\vtanc\in \linform$, since $\partial_{(\state,\param)}
P_t(\state,\param)\cdot (u_\state,u_\param)\in \sblin(\linform,\Param)$.
Let now $\vtanc\in \linform$ be a unit vector that realizes the operator
norm of $\partial_{(\state,\param)}
P_t(\state,\param)\cdot (u_\state,u_\param)$, so that
$\nrmop{\partial_{(\state,\param)}
P_t(\state,\param)}=\nrm{\left(\partial_{(\state,\param)}P_t(\state,\param)\cdot
(u_\state,u_\param)\right)\cdot \vtanc}$ (the last norm is in $\Param$).
Since $P_t(\state,\param)\cdot \vtanc$ is linear in $\vtanc$,
we have $\left(\partial_{(\state,\param)}P_t(\state,\param)\cdot
(u_\state,u_\param)\right)\cdot
v=\partial_{(\state,\param)}\left(P_t(\state,\param)\cdot\vtanc\right)\cdot
(u_\state,u_\param)$ (indeed, both are the limit of $\frac1\eps
(P_t(\state+\eps u_\state,\param+\eps u_\param)\cdot \vtanc-
P_t(\state,\param)\cdot \vtanc)$ when $\eps\to 0$). Therefore,
$\nrmop{\partial_{(\state,\param)}
P_t(\state,\param)}=\nrm{\partial_{(\state,\param)}\left(P_t(\state,\param)\cdot\vtanc\right)\cdot
(u_\state,u_\param)}$. However, since $(u_\state,u_\param)$ is a unit
vector, the latter is at most
$\nrmop{\partial_{(\state,\param)}\left(P_t(\state,\param)\cdot\vtanc\right)}$.
This is $O(1+\norm{\vtanc})$ by \rehyp{fupdrl}, 
but $\vtanc$ is a unit vector
so this is $O(1)$.

Finally, by 
\relem{rtrlstabletubes}, the stable tube for $\state$ and $\param$ is included in the
balls of \rehyp{fupdrl}.
\end{proof}

We now introduce an operator
that represents the first-order change in the computed gradients
$\vtanc_t$, with respect to a change of state at a previous time $s$; this
linearization is computed along a trajectory defined by some $\state_\tinit$
and $\param$. \todo{perhaps
poor choice for notation, index $s$ wrt states $\state_t$}

\begin{definition}[Product of the differentials with respect to the states of the transition operators]
\label{def:prodiffopevolpretats}
Let $\tinit \geq 1$. Let $\param \in \boctopt$, and $\state_\tinit \in \bosyinsopt{\tinit}$. 
For $t \geq \tinit$, abbreviate
$\state_{t}=\opevol_{\tinit:t}\cpl{\state_{\tinit}}{\param}$,
using the notation from \redef{notiterates}. 
We define, for $t \geq s \geq \tinit$, the linear operator
$\prdope{\tinit}{s}{t}(\state_\tinit,\param)$ 
from $\sblin(\Param,\State_s)$ to $\Param$, which sends
$E\in \sblin(\Param,\State_s)$ to
\begin{equation*}
\prdope{\tinit}{s}{t}(\state_\tinit,\param)\,E\deq
P_t(\state_{t},\param)\cdot \left(
\dpartf{\state}{\perte_t}\paren{\state_{t}} \cdot
\paren{\prod_{p=s+1}^{t} \,
\dpartf{\state}{\opevol_p}\paren{\state_{p-1},\,\param}}
E\right)
.
\end{equation*}
(Note that $E\in \sblin(\Param,\State_s)$ so that
$\paren{\prod_{p=s+1}^{t} \,
\dpartf{\state}{\opevol_p}\paren{\state_{p-1},\,\param}}
E$ belongs to $\sblin(\Param,\State_t)$; multiplying this by $\partial_\state
\perte_t(\state_t)\in \sblin(\State_t,\mathbb{R})$ produces an element of $\sblin(\Param,\R)$, from
which $P_t$ produces an element of $\Param$.)
\end{definition}

\begin{lemme}[Expressing the \algonoisy tangent vectors]
\label{lem:expvectanbruitefoncvtanbo}
Under the exact same assumptions and notations as in
\relem{expdiffalgotbruitefemajrtrl}, set
\begin{equation*}
\vtanc_t = \furt{\frac{\partial \perte_t(\state_t)}{\partial
\state}\cdot
\tilde \jope_t
}{\state_t}{\param_{t-1}}
,\qquad
\vtanc'_t = \furt{\frac{\partial \perte_t(\state_t)}{\partial
\state}\cdot
\jope'_t
}{\state_t}{\param_{t-1}}
\end{equation*}
where $\fur_t$ satisfies \rehyp{linearfur}, namely,
$\fur_t(\vtanc,\state,\param)=P_t(\state,\param)\cdot \vtanc$ for some linear
operator
$P_t$.

Then, for all $t \geq 1$, 
\begin{equation*}
\vtanc_t-\vtanc_t'=\sum_{s \leq t} \,
\prdope{0}{s}{t}(\state_0,\param_0) \, E_s + \go{\fmdper{t} \, \sup_{s \leq t} \, \dist{\param_s}{\param_0}}.
\end{equation*}
\end{lemme}

\begin{proof}
Let
$\state''_t\deq
\opevol_{0:t}\cpl{\state_{0}}{\param_0}$ as in
\relem{expdiffalgotbruitefemajrtrl}.
Thanks to the assumption on $\fur_t$,
\begin{equation*}
\ba
\vtanc_t-\vtanc_t'&=P_t(\state_t,\param_{t-1})\cdot\dpartf{\state}{\perte_t}\paren{\state_t} \cdot \paren{\tilde{\jope}_t-\jope_t'}
\\&=P_t(\state''_t,\param_0)\cdot\dpartf{\state}{\perte_t}\paren{\state''_t}
\cdot \paren{\tilde{\jope}_t-\jope_t'}
\\&
+\go{\left(P_t(\state_t,\param_{t-1})\cdot\dpartf{\state}{\perte_t}\paren{\state_t}-P_t(\state''_t,\param_0)\cdot\dpartf{\state}{\perte_t}\paren{\state''_t}\right)
\norm{\tilde{\jope}_t-\jope_t'}}.
\ea
\end{equation*}

By the expression for $\tilde J_t-J'_t$ in
\relem{expdiffalgotbruitefemajrtrl}, and by definition of $
\prdope{0}{s}{t}$ and of $\state''_t$, we have
\begin{multline*}
P_t(\state''_t,\param_0)\cdot\dpartf{\state}{\perte_t}\paren{\state''_t}
\cdot \paren{\tilde{\jope}_t-\jope_t'}
=
\sum_{s\leq t}\prdope{0}{s}{t}(\state_0,\param_0) \, E_s+
\go{\nrmop{P_t(\state''_t,\param_0)}\norm{\dpartf{\state}{\perte_t}\paren{\state''_t}}S_t}
\end{multline*}
where $S_t\deq \sup_{s
\leq t-1} \, \dist{\param_s}{\param_0}$ as in
\relem{expdiffalgotbruitefemajrtrl}.

As in \relem{expdiffalgotbruitefemajrtrl}, $\state''_t$ belongs to the
stable tube.
By \rehyp{regpertes},
$
\nrm{\dpartf{\state}{\perte_t}\paren{\state''_t}}
=\go{\fmdper{t}}$. \relem{boundingextupdateslinear} shows that
$P_t$ is bounded on the table tube. So
$\norm{P_t(\state''_t,\param_0)}\norm{\dpartf{\state}{\perte_t}\paren{\state''_t}}S_t=\go{\fmdper{t}\,S_t}$.

Thus, to reach our expression for $\vtanc_t-\vtanc'_t$, we only have to
prove that
\begin{equation*}
\go{\left(P_t(\state_t,\param_{t-1})\cdot\dpartf{\state}{\perte_t}\paren{\state_t}-P_t(\state''_t,\param_0)\cdot\dpartf{\state}{\perte_t}\paren{\state''_t}\right)
\norm{\tilde{\jope}_t-\jope_t'}}
=\go{\fmdper{t}\,S_t}.
\end{equation*}

As in \relem{expdiffalgotbruitefemajrtrl},
$\norm{\tilde{\jope}_t-\jope_t'}$ is uniformly bounded because both
belong to the stable tube. So we have to prove that
$P_t(\state_t,\param_{t-1})\cdot\dpartf{\state}{\perte_t}\paren{\state_t}-P_t(\state''_t,\param_0)\cdot\dpartf{\state}{\perte_t}\paren{\state''_t}=\go{\fmdper{t}\,S_t}$.

By \rehyp{fupdrl}, $\fur_t$ is $C^1$ so that $P_t$ is $C^1$. 
By \redef{fcpletaprm}, $\perte_t$ is $C^2$. Therefore,
\begin{multline*}
P_t(\state_t,\param_{t-1})\cdot\dpartf{\state}{\perte_t}\paren{\state_t}-P_t(\state''_t,\param_0)\cdot\dpartf{\state}{\perte_t}\paren{\state''_t}=
\\
\go{
\left(\dist{\state_t}{\state''_t}+\dist{\param_{t-1}}{\param_0}\right)
\sup_{(\state,\param)} \norm{\partial_{(\state,\param)} \left(
P_t(\state,\param)\cdot\dpartf{\state}{\perte_t}\paren{\state}
\right)}
}
\end{multline*}
where the supremum over $(\state,\param)$ is on a ball where all assumptions hold, since
$\state_t$, $\state''_t$, $\param_{t-1}$ and $\param_0$ all belong to the
stable tube.

We want to prove that this quantity is $\go{\fmdper{t}\,S_t}$.
We proved in \relem{expdiffalgotbruitefemajrtrl} that
$\dist{\state_t}{\state''_t}$ is $\go{S_t}$, and
$\dist{\param_{t-1}}{\param_0}$ is $\go{S_t}$ by definition of $S_t$. 

So we have to prove that the derivative of
$P_t(\state,\param)\cdot\dpartf{\state}{\perte_t}\paren{\state}$ with
respect to $(\state,\param)$ is $\go{\fmdper{t}}$. By differentiating the product, we have to
bound $P_t(\state,\param)$ and its derivative, as well as
the first and second derivatives of $\perte_t$ with respect to $\state$.

Thanks to \relem{boundingextupdateslinear}, $P_t(\state,\param)$ is bounded on the stable tube, together with its derivative $\partial_{(\state,\param)}P_t(\state,\param)$.

By \rehyp{regpertes}, the first and second derivatives of $\perte_t$ are
controlled by $\go{\fmdper{t}}$ on the stable tube.

This proves that $\partial_{(\state,\param)} \left(
P_t(\state,\param)\cdot\partial_\state \perte_t(\state)
\right)$ is $\go{\fmdper{t}}$ on the stable tube.
%
This
ends the proof.
\end{proof}

\begin{lemme}[Operator norm of $\prdope{\tinit}{s}{t}$]
\label{lem:Pinorm}
There is a constant $M>0$ such that, for any $\tinit>1$, for any 
$\param_\tinit \in \boctopt$ and $\state_\tinit \in \tube_{\tinit}$, for
any $t\geq s\geq \tinit$, the operator 
$\prdope{\tinit}{s}{t}(\state_\tinit,\param_\tinit)$ has operator norm at most
\begin{equation*}
\nrmop{\prdope{\tinit}{s}{t}(\state_\tinit,\param_\tinit)}\leq M\,\fmdper{t}\,(1-\alpha/2)^{(t-s)/\hsr},
\end{equation*}
where $\alpha$ and $\hsr$ are
the spectral radius constants of \rehyp{specrad}, and where the operator norm
of $\prdope{\tinit}{s}{t}(\state_\tinit,\param_\tinit)\in \sblin(\sblin(\Param,\State_s),\Param)$ is defined with respect to the
operator norm on $\sblin(\Param,\State_s)$ and the usual norm on $\Param$.
\end{lemme}

\begin{proof}
Let $E\in \sblin(\Param,\State_s)$. Then
by 
\redef{prodiffopevolpretats},
\begin{equation*}
\prdope{\tinit}{s}{t}(\state_\tinit,\param_\tinit)\,E=P_t(\state_{t},\param_\tinit)\cdot\left(
\dpartf{\state}{\perte_t}\paren{\state_{t}} \cdot
\paren{\prod_{p=s+1}^{t} \,
\dpartf{\state}{\opevol_p}\paren{\state_{p-1},\,\param_\tinit}}E\right)
\end{equation*}
where $\state_t\deq \opevol_{\tinit:t}(\state_\tinit,\param_\tinit)$ for
$t\geq \tinit$.
Since $\state_{\tinit}$ belongs to the stable tube and
$\param_\tinit\in\boctopt$, these states belong to the stable tube.
Therefore, we can apply
\recor{specrad}: 
the product $\prod_{p=s+1}^{t} \,
\dpartf{\state}{\opevol_p}\paren{\state_{p-1},\,\param_\tinit}$ has operator
norm at most $M(1-\alpha/2)^{{(t-s)/\hsr}}$ for some constant
$M>0$.

By \rehyp{regpertes}, the first derivative of $\perte_t$ has operator norm
$\go{\fmdper{t}}$ on the stable tube.
Finally, by \relem{boundingextupdateslinear}, the operator $P_t$ has bounded operator norm on the stable
tube. This proves the claim.
\end{proof}

Let $\paren{\mem_t}_{t\geq\tinit}$ be the sequence produced by the
\algonoisy RTRL algorithm starting at $\mem_{\tinit}$, and with a
parameter $\param_\tinit$. 
Remember the deviation
$\bruitif{\tinit}{t}{\param_\tinit}{\paren{\mem_t}_{t\geq\tinit}}$
introduced in \redef{algobruite}, that measures the effect on $\param_t$
of the difference between the \algonoisy RTRL trajectory $\mem_t$ and
a trajectory closer to RTRL.

\begin{lemme}[Expression of the noise]
\label{lem:expbruit}
Let $\param_\tinit\in \Param$ with $\dist{\prmctrl}{\pctrlopt} \leq
\frac{\rayoptct}{3}$. Let $
\cpl{\state_\tinit}{\jope_\tinit} \in
\tube_{\tinit}\times\tubejope_{\tinit}$.
Assume $\cdvp\leq \brnp$.

Let $\param_t$ and $\mem_t=(\state_t,\tilde \jope_t)$ be the trajectory computed by an
\algonoisy RTRL algorithm (Def.~\ref{def:approxrtrl}) starting at 
$\mem_\tinit=\cpl{\state_\tinit}{\jope_\tinit}$ at time $\tinit$.

Then, for all $\tinit+1 \leq t < \thoriz{\tinit}$,
\begin{equation*}
\bruitif{\tinit}{t}{\param_\tinit}{(\mem_t)}\leq \nrm{\som{s=\tinit+1}{t}
\, \paren{\som{p = s}{t} \,
\pas_p\,\prdope{\tinit}{s}{p}(\state_\tinit,\param_\tinit)}\, E_s} + \go{\paren{\som{s=\tinit+1}{t} \pas_s\,\fmdper{s}}^2}.
\end{equation*}
\end{lemme}

Note that, since we assume $\olr \leq \brnp$, the $O$ term is uniform with respect to $\olr \leq \brnp$.

Moreover, thanks to \relem{itvlgtflng}, we know that, for $t$ large
enough, we have $T_t^{r_\Theta^*}>\flng{t}$. As a result, by definition of
$\tpsk{k+1}$, for $k$ large enough, all results apply to times $t$ in the
interval $[\tpsk{k};\tpsk{k+1}]$.

\begin{proof}
Without loss of generality, we conduct the proof with $\tinit=0$.

\paragraph{First expression of the noise.}
From Definition~\ref{def:approxrtrl}, 
the \algonoisy RTRL trajectory $(\mem_t)=(\state_t,\tilde \jope_t)$ and $(\param_t)$ satisfies
\begin{equation*}
\left\lbrace \ba
\state_{t}&=\opevolt\cpl{\state_{t-1}}{\param_{t-1}} \\
\tilde \jope_{t}&=\frac{\partial \opevol_t(\state_{t-1},\param_{t-1})}{\partial
\state} \,\tilde \jope_{t-1}+\frac{\partial
\opevol_t(\state_{t-1},\param_{t-1})}{\partial
\param}+E_t\\
\vtanc_t &= \furt{\frac{\partial \perte_t(\state_t)}{\partial
\state}\cdot
\tilde \jope_t
}{\state_t}{\param_{t-1}}
\\
\param_{t} &= \opdepct{\param_{t-1}}{\pas_t\,\vtanc_t}.
\ea \right.
\end{equation*}

Remember the definition \ref{def:algobruite} of the deviation
$\bruitif{\tinit}{t}{\param_\tinit}{\paren{\mem_t}_{t\geq\tinit}}$.
To compute the
deviation, we have to compare this to 
%
the regularized trajectory initialized likewise, but satisfying
the recurrence equations 
$\bar\mem_t=\Algo_t(\param_{t-1},\bar \mem_{t-1})$,
$\bar\vtanc_t =\sm{V}_{t}\paren{\param_{t-1},\,\bar\mem_t}$, and
$\bar\param_t=\paramupdate_t(\bar\param_{t-1},\,\pas_t\bar\vtanc_t)$ with
$\Algo_t$ and $\sm{V}_t$ the operators of the RTRL algorithm
(Def.~\ref{def:rtrlasalgo}). The equation on $\bar\mem_t$ amounts to
$\bar\mem_t=(\bar s_t,\bar J_t)$ with
\begin{equation*}
\left\lbrace \ba
\bar\state_{t}&=\opevolt\cpl{\bar\state_{t-1}}{\param_{t-1}}
\\
\bar{\jope}_{t}&={\dpartf{\state}{\opevolt}\paren{\bar\state_{t-1},\,\param_{t-1}}
\cdot \bar{\jope}_{t-1} + \dpartf{\prmctrl}{\opevolt}\paren{\bar\state_{t-1},\,\param_{t-1}}} \\
\ea\right.
\end{equation*}
but the equation for $\bar \state_t$ is the same as for $\state_t$, so
$\bar \state_t=\state_t$ for all $t\geq 0$, and thus
\begin{equation*}
\bar{\jope}_{t}={\dpartf{\state}{\opevolt}\paren{\state_{t-1},\,\param_{t-1}}
\cdot \bar{\jope}_{t-1} +
\dpartf{\prmctrl}{\opevolt}\paren{\state_{t-1},\,\param_{t-1}}}
\end{equation*}
and the evolution equations of $\bar \vtanc_t$ and $\bar\param_t$ become
\begin{equation*}
\bar{\vtanc}_t=
\furt{\frac{\partial \perte_t(\state_t)}{\partial
\state}\cdot
\bar \jope_t
}{\state_t}{\param_{t-1}}
,\qquad
\bar{\param}_{t} = \opdepct{\bar{\param}_{t-1}}{\pas_t\,\bar{\vtanc_t}},
\end{equation*}
by the definition of $\sm{V}_t$ (Def.~\ref{def:rtrlasalgo}).
Then, thanks to \redef{algobruite}, for all $t \geq 1$,
\begin{equation*}
\bruitif{0}{t}{\param_0}{(\mem_t)}=\dist{\param_t}{\bar{\param}_t}.
\end{equation*}

\paragraph{Expressing the noise through a Taylor expansion.}
Thanks to \recor{rtrlstabletubesJ}, we may find a stable tube
$\paren{\tube_t\times\tubejope_t}$ on $\cpl{\state_t}{\tilde J_t}$
suitable for both the RTRL algorithm, and any \algonoisy RTRL algorithm admitting $\sbeg$ as error gauge. As a result, thanks to \relem{maintientpsfinibocont}, for $0 \leq t < \thorizo$, $\param_t$ belongs to $\boctopt$, and $\cpl{\state_t}{\tilde \jope_t}$ belongs to $\tube_t\times\tubejope_t$.

Since all $\prmctrl_t$'s belong to $\boctopt$, 
and since $\tube_t\times\tubejope_t$ is also a stable tube for
(non-\algonoisy)
RTRL,
for all $t \geq 0$,
$\cpl{\state_t}{\bar{\jope}_t}$ also belongs to
$\tube_t\times\tubejope_t$.

As a consequence, 
for every $t\geq 1$, $\vtanc_t$ and $\bar{\vtanc}_t$ belong to $B_{\Tangent_t}$, which was introduced just below the statement of \relem{boundgradrtrl}. 
As a consequence, thanks to \relem{citersbopexp} we have, for all $1 \leq t < \thorizo$,
\begin{equation*}
\param_t-\bar{\param}_t = \sum_{s=1}^{t} \pas_s \,
\paren{\vtanc_s-\bar{\vtanc}_s} + \go{\sum_{s=1}^{t}
\pas_s^2\,\fmdper{s}^2},
\end{equation*}
so that
\begin{equation}
\label{eq:contnormediffpourbruit}
\dist{\param_t}{\bar{\param}_t} \leq \nrm{\sum_{s=1}^{t} \pas_s \,
\paren{\vtanc_s-\bar{\vtanc}_s}} + \go{\sum_{s=1}^{t}
\pas_s^2\,\fmdper{s}^2}.
\end{equation}

\paragraph{Control of the noise.}
The sequences $\state_t$, $\tilde \jope_t$, $\bar \jope_t$, $\vtanc_t$
and $\bar \vtanc_t$ exactly satisfy the recurrence equations appearing in
Lemmas~\ref{lem:expdiffalgotbruitefemajrtrl}
and~\ref{lem:expvectanbruitefoncvtanbo}\todo{ il faudrait probablement harmoniser les $\vtanc_t'$ et les $\bar{\vtanc}_t$}. Therefore,
thanks to \relem{expvectanbruitefoncvtanbo}, for all $1 \leq t < \thorizo$,
\begin{equation*}
\vtanc_t-\bar{\vtanc}_t=\sum_{s=1}^{t} \,
\prdope{0}{s}{t}(\state_0,\param_0) \, E_s + \go{\fmdper{t} \, \sup_{s \leq t} \, \dist{\param_s}{\param_0}}.
\end{equation*}
Now, thanks to the last assertion of \relem{contordredeuxitprm2} (which
holds both for the exact RTRL algorithm and for an \algonoisy RTRL
algorithm with the same stable tube), we have
\begin{equation*}
\sup_{s \leq t} \, \dist{\param_s}{\param_0} \leq \cst_6 \, \sum_{s=1}^{t} \, \pas_s\,\fmdper{s},
\end{equation*}
so that
\begin{equation*}
\vtanc_t-\bar{\vtanc}_t=\sum_{s=1}^{t} \,
\prdope{0}{s}{t}(\state_0,\param_0) \, E_s + \go{\fmdper{t} \, \sum_{s=1}^{t} \pas_s\,\fmdper{s}}.
\end{equation*}
As a consequence, for all $1 \leq t < \thorizo$,
\begin{equation}
\sum_{s=1}^{t} \, \pas_s \, \paren{\vtanc_s-\bar{\vtanc}_s} =
\sum_{s=1}^{t} \, \pas_s  \, \sum_{p=1}^{s} \,
\prdope{0}{p}{s}(\state_0,\param_0) \, E_p + \go{\sum_{s=1}^{t} \, \pas_s \, \fmdper{s} \, \sum_{p=1}^{s} \pas_p\,\fmdper{p}}.
\end{equation}
For the $O$ term, we can bound the sum for $p$ up to $s$
by the sum for $p$ up to $t$,
so that the $O$ term is $\go{\paren{\sum_{s=1}^{t}
\, \pas_s\,\fmdper{s}}^2}$.
Finally, for all $1 \leq t < \thorizo$, 
\begin{equation*}
\sum_{s=1}^{t} \, \pas_s \, \sum_{p=1}^{s} \,
\prdope{0}{p}{s}(\state_0,\param_0) \, E_p=\sum_{p=1}^{t} \,
\paren{\som{s=p}{t} \,\pas_s \, \prdope{0}{p}{s}(\state_0,\param_0)}\, E_p.
\end{equation*}
This concludes the proof.
\end{proof}

The next two
statements concern the measurability of algorithm variables along
\algonoisy RTRL trajectories; namely, we track their dependencies with respect to the noise terms $E_t$.

\begin{lemme}[Measurability for \algonoisy RTRL
algorithms]
\label{lem:measapproxrtrlalgo}
For $t\geq 1$, we denote by $\trib{t}$ the $\sigma$--algebra generated by
$\tribo$ and the $(E_s)_{1\leq s\leq t}$, where $\tribo$ is defined
in \rehyp{wcorrnoiseapprtrl}.

Then
$\acco{\trib{t}}_{t\geq 0}$ is a filtration, and
\rehyp{wcorrnoiseapprtrl} rewrites as $\econd{E_t}{\trib{t-1}}=0$, for every $t\geq 1$.

Moreover, for all $t\geq 1$, $\trib{t}$ contains all the variables $\state_t$, $\param_t$, $\tilde{\jope}_t$ and $E_t$, as well as all the operators $\paren{\opevol_s}_{s\geq 1}$, $\paren{\perte_s}_{s\geq 1}$, $\paren{\fur_s}_{s\geq 1}$ and $\paren{\paramupdate_s}_{s\geq 1}$.
\end{lemme}

\begin{proof}
$\acco{\trib{t}}_{t\geq 0}$ is a filtration by its construction as an
increasing sequence of $\sigma$-algebras.

Moreover, by \rehyp{wcorrnoiseapprtrl}, $\tribo$ contains $\param_0$, $\state_0$, $\tilde{\jope}_0$, as well as all the operators $\paren{\opevol_t}_{t\geq 1}$, $\paren{\perte_t}_{t\geq 1}$, $\paren{\fur_t}_{t\geq 1}$ and $\paren{\paramupdate_t}_{t\geq 1}$.

Then by \redef{approxrtrl}, $\state_1$ is \mesl{0}, and therefore,
\mesl{1}. Moreover, by
\redef{approxrtrl} and
\rehyp{wcorrnoiseapprtrl}, $\tilde{\jope}_1$ is \mesl{1} since $E_1$ is
\mesl{1}. Finally, $\vtanc_1$ and $\param_1$ are also \mesl{1} by
\redef{approxrtrl}.

By induction on $t\geq 1$, the property then holds for all $t\geq 1$.
\end{proof}

\begin{corollaire}[Measurability along \algonoisy trajectories]
\label{cor:measapproxtraj}
Let $k\geq 0$ and let $\state_{\tpsk{k}}$ and $\param_{\tpsk{k}}$ be the
state and parameter obtained by the \algonoisy RTRL algorithm at time
$\tpsk{k}$.

For each $\tpsk{k}+1 \leq t \leq \tpsk{k+1}$, define
\begin{equation*}
c_t\deq\som{p = t}{\tpsk{k+1}} \, \pas_p\,\prdope{\tpsk{k}}{t}{p}(\state_{\tpsk{k}},\param_{\tpsk{k}}).
\end{equation*}
Then, 
for all $\tpsk{k}+1
\leq t \leq \tpsk{k+1}$,
$c_t$ is \mesl{\tpsk{k}} and \mesl{t-1}, where $\acco{\trib{t}}_{t\geq
0}$ is the filtration from \rehyp{wcorrnoiseapprtrl}.
\end{corollaire}

\begin{proof}
Let $\tpsk{k}+1 \leq t \leq \tpsk{k+1}$.
As we saw in Lemma~\ref{lem:measapproxrtrlalgo}, $\state_{\tpsk{k}}$ and
$\param_{\tpsk{k}}$ are \mesl{\tpsk{k}}.

By \redef{prodiffopevolpretats}, for each $\tpsk{k}\leq t \leq p$, the
operator $\prdope{\tpsk{k}}{t}{p}(\state_{\tpsk{k}},\param_{\tpsk{k}})$
is computed from $\state_{\tpsk{k}}$, $\param_{\tpsk{k}}$, $\perte_p$,
$P_p$, the states $\state_l =
\opevol_{\tpsk{k}:l}\cpl{\state_{\tpsk{k}}}{\param_{\tpsk{k}}}$,  and the family of operators $\paren{\opevol_u}_{u\geq 1}$. 

Again thanks to the Lemma~\ref{lem:measapproxrtrlalgo},
the operators $\perte_p$,
$P_p$ (defined by $\fur_p$), and all the operators
$\paren{\opevol_u}_{u\geq 1}$,
are \mesl{\tpsk{k}}. The compound operator
$\opevol_{\tpsk{k}:l}$ is \mesl{\tpsk{k}} too, as a composition of
operators $\paren{\opevol_u}$.

Since $\state_{\tpsk{k}}$, $\param_{\tpsk{k}}$, and $\opevol_{\tpsk{k}:l}$
are \mesl{\tpsk{k}}, 
the states $\state_l =
\opevol_{\tpsk{k}:l}\cpl{\state_{\tpsk{k}}}{\param_{\tpsk{k}}}$ are
\mesl{\tpsk{k}}. This proves that all objects defining
$c_t$ are \mesl{\tpsk{k}}, hence \mesl{t-1} since $t-1 \geq \tpsk{k}$.
\end{proof}

Here we prove that, under the unbiased noise of
\rehyp{wcorrnoiseapprtrl}, the (random) trajectory of an \algonoisy RTRL
algorithm has negligible noise in the sense of
Definition~\ref{def:negnoise}, with arbitrarily high probability.

\begin{lemme}[Noise control for the \algonoisy RTRL algorithm]
\label{lem:nsectrlapprtrl}
For all $k\geq 0$, denote $\mathcal{E}_k=\enstq{\prmctrl \in \epctrl}{\dist{\prmctrl}{\pctrlopt} \leq \frac{\rayoptct}{3}}\times \tube_{\tpsk{k}}\times\tubejope_{\tpsk{k}}$.

Any \algonoisy RTRL algorithm (under the unbiasedness and error
gauge
assumptions of
Section~\ref{sec:approxrtrl}) satisfies the following.
There exists a sequence $\paren{\delta_k}$ tending to $0$,
such that, for every $\eps>0$, there exists $K\geq 0$ such that, for
every $\olr \leq \brnp$, for any trajectory $\paren{\param_t}_{t\geq 0}$, $(\mem_t)_{t\geq 0}$, with $\mem_t=\cpl{\state_t}{\tilde{\jope_t}}$ for all $t\geq 0$, of the \algonoisy
RTRL algorithm, we have
\begin{equation*}
\prob{\mathbbm{1}_{\tpl{\param_{\tpsk{k}}}{\state_{\tpsk{k}}}{\tilde{\jope}_{\tpsk{k}}}\in\mathcal{E}_k}\,\bruitif{\tpsk{k}}{\tpsk{k+1}}{\param_\tpsk{k}}{(\mem_t)}
\leq \delta_k\,\pas_{\tpsk{k}}\,\flng{\tpsk{k}},\quad \forall k\geq K}\geq 1-\eps.
\end{equation*}

Therefore, for any $\eps >0$, for any $\olr \leq \brnp$, with probability
at least $1-\eps$, trajectories of the \algonoisy RTRL algorithm have
negligible noise starting at $K=K(\eps)$, at speed
$\paren{\delta_k}$ (Def.~\ref{def:negnoise}).
\end{lemme}

\begin{proof}
Let us fix some $\bar{\pas}\leq\brnp$. Let $(\param_t)$ and $(\mem_t)$ be
a learning trajectory of the \algonoisy RTRL algorithm, with
$m_t=(\state_t,\tilde{\jope}_t)$.

Thanks to \recor{rtrlsuitstep}, we may apply \relem{minmthoriztpsk} to
obtain that, for all $k \geq 0$, we have $\tpsk{k+1}\leq
\thoriz{\tpsk{k}}(\sm{\pas})$.
Therefore, thanks to \relem{expbruit}, for any $k \geq 0$ such that
\begin{equation*}
\tpl{\param_\tpsk{k}}{\state_\tpsk{k}}{\tilde{\jope}_\tpsk{k}} \in \enstq{\prmctrl \in \epctrl}{\dist{\prmctrl}{\pctrlopt} \leq \frac{\rayoptct}{3}}\times \tube_{\tpsk{k}}\times\tubejope_{\tpsk{k}},
\end{equation*}
we have
\begin{equation*}
\ba
\bruitif{\tpsk{k}}{\tpsk{k+1}}{\param_\tpsk{k}}{(\mem_t)} 
&\leq \nrm{\som{s=\tpsk{k}+1}{\tpsk{k+1}} \, \paren{\som{p =
s}{\tpsk{k+1}} \,
\pas_p\,\prdope{\tpsk{k}}{s}{p}(\state_{\tpsk{k}},\param_{\tpsk{k}})}\, E_s} + \go{\paren{\som{s=\tpsk{k}+1}{\tpsk{k+1}} \pas_s\,\fmdper{s}}^2}
\\&=N_k+\go{\pas_{\tpsk{k}}^2\,\flng{\tpsk{k}}^2\,\fmdper{\tpsk{k}}^2}
\ea
\end{equation*}
by the fifth point of \recor{convzerosompasintervalles} (remembering we
use $\mlipgrad{\cdot}=\fmdper{\cdot}$), and where we have
introduced
\begin{equation*}
N_k
\deq \nrm{\som{t=\tpsk{k}+1}{\tpsk{k+1}} \, \paren{\som{p =
t}{\tpsk{k+1}} \, \pas_p\,\prdope{\tpsk{k}}{t}{p}(\state_{\tpsk{k}},\param_{\tpsk{k}})}\, E_t}.
\end{equation*}
We will now bound $N_k$ in probability via its second moment, conditioned on the event that
$\tpl{\param_\tpsk{k}}{\state_\tpsk{k}}{\tilde{\jope}_\tpsk{k}} \in \mathcal{E}_k$.

\paragraph{Expressing the noise term of the upper bound.}
For $\tpsk{k}+1 \leq t \leq \tpsk{k+1}$, define
\begin{equation*}
c_t\deq\som{p = t}{\tpsk{k+1}} \, \pas_p\,\prdope{\tpsk{k}}{t}{p}(\state_{\tpsk{k}},\param_{\tpsk{k}})
\end{equation*}
so that $N_k=\norm{\som{t=\tpsk{k}+1}{\tpsk{k+1}} \, c_t \, E_t}$. Note
that $c_t\,E_t\in \Param$ by definition of $c_t$ and of the operators
$\prdope{\tpsk{k}}{t}{p}$ (Def.~\ref{def:prodiffopevolpretats}).
Then
\begin{equation*}
\ba
N_k^2
&=\nrm{\som{t=\tpsk{k}+1}{\tpsk{k+1}} \, c_t \, E_t}^2
=\som{t=\tpsk{k}+1}{\tpsk{k+1}}\,\som{s=\tpsk{k}+1}{\tpsk{k+1}} \, \psv{c_t\,E_t}{c_s\,E_{s}}.
\ea
\end{equation*}

As a result, 
\begin{equation*}
\econd{N_k^2
}{\trib{\tpsk{k}}} = \som{t=\tpsk{k}+1}{\tpsk{k+1}}\,\som{s=\tpsk{k}+1}{\tpsk{k+1}} \, \econd{\psv{c_t\,E_t}{c_s\,E_{s}}}{\trib{\tpsk{k}}}.
\end{equation*}
Now, thanks to the unbiasedness assumption for $E_t$, the cross-terms $s\neq t$ vanish: indeed, for every $\tpsk{k} +1 \leq s < t \leq \tpsk{k+1}$,
$c_t$, $c_s$ and $E_s$ are
\mesl{t-1} by \recor{measapproxtraj} and because $t-1\geq s$, and
therefore,
\begin{equation*}
\ba
\econd{\psv{c_t\,E_t}{c_s\,E_{s}}}{\trib{\tpsk{k}}}&=\econd{\econd{\psv{c_t\,E_t}{c_s\,E_{s}}}{\trib{t-1}}}{\trib{\tpsk{k}}} \\
&=\econd{\psv{c_t\,\econd{E_t}{\trib{t-1}}}{c_s\,E_{s}}}{\trib{\tpsk{k}}}\\
&=0,
\ea
\end{equation*}
thanks to \rehyp{wcorrnoiseapprtrl}.
As a consequence,
\begin{equation*}
\ba
\econd{
N_k^2
}{\trib{\tpsk{k}}}&=\som{t=\tpsk{k}+1}{\tpsk{k+1}} \,
\econd{\nrm{c_t\,E_t}^2}{\trib{\tpsk{k}}}
\leq \som{t=\tpsk{k}+1}{\tpsk{k+1}} \, \nrmop{c_t}^2\,\econd{\nrmop{E_t}^2}{\trib{\tpsk{k}}},
\ea
\end{equation*}
again because $c_t$ is $\trib{\tpsk{k}}$-measurable
(\recor{measapproxtraj}), and because
$\norm{c_t\,E_t}\leq \nrmop{c_t}\nrmop{E_t}$ for $E_t\in
\sblin(\Param,\State_t)$ and $c_t\in
\sblin(\sblin(\Param,\State_t),\Param)$.

\paragraph{Control of the $\nrmop{E_t}$'s.}
By \recor{rtrlstabletubesJ}, trajectories of the \algonoisy RTRL
algorithm preserve the stable tube for $\state$ and $\tilde \jope$.
Therefore, if $\tpl{\param_\tpsk{k}}{\state_\tpsk{k}}{\tilde{\jope}_\tpsk{k}} \in
\mathcal{E}_k$, then 
by \relemdeux{maintientpsfinibocontrandomtraj}{minmthoriztpsk},
$\param_t$, $\state_t$ and $\tilde \jope_t$ stay in the stable tube at
least up to time $\tpsk{k+1}$. Since the stable tube is uniformly bounded, the
sequence $\tilde \jope_t$ is bounded for $\tpsk{k} \leq t
\leq \tpsk{k+1}$.
%
%
As a result, thanks to 
\rehyp{errorgauge} and
\relem{localopnorms}, the sequence
$\paren{E_t}$ is bounded in this interval (again, conditioned on
$\tpl{\param_\tpsk{k}}{\state_\tpsk{k}}{\tilde{\jope}_\tpsk{k}} \in
\mathcal{E}_k$).
Therefore,
\begin{equation*}
\econd{
N_k^2}{\trib{\tpsk{k}}}
=\go{ \som{t=\tpsk{k}+1}{\tpsk{k+1}} \, \nrmop{c_t}^2}.
\end{equation*}

\paragraph{Control of the $c_t$'s.}
Now, for all $\tpsk{k}+1 \leq t \leq \tpsk{k+1}$, by definition of $c_t$,
\begin{equation*}
\nrmop{c_t}\leq \som{p = t}{\tpsk{k+1}} \,
\pas_p\,\nrmop{\prdope{\tpsk{k}}{t}{p}(\state_{\tpsk{k}},\param_{\tpsk{k}})},
\end{equation*}
so that
\begin{equation*}
\nrmop{c_t}^2\leq \paren{\som{p = t}{\tpsk{k+1}} \,
\pas_p\,\nrmop{\prdope{\tpsk{k}}{t}{p}(\state_\tpsk{k},\param_\tpsk{k})}}^2.
\end{equation*}

Thanks to \relem{Pinorm},
\begin{equation*}
\nrmop{\prdope{\tpsk{k}}{t}{p}(\state_\tpsk{k},\param_\tpsk{k})}\leq M\,\fmdper{t}\,(1-\alpha/2)^{(p-t)/\hsr}
\end{equation*}
for some constant $M>0$,
where $\alpha$ and $\hsr$ are
the spectral radius constants of \rehyp{specrad}.
Therefore,
\begin{equation*}
\nrmop{c_t}^2 \leq M^2\,\paren{\som{p = t}{\tpsk{k+1}} \, \pas_p\,\fmdper{p} \, \paren{1-\alpha/2}^{{\frac{p-t}{\hsr}}}}^2.
\end{equation*}
Thanks to Schwarz's inequality, writing $\cst_\alpha=\sum_{m\geq
0}\,\paren{1-\alpha/2}^{{\frac{m}{\hsr}}}<\infty$, we have
\begin{equation*}
\ba
\nrmop{c_t}^2 &\leq M^2\,{\som{n=t}{\tpsk{k+1}}\paren{\pas_n\,\fmdper{n}}^2 \, \paren{1-\alpha/2}^{{\frac{n-t}{\hsr}}}\,\som{m=t}{\tpsk{k+1}}\, \paren{1-\alpha/2}^{{\frac{m-t}{\hsr}}}}\\
&\leq M^2\,\cst_\alpha\,{\som{n=t}{\tpsk{k+1}}\paren{\pas_n\,\fmdper{n}}^2 \, \paren{1-\alpha/2}^{{\frac{n-t}{\hsr}}}},
\ea
\end{equation*}
so that
\begin{equation*}
\ba
\sum_{t=\tpsk{k}+1}^{\tpsk{k+1}}\,\nrmop{c_t}^2 &\leq M^2\,\cst_\alpha \, \sum_{t=\tpsk{k}+1}^{\tpsk{k+1}}\,{\som{n=t}{\tpsk{k+1}}\paren{\pas_n\,\fmdper{n}}^2 \, \paren{1-\alpha/2}^{{\frac{n-t}{\hsr}}}}\\
&=M^2\,\cst_\alpha \, \sum_{n=\tpsk{k}+1}^{\tpsk{k+1}}\,\paren{\pas_n\,\fmdper{n}}^2\,{\som{t=\tpsk{k}+1}{n} \,\paren{1-\alpha/2}^{{\frac{n-t}{\hsr}}}}\\
&\leq M^2\,\cst_\alpha^2 \, \sum_{n=\tpsk{k}+1}^{\tpsk{k+1}}\,\paren{\pas_n\,\fmdper{n}}^2.
\ea
\end{equation*}

\paragraph{Upper-bounding the noise term, and applying Borel-Cantelli's lemma.}
As a consequence,
\begin{equation*}
\ba
\econd{N_k
^2}{\trib{\tpsk{k}}}
&= \go{\som{t=\tpsk{k}+1}{\tpsk{k+1}} \, \pas_t^2\,\fmdper{t}^2}
=\go{\pas_{\tpsk{k}}^2\,\fmdper{\tpsk{k}}^2\,\flng{\tpsk{k}}},
\ea
\end{equation*}
as $k$ tends to infinity, thanks to the sixth point of
\recor{convzerosompasintervalles}. Moreover, this bound is uniform over
$\enstq{\prmctrl \in \epctrl}{\dist{\prmctrl}{\pctrlopt} \leq
\frac{\rayoptct}{3}}\times\tube_{\tpsk{k}}\times\tubejope_{\tpsk{k}}$
and over $\cdvp\leq \brnp$, because all of the intermediate results we
used are.

Let us now define a sequence $(\delta_k)$ by
\begin{equation*}
\delta_k\deq k^{-{\frac{{\exli-\exmp}-1/2}{2\paren{1-\exli}}}}
\end{equation*}
where the exponents are those appearing in 
\relem{intvavrg}.
By this lemma,
$\exli>\exmp+1/2$, so that $\delta_k$ tends to $0$.

Then, for any $k \geq 0$, 
thanks to Bienaymé--Chebyshev's inequality, we have 

\begin{multline*}
\prob{\mathbbm{1}_{\tpl{\param_{\tpsk{k}}}{\state_{\tpsk{k}}}{\tilde{\jope}_{\tpsk{k}}}\in\mathcal{E}_k}\,N_k
\geq
\delta_k\,\pas_{\tpsk{k}}\,\flng{\tpsk{k}}}\\
\leq \inv{\delta_k^2\,\pas_{\tpsk{k}}^2\,\flng{\tpsk{k}}^2} \,
\espe{\mathbbm{1}_{\tpl{\param_{\tpsk{k}}}{\state_{\tpsk{k}}}{\tilde{\jope}_{\tpsk{k}}}\in\mathcal{E}_k}\,N_k
^2} \\
= \inv{\delta_k^2\,\pas_{\tpsk{k}}^2\,\flng{\tpsk{k}}^2} \,
\espe{\mathbbm{1}_{\tpl{\param_{\tpsk{k}}}{\state_{\tpsk{k}}}{\tilde{\jope}_{\tpsk{k}}}\in\mathcal{E}_k}\,\econd{N_k
^2}{\trib{\tpsk{k}}}} \\
\leq \go{\frac{\fmdper{\tpsk{k}}^2}{\delta_k^2\,\flng{\tpsk{k}}}}
=\go{\frac{\tpsk{k}^{2\exmp}}{\delta_k^2\,\tpsk{k}^{\exli}}},
\end{multline*}
since $\tpl{\param_{\tpsk{k}}}{\state_{\tpsk{k}}}{\tilde{\jope}_{\tpsk{k}}}$ is
\mesl{\tpsk{k}} as we saw in Lemma~\ref{lem:measapproxrtrlalgo}, and by the definitions of $\fmdper{T}$ and $\flng{T}$
(Section~\ref{sec:stepsizes}).

Now, thanks to \relem{tmsexrtrlalgo}, we know that $\tpsk{k}\sim c\,k^{1/\paren{1-\exli}}$, for some $c>0$, as $k$ tends to infinity. As a result, we have
\begin{equation*}
p_k\deq 
\frac{\fmdper{\tpsk{k}}^2}{\delta_k^2\,\flng{\tpsk{k}}}
=
\delta_k^{-2}\,\tpsk{k}^{2\exmp-\exli}\sim 
c^{2\exmp-\exli}\,k^{\frac{\exmp-1/2}{1-\exli}}
\end{equation*}
by our choice of $\delta_k$.
By \relem{intvavrg}, $\exmp<\exli-1/2$ so that $(\exmp-1/2)/(1-A)<-1$.
Therefore, the series $\sum p_k$ converges.

As a result, for all $K\geq 0$, we have
\begin{equation*}
\prob{\exists k\geq K \text{ such that } \mathbbm{1}_{\tpl{\param_{\tpsk{k}}}{\state_{\tpsk{k}}}{\tilde{\jope}_{\tpsk{k}}}\in\mathcal{E}_k}\,N_k
\geq \delta_k\,\pas_{\tpsk{k}}\,\flng{\tpsk{k}}}\leq \go{\sum_{k\geq K}\,
p_k
}.
\end{equation*}
Let then $\eps > 0$, and let $K_\eps$ such that the upper bound
$\go{\sum_{k\geq K_\eps} p_k}$ is less than
$\eps$. Since all the bounds are uniform over $\olr\leq \brnp$, 
$K_\eps$ is independent of $\olr$. Thus, for any $\olr\leq \brnp$,
\begin{equation*}
\prob{\mathbbm{1}_{\tpl{\param_{\tpsk{k}}}{\state_{\tpsk{k}}}{\tilde{\jope}_{\tpsk{k}}}\in\mathcal{E}_k}\,N_k
\leq \delta_k\,\pas_{\tpsk{k}}\,\flng{\tpsk{k}},\quad \forall k\geq K_\eps}\geq 1-\eps.
\end{equation*}

Remember that
$\bruitif{\tpsk{k}}{\tpsk{k+1}}{\param_\tpsk{k}}{(\mem_t)} 
=N_k+\go{\pas_{\tpsk{k}}^2\,\flng{\tpsk{k}}^2\,\fmdper{\tpsk{k}}^2}$.
Thus, on the set $\acco{\mathbbm{1}_{\tpl{\param_{\tpsk{k}}}{\state_{\tpsk{k}}}{\tilde{\jope}_{\tpsk{k}}}\in\mathcal{E}_k}\,N_k
\leq \delta_k\,\pas_{\tpsk{k}}\,\flng{\tpsk{k}},\; \forall k\geq
K_\eps}$ we have, for every $k\geq K_\eps$,
\begin{multline*}
\mathbbm{1}_{\tpl{\param_{\tpsk{k}}}{\state_{\tpsk{k}}}{\tilde{\jope}_{\tpsk{k}}}\in\mathcal{E}_k}\,\bruitif{\tpsk{k}}{\tpsk{k+1}}{\param_\tpsk{k}}{(\mem_t)}\\
\leq
\paren{\delta_k+\go{\pas_{\tpsk{k}}\,\flng{\tpsk{k}}\,\fmdper{\tpsk{k}}^2}}\,\pas_{\tpsk{k}}\,\flng{\tpsk{k}}.
\end{multline*}

\relem{intvavrg} shows that $\pas_{\tpsk{k}}\,\flng{\tpsk{k}}\,\fmdper{\tpsk{k}}^2$ converges to $0$, as $k$ tends to infinity.
Moreover, as before, the constants of the $O$ term only depend on the constants of the problem, and are uniform with respect to $\olr \leq \brnp$.
Let us write, for $k\geq 0$, $\tilde{\delta}_k\deq {\delta_k+\go{\pas_{\tpsk{k}}\,\flng{\tpsk{k}}\,\fmdper{\tpsk{k}}^2}}$.
Therefore, for every $\olr \leq \brnp$, we have
\begin{equation*}
\prob{\mathbbm{1}_{\tpl{\param_{\tpsk{k}}}{\state_{\tpsk{k}}}{\tilde{\jope}_{\tpsk{k}}}\in\mathcal{E}_k}\,\bruitif{\tpsk{k}}{\tpsk{k+1}}{\param_\tpsk{k}}{(\mem_t)}
\leq {\tilde{\delta}_k}\,\pas_{\tpsk{k}}\,\flng{\tpsk{k}},\quad \forall
k\geq K_\eps}\geq 1-\eps.
\end{equation*}
We have therefore established the claim.
\end{proof}


\subsection{Convergence of the \rtrl Algorithm, \Algonoisy \rtrl Algorithms, and of the TBPTT Algorithm}
\label{sec:proofscvrtrlapprox}

\begin{corollaire}[Convergence of \rtrl, extended \rtrl algorithms, and \algonoisy \rtrl algorithms]
\label{cor:cvrtrlapproxrtrlpreuve}
\rethm{cvapproxrtrlalgo} holds.
\end{corollaire}

\begin{proof}
Remember that \redef{rtrlasalgo} casts the operators associated with
any extended RTRL algorithm in the abstract framework of
\resec{absontadsys}. Thus, 
the proof consists in showing that extended \rtrl algorithms and \algonoisy \rtrl
algorithms satisfy all the assumptions of \rethm{cvalgopti} and \rethm{cvalgopti_noisy} respectively, albeit on a
smaller ball $\boctopt\subset\Param$ than the ball on which the
assumptions of Section~\ref{sec:presentation} hold (because the ball
$\boctopt$ has been reduced several times in the proofs above).

In \resec{absontadsys}, 
we have a (non-\algonoisy) algorithm $\Algo$ that has to satisfy
a series of assumptions. For extended \rtrl
algorithms, we check these assumptions directly.

\Algonoisy \rtrl algorithms are treated somewhat differently.
For every \algonoisy \rtrl algorithm, there is a corresponding
(non-\algonoisy) extended
\rtrl algorithm obtained by setting $E_t=0$ in the definition
(namely,
Defs.~\ref{def:algortrl} and~\ref{def:approxrtrl} with the same system
and update operators, respectively without or with noise $E_t$).
\Algonoisy RTRL
algorithms are random trajectories in the sense of
\redef{randomtraj}; then we apply 
\rethm{cvalgopti_noisy}. For this,
\algonoisy RTRL algorithms do not need to satisfy all the assumptions of
Section~\ref{sec:absontadsys}, only to
respect the stable tube
(Def.~\ref{def:randomtraj}), and to
have negligible noise (Def.~\ref{def:negnoise}) with respect to the
corresponding non-\algonoisy RTRL algorithm. The corresponding non-\algonoisy algorithm
does have to satisfy the assumptions of Section~\ref{sec:absontadsys}.

We now turn to each of the assumptions of Section~\ref{sec:absontadsys}.

Thanks to \recor{rtrlstabletubesJ}, both extended \rtrl algorithms and
\algonoisy \rtrl algorithms admit a stable tube, thus satisfying
\rehyp{stableballs}. In particular, \algonoisy RTRL algorithms with
random noise produce random trajectories which respect the stable tube in
the sense of \redef{randomtraj}.

Thanks to \relem{expforgstjrtrlfxprm} and \recor{rtrllipswrtprm}, the
exponential forgetting and Lipschitz \rehypdeux{expforgetinit}{paramlip}
are satisfied for extended RTRL algorithms.

\relem{boundgradrtrl} proves that the gradient computation operators are bounded on the stable tube, and Lipschitz, so that \rehypdeux{boundedgrads}{contgrad} are satisfied.

The parameter update operators are Lipschitz, thanks to \reprop{prmupdatesfirstorderstepsize}, so that \rehyp{updateop} is satisfied.

Thanks to \rehyp{scfcstepsize}, \relem{intvavrg} and \relem{homsatis}, \rehyp{spdes} on the stepsize sequence is satisfied.

Thanks to \relemdeux{rtrlstableopt}{contractpctrlopt}, 
$\pctrlopt$ satisfies the local optimality conditions of \rehyp{optimprm}.

Thanks to \relem{nsectrlapprtrl}, there exists a sequence
$\paren{\delta_k}$ tending to $0$ such that, for any $\eps > 0$, there exists $K\paren{\eps}$ such that, for any $\bar{\pas}\leq\brnp$ ($\brnp>0$ is introduced in \reprop{prmupdatesfirstorderstepsize}), with probability greater than $1-\eps$, \algonoisy \rtrl algorithms have negligible noise starting at $K\paren{\eps}$, at speed $\paren{\delta_k}$.

Set $\mathcal{N}_{\paramopt}\deq \enstq{\prmctrl \in
\epctrl}{\dist{\prmctrl}{\pctrlopt} \leq \frac{\rayoptct}{4}}$ where
$\rayoptct$ is the radius of $\boctopt$, 
$\mathcal{N}_{\stateopt_0}\deq \tube_0$, and $\mathcal{N}^\jope_{0}\deq\tubejope_0$.
Then, thanks to \rethm{cvalgopti}, there exists $\cdvpmaxconv>0$ such
that an extended RTRL algorithm initialized 
any parameter $\param_0\in \mathcal{N}_{\paramopt}$, any
state $\state_0\in \mathcal{N}_{\stateopt_0}$ and
any differential $\tilde\jope_0\in \mathcal{N}^\jope_{0}$,
produces a sequence of parameters $\param_t$ that converges to
$\paramopt$.
Moreover, by \rethm{cvalgopti_noisy},
for any $\eps>0$, there exists $\cdvpmaxconv>0$ such that, 
for any overall learning rate $\cdvp\leq\cdvpmaxconv$, any initial parameter $\param_0\in \mathcal{N}_{\paramopt}$, any
initial state $\state_0\in \mathcal{N}_{\stateopt_0}$ and
any initial differential $\tilde\jope_0\in \mathcal{N}^\jope_{0}$, with
probability at least $1-\eps$, an \algonoisy \rtrl algorithm produces a sequence of parameters $\param_t$ converging to $\pctrlopt$.
\end{proof}

We now turn to truncated backpropagation through time (TBPTT), which uses
essentially the same approach using the open-loop algorithm.

We start with the following classical result: on a finite interval
$[\tpsk{k};\tpsk{k+1}]$, TBPTT is equivalent to an RTRL algorithm that
updates the parameter only at the end of the interval and initializes
$\jope$ to $0$ at the beginning of the interval. 
However, 
technically we cannot define $\jope_{\tpsk{k}}$ to $0$ at the start of
each interval, because its value
is used to compute the gradient at the end of the previous interval.  So we just
define the next value of $\jope$ as if $\jope_{\tpsk{k}}=0$ in the
formula below.

\begin{proposition}[TBPTT as RTRL on intervals]
\label{prop:tbpttasrtrl}
The TBPTT algorithm (Def.~\ref{def:tbtt}) is equivalent to RTRL with the parameter updated
on steps $\tpsk{k}$, and the influence of
$\jope_t$ on $\jope_{t+1}$ cut at time $\tpsk{k}$, namely:
\begin{equation*}
\left\lbrace
\ba
\state_{t} &= \opevol_t(\state_{t-1},\param_{\tpsk{k}}),
\qquad  \tpsk{k}+1\leq t\leq \tpsk{k+1}
\\
\jope_{\tpsk{k}+1}&=\frac{\partial
\opevol_{\tpsk{k}+1}(\state_{\tpsk{k}},\param_{\tpsk{k}})}{\partial
\param},
\\
\jope_t &= \frac{\partial \opevol_t(\state_{t-1},\param_{\tpsk{k}})}{\partial
\state} \,\jope_{t-1}+\frac{\partial
\opevol_t(\state_{t-1},\param_{\tpsk{k}})}{\partial
\param},\qquad \tpsk{k}+2\leq t\leq \tpsk{k+1}
\\
\vtanc_t &= {\frac{\partial
\perte_t(\state_t)}{\partial \state}\cdot
\jope_t},
\qquad \tpsk{k}+1\leq t\leq \tpsk{k+1},
\ea \right.
\end{equation*}
and parameter update
\begin{equation*}
\param_{\tpsk{k+1}}=\param_{\tpsk{k}}-\pas_{\tpsk{k+1}}\sum_{t=\tpsk{k}+1}^{\tpsk{k+1}}\vtanc_t.
\end{equation*}
\end{proposition}

\begin{proof}
The proof is classical
(as long as the parameter is not updated, RTRL and backpropagation
through time compute the same gradient, by equivalence of forward and
backward gradient computations), and we omit it.
\end{proof}

\begin{corollaire}[Convergence of the TPBTT algorithm]
\rethm{cvtbtt} holds.
\end{corollaire}

\begin{proof}
The proof consists in showing the TBPTT algorithm satisfies the
assumptions of \rethm{cvalgoboucleouverte}.

Remember that \redef{rtrlasalgo} defines the abstract operators $\Algo_t$
and $\sm{V}_t$ for the RTRL algorithm.

Then one checks that the algorithm in
\reprop{tbpttasrtrl} is equivalent to
the open-loop algorithm studied in \rethm{cvalgoboucleouverte},
with $\Algo_t$ the RTRL update on $\mem=(\state,J)$ from
\redef{rtrlasalgo}, with $\sm{V}_t$ as in \redef{rtrlasalgo} using
$\fur_t(\vtanc,\state,\param)=\vtanc$, with
$\paramupdate_t\cpl{\param}{\vtanc}=\param-\vtanc$, and with
\begin{equation*}
\mem'_{\tpsk{k}}\deq \cpl{\state_{\tpsk{k}}}{0}.
\end{equation*}
Indeed, in that case, the
update
\begin{equation*}
\mem_{\tpsk{k}} \leftarrow \mem'_{\tpsk{k}}\in\bmaint{\tpsk{k}}
\end{equation*}
at the beginning of each interval, is equivalent to resetting $J$ to $0$
at the beginning of each interval.
(Note that
resetting the value of $\jope$ to $0$ stays in the stable tube, thanks to
\recor{rtrlstabletubesJ}.) 
 Moreover, since
$\paramupdate_t\cpl{\param}{\vtanc}=\param-\vtanc$, the iterated update
$\param_{t} =
\paramupdate_t\paren{\param_{t-1},\,\pas_{\tpsk{k+1}}\,\vtanc_t}$ in
\rethm{cvalgoboucleouverte} is
equivalent to
\begin{equation*}
\param_{\tpsk{k+1}}=\param_{\tpsk{k}}-\sum_{t=\tpsk{k}+1}^{\tpsk{k+1}}\,\pas_{\tpsk{k+1}} \vtanc_t,
\end{equation*}
which is indeed the parameter update of \reprop{tbpttasrtrl}.

Moreover, the interval lengths $\tpsk{k+1}-\tpsk{k}=\tpsk{k}^\exli$
in the assumptions of \rethm{cvtbtt} match the definition of $\tpsk{k}$
in \relem{tmsexrtrlalgo}, with the constraint on $\exli$ from
\relem{intvavrg}.

The assumptions of \rethm{cvalgoboucleouverte} are the same as those of
\rethm{cvalgopti}; we have checked above in the proof of
\recor{cvrtrlapproxrtrlpreuve} that these assumptions are satisfied for
RTRL. 

Thus, by \rethm{cvalgoboucleouverte} and \reprop{tbpttasrtrl}, the TBPTT algorithm produces a sequence of parameters $\param_{\tpsk{k}}$ converging to $\paramopt$.
\end{proof}

\subsection{NoBackTrack and UORO as \Algonoisy RTRL Algorithms}
\label{sec:nbtuoroasapproxrtrlalgo}

We now prove \relem{nbtuoroapproxrtrl}: NoBackTrack and UORO satisfy the
unbiasedness and bounded noise assumptions.

The notation for NoBackTrack and UORO is introduced in
Section~\ref{sec:defnbtuoro}.
Remember that the noise $E_t$ is defined via
random signs $\veber{t}$ at each time $t$. 
For every $t\geq 1$, we write $\tribd{t}$ the $\sigma$-algebra generated
by $\tribo$ (defined in \rehyp{wcorrnoiseapprtrl}) together with the $\veber{s}$ for $1\leq s \leq t$. Since $E_t$ is
computed from the $\veber{t}$'s, the $\sigma$-algebra $\trib{t}$ generated
by the $(E_s)_{s\leq t}$ is contained in $\tribd{t}$: $\trib{t}\subset
\tribd{t}$. For \rehyp{wcorrnoiseapprtrl}, we want to prove that $\econd{E_t}{\trib{t-1}}=0$; we will
prove the stronger result that $\econd{E_t}{\tribd{t-1}}=0$.

Computing the conditional expectation with respect
to $\tribd{t}$ means integrating with respect to the laws of the
$\veber{s}$'s, for $s > t$.

We are going to study the error $E_t$ at some time $t\geq
1$.
In this section, we fix the following abbreviations.
For $t \geq 1$, let the values of the NoBackTrack objects at time
$t-1$ be $\prmctrl =\param_{t-1}\in \epctrl$, $\state
=\state_{t-1}\in \State_{t-1}$ and
$\jope =\tilde \jope_{t-1}\in \sepjope{t-1}$ with $\jope =
\tens{\vsyst}{\vctrl}$,
and further abbreviate
\begin{equation*}
\sbopmulb\deq\dpartf{\state}{\opevolt}\paren{\state,\,\prmctrl} \quad
\text{and} \quad \lopaddb{i}\deq \dpartf{\prmctrl}{\opevolt^i}\paren{\state,\,\prmctrl}, \quad 1 \leq i \leq \dim \State_{t}
\end{equation*}
where $\opevolt^i$ is the $i$-th component of $\opevolt$ in the basis of
$\State_t$
used to define NoBackTrack.
Finally, abbreviate  $\sigi{i}$ for $\beri{t}$.

\subsubsection{NoBackTrack as an \Algonoisy RTRL Algorithm}

Let us first express the NoBackTrack update.
Let
\begin{equation*}
\cvnbtaxsi\deq \opredt{\vsyst}{\vctrl}{\state}{\prmctrl},
\end{equation*}
where $\sbopred_t$ is the reduction operator for NoBackTrack defined in
Section~\ref{sec:defnbtuoro}.
Then abbreviate
\begin{equation*}
\rho\deq \fegn{\opmulb{\vsyst}}{\vctrl} \quad \text{and} \quad \rho_i\deq \fegn{\vcans_i}{\lopaddb{i}}, \quad 1 \leq i \leq \dim \State_{t},
\end{equation*}
with $\vcans_i$ the basis vectors used to define NoBackTrack.
If ${\opmulb{\vsyst}}=0$, we define $\rho\deq 1$.

Then by definition of the NoBackTrack reduction operator $\sbopred_t$, we
have
\begin{equation*}
\left\lbrace \ba
\vsystax&=\rho \, \opmulb{\vsyst} + \sum_{i=1}^{\dim \State_{t}} \, \sigi{i} \, \rho_i \, \vcans_i \\
\vctrlax&=\rho^{-1} \, \vctrl+\sum_{i=1}^{\dim \State_{t}} \sigi{i} \, \rho_i^{-1} \, \lopaddb{i}.
\ea \right.
\end{equation*}

By induction and by \relem{measapproxrtrlalgo}, 
$\param=\param_{t-1}$, $\state=\state_{t-1}$,
$\jope=\tilde \jope_{t-1}$ and thus
$\sbopmulb$, $\sbopaddb$, $\rho$, and $\rho_i$
are
$\trib{t-1}$-measurable hence $\tribd{t-1}$-measurable. (We assume the
basis $\vcans_i$ is deterministic.) Thus $\vsystax$
and $\vctrlax$ are $\tribd{t}$-measurable, since they also use the
$\sigi{i}$'s at time $t$.
Note that in NoBackTrack, $\vsyst$ and $\vctrl$ at time $t$ 
are $\vsystax$
and $\vctrlax$ from the previous step $t-1$, so that $\vsyst$ and
$\vctrl$ are $\tribd{t-1}$-measurable.

As a result, from the definition of NoBackTrack (Def.~\ref{def:nbtuoro}),
the \algonoisy Jacobian computed by the NoBackTrack update
is
\begin{equation*}
\ba
\tilde\jope_t=
\tens{\vsystax}{\vctrlax} &= \tens{\opmulb{\vsyst}}{\vctrl} + \som{i=1}{\dim \State_{t}} \, \underbrace{\sigi{i}^2}_{=1} \, \tens{\vcans_i}{\lopaddb{i}}\\
&+\rho \, \tens{\opmulb{\vsyst}}{\sum_{i=1}^{\dim \State_{t}} \sigi{i} \, \rho_i^{-1} \, \lopaddb{i}}
+ \rho^{-1} \tens{\vctrl}{\sum_{i=1}^{\dim \State_{t}} \, \sigi{i} \, \rho_i \, \vcans_i}\\
&+ \som{i,j=1,\,i\neq j}{\dim \State_{t}} \tens{\sigi{i} \, \rho_i \, \vcans_i}{\sigi{j} \, \rho_j^{-1} \, \lopaddb{j}},
\ea
\end{equation*}
and
the error term $E_t$ is
\begin{equation*}
\ba
E_t&=\tilde \jope_t
-\paren{\dpartf{\state}{\opevolt}\paren{\state,\,\prmctrl} \cdot \jope + \dpartf{\prmctrl}{\opevolt}\paren{\state,\,\prmctrl}} 
\\&=
\tens{\vsystax}{\vctrlax} - \paren{\tens{\opmulb{\vsyst}}{\vctrl} +
{\sbopaddb}}
\\&=
\tens{\vsystax}{\vctrlax} - \tens{\opmulb{\vsyst}}{\vctrl} -
\som{i=1}{\dim \State_{t}} \, \tens{\vcans_i}{\lopaddb{i}},
\ea
\end{equation*}
namely
\begin{equation}
\ba
E_t&=\rho \, \tens{\opmulb{\vsyst}}{\sum_{i=1}^{\dim \State_{t}} \sigi{i} \, \rho_i^{-1} \, \lopaddb{i}}
+ \rho^{-1} \tens{\vctrl}{\sum_{i=1}^{\dim \State_{t}} \, \sigi{i} \, \rho_i \, \vcans_i}\\
&+ \som{i,j=1,\,i\neq j}{\dim \State_{t}} \tens{\sigi{i} \, \rho_i \, \vcans_i}{\sigi{j} \, \rho_j^{-1} \, \lopaddb{j}}.
\ea
\label{eq:expression-errors-nbt}
\end{equation}

\begin{proof}[NoBackTrack -- unbiasedness of the Jacobian update rule]
Let us show that \rehyp{wcorrnoiseapprtrl} is satisfied with the
filtration $\acco{\tribd{t}}$, hence a fortiori with $\acco{\trib{t}}$.

By construction, at each time $t$ the $\sigi{i}$'s are random
Bernoulli variables that are independent from $\tribd{t-1}$. 
Thus, $\econd{\sigi{i}}{\tribd{t-1}}=0$ and
$\econd{\sigi{i}\sigi{j}}{\tribd{t-1}}=0$ for $i\neq j$.

On the other
hand, we have seen that all the other variables in the expression
\eqref{eq:expression-errors-nbt} for
$E_t$ are
$\tribd{t-1}$-measurable.
Then, taking the conditional expectation with respect to $\tribd{t-1}$ in
\eqref{eq:expression-errors-nbt}, we obtain
$\econd{E_t}{\tribd{t-1}}=0$ as needed.
\end{proof}

\begin{proof}[NoBackTrack -- size of the error]
Next, we deal
with \rehyp{errorgauge} (bounded errors) for NoBackTrack.

From the expression \eqref{eq:expression-errors-nbt} for $E_t$,
since for vectors $v_1$, $v_2$, we have $\nrmop{\tens{v_1}{v_2}}=\nrm{v_1}\nrm{v_2}$, it holds that
\begin{equation*}
\ba
\nrmop{E_t}&\leq \rho \, \nrm{\opmulb{\vsyst}} \, {\sum_{i=1}^{\dim \State_{t}} \, \rho_i^{-1} \, \nrm{\lopaddb{i}}}
+ \rho^{-1} \nrm{\vctrl} \, {\sum_{i=1}^{\dim \State_{t}} \, \rho_i \, \nrm{\vcans_i}}\\
&+ \som{i,j=1,\,i\neq j}{\dim \State_{t+1}} {\rho_i \, \nrm{\vcans_i}}{\rho_j^{-1} \, \nrm{\lopaddb{j}}}\\
&=\sqrt{\nrm{\opmulb{\vsyst}}\nrm{\vctrl}} \, {\sum_{i=1}^{\dim \State_{t}} \, \sqrt{\nrm{\vcans_i}\, \nrm{\lopaddb{i}}}}
+ \sqrt{\nrm{\opmulb{\vsyst}}\nrm{\vctrl}} \, {\sum_{i=1}^{\dim \State_{t}} \, \sqrt{\nrm{\vcans_i}\, \nrm{\lopaddb{i}}}}\\
&+\som{i,j=1,\,i\neq j}{\dim \State_{t+1}} {\sqrt{\nrm{\vcans_i}\, \nrm{\lopaddb{i}}}}{\sqrt{\nrm{\vcans_j}\, \nrm{\lopaddb{j}}}}\\
&\leq 2 \, \sqrt{\nrmop{\sbopmulb}} \, \sqrt{\nrm{\vsyst}\nrm{\vctrl}} \, {\sum_{i=1}^{\dim \State_{t}} \, \sqrt{\nrm{\vcans_i}\, \nrm{\lopaddb{i}}}}
+\som{i,j=1,\,i\neq j}{\dim \State_{t}} {\sqrt{\nrm{\vcans_i}\, \nrm{\lopaddb{i}}}}{\sqrt{\nrm{\vcans_j}\, \nrm{\lopaddb{j}}}}.
\ea
\end{equation*}

Now, $\jope=\tens{\vsyst}{\vctrl}$, so $\nrmop{\jope}=\nrm{\vsyst}\nrm{\vctrl}$. As a result,
\begin{equation*}
\sqrt{\nrm{\vsyst}\nrm{\vctrl}}=\nrmop{\jope}^{1/2}.
\end{equation*}
Next, 
since the $\vcans_i$'s
form and orthonormal basis of vectors of $\epsystins{t}$ according to
\redef{opredinstantt}, we have $\nrm{\vcans_i}=1$ and $\nrm{\lopaddb{i}}\leq
\nrmop{\sbopaddb}\nrm{\vcans_i}=\nrmop{\sbopaddb}$.
Finally, by definition of $\sbopmulb$ and $\sbopaddb$,
\begin{equation*}
\nrmop{\sbopmulb}, \, \nrmop{\sbopaddb} \leq \nrmop{\frac{\partial \opevol_t(\state_{t-1},\param_{t-1})}{\partial(\state,\param)}}.
\end{equation*}

Plugging this into the bound for $\nrmop{E_t}$, we find
\begin{equation*}
\nrmop{E_t}\leq (2\dim \State_t) \,y\,  \nrmop{\jope}^{1/2}+ (\dim
\State_t)^2\, y
\end{equation*}
where $y={\nrmop{\frac{\partial
\opevol_t(\state_{t-1},\param_{t-1})}{\partial(\state,\param)}}}$.

Since $\dim \State_t$ is bounded by
\rehyp{bndimstatespaces},
this shows the size of the error in NoBackTrack is compliant with the
requirements of \redef{errorgauge} and \rehyp{errorgauge}, with
$\feg{x}{y}=C\,(1+x^{1/2})(1+y)$.

Note that we have only proved that $\nrm{E_t}$ is controlled by
$\sqrt{\nrmop{\jope}}$ provided $\jope$ is rank-one. This is true by
construction at all times along the NoBackTrack trajectory, so the bound
above holds at all times on any NoBackTrack trajectory (with probability
$1$); this is all that
is needed for \rehyp{errorgauge}.
\end{proof}

\subsubsection{UORO as an \Algonoisy RTRL Algorithm}

The analysis of UORO is very similar to NoBackTrack. Let us first express the UORO update.
Let
\begin{equation*}
\cvnbtaxsi\deq \opredt{\vsyst}{\vctrl}{\state}{\prmctrl},
\end{equation*}
where $\sbopred_t$ is the reduction operator for UORO defined in \resec{defnbtuoro}.
Then, abbreviate 
\begin{equation*}
\rho_0\deq \fegn{\opmulb{\vsyst}}{\vctrl} \quad \text{and} \quad \rho_1\deq \fegn{\som{i=1}{\dim \epsystins{t}} \sigi{i} \, {\vcans_i}}{\som{i=1}{\dim \epsystins{t}}\,\sigi{i}\,\lopaddb{i}},
\end{equation*}
where the $\vcans_i$'s are the basis vectors used to define UORO. If ${\opmulb{\vsyst}}=0$, we define $\rho\deq 1$.

Then, by definition of the UORO reduction operator $\sbopred_t$, we have
\begin{equation*}
\left\lbrace \ba
\vsystax&=\rho_0 \, \opmulb{\vsyst} + \rho_1\,\sum_{i=1}^{\dim \State_{t}} \, \sigi{i} \, \vcans_i \\
\vctrlax&=\rho_0^{-1} \, \vctrl+\rho_1^{-1} \, \sum_{i=1}^{\dim \State_{t}} \sigi{i} \, \lopaddb{i}.
\ea \right.
\end{equation*}
By induction and by \relem{measapproxrtrlalgo}, 
$\param=\param_{t-1}$, $\state=\state_{t-1}$,
$\jope=\tilde \jope_{t-1}$ and thus
$\sbopmulb$, $\sbopaddb$, $\rho_0$, and $\rho_1$
are
$\trib{t-1}$-measurable hence $\tribd{t-1}$-measurable. (We assume the
basis $\vcans_i$ is deterministic.) Thus $\vsystax$
and $\vctrlax$ are $\tribd{t}$-measurable, since they also use the
$\sigi{i}$'s at time $t$.
Note that in UORO, $\vsyst$ and $\vctrl$ at time $t$ 
are $\vsystax$
and $\vctrlax$ from the previous step $t-1$, so that $\vsyst$ and
$\vctrl$ are $\tribd{t-1}$-measurable.

As a result, from the definition of UORO (Def.~\ref{def:nbtuoro}),
the \algonoisy Jacobian computed by the NoBackTrack update
is
\begin{equation*}
\ba
\tilde{\jope}_t=\tens{\vsystax}{\vctrlax} &= \tens{\opmulb{\vsyst}}{\vctrl} + \som{i=1}{\dim \State_{t}} \, \underbrace{\sigi{i}^2}_{=1} \, \tens{\vcans_i}{\lopaddb{i}}\\
&+\rho_0 \, \tens{\opmulb{\vsyst}}{\rho_1^{-1}\, \sum_{i=1}^{\dim \State_{t}} \sigi{i} \, \lopaddb{i}}
+ \rho_0^{-1} \, \tens{\vctrl}{\rho_1 \, \sum_{i=1}^{\dim \State_{t}} \, \sigi{i} \, \vcans_i}\\
&+ \som{i,j=1,\,i\neq j}{\dim \State_{t}} \tens{\sigi{i} \, \vcans_i}{\sigi{j} \, \lopaddb{j}},
\ea
\end{equation*}
and the error term $E_t$ is
\begin{equation*}
\ba
E_t&=\tilde \jope_t
-\paren{\dpartf{\state}{\opevolt}\paren{\state,\,\prmctrl} \cdot \jope + \dpartf{\prmctrl}{\opevolt}\paren{\state,\,\prmctrl}} 
\\&=
\tens{\vsystax}{\vctrlax} - \paren{\tens{\opmulb{\vsyst}}{\vctrl} +
{\sbopaddb}}
\\&=
\tens{\vsystax}{\vctrlax} - \tens{\opmulb{\vsyst}}{\vctrl} -
\som{i=1}{\dim \State_{t}} \, \tens{\vcans_i}{\lopaddb{i}},
\ea
\end{equation*}
namely
\begin{equation}
\label{eq:expression-errors-uoro}
\ba
E_t&=\rho_0 \, \tens{\opmulb{\vsyst}}{\rho_1^{-1}\, \sum_{i=1}^{\dim \State_{t}} \sigi{i} \, \lopaddb{i}}
+ \rho_0^{-1} \, \tens{\vctrl}{\rho_1 \, \sum_{i=1}^{\dim \State_{t}} \, \sigi{i} \, \vcans_i}\\
&+ \som{i,j=1,\,i\neq j}{\dim \State_{t}} \tens{\sigi{i} \, \vcans_i}{\sigi{j} \, \lopaddb{j}}.
\ea
\end{equation}

\begin{proof}[UORO -- unbiasedness of the Jacobian update rule]
Let us show that \rehyp{wcorrnoiseapprtrl} is satisfied with the
filtration $\acco{\tribd{t}}$, hence a fortiori with $\acco{\trib{t}}$.

By construction, at each time $t$ the $\sigi{i}$'s are random
Bernoulli variables that are independent from $\tribd{t-1}$. 
Thus, $\econd{\sigi{i}}{\tribd{t-1}}=0$ and
$\econd{\sigi{i}\sigi{j}}{\tribd{t-1}}=0$ for $i\neq j$.

On the other
hand, we have seen that all the other variables in the expression
\eqref{eq:expression-errors-uoro} for
$E_t$ are
$\tribd{t-1}$-measurable.
Then, taking the conditional expectation with respect to $\tribd{t-1}$ in
\eqref{eq:expression-errors-uoro}, we obtain
$\econd{E_t}{\tribd{t-1}}=0$ as needed.
\end{proof}

\begin{proof}[UORO -- size of the error]
Next, we deal
with \rehyp{errorgauge} (bounded errors) for UORO.

From the expression \eqref{eq:expression-errors-uoro} for $E_t$,
since for vectors $v_1$, $v_2$, we have $\nrmop{\tens{v_1}{v_2}}=\nrm{v_1}\nrm{v_2}$, it holds that
\begin{equation*}
\ba
\nrmop{E_t}&\leq \rho_0 \, \nrm{\opmulb{\vsyst}} \rho_1^{-1}\, \nrm{\sum_{i=1}^{\dim \State_{t+1}} \, \sigi{i} \, \lopaddb{i}}
+ \rho_0^{-1} \, \nrm{\vctrl} \rho_1 \, \nrm{\sum_{i=1}^{\dim \State_{t}} \, \sigi{i} \, \vcans_i}\\
&+ \som{i,j=1,\,i\neq j}{\dim \State_{t}} \nrm{\vcans_i}\nrm{\lopaddb{j}}\\
&=2\,\sqrt{\nrm{\opmulb{\vsyst}}\nrm{\vctrl}}\, \sqrt{\nrm{\som{i=1}{\dim \epsystins{t+1}} \sigi{i} \, {\vcans_i}}\nrm{\som{i=1}{\dim \epsystins{t}} \sigi{i} \, {\lopaddb{i}}}}\\
&+ \som{i,j=1,\,i\neq j}{\dim \State_{t+1}} \nrm{\vcans_i}\nrm{\lopaddb{j}}\\
&\leq2\,\sqrt{\nrmop{\sbopmulb}}\sqrt{\nrm{\vsyst}\nrm{\vctrl}}\, \sqrt{\nrm{\som{i=1}{\dim \epsystins{t}} \sigi{i} \, {\vcans_i}}\nrm{\som{i=1}{\dim \epsystins{t}} \sigi{i} \, {\lopaddb{i}}}}\\
&+ \som{i,j=1,\,i\neq j}{\dim \State_{t}} \nrm{\vcans_i}\nrm{\lopaddb{j}}.
\ea
\end{equation*}
Now, $\jope=\tens{\vsyst}{\vctrl}$, so $\nrmop{\jope}=\nrm{\vsyst}\nrm{\vctrl}$. As a result,
\begin{equation*}
\sqrt{\nrm{\vsyst}\nrm{\vctrl}}=\nrmop{\jope}^{1/2}.
\end{equation*}
Next, 
since the $\vcans_i$'s
form and orthonormal basis of vectors of $\epsystins{t}$ according to
\redef{opredinstantt}, we have $\nrm{\vcans_i}=1$ and $\nrm{\lopaddb{i}}\leq
\nrmop{\sbopaddb}\nrm{\vcans_i}=\nrmop{\sbopaddb}$.
Finally, by definition of $\sbopmulb$ and $\sbopaddb$,
\begin{equation*}
\nrmop{\sbopmulb}, \, \nrmop{\sbopaddb} \leq \nrmop{\frac{\partial \opevol_t(\state_{t-1},\param_{t-1})}{\partial(\state,\param)}}.
\end{equation*}

Plugging this into the bound for $\nrmop{E_t}$, we find
\begin{equation*}
\nrmop{E_t}\leq (2\dim \State_t) \,y\,  \nrmop{\jope}^{1/2}+ (\dim
\State_t)^2\, y
\end{equation*}
where $y={\nrmop{\frac{\partial
\opevol_t(\state_{t-1},\param_{t-1})}{\partial(\state,\param)}}}$.

Since $\dim \State_t$ is bounded by
\rehyp{bndimstatespaces},
this shows the size of the error in UORO is compliant with the
requirements of \redef{errorgauge} and \rehyp{errorgauge}, with
$\feg{x}{y}=C\,(1+x^{1/2})(1+y)$.

Note that we have only proved that $\nrm{E_t}$ is controlled by
$\sqrt{\nrmop{\jope}}$ provided $\jope$ is rank-one. This is true by
construction at all times along the UORO trajectory, so the bound
above holds at all times on any UORO trajectory (with probability
$1$); this is all that
is needed for \rehyp{errorgauge}.
\end{proof}

\appendix

\section{Positive-Stable Matrices}
\label{sec:positivestable}


We recall several equivalent definitions of positive-stable
matrices (also known, with signs reversed, as Hurwitz matrices).

\begin{definition}[Positive-stable matrix]
A real matrix $A$ is \emph{positive-stable} if one of the following equivalent
conditions is satisfied:
\begin{enumerate}
\item All the eigenvalues of $A$ have positive real part.
\item The solution of the differential equation $\theta'=-A\theta$ converges to $0$ for
any initial value.
\item There exists a symmetric, positive definite matrix $B$ (Lyapunov
function) such that $\transp{\theta}B\theta$ is decreasing along the
solutions of the differential equation $\theta'=-A\theta$.
\item There exists a symmetric positive definite matrix $B$ such that
\begin{equation*}
\transp{\theta} B A \theta>0
\end{equation*}
for all $\theta\neq 0$. (This is the same $B$ as in the previous
condition.)
\item There exists a symmetric positive definite matrix $B$ such that
$B A +\transp{A} B$
is positive definite.
\end{enumerate}
\end{definition}

Stability is invariant by matrix similarity $A\gets C^{-1}A C$, since
this preserves eigenvalues.

Since the
solution of $\theta_t'=-A\theta_t$ is $\theta_t=e^{-tA}\theta_0$, a
Lyapunov function that works is $\transp{\theta_0}B\theta_0\deq \int_{t \geq 0}
\norm{\theta_t}^2$. Indeed this is decreasing, because
$\transp{\theta_t}B\theta_t$ is just the same integral starting at $t$
instead of $0$. Explicitly this is
$B=\int_{t\geq 0} \transp{(e^{-t A})}
e^{-tA}$. It satisfies $BA+\transp{A}B=\id$.

The proofs are classical, and therefore we omit them.

\begin{proposition}
\begin{enumerate}
\label{prop:positivestablecriteria}
\item A symmetric positive definite matrix is positive-stable. (Take $B=\id$
above.)
\item If $A+\transp{A}$ is positive definite and $H$ is symmetric positive
definite, then $AH$ is positive-stable. (Take $B=H$ above.)
\end{enumerate}
\end{proposition}

\section{Equicontinuity of the Extended Hessians in the $C^3$ Case}
\label{sec:eqehbcase}
In this section, which has no bearing on the rest of the proof, we note
that equicontinuity of the
extended Hessians around $\paramopt$ (\rehyp{equicontH}), may be deduced
in a straightforward way under simpler additional assumptions.

\begin{lemme}[Controlling the derivatives of $\fstate_t$ and $\sbperte_t$]
\label{lem:contderivfstate}
Assume that all first, second and third order derivatives of the
$\opevol_t$'s and of the $\sbperte_t$'s exist, and are bounded on a ball
around the target trajectory $(\paramopt,\stateopt_t)$. (In particular,
we can take $\exmp=0$.)

Then all derivatives up to third order of the
$\fstate_t\cpl{\stateopt_0}{\cdot}$'s are bounded on $\boctopt$, 
where $\boctopt$ is defined in \relem{rtrlstabletubes}. Likewise,
all derivatives up to third order of the
$\sbpermc{t}\cpl{\stateopt_0}{\cdot}$'s are bounded as well, so that the
family of Hessians
$\ddpartf{\param}{\sbpermc{t}}\cpl{\stateopt_0}{\cdot}$ is equicontinuous
on $\boctopt$. In
other words, \rehyp{equicontH_simple} is satisfied on the smaller ball
$\boctopt$.
\end{lemme}
\begin{proof}
Let us first bound the derivatives of the $\fstate_t$'s.
\paragraph{Bounding the derivatives of the $\fstate_t$'s.}
For all $t\geq 1$ and $\param\in\Param$, we have
\begin{equation*}
\fstate_t\cpl{\stateopt_0}{\param}=\opevol_t\cpl{\fstate_{t-1}\cpl{\stateopt_0}{\param}}{\param}.
\end{equation*}
Let us write, for all $t\geq 1$ and $\param\in\boctopt$,
$\state_t=\fstate_t\cpl{\stateopt_0}{\param}$. Then, for all $t\geq 1$
and $\param\in\boctopt$, $\state_t$ is in the stable tube $\tube_t$ of
\relem{rtrlstabletubes}. (Note that \rehyp{equicontH} is not used for the
proof of \relem{rtrlstabletubes}.)

Let $t\geq 1$, and $\param\in\boctopt$. Then, we have 
\begin{equation*}
\ba
\dpartf{\param}{\fstate_t}\cpl{\stateopt_0}{\param}
&=\dpartf{\state}{\opevol_t}\cpl{\state_{t-1}}{\param}\cdot\dpartf{\param}{\fstate_{t-1}}\cpl{\stateopt_0}{\param}+\dpartf{\param}{\opevol_t}\cpl{\state_{t-1}}{\param}\\
\ddpartf{\param}{\fstate_t}\cpl{\stateopt_0}{\param}
&=\dpartf{\state}{\opevol_t}\cpl{\state_{t-1}}{\param}\cdot\ddpartf{\param}{\fstate_{t-1}}\cpl{\stateopt_0}{\param}
+\ddpartf{\state}{\opevol_t}\cpl{\state_{t-1}}{\param}\cdot\paren{\dpartf{\param}{\fstate_{t-1}}\cpl{\stateopt_{0}}{\param}}^{\otimes 2}\\
&+2\,\frac{\partial^2\,\opevol_t}{\partial\param\partial\state}\cpl{\state_{t-1}}{\param}\cdot\cpl{\dpartf{\param}{\fstate_{t-1}}\cpl{\stateopt_{0}}{\param}}{\id}
+\ddpartf{\param}{\opevol_t}\cpl{\state_{t-1}}{\param}
\ea
\end{equation*}
\begin{equation*}
\ba
\dtpartf{\param}{\fstate_t}\cpl{\stateopt_0}{\param}
&=\dpartf{\state}{\opevol_t}\cpl{\state_{t-1}}{\param}\cdot\dtpartf{\param}{\fstate_{t-1}}\cpl{\stateopt_0}{\param}\\
&+\ddpartf{\state}{\opevol_t}\cpl{\state_{t-1}}{\param}\cdot\cpl{\ddpartf{\param}{\fstate_{t-1}}\cpl{\stateopt_{0}}{\param}}{\dpartf{\param}{\fstate_{t-1}}\cpl{\stateopt_{0}}{\param}}\\
&+\frac{\partial^2\,\opevol_t}{\partial\param\partial\state}\cpl{\state_{t-1}}{\param}\cdot\cpl{\ddpartf{\param}{\fstate_{t-1}}\cpl{\stateopt_{0}}{\param}}{\id}\\
&+\dtpartf{\state}{\opevol_t}\cpl{\state_{t-1}}{\param}\cdot\paren{\dpartf{\param}{\fstate_{t-1}}\cpl{\stateopt_{0}}{\param}}^{\otimes 3}\\
&+\frac{\partial^3\,\opevol_t}{\partial\param\partial\state^2}\cpl{\state_{t-1}}{\param}\cdot\cpl{\paren{\dpartf{\param}{\fstate_{t-1}}\cpl{\stateopt_{0}}{\param}}^{\otimes 2}}{\id}\\
&+2\,\ddpartf{\state}{\opevol_t}\cpl{\state_{t-1}}{\param}\cdot\cpl{\ddpartf{\param}{\fstate_{t-1}}\cpl{\stateopt_{0}}{\param}}{\dpartf{\param}{\fstate_{t-1}}\cpl{\stateopt_{0}}{\param}}\\
&+2\,\frac{\partial^3\,\opevol_t}{\partial\param\partial\state^2}\cpl{\state_{t-1}}{\param}\cdot\cpl{\paren{\dpartf{\param}{\fstate_{t-1}}\cpl{\stateopt_{0}}{\param}}^{\otimes 2}}{\id}\\
&+3\,\frac{\partial^3\,\opevol_t}{\partial\param^2\partial\state}\cpl{\state_{t-1}}{\param}\cdot\cpl{\dpartf{\param}{\fstate_{t-1}}\cpl{\stateopt_{0}}{\param}}{\id^{\otimes 2}}\\
&+2\,\frac{\partial^3\,\opevol_t}{\partial\param^2\partial\state}\cpl{\state_{t-1}}{\param}\cdot\cpl{\ddpartf{\param}{\fstate_{t-1}}\cpl{\stateopt_{0}}{\param}}{\id}\\
&+\dtpartf{\param}{\opevol_t}\cpl{\state_{t-1}}{\param}.
\ea
\end{equation*}
By \recor{boundedJ} with $E_t=0$ for all $t\geq 1$, the $\dpartf{\param}{\fstate_t}\cpl{\stateopt_0}{\cdot}$'s are bounded uniformly in $t$ over $\boctopt$.

Now, we see the update equation on the $\ddpartf{\param}{\fstate_t}\cpl{\stateopt_0}{\cdot}$'s has the form:
\begin{equation*}
\ddpartf{\param}{\fstate_t}\cpl{\stateopt_0}{\param}
=\dpartf{\state}{\opevol_t}\cpl{\state_{t-1}}{\param}\cdot\ddpartf{\param}{\fstate_{t-1}}\cpl{\stateopt_0}{\param}
+B_t\paren{\param},
\end{equation*}
where $B_t\paren{\param}$ is made up of terms bounded uniformly in $t$ over $\boctopt$. Moreover, the $\dpartf{\state}{\opevol_t}\cpl{\state_{t-1}}{\param}$'s have spectral radius $1-\alpha$. Then, thanks to \reprop{specraditer}, we obtain that the $\ddpartf{\param}{\fstate_t}\cpl{\stateopt_0}{\cdot}$'s are uniformly bounded in $t$ on $\boctopt$.

Finally, the same reasoning applies to the $\dtpartf{\param}{\fstate_t}\cpl{\stateopt_0}{\cdot}$'s.
\paragraph{Bounding the derivatives of the $\sbpermc{t}$'s.}
For all $t\geq 1$ and $\param\in\Param$, we have
\begin{equation*}
{\sbpermc{t}}\cpl{\stateopt_0}{\param}
={\perte_t}\paren{\fstate_t\cpl{\stateopt_0}{\param}}.
\end{equation*}
Let us write as before, for all $t\geq 1$, $\state_t=\fstate_t\cpl{\stateopt_0}{\param}$.
Let $t\geq 1$, and $\param\in\boctopt$. Then
\begin{equation*}
\ba
\dpartf{\param}{\sbpermc{t}}\cpl{\stateopt_0}{\param}
&=\dpartf{\state}{\perte_t}\paren{\state_t}\cdot\dpartf{\param}{\fstate_t}\cpl{\stateopt_0}{\param}
\\
\ddpartf{\param}{\sbpermc{t}}\cpl{\stateopt_0}{\param}
&=\dpartf{\state}{\perte_t}\paren{\state_t}\cdot\ddpartf{\param}{\fstate_t}\cpl{\stateopt_0}{\param}
+\ddpartf{\state}{\perte_t}\paren{\state_t}\cdot\paren{\dpartf{\param}{\fstate_t}\cpl{\stateopt_0}{\param}}^{\otimes 2}\\
\dtpartf{\param}{\sbpermc{t}}\cpl{\stateopt_0}{\param}
&=\dpartf{\state}{\perte_t}\paren{\state_t}\cdot\dtpartf{\param}{\fstate_t}\cpl{\stateopt_0}{\param}\\
&+3\,\ddpartf{\state}{\perte_t}\cpl{\state_t}{\param}\cdot\cpl{\ddpartf{\param}{\fstate_t}\cpl{\stateopt_0}{\param}}{\dpartf{\param}{\fstate_t}\cpl{\stateopt_0}{\param}}\\
&+\dtpartf{\state}{\perte_t}\cpl{\state_t}{\param}\cdot\paren{\dpartf{\param}{\fstate_t}\cpl{\stateopt_0}{\param}}^{\otimes 3}.
\ea
\end{equation*}
Thanks to the first part of the proof, all the derivatives of the
$\fstate_t\cpl{\stateopt_0}{\cdot}$'s are uniformly bounded in $t$ on
$\boctopt$. By assumption, all the derivatives up to third order of the
$\perte_t$'s are bounded on the stable tube. Therefore, all the derivatives up to third order of the $\sbpermc{t}\cpl{\stateopt_0}{\cdot}$'s are uniformly bounded in $t$ on $\boctopt$. Finally, this shows the family of functions $\ddpartf{\param}{\sbpermc{t}}\cpl{\stateopt_0}{\cdot}$ is equicontinuous on $\boctopt$. 
\end{proof}

\begin{lemme}[Equicontinuity satisfied for the extended Hessians with all derivatives bounded]
\label{lem:eqfhwdb}
Assume that all first, second and third order derivatives of the
$\opevol_t$'s and of the $\sbperte_t$'s exist, and are bounded on a ball
around the target trajectory $(\paramopt,\stateopt_t)$. (In particular,
we can take $\exmp=0$.)

Assume moreover that the second derivatives of the extended Hessians
$\fur_t$
are controlled as follows on the balls of \rehyp{fupdrl}. Namely,
we assume that there is a constant $\kappa_\fur>0$ such that for all
$t\geq 1$, 
for all $\vtanc\in\linform$, $\state\in B_{\State_t}$, and $\param \in
B_\Param$, we have
\begin{equation*}
\nrmop{\ddpartf{\vtanc}{\fur_t}\tpl{\vtanc}{\state}{\param}} <
\kappa_\fur,
\qquad
\nrmop{\frac{\partial^2\,\fur_t}{\partial\vtanc\,\partial\cpl{\state}{\param}}\tpl{\vtanc}{\state}{\param}}<\kappa_\fur,
\end{equation*}
and
\begin{equation*}
\nrmop{\ddpartf{\cpl{\state}{\param}}{\fur_t}\tpl{\vtanc}{\state}{\param}}
\leq \kappa_\fur \paren{1+\nrm{\vtanc}}.
\end{equation*}
Then, \rehyp{equicontH} is satisfied on $\boctopt$.
\end{lemme}

Notably, these assumptions on $\fur_t$ cover the preconditioned case
$\fur_t(\vtanc,\state,\param)=P_t(\state,\param)\cdot \vtanc$ with smooth
enough $P_t$.

\begin{proof}
Thanks to \relem{contderivfstate} above, we know the $\ddpartf{\param}{\sbperte_{\leadsto t}}\cpl{\stateopt_0}{\cdot}$'s are equicontinuous on $\boctopt$.

For $t\geq 0$ and $\param\in\Param$, let us write $g_t\paren{\param}=\tpl{\frac{\partial}{\partial \param}\sbpermc{t}(\stateopt_0,\param)}{\fstate_t(\stateopt_0,\param)}{\param}$. Then, for $\param\in\Param$, we have $\fht{\param}=\dpartf{\param}{\fur_t\circ g_t}\paren{\param}$, and (writing $g=\tpl{\vtanc}{\state}{\param}$)
\begin{equation}
\ba
{\sbfh_t}\paren{\param}-\sbfh_t\paren{\paramopt}
&={\dpartf{g}{\fur_t}\paren{g_t\paren{\param}}}\cdot\dpartf{\param}{g_t}\paren{\param}
-{\dpartf{g}{\fur_t}\paren{g_t\paren{\paramopt}}}\cdot\dpartf{\param}{g_t}\paren{\paramopt}\\
&=\paren{{\dpartf{g}{\fur_t}\paren{g_t\paren{\param}}}-{\dpartf{g}{\fur_t}\paren{g_t\paren{\paramopt}}}}\cdot\dpartf{\param}{g_t}\paren{\param}\\
&-{\dpartf{g}{\fur_t}\paren{g_t\paren{\paramopt}}}\cdot\paren{\dpartf{\param}{g_t}\paren{\paramopt}-\dpartf{\param}{g_t}\paren{\param}}.
\ea
\label{eq:fcyhesstwoterms}
\end{equation}
Moreover, for all $\param\in\Param$, we have
\begin{equation}
\label{eq:gradgt}
\dpartf{\param}{g_t}\paren{\param}=
\tpl{\frac{\partial^2}{\partial \param^2}\sbpermc{t}(\stateopt_0,\param)}{\dpartf{\param}{\fstate_t}(\stateopt_0,\param)}{\id}.
\end{equation}
Let us first control the second term of \eqref{eq:fcyhesstwoterms}. We have
\begin{multline*}
\nrm{{\dpartf{g}{\fur_t}\paren{g_t\paren{\paramopt}}}\cdot\paren{\dpartf{\param}{g_t}\paren{\paramopt}-\dpartf{\param}{g_t}\paren{\param}}}
\leq \nrm{{\dpartf{\vtanc}{\fur_t}\paren{g_t\paren{\paramopt}}}\cdot\paren{\frac{\partial^2}{\partial \param^2}\sbpermc{t}(\stateopt_0,\param)-\frac{\partial^2}{\partial \param^2}\sbpermc{t}(\stateopt_0,\paramopt)}}\\
+\nrm{{\dpartf{\cpl{\state}{\param}}{\fur_t}\paren{g_t\paren{\paramopt}}}\cdot\cpl{\dpartf{\param}{\fstate_t}(\stateopt_0,\param)-\dpartf{\param}{\fstate_t}(\stateopt_0,\paramopt)}{0}}\\
\leq \nrmop{\dpartf{\vtanc}{\fur_t}\paren{g_t\paren{\paramopt}}}\,\nrm{\frac{\partial^2}{\partial \param^2}\sbpermc{t}(\stateopt_0,\param)-\frac{\partial^2}{\partial \param^2}\sbpermc{t}(\stateopt_0,\paramopt)}
\\
+\go{1+\nrm{\dpartf{\param}{\sbpermc{t}}\cpl{\stateopt_0}{\paramopt}}}\,\nrm{\dpartf{\param}{\fstate_t}(\stateopt_0,\param)-\dpartf{\param}{\fstate_t}(\stateopt_0,\paramopt)}.
\end{multline*}
by the control of $\partial\fur_t/\partial (\state,\param)$ in
\rehyp{fupdrl}.
Now, the operator norm of $\dpartf{\vtanc}{\fur_t}$ is bounded on the
stable tube thanks to
\rehyp{fupdrl}. Moreover, thanks to \relem{contderivfstate}, the third
derivatives of $\sbpermc{t}$ and the second derivatives of $\fstate_t$
are bounded, so that the differences between $\param$ and $\paramopt$ in
this expression are
$\go{\param-\paramopt}$.
As a result, on the stable tube, the second term of \reeq{fcyhesstwoterms} is bounded by some $\rho_2\paren{\nrm{\param-\paramopt}}$.

Let us now control the first term of \reeq{fcyhesstwoterms}.
We have
\begin{multline*}
\nrm{\paren{{\dpartf{g}{\fur_t}\paren{g_t\paren{\param}}}-{\dpartf{g}{\fur_t}\paren{g_t\paren{\paramopt}}}}\cdot\dpartf{\param}{g_t}\paren{\param}}
\leq \nrmop{{\dpartf{g}{\fur_t}\paren{g_t\paren{\param}}}-{\dpartf{g}{\fur_t}\paren{g_t\paren{\paramopt}}}}
\,\nrm{\dpartf{\param}{g_t}\paren{\param}}.
\end{multline*}
Now, \reeq{gradgt} together with \relem{contderivfstate} show that the gradients of the $g_t$'s are bounded on $\boctopt$. 
Next, thanks to our assumptions on the second derivatives of $\fur_t$, by
decomposing $g_t$, we have
(with suprema taken on the stable tube)
\begin{multline*}
\nrmop{{\dpartf{g}{\fur_t}\paren{g_t\paren{\param}}}-{\dpartf{g}{\fur_t}\paren{g_t\paren{\paramopt}}}}\\
\leq \paren{\sup\nrmop{\ddpartf{\vtanc}{\fur_t}}}\,\nrm{g_t\paren{\param}-g_t\paren{\paramopt}}
+\paren{\sup_{\param'\in\boctopt}\,\nrmop{\ddpartf{\cpl{\state}{\param}}{\fur_t}\paren{g_t\paren{\param'}}}}\,\nrm{g_t\paren{\param}-g_t\paren{\paramopt}}\\
+2\,{\sup_{\param'\in\boctopt}\,\nrmop{\frac{\partial^2\,\fur_t}{\partial\vtanc\,\partial\cpl{\state}{\param}}\paren{g_t\paren{\param'}}}}\,\nrm{g_t\paren{\param}-g_t\paren{\paramopt}}\\
\leq \paren{\sup\,\nrmop{\ddpartf{\vtanc}{\fur_t}}}\,\nrm{g_t\paren{\param}-g_t\paren{\paramopt}}
+\go{1+\sup_{\boctopt}\nrm{\dpartf{\param}{\sbpermc{t}\cpl{\stateopt_0}{\cdot}}}}\,\nrm{g_t\paren{\param}-g_t\paren{\paramopt}}\\
+2\,{\sup_{\param'\in\boctopt}\,\nrmop{\frac{\partial^2\,\fur_t}{\partial\vtanc\,\partial\cpl{\state}{\param}}\paren{g_t\paren{\param'}}}}\,\nrm{g_t\paren{\param}-g_t\paren{\paramopt}}.
\end{multline*}
Now, by assumption, the second derivative of $\fur_t$ with respect to
$\vtanc$, and its cross-derivative with respect to $\vtanc$ and
$(\state,\param)$, are both bounded on the stable tube. Moreover, by \relem{contderivfstate}, the first derivative of $\sbpermc{t}\cpl{\stateopt_0}{\cdot}$ is bounded on $\boctopt$, while the $g_t$'s are equicontinuous on the same ball. As a result, on the stable tube, the first term of \reeq{fcyhesstwoterms} is bounded by some $\rho_1\paren{\nrm{\param-\paramopt}}$.

Gathering the controls of the two terms of \reeq{fcyhesstwoterms} we have obtained, we see that the extended Hessians are indeed equicontinuous on $\boctopt$, so that \rehyp{equicontH} is indeed satisfied on this ball.
\end{proof}

\bibliography{./biblio.bib}

\end{document}